\documentclass[a4paper, 12pt]{article}
 
\usepackage[top=2cm, bottom=2cm, left=2cm, right=2cm]{geometry}
\usepackage{graphicx, array, amsmath, amsthm, mathrsfs, framed, stmaryrd, marvosym, bbm, amssymb, mathtools, verbatim, tikz}
\usepackage[all]{xy}
\tikzset{node distance=2cm, auto}
\usepackage{enumerate}
\usepackage{color}
\usepackage{youngtab}
\allowdisplaybreaks

\input xy
\xyoption{matrix}\xyoption{arrow}
\def\edge{\ar@{-}}
\def\dedge{\ar@{.}}

\long\def\ignore#1{#1}

\newcommand{\st}{{\rm st}}

\newtheorem{theorem}{Theorem}[section]

\newtheorem{proposition}[theorem]{Proposition}

\newtheorem{lemma}[theorem]{Lemma}
\newtheorem{corollary}[theorem]{Corollary}

\theoremstyle{definition}
\newtheorem{definition}[theorem]{Definition}
\newtheorem{example}[theorem]{Example}
\newtheorem{remark}[theorem]{Remark}

\newtheorem{conjecture}[theorem]{Conjecture}
\newtheorem{notation}[theorem]{Notation}
\newtheorem{convention}[theorem]{Convention}

\def\f{\mathbb{F}}

\def\k{{\mathbb K}}

\def\row{{\rm row}}
\def\col{{\rm col}}
\def\spec{{\rm Spec}}
\def\hspec{\ch\!-\!\spec} 
\def\cdiag{{\rm CD}}

\def\cf{{\mathcal F}}
\def\ci{{\mathcal I}}
\def\ch{{\mathcal H}}

\def\cl{{\mathcal L}}
\def\cm{{\mathcal M}}
\def\cn{{\mathcal N}}
\def\co{{\mathcal O}}
\def\cp{{\mathcal P}}
\def\cq{{\mathcal Q}}
\def\cs{{\mathcal S}}

\def\mc{{\mathbb C}}
\def\mn{{\mathbb N}}

\def\mr{{\mathbb R}}
\def\mz{{\mathbb Z}}

\newcommand{\R}{\mathbb{R}}
\newcommand{\N}{\mathbb{N}}

\newcommand{\Z}{\mathbb{Z}}

\newcommand{\ang}[1]{\langle #1 \rangle}

\newcommand{\llb}{\llbracket}
\newcommand{\rrb}{\rrbracket}

\newcommand*{\TakeFourierOrnament}[1]{{%
\fontencoding{U}\fontfamily{futs}\selectfont\char#1}}
\newcommand*{\danger}{\TakeFourierOrnament{66}}
\newcommand{\intint}[1]{\llb #1 \rrb}
\newcommand{\piecewise}[5]{ #1 = \begin{cases} #2 & #3 \\ #4 & #5 \end{cases} }
\newcommand{\piecewiseseven}[7]{ #1 = \begin{cases} #2 & #3 \\ #4 & #5 \\ #6 & #7 \end{cases} }

\newcommand{\q}{\medskip\par\noindent}

\def\oq{\mathcal{O}_{q}}

\def\oq{{\cal O}_q}

\def\oqmmnf{\co_q(M_{m,n}(\f))}
\def\oqmmnk{\co_q(M_{m,n}(\k))}
\def\oqmmnr{\co_q(M_{m,n}(R))}

\def\oqmm13{\co_q(M_{1,3})}
\def\oqm23{\co_q(M_{2,3})}

 \def\oqmnk{{\co_q(M_n(\k))}}
 \def\oqmnr{{\co_q(M_n(R))}}

\def\oqglnk{\co_q(GL_n(\k))}

\def\ylk{\co_q(Y_\lambda(\k))}
\def\ylr{\co_q(Y_\lambda(R))} 
\def\ylf{\co_q(Y_\lambda(\f))}
\def\yldashf{\co_q(Y_{\lambda'}(\f))}
\def\yldashr{\co_q(Y_{\lambda'}(R))}

\def\oqgmnr{\oq(G_{mn}(R))}
\def\oqgmnf{\oq(G_{mn}(\f))}
\def\oqgmnk{\oq(G_{mn}(\k))}

\def\oqgmnkone{\oq(G_{mn}(K_1))} 
\def\oqgmnktwo{\oq(G_{mn}(K_2))}
\def\oqgmnki{\oq(G_{mn}(K_i))}

\def\sgr{S(\gamma)_R}
\def\sgf{S(\gamma)_{\f}}
\def\sgro{S^0(\gamma)_R}
\def\sgfo{S^0(\gamma)_{\f}}

\def\qdot{(-q)^{\bullet}}

\def\ogmnr{{G_{mn}(\mr)}}
\def\ogmntnn{{G_{mn}^{{\rm tnn}}(\mr)}}

\def\widebar{\overline}
\def\pcl{\Pi_{\lambda}^{C}}
\def\pcldash{\Pi_{\lambda'}^{C'}}
\def\Pcl{P_{\lambda}^{C}}

\def\idealpclr{\langle\Pi_{\lambda}^{C}\rangle_R}

\def\idealPclr{\langle\Pcl\rangle_R}
\def\idealpclrdash{\langle\Pi_{\lambda'}^{C'}\rangle_R}
\def\idealpclfdash{\langle\Pi_{\lambda'}^{C'}\rangle_{\f}}
\def\inoti{I\backslash\{i\}}
\def\inotiwithj{I\backslash\{i\}\sqcup\{j\}}
\def\jwithi{J\sqcup\{i\}} 
\def\setn{\{1,\dots,n\}}

\def\goesto{\longrightarrow}
\def\mapsto{\leadsto} 
\newcommand{\ideal}[1]{\langle#1\rangle}

\def\gqtwofour{\mathcal{O}_q(G_{24}(\k))}

\def\postc{{\rm Post}(C)}
\def\neck{{\rm Neck}}

\def\ker{\rm ker}
\def\goesto{\longrightarrow}

\def\Spec{{\rm Spec}}

\def\lessi{<_i}
\def\leqqi{\leq_i}
\def\geqqi{\geq_i}

\def\lessone{<_1}
\def\leqqone{\leq_1}

\def\lesstwo{<_2}
\def\leqqtwo{\leq_2}

\newcommand{\la}{\lambda}

\newcommand{\id}{\text{id}}

\newcommand{\bs}{\setminus}
\newcommand{\be}{\begin{equation}}
\newcommand{\ee}{\end{equation}}

\newcommand{\K}{\mathbb{K}}

\newcommand{\ol}{\overline}
\newcommand{\ul}{\underline}

\newcommand{\Le}{\mathbin{\rotatebox[origin=c]{180}{$\Gamma$}}}

\newcommand{\lex}{\tx{lex}}

\newcommand{\qmatr}{\co_q(M_{m,n}(R))}

\newcommand{\ladder}{\co_q(M_{m,n,\gamma}(\f))}
\newcommand{\ladderk}{\co_q(M_{m,n,\gamma}(\k))}
\newcommand{\partitionk}{\co_{q^{-1}}(Y_{\lambda}(\k))}
\newcommand{\partition}{\co_{q^{-1}}(Y_{\lambda}(\f))}
\newcommand{\partitionr}{\co_{q^{-1}}(Y_{\lambda}(R))}

\newcommand{\qimf}{\co_{q^{-1}}(M_{m,n-m}(\f))}

\newcommand{\oqmmnminusmk}{\co_{q^{-1}}(M_{m,n-m}(\mathbb{K}))}

\DeclareMathOperator*{\Dhom}{Dhom}
\newcommand{\wt}{\widetilde}
 \def\mat{{M_{m,n}({\mz})}}
 \def\matnonneg{{M_{m,n}({\mz_{\geq 0})}}}
 \newcommand{\vect}{\boldsymbol}
 \def\lm{{\rm LM}}
  \def\lexp{{\rm LE}}

\newcommand{\tx}{\text}
\newcommand{\ds}{\displaystyle}
\newcommand{\ts}{\textstyle}

\usepackage{verbatim}

\def\oqgmnkone{\oq(G_{mn}(K_1))} 
\def\oqgmnktwo{\oq(G_{mn}(K_2))}
\def\oqgmnki{\oq(G_{mn}(K_i))}

\newcommand{\gc}{ [ \hspace{-0.65mm} [}
\newcommand{\dc}{]  \hspace{-0.65mm} ]}

\usepackage{version}

\includeversion{full} 

\title{Total positivity is a quantum phenomenon: the grassmannian case} 
\author{S Launois\thanks{\,The research of the first named author was supported by EPSRC grant EP/N034449/1.},~ T H Lenagan\thanks{\,The research of the second named author was partially supported by a Leverhulme Trust Emeritus Fellowship.} ~and B M Nolan}
\date{}

\begin{document} 
\maketitle
\begin{abstract} The main aim of this paper is to establish a deep link between the totally nonnegative grassmannian and the quantum grassmannian. More precisely, under the assumption that the deformation parameter $q$ is transcendental, we show that ``quantum positroids'' are completely prime ideals in the quantum grassmannian $\oqgmnf$. As a consequence, we obtain that torus-invariant prime ideals 
in the quantum grassmannian are generated by polynormal sequences of quantum Pl\"ucker coordinates and give a combinatorial description of these generating sets. We also give a topological description of the poset of torus-invariant prime ideals in  $\oqgmnf$, and prove a version of the orbit method for torus-invariant objects. Finally, we construct separating Ore sets for all torus-invariant primes in  $\oqgmnf$. The latter is the first step in the Brown-Goodearl strategy to establish the orbit method for (quantum) grassmannians.  
\end{abstract}





\section{Introduction} 

The quantum grassmannian $\oqgmnf$ is a noncommutative algebra that is a deformation of the homogeneous coordinate ring of the grassmannian variety of $m$-dimensional subspaces in $n$-dimensional space. As such, it is generated as an $\f$-algebra by the so-called quantum Pl\"ucker coordinates. One important goal in the study of the quantum grassmannian is to understand the structure of the prime spectrum. This study is aided by the presence of a natural action of a torus group $\ch:=(\f^*)^n$ on the quantum grassmannian for which the quantum Pl\"ucker coordinates are eigenvectors. The Goodearl-Letzter stratification theory, see \cite{bg-book}, suggests that one should first understand the prime ideals that are invariant under the action of $\ch$ before going on to a more detailed study of the whole prime spectrum. \\

In earlier work \cite{llr-grass}, the first two authors and Rigal have shown that the $\ch$-prime ideals of $\oqgmnf$ are parameterised by Cauchon diagrams on Young diagrams that fit into an $m\times (n-m)$ rectangular array. These diagrams first appear in work by Cauchon \cite{cauchon2}  in the study of the prime spectrum of quantum matrices. Remarkably, the same  diagrams also appear in the ground-breaking work of Postnikov \cite{post} on the positroid cell stratification of the totally nonnegative grassmannian, under the name of Le-diagrams. These diagrams will play a key role in this paper, and we propose to call them Cauchon-Le diagrams in recognition of their two independent appearances in the work of Cauchon and Postnikov. \\

The positroid cell stratification of the totally nonnegative grassmannian has been intensively studied over the last dozen or so years following Postnikov's paper, see, for example, \cite{arw,kls,lam,oh,tal}. Besides its intrinsic beauty, there are also
applications in the study of partial differential equations \cite{kw}, scattering amplitudes \cite{arkani}, and juggling \cite{kls}. \\

 On seeing Postnikov's paper, the first two authors observed that the Le diagrams that Postnikov introduced to parameterise the positroid cells were the same as the diagrams introduced by Cauchon to study the $\ch$-prime spectrum of quantum matrices. This lead to several papers investigating this connection, culminating in the present work. In the first paper in the series, \cite{llr-grass}  it was shown that the positroid cells of the totally nonnegative grassmannian were in natural bijection with the $\ch$-prime spectrum of the quantum grassmannian via what we are now calling Cauchon-Le diagrams. We also conjectured at the time that this bijection would be a homeomorphism between the partially ordered sets provided by the positroid cells under closure and the $\ch$-prime spectrum under inclusion. Further,  we conjectured that the $\ch$-primes would be generated by the the quantum Pl\"ucker coordinates that they contain, and that the containment of a quantum Pl\"ucker coordinate in an $\ch$-prime ideal could be read off from the Postnikov graph. All of these conjectures have now been answered in the present work. However, at the time, we did not have the tools in the quantum setting to verify these conjectures. One setting where we were able to make progress was for quantum matrices, which occur as the ``big cell'' in the quantum grassmannian. The first author had already shown that, in the case of a transcendental deformation parameter $q$, each $\ch$-prime ideal was generated by the quantum minors that it contained, \cite{launois-generation}, verifying a conjecture of Goodearl and the second author, \cite{gl-winding}. Yakimov also produced a proof of this result, \cite{yak-invariant-primes}, and Casteels replaced the transcendental restriction by the natural condition that $q$ be a non root of unity, \cite{cas2}. Casteels' use of (noncommutative) Gr\"obner basis techniques was crucial to our present work. 

In two papers with Goodearl, the first two authors were able to show that  the membership problem for quantum minors in $\ch$-prime ideals in quantum matrices exactly corresponded to the corresponding problem for minors belonging to a positroid cell, \cite{gl1, gl2}.

The main aim of this article is to prove that $\ch$-primes in quantum grassmannians are generated (as left or right ideals) by quantum Pl\"ucker coordinates and to describe explicitly which quantum Pl\"ucker coordinates belong to a given $\ch$-prime. This aim is achieved under the assumption that the deformation parameter is transcendental. \\

\begin{theorem}\label{theorem-main-intro}
 Let $\f$ be a field, and let $q$ be a nonzero element of $\f$. Assume that $q$ is transcendental over the prime field of $\f$. 
 
Then every $\ch$-prime of $\oqgmnf$ is generated by an explicit polynormal sequence of quantum Pl\"ucker coordinates.  
\end{theorem}

Once our main result is established, we discuss various applications. In particular, we discuss the link between positroid cells in totally nonegative grassmannians and $\ch$-primes in quantum grassmannians. This allows us to describe the poset of $\ch$-primes in quantum grassmannians, and prove that, in the spirit of the orbit method, the set of $\ch$-primes is homeomorphic to the set of torus-orbits of symplectic leaves in the corresponding grassmannian variety. Finally, we use our main result to construct separating Ore sets for all $\ch$-primes in the quantum grassmannian. These sets are central in the Brown-Goodearl strategy \cite{bg} to establish an homeomorphism between the primitive spectrum of  the quantum grassmannian and the set of symplectic leaves in the corresponding grassmannian variety.\\

Before we explain the strategy of the proof of the main result, we need to introduce necessary notation. 

Let $\f$ be a field, and let $q$ a nonzero element of $\f$. We assume that $q$ is not a root of unity. Let $m \leq n$ be positive integers and let $\oqmmnf$ denote the quantum deformation of the affine coordinate ring on $m\times n$ matrices. The quantum deformation of the homogeneous coordinate ring of the grassmannian, denoted $\oqgmnf$, is defined as the subalgebra of $\oqmmnf$ generated by the maximal quantum minors of the generic matrix of $\oqmmnf$. To simplify, these algebras will be referred to as the algebra of quantum matrices and the quantum grassmannian, respectively. Moreover, the maximal minors of $\oqmmnf$ are called quantum Pl\"ucker coordinates. \\

The starting point of this work is the study of the prime spectrum of the quantum grassmannian. This algebra is naturally endowed with the action of a torus $\ch$, and it is well known that the stratification theory as developed by Goodearl and Letzter applies to this algebra, that is the prime spectrum of the quantum grassmannian admits a partition into finitely many $\ch$-strata, each stratum being indexed by an $\ch$-prime ideal (equivalently, a prime ideal invariant under the action of $\ch$ in the cases that we will consider). \\

The finiteness of the set $\ch$-$\spec \oqgmnf$ of $\ch$-primes in $\oqgmnf$ was established in \cite{llr-grass}, where it was proved that there is a natural one-to-one correspondence between $\ch$-$\spec \oqgmnf$ and Cauchon-Le diagrams defined on Young diagrams fitting into a rectangular $m\times (n-m)$ Young diagram. This bijection will feature heavily in the present work, and so we briefly recall its construction. \\

We denote by $\Pi$ the set of quantum Pl\"ucker coordinates in $\oqgmnf$. Each quantum Pl\"ucker coordinate $\gamma$ is an $m\times m$ quantum minor of $\oqgmnf$, and so $\gamma$ is characterised by the $m$ columns chosen to form this quantum minor. Thus $\Pi$ is often identified with the set of all $m$-element subsets of the set $\gc 1,n \dc :=\{1, 2, \dots , n\}$. This allows one to define a natural partial order on $\Pi$ by
$$\gamma=[\gamma_1 < \dots < \gamma_m] \leq \gamma'=[\gamma'_1 < \dots < \gamma'_m] \mbox{ if and only if } \gamma_i \leq \gamma'_i \mbox{ for all } i \in \gc 1 , m \dc .$$ 
As an example, the partial order for $\oq(G_{3,6}(\f))$ is illustrated in Figure~\ref{figure-3x6-ordering}.

\begin{figure}[ht]
\ignore{
$$\xymatrixrowsep{2.4pc}\xymatrixcolsep{3.2pc}\def\objectstyle{\scriptstyle}
\xymatrix@!0{
 &&  [456 ]\edge[d]\\
 &&   [356 ] \edge[dl] \edge[dr]\\
 &   [346 ] \edge[dl] \edge[dr] &
 &  [256 ] \edge[dl] \edge[dr]\\
  [345 ] \edge[dr] &&   [246 ] \edge[dl] \edge[dr] \edge[d] 
&&   [156 ] 
  \edge[dl]\\
 &   [245 ] \edge[d] \edge[dr] 
 &   [236 ] \edge[dl]|\hole \edge[dr]|\hole &  [146 ]
\edge[dl] \edge[d]\\
 &  [235 ] \edge[dl] \edge[dr] 
 &  [145 ] \edge[d] &  [136 ] \edge[dl] \edge[dr]\\
  [234 ] \edge[dr] && [135 ] \edge[dl] \edge[dr] 
&& [126 ]\edge[dl]\\
 & [134 ] \edge[dr] && [125 ] \edge[dl]\\
 && [124 ] \edge[d]\\
 && [123 ]
}$$
\caption{The standard order for $\oq(G_{3,6}(\f))$.}
\label{figure-3x6-ordering}
}
\end{figure}
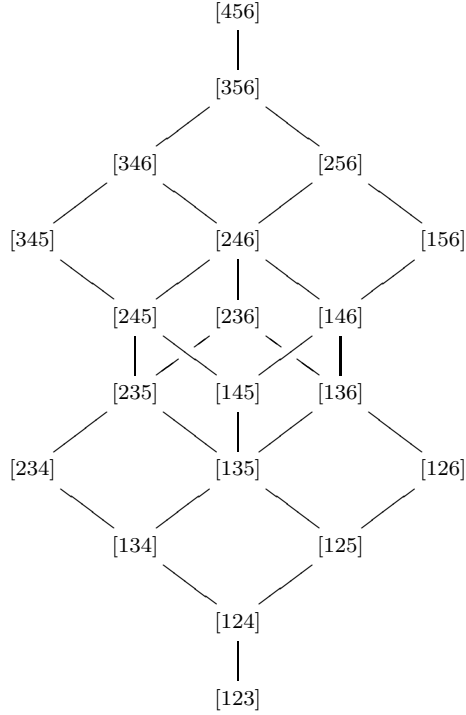

Fix a quantum Pl\"ucker coordinate $\gamma=[\gamma_1 < \dots < \gamma_m]$ of $\oqgmnf$. The poset ideal associated to $\gamma$, defined by $\Pi_{\gamma}:= \{ \alpha \in \Pi: \alpha \not\geq \gamma\}$, generates a completely prime ideal of $\oqgmnf$ \cite{lr3}. The corresponding factor algebra $S(\gamma)$ is the so-called {\em quantum Schubert variety} associated to $\gamma$ \cite{lr3}. The coset $\overline{\gamma}$ becomes a nonzero normal element of the noetherian domain $S(\gamma)$. Thus one can form the noncommutative localised algebra: $S(\gamma) [\overline{\gamma}^{-1}] $. Computation in this algebra is facilitated by the fact that it is close to being a noncommutative polynomial algebra or quantum nilpotent algebra. More precisely, the {\em noncommutative dehomogenisation} theory developed in \cite{klr} shows that 
\begin{eqnarray}\label{iso-intro}
\Phi_{\lambda}:S(\gamma) [\overline{\gamma}^{-1}] \xrightarrow{\cong}  \partition[Z^{\pm 1};\sigma],
\end{eqnarray}
where the so-called {\em partition subalgebra} $\partition$ is the subalgebra of $\qimf$ generated by the canonical generators of $\qimf$ that sit in the Young tableau $Y_{\lambda}$ associated to $\gamma$. In particular, when $\gamma =[1,\dots , m]$, the partition subalgebra coincides with $\qimf$. The Young tableau $Y_{\lambda}$ can easily be constructed from $\gamma$ as Figure~\ref{figure-135-tableau} below illustrates for $\gamma =[135]$: here, $\gamma_1,\dots, \gamma_m$ correspond to the vertical steps on a path of length $n$ with each step being either horizontal towards the left or 
vertically downwards.  \\

\begin{figure}
\begin{tikzpicture}[xscale=1, yscale=1]

\draw[color=white] (0,0) rectangle (0,0); 
\draw[color=gray] (7,2) rectangle (8,3);            
\draw[color=gray] (8,2) rectangle (9,3);            
\draw[color=gray] (9,2) rectangle (10,3);            

\draw[color=gray] (7,1) rectangle (8,2);               
\draw[color=gray] (8,1) rectangle (9,2);               

\draw[color=gray] (7,0) rectangle (8,1);            

\node[scale=1] at (10.2,2.5) {$\color{red}1$};
\node[scale=1] at (9.6,1.8) {$2$};
\node[scale=1] at (9.2,1.5) {$\color{red}3$};
\node[scale=1] at (8.6,0.8) {$4$};
\node[scale=1] at (8.2,0.5) {$\color{red}5$};
\node[scale=1] at (7.5,-0.3) {$6$};
\end{tikzpicture}
\caption{Young tableau associated to quantum Pl\"ucker coordinate $\gamma =[135]$}
\label{figure-135-tableau}
\end{figure}
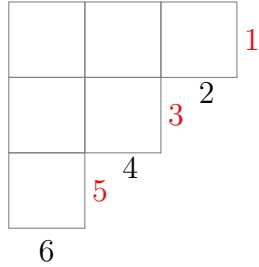

In \cite{llr-grass}, it was shown that each $\ch$-prime (other than the irrelevant ideal) in $\oqgmnf$ survives in exactly one localisation $S(\gamma) [\overline{\gamma}^{-1}]$. This leads, with some slightly abusive notation, to a partition of $\ch$-$\spec\oqgmnf$:
       $$\ch-\spec\oqgmnf =\ang{\Pi} \bigsqcup \bigsqcup_{\gamma \in \Pi } \ch-\spec\partition$$

As each partition subalgebra is a quantum nilpotent algebra (quantum nilpotent algebras also appear in the literature under the name CGL extension) \cite{llr-ufd}, Cauchon's theory of deleting derivations \cite{c1} applies to these algebras, and was used to prove that $ \ch-\spec\partition$ is in bijection with the set of Cauchon-Le diagrams that sit in the Young tableau $Y_{\lambda}$. \\

Given the above discussion, there is a clear strategy to prove that $\ch$-primes are generated by quantum Pl\"ucker coordinates. In view of the isomorphisms \eqref{iso-intro}, we should prove that the $\ch$-primes of 
$\ch-\spec\partition$
are generated by the images of the quantum Pl\"ucker coordinates under $\Phi_{\lambda}$. So one of the first tasks is to identify these images in $\partition$. This leads to the concept of {\em pseudo quantum minor} which is a generalisation of the notion of quantum minor. In the case where  $\gamma =[1,\dots , m]$, the partition subalgebra coincides with $\qimf$, and the pseudo quantum minors are exactly the quantum minors of $\qimf$. However, in general, there are more pseudo quantum minors than quantum minors. Care will be needed when dealing with pseudo quantum minors as not all of the usual formulae for quantum minors such as quantum Laplace expansions are true for pseudo quantum minors. \\

In Section \ref{Gens_for_H_primes_in_part_subalgs}, we prove that $\ch$-primes in partition subalgebras are generated as right (respectively, left) ideals by pseudo quantum minors. In order to do this, we start from the case of quantum matrices $\qimf$ and do a decreasing induction on the number of boxes in the Young tableau $Y_{\lambda}$. This part makes heavy use of Cauchon's theory of deleting derivations. \\

As a consequence, transferring through the dehomogenisation isomorphism \eqref{iso-intro}, we obtain that  $\ch$-primes in $S(\gamma) [\overline{\gamma}^{-1}]$ are generated by quantum Pl\"ucker coordinates. \\

Of course, this is not enough to conclude that  $\ch$-primes in $S(\gamma) $ are generated by quantum Pl\"ucker coordinates: we have to deal with a very difficult torsion problem! \\

In order to tackle this problem, we will make use of a result of Goodearl and the second author \cite[Proposition 2.1]{good-len-q-transc}, which will allow us to pass information from the classical setting (that is, when $q=1$) to the generic setting (that is, when $q$ is transcendental) in specific circumstances that we will detail below. In particular, to use this result, we need to have information about prime ideals in the quantum grassmannian not only over a field, but also over a commutative Laurent polynomial ring $\k[q^{\pm 1}]$, as well as information about which Pl\"ucker coordinates generate prime ideals in the homogeneous coordinate ring of the classical grassmannian. \\

Working over $\k[q^{\pm 1}]$ brings extra technicalities. For instance, we cannot use the full strength of the $\ch$-action and talk about $\ch$-primes. \\

We proceed as follows. For each Cauchon-Le diagram $C$ in $Y_{\lambda}$, we give a graph-theoretic way to recognise whether a quantum Pl\"ucker coordinate belongs to the $\ch$-prime $P_C$ of $\oqgmnf$ associated to $C$. More precisely, we make use of the Postnikov graph associated to a Cauchon-Le diagram as defined by Postnikov in \cite{post}, see also \cite{cas1,cas2}. In each white box of the Cauchon-Le diagram $C$, we put a vertex and then draw a $\Gamma$-shaped hook at each of these vertices. Labelling the Cauchon-Le diagram as in Figure~\ref{figure-135-tableau}, we obtain a planar network whose sources are indexed by $\gamma_1 , \dots , \gamma_m$ and sinks by $\gc 1,n \dc \setminus \gamma$. Figure~\ref{figure-intro-post} illustrates this construction. 

\begin{figure}
\begin{tikzpicture}[xscale=1, yscale=1]

\draw[color=white] (0,0) rectangle (0,0);   

\draw[color=gray] (7,2) rectangle (8,3);            
\draw[color=gray] (8,2) rectangle (9,3);            
\draw[color=gray] (9,2) rectangle (10,3);            

\draw[color=gray] (7,1) rectangle (8,2);               
\draw[color=gray] (8,1) rectangle (9,2);               

\draw[color=gray] (7,0) rectangle (8,1);            


\draw[fill=gray] (7,2) rectangle (8,3);               




\node[scale=1] at (10.2,2.5) {$\color{red}1$};
\node[scale=1] at (9.6,1.8) {$2$};
\node[scale=1] at (9.2,1.5) {$\color{red}3$};
\node[scale=1] at (8.6,0.8) {$4$};
\node[scale=1] at (8.2,0.5) {$\color{red}5$};
\node[scale=1] at (7.5,-0.3) {$6$};

\node at (8.5, 2.5) {$\bullet$}; 
\node at (9.5, 2.5) {$\bullet$}; 
\node at (7.5, 1.5) {$\bullet$}; 
\node at (8.5, 1.5) {$\bullet$}; 
\node at (7.5, 0.5) {$\bullet$}; 

\draw [<-, thick, black] (9.6,2.5)--(10.1,2.5); 
\draw [<-, thick, black] (8.6,2.5)--(9.4,2.5);

\draw [<-, thick, black] (8.6,1.5)--(9.1,1.5);
\draw [<-, thick, black] (7.6,1.5)--(8.4,1.5);

\draw [<-, thick, black] (7.6,0.5)--(8.1,0.5);

\draw [->, thick, black] (8.5,2.4)--(8.5,1.6);  
\draw [->, thick, black] (9.5,2.4)--(9.5,1.9);  
\draw [->, thick, black] (7.5,1.4)--(7.5,0.6);  
\draw [->, thick, black] (8.5,1.4)--(8.5,0.9);  
\draw [->, thick, black] (7.5,0.4)--(7.5,-0.1);  

\end{tikzpicture}
\caption{Example of Postnikov graph}
\label{figure-intro-post}
\end{figure}
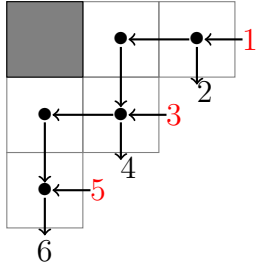

We are now ready to state the main result of Section \ref{qgrass}.

\begin{theorem}\label{theorem-main-pluckers-2-intro}
Assume that $q \in \f^*$ is not a root of unity. Let $P\neq \ang{\Pi}$ be the $\ch$-prime ideal of $\oqgmnf$ associated to the Cauchon-Le diagram $C$ in $Y_{\lambda}$.  Set $\{a_1<\cdots<a_{n-m}\}=\llb 1,n \rrb\bs\gamma$. Let $\alpha\in \Pi$ be such that $\alpha>\gamma$.  Write 
 $\alpha=[(\gamma\bs \{\gamma_{i_1},\ldots,\gamma_{i_t}\})\sqcup\{a_{j_1},\ldots,a_{j_t}\}]$ where 
$1\leq i_1<\cdots < i_t\leq m$ and $1\leq j_1<\cdots <j_t\leq n-m$ with 
$a_{j_l}>\gamma_{i_l}$ for all $l\in \llb 1,t\rrb$.

 Then the quantum Pl\"ucker coordinate $\alpha$ belongs to $P$ if and only if there does not exist a vertex-disjoint path system from sources indexed by $\gamma_{i_1},\ldots,\gamma_{i_t}$ to sinks associated to $a_{j_1},\ldots,a_{j_t}$ in the Postnikov graph of $C$.
\end{theorem} 

For instance, in the Cauchon-Le diagram $C$ of Figure \ref{figure-intro-post}, there is a vertex disjoint set of paths from $\{{\color{red}1},{\color{red}3}\}$ to $\{2,4\}$ and so Theorem \ref{theorem-main-pluckers-2-intro} shows that the quantum Pl\"ucker coordinate $[245]$ is not in the $\ch$-prime associated to $C$. 

On the other hand, there is no vertex disjoint set of paths from  $\{{\color{red}1},{\color{red}3}\}$  to $\{4,6\}$ so Theorem \ref{theorem-main-pluckers-2-intro} shows that the quantum Pl\"ucker coordinate $[456]$ is in the $\ch$-prime associated to $C$. \\

In view of Theorem \ref{theorem-main-pluckers-2-intro} we introduce the following sets of quantum Pl\"ucker coordinates: for each Cauchon-Le diagram $C$ in $Y_{\lambda}$, denote by $P^C_{\lambda}$ the set of quantum Pl\"ucker coordinates such that there does not exist a vertex-disjoint path system from sources indexed by $\gamma_{i_1},\ldots,\gamma_{i_t}$ to sinks associated to $a_{j_1},\ldots,a_{j_t}$ in the Postnikov graph of $C$. Theorem \ref{theorem-main-pluckers-2-intro} shows that $P^C_{\lambda}$ is exactly the set of quantum Pl\"ucker coordinates that belong to the $\ch$-prime associated to $C$. Of course, we would like to prove that the ideal generated by $P^C_{\lambda}$ in $\oqgmnf$ is (completely) prime. \\

We will prove this in a rather indirect way. Let $\k$ be a field, set $R:=\k[q^{\pm 1}]$ and let $\f$ be the field of fractions of $R$. All the algebras we have defined so far can be defined over $R$, and the dehomogenisation isomorphisms \eqref{iso-intro}  restrict to isomorphisms of $R$-algebras. We use a subscript $R$ to indicate that we are working with the $R$-algebra. \\

In order to prove that $P_C^{\lambda}$ generates a completely prime ideal of  $S(\gamma)_R [\overline{\gamma}^{-1}]$, we transfer this problem into a question about certain families of pseudo quantum minors generating completely prime ideals in $\partitionr$.\\

Using the deleting derivations algorithm in a way similar to that which we used when proving that 
 $\ch$-primes in $\partition$ are generated by pseudo quantum minors, the problem is reduced to proving that certain families of quantum minors generate completely prime ideals in quantum matrices over $R$. The fact that these families of quantum minors generate prime ideals in quantum matrices over a field was established by Casteels, \cite{cas2}, and we make use of the fact that they form a Gr\"obner basis to push down the result over $R$ in Section \ref{Primes in oqmmnr from Cauchon-Le diagrams}. As a consequence, we prove that the right ideal generated by $P^C_{\lambda}$ is a completely prime ideal of $S(\gamma)_R [\overline{\gamma}^{-1}]$.  \\

Set $B_R:= (\idealPclr/\langle\Pi_\gamma\rangle_R)[\overline{\gamma}^{-1}]
\cap\oqgmnr/\langle\Pi_\gamma\rangle_R$, and note that $B_R$ is completely 
prime. Set $A_R:= \idealPclr/\langle\Pi_\gamma\rangle_R$. Clearly, $A_R\subseteq B_R$, and this inclusion gives  a complex of right $R$-modules:
$$A_R \rightarrow B_R \rightarrow 0.$$
At this stage, we are almost ready to apply \cite[Proposition 2.1]{good-len-q-transc}. It only remains to prove that the complex 
$$A_R/(q-1)A_R \rightarrow B_R/(q-1)B_R\rightarrow 0$$
is exact. In order to achieve this, we use the complete primeness of $B_R$ together with results of Knutson-Lam-Speyer \cite{kls-proj,kls} that show that the sets $P^C_{\lambda}$ generate  prime ideals of the homogeneous coordinate ring of the (classical) grassmannian over an algebraically closed field (they are the vanishing ideals of the so-called {\em positroid varieties}). 

As a consequence, we obtain that $A_{\f}=B_{\f}$, so that $\ang{P_\lambda^C}_{\f}$ is completely prime in $\oqgmnf$ under the assumption that $q$ is transcendental and $\k$ is algebraically closed. The later is necessary at this stage as the results of Knutson-Lam-Speyer are only available under the assumption that $\k$ is algebraically closed. However, using ideas from \cite{gll2}, we remove this assumption to obtain Theorem \ref{theorem-main-intro}. \\

In the final section of this paper, we investigate consequences of our main result. First, we confirm that the set of quantum Pl\"ucker coordinates that belong to an $\ch$-prime of $\oqgmnf$ are precisely the complements of positroids (in the totally nonnegative grassmannian). This is a grassmannian analogue of the main result of \cite{gll1,gll2}. \\

Next, we turn our attention to the poset structure of $\hspec (\oqgmnf)$. The parallel with the totally nonnegative world still features heavily here. Indeed, we use the notion of Grassmann necklace introduced by Postnikov, \cite{post}, in the  study of positroids to give a criterion for one $\ch$-prime to be contained in another $\ch$-prime. This allows us to describe the poset structure of $\hspec (\oqgmnf)$, and in particular to prove that it is isomorphic to the poset of torus-orbits of symplectic leaves in the grassmannian. This partly answers a conjecture of Yakimov \cite{yak-qflag}. This result was certainly expected and is in the spirit of the orbit method. \\

In order to prove that the primitive spectra of quantum algebras are homeomorphic to the space of symplectic leaves of their semi-classical limit, Brown and Goodearl \cite{bg} have developed a strategy based on the $\ch$-stratification and the so-called notion of separating Ore sets. Our final main result gives an explicit construction of separating Ore sets for all $\ch$-primes in $\oqgmnf$. \\

The results of this paper have potential applications in theoretical physics, via a recently established connection between scattering amplitudes and the totally nonnegative grassmannian, see, for example, \cite{arkani}, and the very recent paper by Movshev and Schwarz, \cite{ms}, which takes this connection onwards to the quantum grassmannian as a result of our work.\\

The paper is organised as follows. In Section 2, we develop the notion of pseudo quantum minor in a partition subalgebra. Section 3 is dedicated to general results about deleting derivations. In Section 4, we give a graph-theoretic interpretation of pseudo quantum minors. These sections all feed into Section 5 where we give a graph-theoretic characterisation of those pseudo quantum minors that belong to an $\ch$-prime in a partition subalgebra (over a field). In Section 6, we prove that certain families of quantum minors generate completely prime ideals in $\oqmmnr$. We prove that $\ch$-primes in a partition subalgebra over a field are generated as right ideals by pseudo quantum minors in Section 7. In Section 8, we prove that certain families of pseudo quantum minors generate completely prime ideals in $\partitionr$. Sections 9 and 10 are dedicated to the quantum grassmannian. We prove Theorem \ref{theorem-main-pluckers-2-intro} in Section 9 and Theorem \ref{theorem-main-intro} in Section 10. Finally, Section 11 is dedicated to applications of our main results. \\

\noindent {\bf Acknowledgment:} SL and THL would like to thank MFO (Oberwolfach), CIRM (Luminy) and ICMS (Edinburgh), where parts of this work were developed over the last few years. The authors thank Thomas Lam for his explanations about positroid varieties (see discussion after Proposition 10.8).


\section{Quantum matrices and partition subalgbras thereof}


\subsection{Basic definitions and $q$-Laplace expansions} 
\label{partition subalgebra}

Let $\k$ be a field, and let $q$ a nonzero element of $\k$. 
The algebra of $m\times n$ quantum matrices over $\k$, denoted by $\oqmmnk$, is 
the algebra generated over $\k$ by 
$mn$ indeterminates 
$x_{ij}$, with $1 \le i \le m$ and $1 \le j \le n$,  which commute with the elements of 
$\k$ and are subject to the relations:
\[
\begin{array}{ll}  
x_{ij}x_{il}=qx_{il}x_{ij},&\mbox{ for }1\le i \le m,\mbox{ and }1\le j<l\le
n\: ;\\ 
x_{ij}x_{kj}=qx_{kj}x_{ij}, & \mbox{ for }1\le i<k \le m, \mbox{ and }
1\le j \le n \: ; \\ 
x_{ij}x_{kl}=x_{kl}x_{ij}, & \mbox{ for } 1\le k<i \le m,
\mbox{ and } 1\le j<l \le n \: ; \\
x_{ij}x_{kl}-x_{kl}x_{ij}=(q-q^{-1})x_{il}x_{kj}, & \mbox{ for } 1\le i<k \le
m, \mbox{ and } 1\le j<l \le n.
\end{array}
\]
Informally, we will refer to the final relation above as the {\em nasty relation}, as it is responsible for most of the difficulties in computations in $\oqmmnk$. 

When $q\in \k^*$ is not a root of unity, one may write $\oqmmnk$ as a Quantum Nilpotent Algebra (QNA for short) or Cauchon-Goodearl-Letzter extension in the sense of \cite[Definition 3.1]{llr-grass}, by adjoining the generators $x_{ij}$ in lexicographic order:
\[
\oqmmnk=\k[x_{11}][x_{12};\sigma_{12},\delta_{12}]\cdots [x_{mn};\sigma_{mn},\delta_{mn}],
\]
(here, the $\sigma_{ij}$ are automorphisms and the $\delta_{ij}$ are left $\sigma_{ij}$-derivations of the appropriate subalgebras.) The exact definition (of QNA) will be given later in Section \ref{section-cgl} .

The algebraic torus $\ch=(\k^\times)^{m+n}$ acts by automorphisms on 
$\oqmmnk$ as follows:
\[
(\alpha_1,\ldots,\alpha_m,\beta_1,\ldots,\beta_n)\cdot x_{ij}=\alpha_i\beta_j x_{ij}.
\]
 By \cite[Theorem II.2.7]{bg}, the action of $\ch$ on $\oqmmnk$ is rational in the sense of \cite[Definition II.2.6]{bg}. All actions of tori which appear in this paper can be checked to be rational using \cite[Theorem II.2.7]{bg} along with elementary arguments.

Let $R$ be a commutative noetherian domain and let $q$ be an invertible element of $R$. Let $\f$ be the field of fractions of $R$. 
(Often, $q$ will be restricted to be a non-root of unity or, even more restrictedly, a transcendental element over some base field contained in $R$. We will make the precise requirements clear whenever we need to restrict $R$ and $q$.)
The ring, $\oqmmnr$,  of quantum matrices over $R$ 
is the subring of $\oqmmnf$ generated over $R$ by the 
$x_{ij}$. 

The ring  $\oqmmnr$ is a noetherian domain, as it is an iterated Ore extension over the ring 
noetherian ring $R$.

If a ring  $A$ is generated over $R$ by elements $x_{ij}$ which satisfy 
all of the above relations except possibly the nasty relation (and maybe has 
other relations as well) then we will say 
that $A$ is a {\em pre-quantum matrix algebra}.


\begin{definition}\label{definition-partition-subalgebra} 
Let $\lambda = \{\lambda_1\geq \lambda_2\geq\dots\geq\lambda_m\}$ be a partition with associated Young diagram $Y_\lambda$. 
In $\oqmmnr$, with $n\geq \lambda_1$, look at the subring  
$\ylr$ generated 
over $R$ by those 
$x_{ij}$ that fit  into the Young diagram for $\lambda$. We call this subalgebra the {\em partition 
subalgebra of $\oqmmnr$
 associated with the partition $\lambda$}.
\end{definition} 

When $q\in \k^*$ is not a root of unity, the partition subalgebra  $\ylk$ can be presented as a QNA with the variables $x_{ij}$ added in lexicographic order, and with the torus $\ch$ acting via restriction from the action on $\oqmmnk$. Other orderings of the variables are permissible while maintaining the QNA condition, and some will be used later in the paper. 
A consequence of the QNA condition is that important tools such as Cauchon's deleting derivations procedure and the Goodearl-Letzter stratification theory are available. 
(This is the origin of the CGL terminology, see \cite{llr-ufd}.) 
These ideas will be introduced in detail later. At the moment, we merely note that it is immediate that 
partition subalgebras of $\oqmmnr$ are noetherian domains.

We can augment the generating elements in a partition subalgebra to 
obtain a pre-quantum matrix algebra of size $m\times n$ by setting 
$x_{ij}:= 0$ whenever $(i,j)$ lies outside the Young diagram $Y_\lambda$.


\begin{example}\label{def-partition-subalgebra}
 Let $\lambda =\{4,3,1\}$. Then the partition subalgebra 
associated with $\lambda$ is generated by the following variables:

\begin{center}

\begin{tikzpicture}[xscale=1, yscale=1]




\draw[color=gray] (0,2) rectangle (1,3);            
\draw[color=gray] (1,2) rectangle (2,3);            
\draw[color=gray] (2,2) rectangle (3,3);            
\draw[color=gray] (3,2) rectangle (4,3);            

\draw[color=gray] (0,1) rectangle (2,2);               
\draw[color=gray] (1,1) rectangle (1,2);               
\draw[color=gray] (2,1) rectangle (3,2);               

\draw[color=gray] (0,0) rectangle (1,1);            




\node at (0.5, 2.5) {\Large{$x_{11}$}}; 
\node at (1.5, 2.5) {\Large{$x_{12}$}}; 
\node at (2.5, 2.5) {\Large{$x_{13}$}}; 
\node at (3.5, 2.5) {\Large{$x_{14}$}}; 

\node at (0.5, 1.5) {\Large{$x_{21}$}}; 
\node at (1.5, 1.5) {\Large{$x_{22}$}}; 
\node at (2.5, 1.5) {\Large{$x_{23}$}}; 

\node at (0.5, 0.5) {\Large{$x_{31}$}}; 

\end{tikzpicture}
\end{center}
In the corresponding $3\times 4$ pre-quantum matrix algebra, 
$x_{24}=x_{32}=x_{33}=x_{34}=0$. 
\end{example}

Later, we also need the following algebras associated to a Young diagram $Y$. 
The {\em quantum affine space corresponding to the Young diagram $Y$} 
is the algebra generated over $R$ or $\k$ (as appropriate) by variables $t_{ij}$, one for each box $(i,j)$ 
of the Young diagram, subject to the following commutation relations: 
$t_{ij}t_{il} = qt_{il}t_{ij}$, for each $j<l$, and $t_{ij}t_{kj}=qt_{kj}t_{ij}$ for each 
$i<k$, while all other pairs  $t_{ij}, t_{kl}$ commute. The {\em quantum 
torus corresponding to $Y$} is the localisation of the quantum affine space corresponding to  $Y$ obtained by inverting each of the $t_{ij}$ (this is 
possible, as each $t_{ij}$ is a normal element).

\q
An {\it index pair} is a pair $(I,J)$ such
that $I \subseteq \{1,\dots,m\}$ and $J \subseteq \{1,\dots,n\}$ are subsets
with the same cardinality. Hence, an index pair is given by an integer $t$
such that $1 \le t \le \min\{m,n\}$ and ordered sets 
$I=\{i_1 < \dots < i_t\} \subseteq \{1,\dots,m\}$
and $J=\{j_1 < \dots < j_t\} \subseteq \{1,\dots,n\}$. 
In a pre-quantum matrix algebra arising from a partition subalgebra, to any such index pair
we associate the {\em pseudo quantum minor} 
\[ 
[I|J] = \sum_{\sigma\in S_t}
(-q)^{\ell(\sigma)} x_{i_1j_{\sigma(1)}} \cdots x_{i_tj_{\sigma(t)}} . 
\] 
with the convention that $x_{ij}=0$ whenever $x_{ij}$ falls outside the 
 partition. 
 
For example, in Example~\ref{def-partition-subalgebra}, 
$[12|12]=x_{11}x_{22} -qx_{12}x_{21}$, which is the usual quantum minor, 
whereas $[12|34]= -qx_{14}x_{23}$. \\


\begin{remark} 
Care is needed when manipulating pseudo quantum minors. With the 
 $x_{ij}=0$ convention used above, the variables do not form a quantum matrix. 
 For example, the nasty relation fails:  in Example~\ref{def-partition-subalgebra}
 above, as $x_{24}=0$ we see that if the nasty 
 relation were to hold then 
 $0= x_{13}x_{24} -x_{24}x_{13} = (q-q^{-1})x_{14}x_{23}$, which is a contradicion.

So, we are not free to automatically take over known results such as quantum 
Laplace expansions for quantum minors to pseudo quantum minors. 
To see this, notice that if we consider the partition subalgebra of 
$2\times 2$ quantum matrices corresponding to 
$\lambda =\{2,1\}$ and so generated by $x_{11},x_{12}, x_{21}$ and 
$x_{22}=0$, then, reasoning as above we see that the 
pseudo quantum minor $[12|12]= x_{11}0-qx_{12}x_{21} \neq 
q^{-1}x_{21}x_{12} = x_{22}x_{11} -q^{-1}x_{21}x_{12}
$. Thus, we don't have the usual quantum 
Laplace expansion on the first row, with 
the first row occurring on the right. However, there are two quantum Laplace 
row expansions which do follow easily from the definition, and these are 
noted in the next lemma. 
\end{remark}


\begin{lemma}[quantum Laplace expansion with rows]\label{row Laplace}
Let $I=\{i_1<\ldots<i_l\}\subseteq \llb 1, m\rrb$ and $J=\{j_1<\ldots<j_l\}\subseteq \llb 1,n\rrb$.
The following quantum Laplace expansions hold for pseudo quantum minors over
a pre-quantum matrix algebra arising from a partition subalgebra:
\begin{enumerate}[(1)]
\item $[I\ |\ J]=\ds{\sum_{p=1}^l}(-q)^{p-1}x_{i_1,j_p}[i_2\cdots i_l\ |\ j_1\cdots \widehat{j_p}\cdots j_l]$;
\item $[I\ |\ J]=\ds{\sum_{p=1}^l}(-q)^{l-p}[i_1\cdots i_{l-1}\ |\ j_1\cdots\widehat{j_p}\cdots j_l]x_{i_l,j_p}$.
\end{enumerate}
\begin{proof}
We treat part (1) only; part (2) is similar. Let us assume for ease of notation that $I=J=\llb 1,l\rrb$ (the proof for general $I$ and $J$ is the same but the notation is more unwieldy). We have 
\begin{align*}
[1\cdots l \mid 1\cdots l] &= \sum_{\sigma\in S_l}(-q)^{\ell(\sigma)}x_{1,\sigma(1)}\cdots x_{l,\sigma(l)} \\
                           &= \sum_{p=1}^l \sum_{\tiny
                         \begin{array}{c}
                         \sigma\in S_l \\ \sigma(1)=p
                         \end{array}
                         } (-q)^{\ell(\sigma)}x_{1,p}x_{2,\sigma(2)}\cdots x_{l,\sigma(l)} \\
                         &= \sum_{p=1}^l x_{1,p} \sum_{\tiny
                         \begin{array}{c}
                         \sigma\in S_l \\ \sigma(1)=p
                         \end{array}
                         } (-q)^{\ell(\sigma)}x_{2,\sigma(2)}\cdots x_{l,\sigma(l)}                          
\end{align*}
Set $\{a_1<\cdots<a_{l-1}\}:=\{2<\cdots<l\}$ and $\{j^p_1<\cdots<j^p_{l-1}\}:=\{1<\cdots<\widehat{p}<\cdots<l\}$, so that 
\begin{align*}
[1\cdots l \mid 1\cdots l] &= \sum_{p=1}^l x_{1,p} \sum_{\rho\in S_{l-1}}(-q)^{\ell(\rho)+p-1}x_{a_1,j_{\rho(1)}^p}\cdots x_{a_{l-1},j_{\rho(l-1)}^p} \\
                           &= \sum_{p=1}^l (-q)^{p-1} x_{1,p} \sum_{\rho\in S_{l-1}}(-q)^{\ell(\rho)}x_{a_1,j_{\rho(1)}^p}\cdots x_{a_{l-1},j_{\rho(l-1)}^p} \\
                           &= \sum_{p=1}^l (-q)^{p-1} x_{1,p} [a_1\cdots a_{l-1}\mid j_1^p\cdots j_{l-1}^p] \\
                           &= \sum_{p=1}^l (-q)^{p-1} x_{1,p} [2 \cdots l \mid 1\cdots \widehat{p} \cdots l] 
\end{align*}
\end{proof}
\end{lemma} 


We shall also need quantum Laplace expansions on columns, and the next 
result is preparation for such results. 

\begin{lemma}\label{flipped expression for pseudo q minors}
If $I=\{i_1<\ldots< i_l\}\subseteq \llb 1, m\rrb$ and $J=\{j_1<\ldots<j_l\}\subseteq \llb 1,n\rrb$. The following expression holds for pseudo quantum minors over
a pre-quantum matrix algebra arising from a partition subalgebra:
\[
[I \ |\ J]=\sum_{\sigma\in S_l}(-q)^{\ell(\sigma)}x_{i_{\sigma(1)},j_1}x_{i_{\sigma(2)},j_2}\cdots x_{i_{\sigma(l)},j_l}.
\]
\begin{proof}
Let us assume for ease of notation that $I=J=\llb 1,l\rrb$ (the proof for general $I$ and $J$ is the same but the notation is more unwieldy).
Let us set $\{1\cdots l\ |\ 1\cdots l\}=\sum_{\sigma\in S_l}(-q)^{\ell(\sigma)}x_{\sigma(1),1}\cdots x_{\sigma(l),l}$. Our claim is that $\{1\cdots l\ |\ 1\cdots l\}=[1\cdots l\ |\ 1\cdots l]$. This claim clearly holds if $l=1$ and we proceed by induction on $l$. 
We have 
\begin{align*}
[1\cdots l\ |\ 1\cdots l]&=\sum_{p=1}^l(-q)^{p-1}x_{1,p}[2\cdots l\ |\ 1\cdots \widehat{p}\cdots l]\ \ \ \tx{(Lemma \ref{row Laplace}(1))} \\
                         &=\sum_{p=1}^l(-q)^{p-1}x_{1,p}\{2\cdots l\ |\ 1\cdots \widehat{p}\cdots l\}.\ \ \ \tx{(induction hypothesis)} \\
\end{align*} 
Set $\{a_1<\cdots<a_{l-1}\}:=\{2<\cdots<l\}$ and $\{j^p_1<\cdots<j^p_{l-1}\}:=\{1<\cdots<\widehat{p}<\cdots<l\}$, so that 
\begin{align*}
[1\cdots l\ |\ 1\cdots l]&=\sum_{p=1}^l(-q)^{p-1}x_{1,p}\left( 
\sum_{\rho\in S_{l-1}}(-q)^{\ell(\rho)}x_{a_{\rho(1)},j_1^p}\cdots x_{a_{\rho(l-1)},j^p_{l-1}}
\right) \\
&=\sum_{p=1}^l\sum_{\rho\in S_{l-1}}(-q)^{p-1+\ell(\rho)}x_{1,p}x_{a_{\rho(1)},j_1^p}\cdots x_{a_{\rho(l-1)},j^p_{l-1}}.
\end{align*}
For all $j<p$ and all $s=1,\ldots,l-1$, the relations of the pre-quantum matrix algebra arising from a partition subalgebra show that $x_{1,p}$ commutes with $x_{a_{\rho(s)},j}$ and hence 
\begin{align*}
[1\cdots l\ |\ 1\cdots l]&=\sum_{p=1}^l\sum_{\rho\in S_{l-1}}(-q)^{p-1+\ell(\rho)}x_{a_{\rho(1)},1}\cdots x_{a_{\rho(p-1)},p-1}x_{1,p}x_{a_{\rho(p)},p+1}  \cdots x_{a_{\rho(l-1)},l} \\
                         &=\sum_{p=1}^t\sum_{\tiny
                         \begin{array}{c}
                         \sigma\in S_l \\ \sigma(p)=1
                         \end{array}
                         }
                         (-q)^{\ell(\sigma)}
                         x_{\sigma(1),1}\cdots x_{\sigma(l),l} \\
                         &=\sum_{\sigma\in S_l}(-q)^{\ell(\sigma)}x_{\sigma(1),1}\cdots x_{\sigma(l),l} \\
                         &=\{1\cdots l\ |\ 1\cdots l\},
\end{align*}
as required.
\end{proof}
\end{lemma}


\begin{corollary}[quantum Laplace expansion with columns]
\label{column Laplace}
Suppose that $\{i_1<\cdots <i_l\}\subseteq \llb 1,m\rrb$ and that $\{j_1<\cdots <j_l\}\subset \llb 1,n\rrb$. The following quantum Laplace expansions hold for pseudo quantum minors over
a pre-quantum matrix algebra arising from a partition subalgebra:
\begin{enumerate}[(1)]
\item $[i_1\cdots i_l\ |\ j_1\cdots j_l]=\ds{\sum_{p=1}^l}(-q)^{p-1}x_{i_p,j_1}[i_1\cdots \widehat{i_p} \cdots i_l\ |\ j_2\cdots j_l]$;
\item $[i_1\cdots i_l\ |\ j_1\cdots j_l]=\ds{\sum_{p=1}^{l}}(-q)^{l-p} [i_1\cdots \widehat{i_p} \cdots i_l\ |\ j_1\cdots j_{l-1}]x_{i_p,j_l}$.
\end{enumerate}

\begin{proof} 
\begin{full} 
Follows easily from Lemma \ref{flipped expression for pseudo q minors}, with the proof as in 
Lemma~\ref{row Laplace}.
\end{full} 
\end{proof}

\end{corollary}




\section{Quantum Nilpotent Algebras and their $\ch$-primes}\label{section-cgl}

In this section, we review the two main tools needed to study the prime spectrum of a quantum nilpotent algebra: the $\ch$-stratification of Goodearl-Letzter and Cauchon's Deleting Derivations Algorithm. 

\subsection{Quantum Nilpotent Algebras}

Throughout this section, $N$ denotes a positive integer and
 $A$ is an iterated Ore extension; that is,
\begin{equation}
A\ = \ \k[x_1][x_2;\sigma_2,\delta_2]\cdots[x_N;\sigma_N,\delta_N],
\end{equation}
 where $\sigma_j$ is an automorphism of the $\k$-algebra $A_{j-1}:=\k[x_1][x_2;\sigma_2,\delta_2]\dots[x_{j-1};\sigma_{j-1},\delta_{j-1}]$
 and  $\delta_j$ is a  $\k$-linear $\sigma_j$-derivation of
 $A_{j-1}$ for all $j\in \gc 2 ,N \dc$. In other words, $A$ is a skew polynomial ring whose multiplication is defined by:
 $$x_j a = \sigma_j(a) x_j + \delta_j(a)$$
 for all $j\in \gc 2 ,N \dc$ and $a \in A_{j-1}$. Thus $A$ is a noetherian domain.  Henceforth, we
 assume that $A$ is a quantum nilpotent algebra (a.k.a. CGL extension), as in the following definition. 

\begin{definition}[\cite{llr-ufd}]
\label{defCGL}
The iterated Ore extension $A$ is said to be a \emph{quantum nilpotent algebra (QNA for short)} if it is equipped with a rational action of a  torus $\ch=(\k^*)^l$ by $\k$-automorphisms satisfying the following conditions:
\begin{enumerate}
\label{hypofond}
\item[(i)]  The elements $x_1, \ldots, x_N$ are $\ch$-eigenvectors.
\item[(ii)] For every $j \in \gc 2,N \dc$, $\delta_j$ is a locally nilpotent $\sigma_j$-derivation of $A_{j-1}$. 
\item[(iii)] For every $j \in \gc 1,N \dc$, there exists $h_j \in \ch$ such that $(h_j\cdot)|_{A_{j-1}} = \sigma_j$ and
$h_j \cdot x_j = q_j x_j$ for some $q_j \in \k^*$ which is not a root of unity. 
\end{enumerate}
(We have omitted the condition $\sigma_j \circ \delta_j = q_j \delta_j \sigma_j$ from the original definition, as it follows from the other conditions; see, e.g., \cite[Eq. (3.1); comments, p.694]{gy16}.) From (i) and (iii), there exist scalars $\la_{j,i} \in \k^*$ such that $\sigma_j(x_i) = \la_{j,i} x_i$ for all $i < j$ in $\gc1,N \dc$.
\end{definition}

If, in addition, the subgroup of $\k^*$ generated by the $\lambda_{ji}$ is
torsion-free then we will say that $A$ is a {\em torsion-free QNA}.

For a discussion of rational actions of tori, see \cite[Chapter II.2]{bg}.

It is easy to check that all of these conditions are satisfied for partition
subalgebras with the variables ordered lexicographically.

\begin{proposition} \label{prop-partition-cgl}
Partition subalgebras of quantum matrix algebras are torsion-free QNAs when the parameter $q$ is not a root of unity.
\end{proposition}

A two-sided ideal $I$ of $A$ is said to be {\em $\ch$-invariant} if $h\cdot I=I$ for all
$h \in \ch$.  An {\em $\ch$-prime ideal} of $A$ is a proper $\ch$-invariant ideal
$J$ of $A$ such that if  $J$ contains the product of two
$\ch$-invariant ideals of $A$ then $J$ contains at least one of them. We denote
by $\ch$-$\spec(A)$ the set of all $\ch$-prime ideals of $A$. Observe
that if $P$ is a prime ideal of $A$ then
\begin{equation}
(P:\ch)\ := \ \bigcap_{h\in \ch} h\cdot P
\end{equation}
 is an $\ch$-prime ideal of
$A$. Indeed, let $J$ be an $\ch$-prime ideal of $A$. We denote
by $\spec_J (A)$ the {\em $\ch$-stratum} associated  to $J$; that is,  
\begin{equation}
\spec_J (A)=\{ P \in \spec(A) \mbox{ $\mid$ } (P:\ch)=J \}.
\end{equation}
Then the $\ch$-strata of $\spec(A)$ form a partition of $\spec(A)$
\cite[Chapter II.2]{bg}; that
is,
\begin{equation}
\label{eq:Hstratification}
 \spec(A)= \bigsqcup_{J \in \ch\mbox{-}\spec(A)}\spec_J(A).
 \end{equation}
This partition is the so-called {\em $\ch$-stratification} of $\spec(A)$.

It follows from work of Goodearl and Letzter \cite{gl2} that every $\ch$-prime ideal of $A$ is completely prime, so $\ch$-$\spec(A)$ coincides with the set of $\ch$-invariant completely prime ideals of $A$. Moreover there are at most $2^N$ $\ch$-prime ideals in $A$. As a consequence,  the prime spectrum of $A$ into a finite number of parts, the $\ch$-strata. 

For each $\ch$-prime ideal $J$ of $A$, the space $\spec_J(A)$ is homeomorphic to ${\rm Spec}(\k[z_1^{\pm 1},\ldots ,z_d^{\pm 1}])$ for some $d$ which depends on $J$ \cite[Theorems II.2.13 and II.6.4]{bg}, and the primitive ideals of $A$ are precisely the prime ideals that are maximal in their $\ch$-strata \cite[Theorem II.8.4]{bg}.


\subsection{Cauchon's deleting derivations algorithm}
\label{canonicalembedding}
\vskip 2mm

We keep the notation of the previous section. In particular,  $A= \k[x_1][x_2;\sigma_2,\delta_2]\cdots[x_N;\sigma_N,\delta_N]$ still denotes a QNA with associated torus $\ch$.

As we have seen in the previous section, the $\ch$-prime ideals of a QNA $A$ are key in studying the whole prime 
spectrum. Cauchon's deleting derivations algorithm \cite[Section 3.2]{c1}, which we summarise below, provides a powerful way of studying the $\ch$-prime ideals of $A$.

The deleting derivations algorithm constructs, for each $j\in \{ N+1, N, \dots, 2 \} $, a family 
$\{x_1^{(j)},\dots,x_N^{(j)}\}$ of elements of the division ring of fractions ${\rm Fract}(A)$ of $A$ defined as follows:

\begin{enumerate}
\item When $j=N+1$, we set $(x_1^{(N+1)},\dots,x_N^{(N+1)})=(x_1,\dots,x_N)$.
\item Assume that $j<N+1$ and that the $x_i^{(j+1)}$ ($i \in \gc 1,N \dc$) are already constructed. 
Then it follows from \cite[Th\'eor\`eme 3.2.1]{c1} that $x_j^{(j+1)} \neq 0$ and that, for each $i\in \gc 1,N \dc$, we have
$$x_i^{(j)}=\left\{ 
\begin{array}{ll}
x_i^{(j+1)} & \quad \mbox{ if }i \geq j \\
\displaystyle{\sum_{k=0}^{+\infty } \frac{(1-q_j)^{-k}}{[k]!_{q_j}} }
\delta_j^k \circ \sigma_j^{-k} (x_i^{(j+1)}) (x_j^{(j+1)})^{-k} & \quad
\mbox{ if }i < j,
\end{array} \right. $$
\end{enumerate}
where $[k]!_{q_j}=[0]_{q_j} \times \dots \times [k]_{q_j}$ with 
$[0]_{q_j}=1$ and $[i]_{q_j}=1+q_j+\dots+q_j^{i-1}$ when $i \geq 1$.

For all $j \in \gc 2,N+1 \dc$, we denote by $A^{(j)}$ the subalgebra of ${\rm Fract}(A)$ generated by the $x_i^{(j)}$; that is,
$$A^{(j)}:= \f\langle x_1^{(j)},\dots,x_N^{(j)} \rangle .$$

The following results were proved by Cauchon \cite[Th\'eor\`eme 3.2.1 and Lemme 4.2.1]{c1}. 

For $j \in \gc 2,N+1 \dc$, we have

\begin{enumerate}
\item $A^{(j)}$ is isomorphic to an iterated Ore extension of the form $$\k[y_1]\dots[y_{j-1};\sigma_{j-1},\delta_{j-1}][y_j;\tau_j]\cdots[y_N;\tau_N]$$ 
by an isomorphism that sends $x_i^{(j)}$ to $y_i$ ($1 \leq i \leq N$), where 
$\tau_j,\dots,\tau_N$ denote the $\k $-linear automorphisms  such that 
$\tau_{\ell}(y_i)=\lambda_{\ell,i} y_i$ ($1 \leq i \leq \ell$).
\item Assume that $j \neq N+1$ and set $S_j:=\{(x_j^{(j+1)})^n ~|~  n\in \mathbb{N} \}=
\{(x_j^{(j)})^n  ~|~  n\in \mathbb{N} \}$.
\\This is a multiplicative system of regular elements of $A^{(j)}$ and $A^{(j+1)}$, that satisfies the Ore condition
in $A^{(j)}$ and $A^{(j+1)}$. Moreover we have 
$$A^{(j)}S_j^{-1}=A^{(j+1)}S_j^{-1}.$$
\end{enumerate}

It follows from these results that $A^{(j)}$ is a noetherian domain, for all $j\in \gc 2,N+1 \dc$.

As in \cite{c1}, we use the following notation.

\begin{notation}
We set $\overline{A}:=A^{(2)}$ and $t_i:=x_i^{(2)}$ for all $i\in \gc 1,N \dc$.
\end{notation}

It follows from \cite[Proposition 3.2.1]{c1} that  $\overline{A}$ is a quantum affine space in the indeterminates $t_1,\dots,t_N$: it is for this reason that Cauchon used the expression ``effacement des d\'erivations". More precisely, let $\Lambda=\left( \mu_{i,j} \right) \in M_N(\k ^*)$ be the
multiplicatively antisymmetric matrix whose entries are defined as follows. 
$$\mu_{j,i}=\left\{ \begin{array}{ll}
\lambda_{j,i} & \mbox{ if }i<j \\
1 & \mbox{ if } i=j \\
\lambda_{i,j}^{-1} & \mbox{ if }i> j,
\end{array}\right.$$
where the  $\lambda_{j,i}$ with $i<j$ come from the QNA structure of $A$. Then we have \begin{equation}\overline{A}=\k_\Lambda [t_1,\dots,t_N]=\mathcal{O}_{\Lambda}(\k ^N).
,  \end{equation}
where $\k_\Lambda [t_1,\dots,t_N]=\mathcal{O}_{\Lambda}(\k ^N)$ denotes the $\k$-algebra generated by $t_1, \dots , t_n$ with relations $t_it_j = \mu_{i,j}t_jt_i$ for all $i,j$. 
 
\subsection{Canonical embedding}

The deleting derivations algorithm was used by Cauchon in order to relate the prime spectrum of a QNA $A$ to the prime spectrum of the associated quantum affine space $\overline{A}$. More precisely, Cauchon  has used this algorithm to construct embeddings 
\begin{equation}
\varphi_j:{\rm Spec}(A^{(j+1)}) \longrightarrow {\rm Spec}(A^{(j)}) \qquad {\rm for ~} j \in \gc 2,N \dc.
\end{equation} 
Recall from \cite[Section 4.3]{c1} that these embeddings are defined as follows.

Let $P \in  {\rm Spec}(A^{(j+1)})$. Then 
$$ 
\varphi_j (P) = \left\{\begin{array}{ll} 
PS_j^{-1} \cap A^{(j)} & \mbox{ if } x_j^{(j+1)} \notin P \\
g_j^{-1} \left( P/\langle x_j^{(j+1)}\rangle \right) & \mbox{ if } x_j^{(j+1)} \in P \\
\end{array}\right.
$$
where $g_j$ denotes the surjective homomorphism 
$$
g_j:A^{(j)}\rightarrow A^{(j+1)}/\langle x_j^{(j+1)}\rangle
$$ 
defined by 
$$
g_j(x_i^{(j)}):=x_i^{(j+1)} + \langle x_j^{(j+1)}\rangle.
$$
 (For more details see \cite[Lemme 4.3.2]{c1}.)  It was proved by Cauchon \cite[Proposition 4.3.1]{c1} that $\varphi_j$ induces an increasing homeomorphism from the topological space $$\{P \in  {\rm Spec}(A^{(j+1)}) \mid x_j^{(j+1)} \notin P \}$$ onto $$\{Q \in  {\rm Spec}(A^{(j)}) \mid x_j^{(j)} \notin Q \}$$ whose inverse is also an increasing homeomorphism; also, $\varphi_j$ induces an increasing homeomorphism from $$\{P \in  {\rm Spec}(A^{(j+1)}) \mid x_j^{(j+1)} \in P \}$$ onto its image by $\varphi_j$ whose inverse similarly is an increasing homeomorphism. Note however that, in general, $\varphi_j$ is not an homeomorphism from ${\rm Spec}(A^{(j+1)})$ onto its image.

Composing these embeddings, we get an embedding \begin{equation} 
\varphi:=\varphi_2 \circ \dots \circ \varphi_N : 
{\rm Spec}(A) \longrightarrow {\rm Spec}(\overline{A}), \end{equation} which is called the  \emph{canonical embedding} from ${\rm Spec}(A)$ into 
${\rm Spec}(\overline{A})$. 

The canonical embedding $\varphi$ is $\ch$-equivariant so that $\varphi (\hspec(A)) \subseteq \hspec(\overline{A})$. Interestingly, the set $\hspec(\overline{A})$ has been described by Cauchon. More precisely, for any subset $C$ of $\{1,\dots,N\}$, let $K_C$ denote the $\ch$-prime ideal of 
 $\widebar{A}$  generated by the $t_i$ with $i\in C$, that is
 \[
K_C= \langle t_i\mid i\in C\rangle .
\]
Then Cauchon proved (see \cite[Proposition 5.5.1]{c1}): 
$$\hspec(\overline{A}) = \{K_C ~|~ C \subseteq \{1,\dots,N\} \},$$
so that 
$$\varphi (\hspec(A)) \subseteq \{K_C ~|~ C \subseteq \{1,\dots,N\} \}.$$

\subsection{Cauchon diagrams}
\label{section-Cauchon diagrams for a QNA}

 We are now turning our attention to describing the set $\varphi (\hspec(A))$. 
 
A subset $C\subseteq\{1,\dots,N\}$ is said to be a {\em Cauchon diagram} for $A$ if 
\[
K_C= \langle t_i\mid i\in C\rangle \in \varphi(\hspec(A)).
\]
In this case, $J_C$ denotes the (unique) $\ch$-prime ideal of $A$ such that $\varphi(J_C)=K_C$. 

Cauchon proved that 
$$\hspec(A) = \{ J_C ~|~ C \mbox{ is a Cauchon diagram}\}. $$ 

If $J $ is an $\ch$-prime ideal of $A$, we denote by $\cdiag(J)$ its associated Cauchon diagrams; that is, $\cdiag(J)$ is the set $C$ such that 
$\varphi(J)=K_C=\ideal{t_i\mid i\in C}$. Note that the set of Cauchon diagrams of $A$ depends on the QNA structure of $A$ (and not just on $A$).  

A useful way to represent a Cauchon diagram $C$ is as follows. 
Draw $N$ boxes in a row, and  colour the $i$-th box black if and only 
$i\in C$; the remaining boxes are coloured white. For example, if $N=5$ and $C=\{1,2,5\}$, then we draw the diagram:

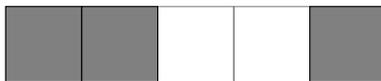
\begin{center}

\begin{tikzpicture}[xscale=1, yscale=1]


\draw[color=gray] (0,2) rectangle (1,3);            
\draw[color=gray] (1,2) rectangle (2,3);            
\draw[color=gray] (2,2) rectangle (3,3);            
\draw[color=gray] (3,2) rectangle (4,3);            
\draw[color=gray] (4,2) rectangle (5,3);            


\draw[fill=gray] (0,2) rectangle (1,3);               
\draw[fill=gray] (1,2) rectangle (2,3);               
\draw[fill=gray] (4,2) rectangle (5,3);               


\end{tikzpicture}
\end{center}

The Cauchon diagrams that arise for partition subalgebras $\ylf$ have been studied in \cite{llr-grass}.  In the next sections, we will explain the results of \cite{llr-grass} which require modifying the visual presentation of Cauchon diagrams to take advantage of the Young diagrams that are intrinsic to the presentation of partitions subalgebras.

\subsection{Some relationships between Cauchon diagrams}


Let $A:=\k[x_1][x_2;\sigma_2,\delta_2]\dots[x_N;\sigma_N,\delta_N]$  be a QNA, and let $\varphi=\varphi_2\circ\dots\circ\varphi_N: \hspec(A)\goesto\hspec(\widebar{A})$ be the canonical embedding.

Suppose that $J$ is an $\ch$-prime ideal that does not contain $x_N$. In this case, observe that $\cdiag(J)\subseteq\{1,\dots,N-1\}$. 

After the first step in the deleting derivations algorithm, we have
\[
 A^{(N)} =\k[x_1^{(N)}][x_2^{(N)};\sigma_2,\delta_2]\dots
[x_N^{(N)};\sigma_N]
\]

Set $y_i:=x_i^{(N)}$ and write $B:= \k[y_1]\dots[y_{N-1};\sigma_{N-1},\delta_{N-1}]$. 
Then $B$ is a quantum nilpotent algebra and 
$A\cong B[x_{N};\sigma_{N},\delta_{N}]$, while 
$A^{(N)}=B[y_{N};\sigma_{N}]$. (Note that $y_{N}=x_{N}$.)\\

Let $I:=\varphi_{N}(J)\cap B$ and note that $I\in\hspec(B)$. 


\begin{lemma} \label{lemma-equal-cauchon-diagrams} 

With the notation above, let $J$ be an $\ch$-prime ideal that does not contain $x_N$. Then
\[
\cdiag(I)=\cdiag(J)
\]
\end{lemma}


\begin{proof} Note that 
\[
\varphi_N(J)= \bigoplus_{i \in \mn}\, Iy_{N}^i
\]
by \cite[Corollary 2.4]{llr-ufd}.

Note that $B$ is a QNA and so we can apply the deleting-derivations algorithm to $B$. It is easy to check that $A^{(k)} = B^{(k)} [y_N; \sigma_N]$ (with a slightly abusive notation regarding the automorphism $\sigma_N$). 

Next we prove by decreasing induction that for all $k \in \{2,\dots, N\}$, we have
$$\varphi_k \circ \dots \circ \varphi_N(J)= \bigoplus_{i \in \mn}\, \psi_k \circ \dots \circ \psi_{N-1}(I)y_{N}^i,$$
where $\psi = \psi_2 \circ \dots \circ \psi_{N-1}$ is the canonical embedding of $B$. 

To simplify the notation, we set $J^{(k)}:=\varphi_k \circ \dots \circ \varphi_N(J)$ and  $I^{(k)}:=\psi_k \circ \dots \circ \psi_{N-1}(I)$. Note that $J^{(k)}$ is an $\ch$-prime of $A^{(k)}$ that does not contain $y_N$ and $I^{(k)}$ is an $\ch$-prime of $B^{(k)}$. 

The case $k=N$ has already been proved, so we assume that 
$$\varphi_{k+1} \circ \dots \circ \varphi_N(J)= \bigoplus_{i \in \mn}\, \psi_{k+1} \circ \dots \circ \psi_{N-1}(I)y_{N}^i$$
for some $k < N$, and set $t_k:= y_k^{(k)} = y_k^{(k+1)}$. To proceed with the induction step, we distinguish between two cases. 

First, if $t_k \notin J^{(k+1)}$, then $J^{(k)} = J^{(k+1)} [t_k^{-1}] \cap A^{(k)}$. Moreover, as $t_k \notin J^{(k+1)}$, we get that 
$t_k \notin I^{(k+1)} = J^{(k+1)} \cap B^{(k+1)}$. Hence $I^{(k)} = I^{(k+1)} [t_k^{-1}] \cap B^{(k)}$, so that we have:
\begin{eqnarray*}
J^{(k)} & = & J^{(k+1)} [t_k^{-1}] \cap A^{(k)} \\
& = & (\bigoplus_{i \in \mn}\, I^{(k+1)}y_{N}^i ) [t_k^{-1}] \cap A^{(k)} \\
& = & (\bigoplus_{i \in \mn}\, I^{(k+1)}[t_k^{-1}]  y_{N}^i ) \cap A^{(k)}  \mbox{ (since $t_k$ and $y_N$ $q$-commute)}\\
& = & (\bigoplus_{i \in \mn}\, I^{(k+1)}[t_k^{-1}]  y_{N}^i ) \cap (\bigoplus_{i \in \mn}\, B^{(k)}  y_{N}^i ) \\
& = & \bigoplus_{i \in \mn}\, (I^{(k+1)}[t_k^{-1}] \cap B^{(k)} ) y_{N}^i  \\
& = & \bigoplus_{i \in \mn}\, I^{(k)}  y_{N}^i \\
\end{eqnarray*}
as desired. 

Next, if $t_k \in J^{(k+1)}$, then $J^{(k)} = g_{A,k}^{-1}(\frac{J^{(k+1)}}{\langle t_k \rangle})$, where $ g_{A,k}: A^{(k)} \rightarrow \frac{A^{(k+1)}}{\langle t_k \rangle}$ is the surjective homomorphism defined by $g_{A,k}(y_i^{(k)})= y_i^{(k+1)} + \langle t_k \rangle$. 

As $t_k \in J^{(k+1)}$, we get that $t_k \in I^{(k+1)} = J^{(k+1)} \cap B^{(k+1)}$. This implies that  $I^{(k)} = g_{B,k}^{-1}(\frac{I^{(k+1)}}{\langle t_k \rangle_B})$, where $ g_{B,k}: B^{(k)} \rightarrow \frac{B^{(k+1)}}{\langle t_k \rangle_B}$ is the surjective homomorphism defined by $g_{B,k}(y_i^{(k)})= y_i^{(k+1)} + \langle t_k \rangle_B$. (Note that, to avoid any confusion, we have denoted by $\langle t_k \rangle$ the ideal of $A^{(k+1)}$ generated by $t_k$, and by $\langle t_k \rangle_B$ the ideal of $B^{(k+1)}$ generated by $t_k$.)

Observe that $\frac{A^{(k+1)}}{\langle t_k \rangle }= \bigoplus_{i \in \mn}\, \frac{B^{(k+1)}}{\langle t_k \rangle_B} \overline{y_N}^i$, where $\overline{y_N} = y_N + \langle t_k \rangle$. Hence, we can see $g_{B,k}$ as the restriction of $g_{A,k}$ to $B^{(k)}$.

Let $u \in J^{(k)}$. We can write $u = \bigoplus_{i \in \mn}\, u_i y_N^i$ with $u_i \in B^{(k)}$. We need to show that $u_i \in I^{(k)}$ for all $i$. Observe that: 
$$\bigoplus_{i \in \mn}\, g_{B,k}(u_i) \overline{y_N}^i = g(u) \in \frac{J^{(k+1)}}{\langle t_k \rangle} = \bigoplus_{i \in \mn}\, \frac{I^{(k+1)}}{\langle t_k \rangle_B} \overline{y_N}^i .$$
Hence, $g_{B,k}(u_i) \in  \frac{I^{(k+1)}}{\langle t_k \rangle_B}$ for all $i$, so that $u_i \in I^{(k)}$ for all $i$, as required. 

To summarise, we have proved that 
$$J^{(k)}= \bigoplus_{i \in \mn}\, I^{(k)} y_{N}^i$$
for all $k\in \{2, \dots, N\}$. 

The case $k=2$ allows us to conclude that $\cdiag(I)=\cdiag(J)$, since $J^{(2)} = \langle y_i^{(2)} ~|~ i \in \cdiag(J)\rangle$ and $I^{(2)} = \langle y_i^{(2)} ~|~ i \in \cdiag(I)\rangle$.

\end{proof} ~\\


\section{Cauchon-Le diagrams and Postnikov graphs}\label{Cauchon-Le diagrams and Postnikov graphs}

Before we show that each $\ch$-prime ideal in a partition subalgebra is generated by 
the pseudo quantum minors that it contains, we shall identify these pseudo quantum minors. 
There are  strong pointers to the result that we 
should expect. In \cite{cas1,cas2}, Casteels has shown that the membership problem 
for an $\ch$-prime ideal $P$ in $\oqmmnk$ is decided by the 
non-existence or existence of certain families of vertex 
disjoint paths in the Postnikov 
graph of the Cauchon diagram belonging to $P$. Exactly the same result has 
been obtained by Oh, \cite{oh}, for the membership problem for Pl\"ucker 
coordinates in positroid cells of the totally nonnegative grassmannian. In 
\cite{gll1,gll2} it is shown that the membership problem 
for cells in the space of totally nonnegative matrices and the membership 
problem for invariant prime ideals in quantum matrices have exactly the same 
solution; so this is what we conjectured happens in the grassmannian setting. 
The results in this section  achieve this hoped for result for partition 
subalgebras. This then gives the desired result for the quantum Schubert 
cells of the quantum grassmannian, and, later, we will pull these results 
back to the quantum grassmannian. 

In this section, $q\in \k$ is just assumed to be nonzero.

\subsection{Cauchon-Le diagrams and Postnikov graphs} 
\label{subsection:CauchonLeDiag-PostGraph}

We consider Young diagrams that fit into an $m\times n$ array. Each such 
Young diagram corresponds to a partition 
$\lambda = (\lambda_1, \dots, \lambda_m)$ with $\lambda_1\leq n$. For example, the Young diagram in Example~\ref{def-partition-subalgebra}
corresponds to the partition $\lambda=(4,3,1)$ and fits inside a 
$3 \times 4$ array (or, an $m\times n$ array for any $m\geq 3, n\geq 4$). We number the rows 
of Young diagrams from  top to bottom, and the columns from left to right, as 
in the usual matrix notation. 

A {\em Cauchon-Le diagram} $C$ on the Young diagram $Y$ is a colouring of the squares 
of $Y$ with black or white, with the restriction that  
if a square is coloured black then either each square to its left and in the same row is also black, or, each square above and in the same column is black. The set of black squares in $C$ will be denoted by $B_C$ and the set of white squares of $C$ by $W_C$. Here is an example of a Cauchon-Le diagram $C$ on the Young diagram $Y$ associated with the partition $\lambda=\{4,3,1\}$:

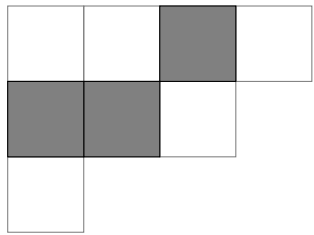
\begin{center}

\begin{tikzpicture}[xscale=1, yscale=1]


\draw[color=gray] (0,2) rectangle (1,3);            
\draw[color=gray] (1,2) rectangle (2,3);            
\draw[color=gray] (2,2) rectangle (3,3);            
\draw[color=gray] (3,2) rectangle (4,3);            

\draw[color=gray] (0,1) rectangle (2,2);               
\draw[color=gray] (1,1) rectangle (1,2);               
\draw[color=gray] (2,1) rectangle (3,2);               

\draw[color=gray] (0,0) rectangle (1,1);            


\draw[fill=gray] (2,2) rectangle (3,3);               
\draw[fill=gray] (1,1) rectangle (2,2);               
\draw[fill=gray] (0,1) rectangle (1,2);               


\end{tikzpicture}
\end{center}
We see that $B_C=\{(1,3),(2,1), (2,2)\}$ and $W_C=\{(1,1), (1,2), (1,4), (2,3), (3,1)\}$.

If $A$ is the partition subalgebra associated with the partition $\lambda$ and $q$ is not a root of unity then the $\ch$-invariant prime ideals of $A$ are in bijection with 
Cauchon-Le diagrams on $Y_\lambda$, the Young diagram corresponding 
to $\lambda$, by \cite[Theorem 3.5]{llr-grass}. We will come back to this in more details in Section \ref{section: Cauchon-Le diagrams in partition subalgebras}.

\begin{definition} \label{torus-rels}
{\rm The {\em quantum torus associated with the Cauchon-Le diagram $C$ on $Y_\lambda$} is the $\k$-algebra 
generated by elements $t_{ij}^{\pm1}$ with $(i,j)\in W_C$ and subject to the relations 
$t_{ij}t_{ik} =qt_{ik}t_{ij}$ for $j<k$ and  $t_{ij}t_{lj} = qt_{lj}t_{ij}$ for $i<l$ while all other pairs of generators commute. An easy way to remember these relations is as follows: let $a, b$ be squares in $W_C$ 
with $a<_{{\rm lex}} b$. Then $t_a$ and $t_b$ commute unless $a$ and $b$ are in the same row or 
column, in which case $t_at_b=qt_bt_a$. 
}\end{definition}

In the case that all boxes of $C$ are white, this quantum torus is (isomorphic to) the localisation obtained by inverting the generators of the $\k$-algebra $\widebar{A}=A^{(2)}=\k\langle t_{ij} \mid (i,j)\in C\rangle$ that occurs at the end of the deleting derivation process on $A:=\ylk$. In general, there is a natural surjective homomorphism from $\widebar{A}$ onto $\k\langle t_{ij}\mid(i,j)\in W_C \rangle$ given by $t_{ij}\goesto t_{ij}$ for $(i,j)\in W_C$ and $t_{ij}\goesto 0$ for $(i,j)$ in $B_C$, with kernel $K_C = \langle t_{ij} \mid (i,j) \in B_C\rangle$.

\begin{definition} \label{definition-postnikov-graph} 
The {\em Postnikov graph}, $\postc$, of a Cauchon-Le diagram $C$ on the Young 
diagram $Y_\lambda$ is a directed planar graph with weighted edges, constructed from $C$, where the weights come from the 
quantum torus  corresponding to $C$. The Postnikov 
graph is constructed in the following way. If $(i,j)$ is a white square in 
$C$ then let $(i,j^-)$ be the first white square to the 
left of $(i,j)$ (if it exists) and let $(i,j^+)$ be the first white square to the 
right of $(i,j)$ (if it exists). Similarly, $(i^+,j)$ is the first white square 
below $(i,j)$ (if one exists).

Firstly, place a vertex in each white square of the Cauchon-Le diagram. These vertices are called {\em internal}. 
Then, place vertices $r_i$ immediately 
to the right of the last box in row $i$, and also vertices $c_j$ immediately 
below the last box in column $j$. Vertices $r_i$ and $c_j$ are called {\em boundary vertices}. For each row index $i$, put a directed 
edge from $r_i$ to the rightmost white square in row $i$ (if it exists), 
say square $(i,k)$. Weight 
this edge with the weight $t_{ik}$. For each column index $j$, put a directed 
edge from the bottom-most white square in column $j$ (if it exists), 
to the vertex $c_j$. Weight this edge with weight $1$. For each white square 
$(i,j)$, put a directed edge from $(i,j)$ to $(i,j^-)$ (if it exists) and give this edge 
the weight $t_{i,j}^{-1}t_{i,j^-}$. Finally, 
for each white square 
$(i,j)$, put a directed edge from $(i,j)$ to $(i^+,j)$ (if it exists) and give this edge 
the weight $1$. 
\end{definition} 
Here is an example of a Cauchon-Le diagram and its associated 
Postnikov graph. Note that by convention the weight of weight 1 edges are omitted, and so we only record the weight of horizontal edges when drawing Postnikov graphs.  \\

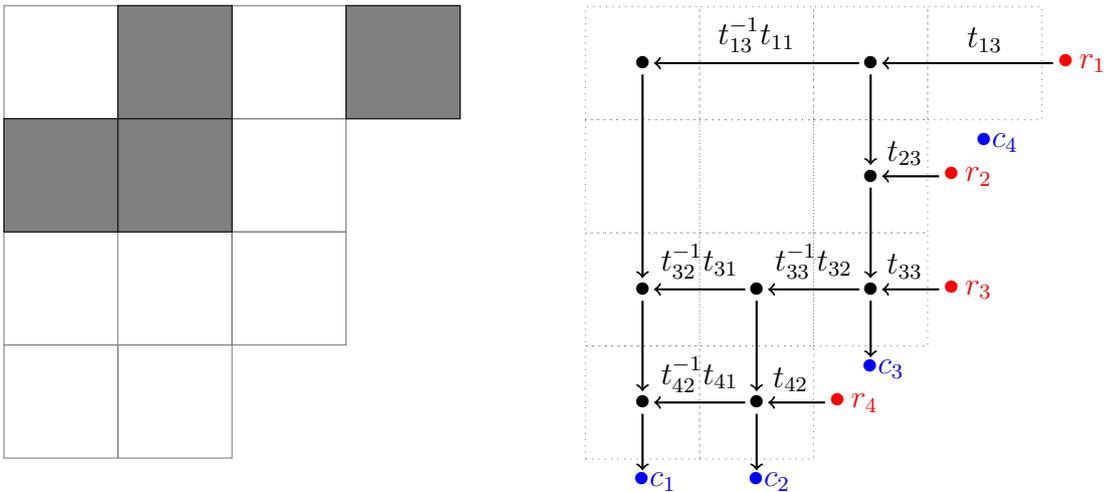
\begin{center}
 \begin{minipage}{30ex}
  \begin{center}
   \begin{tikzpicture}[xscale=1.5, yscale=1.5]

\draw[color=gray] (0,3) rectangle (1,4);            
\draw[color=gray] (1,3) rectangle (2,4);            
\draw[color=gray] (2,3) rectangle (3,4);            
\draw[color=gray] (3,3) rectangle (4,4);            
\draw[color=gray] (0,2) rectangle (2,3);               
\draw[color=gray] (1,2) rectangle (1,3);               
\draw[color=gray] (2,2) rectangle (3,3);               
\draw[color=gray] (0,1) rectangle (2,2);               
\draw[color=gray] (1,1) rectangle (1,2);               
\draw[color=gray] (2,1) rectangle (3,2);               
\draw[color=gray] (0,0) rectangle (1,1);            
\draw[color=gray] (1,0) rectangle (2,1);            

\draw[fill=gray] (1,3) rectangle (2,4);               
\draw[fill=gray] (3,3) rectangle (4,4);               
\draw[fill=gray] (0,2) rectangle (1,3);               
\draw[fill=gray] (1,2) rectangle (2,3);               

\node[color=blue]  at (0.5,-0.3) {$\hspace{0.32cm}$};

   \end{tikzpicture}
  \end{center}
 \end{minipage}
~~~~~~~~~~~~~~
 \begin{minipage}{35ex}
  \begin{center}
   \begin{tikzpicture}[xscale=1.5, yscale=1.5]

\draw[color=gray, dotted] (0,3) rectangle (1,4);            
\draw[color=gray, dotted] (1,3) rectangle (2,4);            
\draw[color=gray, dotted] (2,3) rectangle (3,4);            
\draw[color=gray, dotted] (3,3) rectangle (4,4);            
\draw[color=gray, dotted] (0,2) rectangle (2,3);               
\draw[color=gray, dotted] (1,2) rectangle (1,3);               
\draw[color=gray, dotted] (2,2) rectangle (3,3);               
\draw[color=gray, dotted] (0,1) rectangle (2,2);               
\draw[color=gray, dotted] (1,1) rectangle (1,2);               
\draw[color=gray, dotted] (2,1) rectangle (3,2);               
\draw[color=gray, dotted] (0,0) rectangle (1,1);            
\draw[color=gray, dotted] (1,0) rectangle (2,1);            

\node at (0.5, 3.5) {$\bullet$}; 
\node at (2.5, 3.5) {$\bullet$}; 
\node at (2.5, 2.5) {$\bullet$}; 
\node at (0.5, 1.5) {$\bullet$}; 
\node at (1.5, 1.5) {$\bullet$}; 
\node at (2.5, 1.5) {$\bullet$}; 
\node at (0.5, 0.5) {$\bullet$}; 
\node at (1.5, 0.5) {$\bullet$}; 

\node[color=red]  at (4.35,3.5) {$\bullet\, r_1$};
\node[color=red]  at (3.35,2.5) {$\bullet\, r_2$};
\node[color=red]  at (3.35,1.5) {$\bullet\, r_3$};
\node[color=red]  at (2.35,0.5) {$\bullet\, r_4$};

\node[color=blue]  at (0.5,-0.2) {$\hspace{0.32cm}\bullet c_1$};
\node[color=blue]  at (1.5,-0.2) {$\hspace{0.32cm}\bullet c_2$};
\node[color=blue]  at (2.5,0.8) {$\hspace{0.32cm}\bullet c_3$};
\node[color=blue]  at (3.5,2.8) {$\hspace{0.32cm}\bullet c_4$};

\draw [<-, thick, black] (0.6,3.5)--(2.4,3.5);
\draw [<-, thick, black] (2.6,3.5)--(4.1,3.5);
\draw [<-, thick, black] (2.6,2.5)--(3.1,2.5);
\draw [<-, thick, black] (0.6,1.5)--(1.4,1.5);
\draw [<-, thick, black] (1.6,1.5)--(2.4,1.5);
\draw [<-, thick, black] (2.6,1.5)--(3.1,1.5);
\draw [<-, thick, black] (0.6,0.5)--(1.4,0.5);
\draw [<-, thick, black] (1.6,0.5)--(2.1,0.5);

\draw [->, thick, black] (0.5,3.4)--(0.5,1.6);
\draw [->, thick, black] (0.5,1.4)--(0.5,0.6);
\draw [->, thick, black] (0.5,0.4)--(0.5,-0.1);
\draw [->, thick, black] (1.5,1.4)--(1.5,0.6);
\draw [->, thick, black] (1.5,0.4)--(1.5,-0.1);
\draw [->, thick, black] (2.5,3.4)--(2.5,2.6);
\draw [->, thick, black] (2.5,2.4)--(2.5,1.6);
\draw [->, thick, black] (2.5,1.4)--(2.5,0.9);

\node [above] at (1.5,3.5) {$t_{13}^{-1}t_{11}$};
\node [above] at (3.5,3.5) {$t_{13}$};
\node [above] at (2.8,2.5) {$t_{23}$};
\node [above] at (1,1.5) {$t_{32}^{-1}t_{31}$};
\node [above] at (2,1.5) {$t_{33}^{-1}t_{32}$};
\node [above] at (2.8,1.5) {$t_{33}$};
\node [above] at (1,0.5) {$t_{42}^{-1}t_{41}$};
\node [above] at (1.8,0.5) {$t_{42}$};

   \end{tikzpicture}
  \end{center}
 \end{minipage} \\~\\
\end{center}
It is clear from the construction of the Postnikov graph that it is represented in the plane. We shall 
always assume that Postnikov graphs are embedded in the plane in this way, and this allows us to 
use terms such as horizontal, vertical, left, right, North, South, East, West, etc. in reference to vertices and edges of a 
Postnikov graph.

\begin{notation}\label{col notations}{\rm 
\begin{itemize}
\item For us, ``path'' and ``edge'' shall always mean ``directed path'' and ``directed edge'' respectively.
\item Let $v$ and $v'$ be vertices of $\postc$. There is clearly at most one edge from $v$ to $v'$ and if it exists, we denote it by $(v,v')$. 
\item We denote the weight of an edge $e$ of $\postc$ by $w(e)$.
\item Suppose that $(i,j),(i,j')\in W_C$ and that there is an edge $e=((i,j),(i,j'))$ in $\postc$ (notice that this forces $j>j'$). Then we set $\row(e)=\{i\}$, $\col_1(e)=\{j\}$, $\col_2(e)=\{j'\}$. 
\item Let $e$ be an edge of $\postc$ from $r_i$ to $(i,j)\in W_C$. Then we define $\col_1(e)=\emptyset$ and $\col_2(e)=\{j\}$. 
\item For any edge $e$ in $\postc$, we define $\col(e)=\col_1(e)\cup\col_2(e)$.
\item If $v_0,v_1,\ldots,v_k$ are vertices of $\postc$ such that the edges $(v_0,v_1),(v_1,v_2),\ldots,(v_{k-1},v_k)$ exist, then we write $(v_0,v_1,\ldots,v_k)$ for the path from $v_0$ to $v_k$ given by the concatenation of the edges $(v_0,v_1),(v_1,v_2),\ldots,(v_{k-1},v_k)$. 
\item For vertices $v$ and $v'$ of $\postc$, we write $P: v\Longrightarrow v'$ to mean that $P$ is a path from $v$ to $v'$.  
\end{itemize} 
}
\end{notation}
\begin{definition} 
By the \emph{weight} of a path $(v_0,v_1,\ldots,v_k)$ in $\postc$, we mean the ordered product 
\[
w(v_0,v_1)w(v_1,v_2)\cdots w(v_{k-1},v_k)
\]
of the weights of its edges. We denote the weight of a path $P$ in $\postc$ by $w(P)$.
\end{definition}

\begin{definition}
An edge or path in $\postc$ is called \emph{internal} if its beginning and end vertices belong to $W_C$. An edge (resp. path) that is not internal is called a {\em boundary} edge (resp. path). 
\end{definition}

\begin{remark}
Since vertical edges in $\postc$ have weight $1$, only horizontal edges contribute to the weight of any path in $\postc$. We shall often use this fact without explicit mention. 
\end{remark}

\begin{example}\label{partition-ex''}
Let $c=d=4$, let $\lambda=(4,3,3,1)$.

In the figure below, we have shown a Cauchon-Le diagram $C$ on $Y_\lambda$ with its 
Postnikov graph $\postc$ superimposed onto $C$:

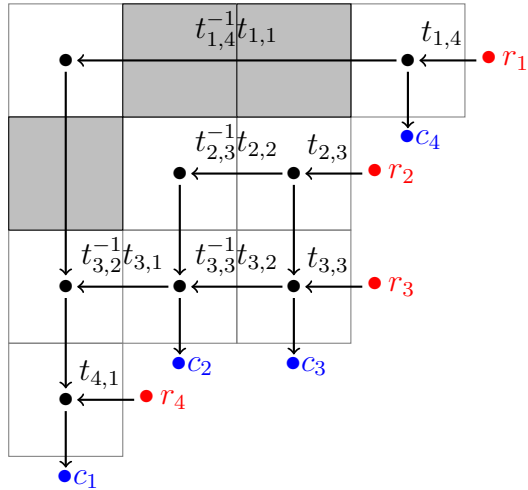
\begin{center}
   \begin{tikzpicture}[xscale=1.5, yscale=1.5]
\draw[color=gray] (0,3) rectangle (1,4);            
\draw[color=gray] (1,3) rectangle (2,4);            
\draw[color=gray] (2,3) rectangle (3,4);            
\draw[color=gray] (3,3) rectangle (4,4);            
\draw[color=gray] (0,2) rectangle (2,3);               
\draw[color=gray] (1,2) rectangle (1,3);               
\draw[color=gray] (2,2) rectangle (3,3);               
\draw[color=gray] (0,1) rectangle (2,2);               
\draw[color=gray] (1,1) rectangle (1,2);               
\draw[color=gray] (2,1) rectangle (3,2);               
\draw[color=gray] (0,0) rectangle (1,1);            

\draw[fill=lightgray] (1,3) rectangle (2,4);               
\draw[fill=lightgray] (0,2) rectangle (1,3);               
\draw[fill=lightgray] (2,3) rectangle (3,4);               

\node at (0.5, 3.5) {$\bullet$}; 
\node at (3.5, 3.5) {$\bullet$}; 
\node at (1.5, 2.5) {$\bullet$}; 
\node at (2.5, 2.5) {$\bullet$}; 
\node at (0.5, 1.5) {$\bullet$}; 
\node at (1.5, 1.5) {$\bullet$}; 
\node at (2.5, 1.5) {$\bullet$}; 
\node at (0.5, 0.5) {$\bullet$}; 

\node[color=red]  at (4.35,3.5) {$\bullet\, r_1$};
\node[color=red]  at (3.35,2.5) {$\bullet\, r_2$};
\node[color=red]  at (3.35,1.5) {$\bullet\, r_3$};
\node[color=red]  at (1.35,0.5) {$\bullet\, r_4$};

\node[color=blue]  at (0.5,-0.2) {$\hspace{0.32cm}\bullet c_1$};
\node[color=blue]  at (1.5,0.8) {$\hspace{0.32cm}\bullet c_2$};
\node[color=blue]  at (2.5,0.8) {$\hspace{0.32cm}\bullet c_3$};
\node[color=blue]  at (3.5,2.8) {$\hspace{0.32cm}\bullet c_4$};

\draw [<-, thick, black] (0.6,3.5)--(3.4,3.5);
\draw [<-, thick, black] (3.6,3.5)--(4.1,3.5);
\draw [<-, thick, black] (2.6,2.5)--(3.1,2.5);
\draw [<-, thick, black] (1.6,2.5)--(2.4,2.5);
\draw [<-, thick, black] (0.6,1.5)--(1.4,1.5);
\draw [<-, thick, black] (1.6,1.5)--(2.4,1.5);
\draw [<-, thick, black] (2.6,1.5)--(3.1,1.5);
\draw [<-, thick, black] (0.6,0.5)--(1.1,0.5);

\draw [->, thick, black] (0.5,3.4)--(0.5,1.6);
\draw [->, thick, black] (0.5,1.4)--(0.5,0.6);
\draw [->, thick, black] (0.5,0.4)--(0.5,-0.1);
\draw [->, thick, black] (1.5,2.4)--(1.5,1.6);
\draw [->, thick, black] (1.5,1.4)--(1.5,0.9);
\draw [->, thick, black] (2.5,2.4)--(2.5,1.6);
\draw [->, thick, black] (2.5,1.4)--(2.5,0.9);
\draw [->, thick, black] (3.5,3.4)--(3.5,2.9);

\node [above] at (3.8,3.5) {$t_{1,4}$};
\node [above] at (2,3.5) {$t_{1,4}^{-1}t_{1,1}$};
\node [above] at (2,2.5) {$t_{2,3}^{-1}t_{2,2}$};
\node [above] at (2.8,2.5) {$t_{2,3}$};
\node [above] at (1,1.5) {$t_{3,2}^{-1}t_{3,1}$};
\node [above] at (2,1.5) {$t_{3,3}^{-1}t_{3,2}$};
\node [above] at (2.8,1.5) {$t_{3,3}$};
\node [above] at (0.8,0.5) {$t_{4,1}$};

   \end{tikzpicture}
  \end{center}
\end{example}

We shall often superimpose Postnikov graphs onto their Cauchon-Le diagrams in this way.

\begin{proposition}[cf. Proposition 3.3 of \cite{cas1}]\label{basic facts}
The Postnikov graph $\postc$ has the following properties:
\begin{enumerate}[(1)]
\item $\postc$ is acyclic; that is, $\postc$ has no directed cycles.

\item The embedding of the Postnikov graph $\postc$ in the plane 
described above is a planar embedding; 
that is, all edge crossings occur at vertices.

\item An internal horizontal path $P:(i,j_2)\implies (i,j_1)$ has weight $t_{i,j_2}^{-1}t_{i,j_1}$.

\item A path $r_i\implies (i,j)$ beginning at a row vertex and consisting solely of horizontal edges has weight $t_{i,j}$. 

\end{enumerate}
\end{proposition} 

\begin{proof}
\begin{full} 
\begin{enumerate}[(1)]
\item Because all edges are directed leftwards or downwards, the graph $\postc$ cannot have a directed cycle. 
\item If two edges cross, then one edge must be vertical and the other horizontal. Let a vertical edge $e_1=((i_1,j),(i_2,j))$ cross a horizontal edge $e_2=((i,j_2),(i,j_1))$ at a black square $(i,j)$ of the Cauchon-Le diagram $C$. The black square $(i,j)$ has the white square $(i_1,j)$ above it and the white square $(i,j_1)$ to its left, contradicting the definition of a Cauchon-Le diagram. It follows that the square $(i,j)$ must be white and that the edges $e_1$ and $e_2$ cross at the vertex $(i,j)$.
\item If the path $P$ consists of a single edge, then the result follows from the definition of the Postnikov graph $\postc$. Suppose that the path $P$ consists of $n>1$ edges and that the desired result holds for all internal
 horizontal paths in $\postc$ consisting of fewer than $n$ edges. Let $(i,k)$ be an internal vertex of $P$. When $P'$ and $P''$ are the horizontal paths given by $P':(i,j_2)\Longrightarrow (i,k)$ and $P'': (i,k)\Longrightarrow (i,j_1)$, we have $P=P'P''$. Now the inductive hypothesis gives $w(P)=w(P')w(P'')=t_{i,j_2}^{-1}t_{i,k}t_{i,k}^{-1}t_{i,j_1}=t_{i,j_2}^{-1}t_{i,j_1}$.
\item This follows from part (3) and the definition of the Postnikov graph.
\end{enumerate}
\end{full}
\end{proof}


\subsection{Commutation relations between weights of paths}

In this subsection, we develop commutation relations between edges and certain paths in 
Postnikov graphs. Casteels, \cite[Lemmas 3.4, 3.5, 3.6]{cas1}, obtains similar results for the case that the Young diagram is a rectangle. The proofs for our results are similar, but we record them for the convenience of the reader. 

We refer the reader to Notation \ref{col notations} for the meanings of $\col_1(a),\col_2(a),\col(a)$ for internal and boundary edges $a$.

\begin{lemma}[cf. Lemma 3.4 of \cite{cas1}]\label{how edges commute}
Let $e$ and $f$ be distinct horizontal edges in $\postc$ such that $\row(f)\leq \row(e)$.
\begin{enumerate}[(1)]
\item If $\col(e)\cap \col(f)=\emptyset$, then $w(f)w(e)=w(e)w(f)$.
\item Suppose that $|\col(e)\cap \col(f)|=1$.
\begin{enumerate}[(i)]
\item If $\col_1(e)=\col_1(f)$ or $\col_2(e)=\col_2(f)$, then $w(f)w(e)=qw(e)w(f)$; 
\item if $\col_1(e)=\col_2(f)$ or $\col_2(e)=\col_1(f)$ and $\row(e)\neq \row(f)$, then $w(f)w(e)=q^{-1}w(e)w(f)$;
\item if $\col_2(f)=\col_1(e)$ and $\row(f)=\row(e)$, then $w(f)w(e)=q^{-1}w(e)w(f)$.
\end{enumerate}
\item If $|\col(e)\cap \col(f)|=2$, then $w(f)w(e)=q^2w(e)w(f)$.
\end{enumerate}
\begin{proof}
\begin{full} 
The proof consists of an investigation of all possible relative positions of $e$ and $f$. 

Notice that if $d$ is any internal horizontal edge in the graph $\postc$ and $col_2(d)<j<\col_1(d)$, then the square $(\row(d),j)$ is a black square in $C$ which has the white square $(\row(d),\col_2(d))$ to its left, so that for all $i\leq \row(d)$, the square $(i,j)$ is black. This observation will be used several times below and is the reason why certain cases cannot occur. For instance, the following configuration cannot occur, because the square immediately to the right of vertex $b$ fails the Cauchon-Le condition.

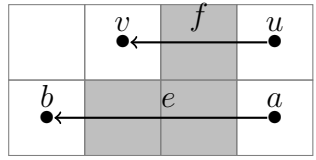
\begin{center}
\begin{tikzpicture}

\fill[color=lightgray] (2,1) rectangle (3,2);
\fill[color=lightgray] (2,0) rectangle (3,1);
\fill[color=lightgray] (1,0) rectangle (2,1);

\draw[color=gray] (0,1) rectangle (1,2); 
\draw[color=gray] (1,1) rectangle (2,2); 
\draw[color=gray] (2,1) rectangle (3,2); 
\draw[color=gray] (3,1) rectangle (4,2); 
\draw[color=gray] (0,0) rectangle (1,1); 
\draw[color=gray] (1,0) rectangle (2,1); 
\draw[color=gray] (2,0) rectangle (3,1); 
\draw[color=gray] (3,0) rectangle (4,1); 

\node at (1.5,1.5) (w2) {$\bullet$}; 
\node at (3.5,1.5) (v2) {$\bullet$}; 
\node at (0.5,0.5) (b) {$\bullet$}; 
\node at (3.5,0.5) (a) {$\bullet$}; 

\node [black,above] at (w2) {$v$};
\node [black,above] at (v2) {$u$};
\node [black,above] at (a) {$a$};
\node [black,above] at (b) {$b$};

\draw [<-,thick,black] (1.6,1.5) -- (3.4,1.5);
\draw [<-,thick,black] (0.6,0.5) -- (3.4,0.5);

\node [black, above] at (2.5,1.5) {$f$};
\node [black, above] at (2.1,0.5) {$e$};
\end{tikzpicture}
\end{center}

Let $a,b,u,v$ be the vertices of $\postc$ such that $e=(a,b)$ and $f=(u,v)$. \\

{\bf Case 0:} $\col(e)\cap \col(f)=\emptyset$. \\

$\bullet$ If $\row(e)\neq \row(f)$, then the result follows immediately from the relations in Definition~\ref{torus-rels} because $t_a$ and $t_b$ commute with $t_u$ and $t_v$. 

$\bullet$ Suppose that $\row(e)=\row(f)$ and notice that we may assume without loss of generality that $u$ and $v$ lie west of $a$ and $b$. We distinguish between two cases.

$\bullet$$\bullet$ Assume first that $e$ is internal. The following diagram illustrates the situation:
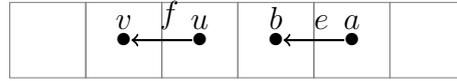
\begin{center}
\begin{tikzpicture}[xscale=1,yscale=1]

\draw[color=gray] (0,0) rectangle (1,1); 
\draw[color=gray] (1,0) rectangle (2,1); 
\draw[color=gray] (2,0) rectangle (3,1); 
\draw[color=gray] (3,0) rectangle (4,1); 
\draw[color=gray] (4,0) rectangle (5,1); 
\draw[color=gray] (5,0) rectangle (6,1); 

\node at (1.5,0.5) (w) {$\bullet$}; 
\node at (2.5,0.5) (v) {$\bullet$}; 
\node at (3.5,0.5) (b) {$\bullet$}; 
\node at (4.5,0.5) (a) {$\bullet$}; 
\node [black,above] at (w) {$v$};
\node [black,above] at (v) {$u$};
\node [black,above] at (a) {$a$};
\node [black,above] at (b) {$b$};
\draw [<-,thick,black] (1.6,0.5) -- (2.4,0.5);
\draw [<-,thick,black] (3.6,0.5) -- (4.4,0.5);
\node [black, above] at (2.1,0.5) {$f$};
\node [black, above] at (4.1,0.5) {$e$};
\end{tikzpicture}
\end{center}
The relations in Definition~\ref{torus-rels} now give $t_ut_b=qt_bt_u$, $t_ut_a=qt_at_u$, $t_vt_b=qt_bt_v$, and $t_vt_a=qt_at_v$. Hence
\[
w(e)w(f)=t_a^{-1}t_bt_u^{-1}t_v=qq^{-1}t_a^{-1}t_u^{-1}t_vt_b=qq^{-1}qq^{-1}t_u^{-1}t_vt_a^{-1}t_b=w(f)w(e).
\]

$\bullet$$\bullet$ Next assume that $e$ is boundary. The following diagram illustrates the situation:
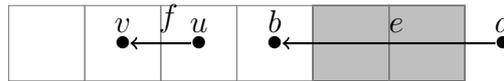
\begin{center}
\begin{tikzpicture}[xscale=1,yscale=1]
\draw[fill=lightgray] (4,0) rectangle (5,1);               
\draw[fill=lightgray] (5,0) rectangle (6,1);               

\draw[color=gray] (0,0) rectangle (1,1); 
\draw[color=gray] (1,0) rectangle (2,1); 
\draw[color=gray] (2,0) rectangle (3,1); 
\draw[color=gray] (3,0) rectangle (4,1); 
\draw[color=gray] (4,0) rectangle (5,1); 
\draw[color=gray] (5,0) rectangle (6,1); 

\node at (1.5,0.5) (w) {$\bullet$}; 
\node at (2.5,0.5) (v) {$\bullet$}; 
\node at (3.5,0.5) (b) {$\bullet$}; 
\node at (6.5,0.5) (a) {$\bullet$}; 
\node [black,above] at (w) {$v$};
\node [black,above] at (v) {$u$};
\node [black,above] at (a) {$a$};
\node [black,above] at (b) {$b$};
\draw [<-,thick,black] (1.6,0.5) -- (2.4,0.5);
\draw [<-,thick,black] (3.6,0.5) -- (6.4,0.5);
\node [black, above] at (2.1,0.5) {$f$};
\node [black, above] at (5.1,0.5) {$e$};

\end{tikzpicture}
\end{center}
The relations in Definition~\ref{torus-rels} now give $t_ut_b=qt_bt_u$, and $t_vt_b=qt_bt_v$. Hence
\[
w(e)w(f)=t_bt_u^{-1}t_v=qq^{-1}t_u^{-1}t_vt_b=w(f)w(e).
\]

{\bf Case 1:} $|\col(e)\cap \col(f)|=1$. \\
We distinguish between 8 cases. 
\begin{enumerate}[(i)]
\item If we are in the situation of the diagram below ($\col_2(e)=\col_2(f)$ and $u,a$ internal),
\begin{center}
\begin{tikzpicture}
\draw[color=gray] (0,1) rectangle (1,2); 
\draw[color=gray] (1,1) rectangle (2,2); 
\draw[color=gray] (2,1) rectangle (3,2); 
\draw[color=gray] (3,1) rectangle (4,2); 
\draw[color=gray] (0,0) rectangle (1,1); 
\draw[color=gray] (1,0) rectangle (2,1); 
\draw[color=gray] (2,0) rectangle (3,1); 
\draw[color=gray] (3,0) rectangle (4,1); 

\fill[color=lightgray] (2,1) rectangle (3,2);

\node at (1.5,1.5) (w2) {$\bullet$}; 
\node at (3.5,1.5) (v2) {$\bullet$}; 
\node at (1.5,0.5) (b) {$\bullet$}; 
\node at (2.5,0.5) (a) {$\bullet$}; 

\node [black,above] at (w2) {$v$};
\node [black,above] at (v2) {$u$};
\node [black,above] at (a) {$a$};
\node [black,above] at (b) {$b$};

\draw [<-,thick,black] (1.6,1.5) -- (3.4,1.5);
\draw [<-,thick,black] (1.6,0.5) -- (2.4,0.5);

\node [black, above] at (2.5,1.5) {$f$};
\node [black, above] at (2.1,0.5) {$e$};
\end{tikzpicture}
\end{center}
then the relations in Definition~\ref{torus-rels} now give $t_at_u=t_ut_a$, $t_at_v=t_vt_a$, $t_bt_u=t_ut_b$, and $t_vt_b=qt_bt_v$. Hence
\[
w(e)w(f)=t_a^{-1}t_bt_u^{-1}t_v=q^{-1}t_u^{-1}t_vt_a^{-1}t_b=q^{-1}w(f)w(e),
\]
as expected in the case where $\col_2(e)=\col_2(f)$. 

\item If we are in the situation of the diagram below ($\col_2(e)=\col_2(f)$, $u$ internal and $a$ boundary),
\begin{center}
\begin{tikzpicture}
\draw[color=gray] (0,1) rectangle (1,2); 
\draw[color=gray] (1,1) rectangle (2,2); 
\draw[color=gray] (2,1) rectangle (3,2); 
\draw[color=gray] (3,1) rectangle (4,2); 
\draw[color=gray] (0,0) rectangle (1,1); 
\draw[color=gray] (1,0) rectangle (2,1); 

\fill[color=lightgray] (2,1) rectangle (3,2);

\node at (1.5,1.5) (w2) {$\bullet$}; 
\node at (3.5,1.5) (v2) {$\bullet$}; 
\node at (1.5,0.5) (b) {$\bullet$}; 
\node at (2.5,0.5) (a) {$\bullet$}; 

\node [black,above] at (w2) {$v$};
\node [black,above] at (v2) {$u$};
\node [black,above] at (a) {$a$};
\node [black,above] at (b) {$b$};

\draw [<-,thick,black] (1.6,1.5) -- (3.4,1.5);
\draw [<-,thick,black] (1.6,0.5) -- (2.4,0.5);

\node [black, above] at (2.5,1.5) {$f$};
\node [black, above] at (2.1,0.5) {$e$};
\end{tikzpicture}
\end{center}
then the relations in Definition~\ref{torus-rels} now give $t_bt_u=t_ut_b$, and $t_vt_b=qt_bt_v$. Hence
\[
w(e)w(f)=t_bt_u^{-1}t_v=q^{-1}t_u^{-1}t_vt_b=q^{-1}w(f)w(e),
\]
as expected in the case where $\col_2(e)=\col_2(f)$. 

\item If we are in the situation of the diagram below ($\col_2(e)=\col_2(f)$, $a$ internal and $u$ boundary),
\begin{center}
\begin{tikzpicture}

\fill[color=lightgray] (2,1) rectangle (3,2);
\fill[color=lightgray] (3,1) rectangle (4,2);

\draw[color=gray] (0,1) rectangle (1,2); 
\draw[color=gray] (1,1) rectangle (2,2); 
\draw[color=gray] (2,1) rectangle (3,2); 
\draw[color=gray] (3,1) rectangle (4,2); 
\draw[color=gray] (0,0) rectangle (1,1); 
\draw[color=gray] (1,0) rectangle (2,1); 
\draw[color=gray] (2,0) rectangle (3,1); 

\node at (1.5,1.5) (w2) {$\bullet$}; 
\node at (4.5,1.5) (v2) {$\bullet$}; 
\node at (1.5,0.5) (b) {$\bullet$}; 
\node at (2.5,0.5) (a) {$\bullet$}; 

\node [black,above] at (w2) {$v$};
\node [black,above] at (v2) {$u$};
\node [black,above] at (a) {$a$};
\node [black,above] at (b) {$b$};

\draw [<-,thick,black] (1.6,1.5) -- (4.4,1.5);
\draw [<-,thick,black] (1.6,0.5) -- (2.4,0.5);

\node [black, above] at (3.5,1.5) {$f$};
\node [black, above] at (2.1,0.5) {$e$};
\end{tikzpicture}
\end{center}
then the relations in Definition~\ref{torus-rels} now give $t_at_v=t_vt_a$, and $t_vt_b=qt_bt_v$. Hence
\[
w(e)w(f)=t_a^{-1}t_bt_v=q^{-1}t_vt_a^{-1}t_b=q^{-1}w(f)w(e),
\]
as expected in the case where $\col_2(e)=\col_2(f)$.

\item If we are in the situation of the diagram below ($\col_2(e)=\col_2(f)$, $a$ and $u$ boundary),
\begin{center}
\begin{tikzpicture}

\fill[color=lightgray] (2,1) rectangle (3,2);
\fill[color=lightgray] (3,1) rectangle (4,2);

\draw[color=gray] (0,1) rectangle (1,2); 
\draw[color=gray] (1,1) rectangle (2,2); 
\draw[color=gray] (2,1) rectangle (3,2); 
\draw[color=gray] (3,1) rectangle (4,2); 
\draw[color=gray] (0,0) rectangle (1,1); 
\draw[color=gray] (1,0) rectangle (2,1); 

\node at (1.5,1.5) (w2) {$\bullet$}; 
\node at (4.5,1.5) (v2) {$\bullet$}; 
\node at (1.5,0.5) (b) {$\bullet$}; 
\node at (2.5,0.5) (a) {$\bullet$}; 

\node [black,above] at (w2) {$v$};
\node [black,above] at (v2) {$u$};
\node [black,above] at (a) {$a$};
\node [black,above] at (b) {$b$};

\draw [<-,thick,black] (1.6,1.5) -- (4.4,1.5);
\draw [<-,thick,black] (1.6,0.5) -- (2.4,0.5);

\node [black, above] at (3.5,1.5) {$f$};
\node [black, above] at (2.1,0.5) {$e$};
\end{tikzpicture}
\end{center}
then the relations in Definition~\ref{torus-rels} now give $t_bt_u=t_ut_b$, and $t_vt_b=qt_bt_v$. Hence
\[
w(e)w(f)=t_bt_u^{-1}t_v=q^{-1}t_u^{-1}t_vt_b=q^{-1}w(f)w(e),
\]
as expected in the case where $\col_2(e)=\col_2(f)$. 

\item If we are in the situation of the diagram below ($\col_1(e)=\col_1(f)$, $a$ and $u$ internal),
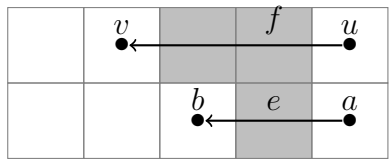
\begin{center}
\begin{tikzpicture}

\fill[color=lightgray] (2,1) rectangle (3,2);
\fill[color=lightgray] (3,1) rectangle (4,2);
\fill[color=lightgray] (3,0) rectangle (4,1);

\draw[color=gray] (0,1) rectangle (1,2); 
\draw[color=gray] (1,1) rectangle (2,2); 
\draw[color=gray] (2,1) rectangle (3,2); 
\draw[color=gray] (3,1) rectangle (4,2); 
\draw[color=gray] (4,1) rectangle (5,2); 
\draw[color=gray] (0,0) rectangle (1,1); 
\draw[color=gray] (1,0) rectangle (2,1); 
\draw[color=gray] (2,0) rectangle (3,1); 
\draw[color=gray] (3,0) rectangle (4,1); 
\draw[color=gray] (4,0) rectangle (5,1); 

\node at (1.5,1.5) (w2) {$\bullet$}; 
\node at (4.5,1.5) (v2) {$\bullet$}; 
\node at (2.5,0.5) (b) {$\bullet$}; 
\node at (4.5,0.5) (a) {$\bullet$}; 

\node [black,above] at (w2) {$v$};
\node [black,above] at (v2) {$u$};
\node [black,above] at (a) {$a$};
\node [black,above] at (b) {$b$};

\draw [<-,thick,black] (1.6,1.5) -- (4.4,1.5);
\draw [<-,thick,black] (2.6,0.5) -- (4.4,0.5);

\node [black, above] at (3.5,1.5) {$f$};
\node [black, above] at (3.5,0.5) {$e$};
\end{tikzpicture}
\end{center}
then the relations in Definition~\ref{torus-rels} now give $t_bt_u=t_ut_b$, $t_vt_b=t_bt_v$, $t_at_v=t_vt_a$ and $t_ut_a=qt_at_u$. Hence
\[
w(e)w(f)=t_a^{-1}t_bt_u^{-1}t_v=q^{-1}t_u^{-1}t_vt_a^{-1}t_b=q^{-1}w(f)w(e),
\]
as expected in the case where $\col_1(e)=\col_1(f)$.

\item Suppose that $\row(e)\neq \row(f)$ and that $\col_1(f)=\col_2(e)$ (the case where $\row(e)\neq \row(f)$ and $\col_2(f)=\col_1(e)$ is similar). The following diagram illustrates the situation when $a$ is an internal vertex:
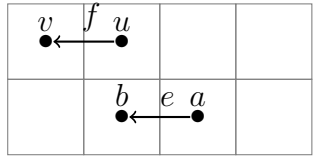
\begin{center}
\begin{tikzpicture}
\draw[color=gray] (0,1) rectangle (1,2); 
\draw[color=gray] (1,1) rectangle (2,2); 
\draw[color=gray] (2,1) rectangle (3,2); 
\draw[color=gray] (3,1) rectangle (4,2); 
\draw[color=gray] (0,0) rectangle (1,1); 
\draw[color=gray] (1,0) rectangle (2,1); 
\draw[color=gray] (2,0) rectangle (3,1); 
\draw[color=gray] (3,0) rectangle (4,1); 

\node at (0.5,1.5) (w1) {$\bullet$}; 
\node at (1.5,1.5) (v1) {$\bullet$}; 
\node at (1.5,0.5) (b) {$\bullet$}; 
\node at (2.5,0.5) (a) {$\bullet$}; 

\node [black,above] at (w1) {$v$};
\node [black,above] at (v1) {$u$};
\node [black,above] at (a) {$a$};
\node [black,above] at (b) {$b$};

\draw [<-,thick,black] (0.6,1.5) -- (1.4,1.5);
\draw [<-,thick,black] (1.6,0.5) -- (2.4,0.5);

\node [black, above] at (1.1,1.5) {$f$};
\node [black, above] at (2.1,0.5) {$e$};
\end{tikzpicture}
\end{center}
The relations \eqref{torus-rels} now give $t_at_u=t_ut_a$, $t_at_v=t_vt_a$, $t_bt_v=t_vt_b$, and $t_ut_b=qt_bt_u$. It follows that 
\[
w(e)w(f)=t_a^{-1}t_bt_u^{-1}t_v=qt_u^{-1}t_vt_a^{-1}t_b=qw(f)w(e), 
\]
as desired.

\item Suppose that $\row(e)\neq \row(f)$ and $\col_2(f)=\col_1(e)$. The following diagram illustrates the situation when $u$ is an internal vertex:
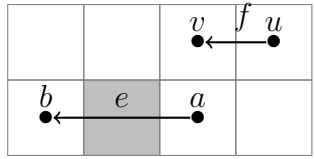
\begin{center}
\begin{tikzpicture}
\fill[color=lightgray] (1,0) rectangle (2,1);
\draw[color=gray] (0,1) rectangle (1,2); 
\draw[color=gray] (1,1) rectangle (2,2); 
\draw[color=gray] (2,1) rectangle (3,2); 
\draw[color=gray] (3,1) rectangle (4,2); 
\draw[color=gray] (0,0) rectangle (1,1); 
\draw[color=gray] (1,0) rectangle (2,1); 
\draw[color=gray] (2,0) rectangle (3,1); 
\draw[color=gray] (3,0) rectangle (4,1); 

\node at (2.5,1.5) (w1) {$\bullet$}; 
\node at (3.5,1.5) (v1) {$\bullet$}; 
\node at (0.5,0.5) (b) {$\bullet$}; 
\node at (2.5,0.5) (a) {$\bullet$}; 

\node [black,above] at (w1) {$v$};
\node [black,above] at (v1) {$u$};
\node [black,above] at (a) {$a$};
\node [black,above] at (b) {$b$};

\draw [<-,thick,black] (2.6,1.5) -- (3.4,1.5);
\draw [<-,thick,black] (0.6,0.5) -- (2.4,0.5);

\node [black, above] at (3.1,1.5) {$f$};
\node [black, above] at (1.5,0.5) {$e$};
\end{tikzpicture}
\end{center}
The relations \eqref{torus-rels} now give $t_at_u=t_ut_a$, $t_at_v=q^{-1}t_vt_a$, $t_bt_v=t_vt_b$, and $t_ut_b=t_bt_u$. It follows that 
\[
w(e)w(f)=t_a^{-1}t_bt_u^{-1}t_v=qt_u^{-1}t_vt_a^{-1}t_b=qw(f)w(e).
\]

\item Suppose that $\row(e)=\row(f)$ and that $\col_2(e)=\col_1(f)$. Then the end vertex of $e$ is the starting vertex of $f$, that is $b=u$. The following diagram illustrates the situation when $a$ is internal:
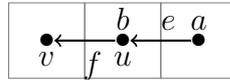
\begin{center}
\begin{tikzpicture}
\draw[color=gray] (0,0) rectangle (1,1); 
\draw[color=gray] (1,0) rectangle (2,1); 
\draw[color=gray] (2,0) rectangle (3,1); 

\node at (0.5,0.5) (w) {$\bullet$}; 
\node at (1.5,0.5) (b) {$\bullet$}; 
\node at (2.5,0.5) (a) {$\bullet$}; 

\draw [<-,thick,black] (1.6,0.5) -- (2.4,0.5);
\draw [<-,thick,black] (0.6,0.5) -- (1.4,0.5);

\node [black,below] at (w) {$v$};
\node [black,below] at (b) {$u$};
\node [black,above] at (b) {$b$};
\node [black,above] at (a) {$a$};

\node [black, above] at (2.1,0.5) {$e$};
\node [black, below] at (1.1,0.5) {$f$};
\end{tikzpicture}
\end{center}

If $a$ is internal, then the relations in Definition~\ref{torus-rels} give $t_vt_b=qt_bt_v$, $t_vt_a=qt_at_v$, and $t_bt_a=qt_at_b$. Hence 
\[
w(f)w(e)=t_b^{-1}t_vt_a^{-1}t_b=q^2t_vt_a^{-1}=qt_a^{-1}t_v=qt_a^{-1}t_bt_b^{-1}t_v=qw(e)w(f).
\]
If $a$ is boundary, then the relations in Definition~\ref{torus-rels} give $t_vt_b=qt_bt_v$. Hence 
\[
w(f)w(e)=t_b^{-1}t_vt_b=qt_bt_b^{-1}t_v=qw(e)w(f).
\]
This ends Case 1.
\end{enumerate}

{\bf Case 2:} $|\col(e)\cap \col(f)|=2$. The following diagram illustrates the situation:
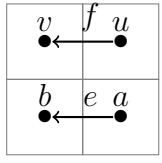
\begin{center}
\begin{tikzpicture}
\draw[color=gray] (0,1) rectangle (1,2); 
\draw[color=gray] (1,1) rectangle (2,2); 

\draw[color=gray] (0,0) rectangle (1,1); 
\draw[color=gray] (1,0) rectangle (2,1); 

\node at (0.5,1.5) (w) {$\bullet$}; 
\node at (1.5,1.5) (v) {$\bullet$}; 
\node at (0.5,0.5) (b) {$\bullet$}; 
\node at (1.5,0.5) (a) {$\bullet$}; 
\draw [<-,thick,black] (0.6,1.5) -- (1.4,1.5);
\draw [<-,thick,black] (0.6,0.5) -- (1.4,0.5);
\node [black,above] at (w) {$v$};
\node [black,above] at (v) {$u$};
\node [black,above] at (a) {$a$};
\node [black,above] at (b) {$b$};
\node [black, above] at (1.1,1.5) {$f$};
\node [black, above] at (1.1,0.5) {$e$};
\end{tikzpicture}
\end{center}
The relations in Definition~\ref{torus-rels} show that $t_at_v=t_vt_a$, $t_bt_u=t_ut_b$, $t_vt_b=qt_bt_v$, and $t_ut_a=qt_at_u$. It follows that 
\[
w(e)w(f)=t_a^{-1}t_bt_u^{-1}t_v=q^{-1}t_a^{-1}t_u^{-1}t_vt_b=q^{-2}t_u^{-1}t_vt_a^{-1}t_b=q^{-2}w(f)w(e).
\]

\end{full} 
\end{proof}
\end{lemma}
\begin{remark}
The reader may notice that part (2) of Lemma \ref{how edges commute} differs from part 2 of \cite[Lemma 3.4]{cas1}. This is to clear up a slight ambiguity in part 2(ii) of \cite[Lemma 3.4]{cas1}, namely that in the case where $\row(e)=\row(f)$, part 2(ii) of \cite[Lemma 3.4]{cas1} only holds if $e$ begins where $f$ ends. 
\end{remark}

\begin{lemma}[cf. Lemma 3.5 of \cite{cas1}]\label{how the head an tail commute lemma}
Let $K: v_0\implies v$ and $L: v\implies v_t$ be paths in $\postc$.
\begin{enumerate}[(1)]
\item If either $K$ or $L$ contains only vertical edges, then $w(K)w(L)=w(L)w(K)$.
\item If $K$ contains a horizontal edge and $L$ contains a horizontal edge, then $w(K)w(L)=q^{-1}w(L)w(K)$.
\end{enumerate}
\begin{proof}
\begin{full} 
\begin{enumerate}[(1)]
\item This follows immediately from the fact that vertical edges in $\postc$ have weight $1$.
\item Since all vertical edges have weight one, only the horizontal edges of $K$ and $L$ contribute to their weights. Let $v_1$ be the rightmost vertex in $K$ in the same row as $v$. (Note that $v_1$ can be equal to $v$ as illustrated by the dotted alternative in the following diagram.) We denote by $K'$ the subpath of $K$ that starts at $v_0$ and ends at $v_1$, and by $K''$ the subpath of $K$ that starts at $v_1$ and ends at $v_t$. Clearly, we have $w(K)=w(K')w(K'')$.

Similarly, let $v_2$ be the leftmost vertex in $L$ in the same row as $v$. (Note that $v_2$ can be equal to $v$ as illustrated by the dashed alternative in the following diagram.) We denote by $L''$ the subpath of $L$ that starts at $v$ and ends at $v_2$, and by $L'$ the subpath of $L$ that starts at $v_2$ and ends at $v_t$. Clearly, we have $w(L)=w(L'')w(L')$.

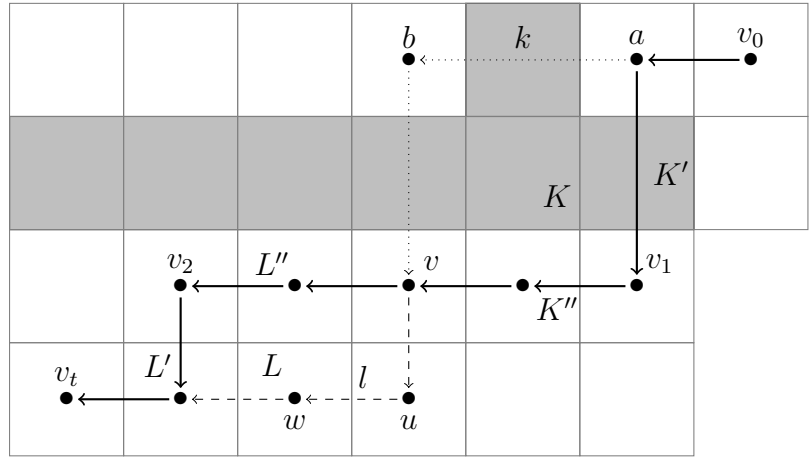
\begin{center}
 \begin{minipage}{45ex}
   \begin{tikzpicture}[xscale=1.5, yscale=1.5]
   
\draw[fill=lightgray] (-1,1) rectangle (5,2);   
\draw[fill=lightgray] (3,2) rectangle (4,3);   

\draw[color=gray] (0,2) rectangle (1,3);               
\draw[color=gray] (1,2) rectangle (2,3);               
\draw[color=gray] (2,2) rectangle (3,3);               
\draw[color=gray] (3,2) rectangle (4,3);               
\draw[color=gray] (4,2) rectangle (5,3);               
\draw[color=gray] (5,2) rectangle (6,3);               
\draw[color=gray] (0,1) rectangle (1,2);               
\draw[color=gray] (1,1) rectangle (2,2);               
\draw[color=gray] (2,1) rectangle (3,2);               
\draw[color=gray] (3,1) rectangle (4,2);               
\draw[color=gray] (4,1) rectangle (5,2);               
\draw[color=gray] (5,1) rectangle (6,2);               
\draw[color=gray] (0,0) rectangle (1,1);               
\draw[color=gray] (1,0) rectangle (2,1);               
\draw[color=gray] (2,0) rectangle (3,1);               
\draw[color=gray] (3,0) rectangle (4,1);               
\draw[color=gray] (4,0) rectangle (5,1);               
\draw[color=gray] (0,-1) rectangle (1,0);               
\draw[color=gray] (1,-1) rectangle (2,0);               
\draw[color=gray] (2,-1) rectangle (3,0);               
\draw[color=gray] (3,-1) rectangle (4,0);               
\draw[color=gray] (4,-1) rectangle (5,0);               
\draw[color=gray] (-1,-1) rectangle (0,0);               
\draw[color=gray] (-1,0) rectangle (0,1);               
\draw[color=gray] (-1,1) rectangle (0,2);               
\draw[color=gray] (-1,2) rectangle (0,3);               

\node at (5.5,2.5) {$\bullet$};
\node at (2.5,2.5) {$\bullet$};
\node at (4.5,2.5) {$\bullet$}; 
\node at (4.5,0.5) {$\bullet$};
\node at (3.5,0.5) {$\bullet$};
\node at (2.5,0.5) {$\bullet$}; 
\node at (1.5,0.5) {$\bullet$}; 
\node at (0.5,0.5) {$\bullet$}; 
\node at (0.5,-0.5) {$\bullet$}; 
\node at (-0.5,-0.5) {$\bullet$}; 
\node at (1.5,-0.5) {$\bullet$}; 
\node at (2.5,-0.5) {$\bullet$}; 
\draw [->, thick, black] (5.4,2.5)--(4.6,2.5);
\draw [->, thick, black] (4.4,0.5)--(3.6,0.5);
\draw [->, thick, black] (3.4,0.5)--(2.6,0.5);
\draw [->, thick, black] (2.4,0.5)--(1.6,0.5);
\draw [->, thick, black] (1.4,0.5)--(0.6,0.5);
\draw [->, thick, black] (0.4,-0.5)--(-0.4,-0.5);

\draw [->, dotted, black] (4.4,2.5)--(2.6,2.5);
\draw [->, dotted, black] (2.5,2.4)--(2.5,0.6);

\draw [->, dashed, black] (2.5,0.4)--(2.5,-0.4);
\draw [->, dashed, black] (2.4,-0.5)--(1.6,-0.5);
\draw [->, dashed, black] (1.4,-0.5)--(0.6,-0.5);

\draw [->, thick, black] (4.5,2.4)--(4.5,0.6);
\draw [->, thick, black] (0.5,0.4)--(0.5,-0.4);

\node at (5.5,2.7) {$v_0$}; 
\node at (2.7,0.7) {$v$};
\node at (0.5,0.7) {$v_2$}; 
\node at (-0.5,-0.3) {$v_t$}; 
\node at (4.7,0.7) {$v_1$}; 
\node at (1.5,-0.7) {$w$}; 
\node at (2.5,-0.7) {$u$}; 
\node at (4.5,2.7) {$a$}; 
\node at (2.5,2.7) {$b$}; 

\node at (3.5,2.7) {$k$}; 
\node at (2.1,-0.3) {$l$};

\node at (4.8,1.5) {$K'$}; 
\node at (3.8,0.3) {$K''$}; 
\node at (3.8,1.3) {$K$}; 
\node at (1.3,0.7) {$L''$}; 
\node at (1.3,-0.2) {$L$}; 
\node at (0.3,-0.2) {$L'$}; 
  \end{tikzpicture}

 \end{minipage}

\end{center} 

We distinguish between two cases.

$\bullet$ Case 1: $K'' \neq \emptyset$ and $L'' \neq \emptyset$. \\
By Lemma \ref{how edges commute}(1), the weights of all edges in $K'$ commute with the weights of all edges in $L$ and the weights of all edges in $L'$ commute with the weights of all edges in $K$. By Lemma \ref{how edges commute}(2)(ii) or Lemma \ref{how edges commute}(2)(iii), we have $w(K'')w(L'')=q^{-1}w(L'')w(K'')$. Now 
\begin{align*}
w(K)w(L)&=w(K')w(K'')w(L'')w(L') \\
        &=q^{-1}w(K')w(L'')w(K'')w(L') \\
        &=q^{-1}w(L'')w(L')w(K')w(K'') \\
        &=q^{-1}w(L)w(K). \\
\end{align*}

$\bullet$ Case 2: $K'' = \emptyset$ or $L'' = \emptyset$. \\
In this case, we denote by $k:a \rightarrow b$ the last horizontal edge in $K$ and by $l:u\rightarrow w$ be the first horizontal edge in $L$. We denote by $K\bs\{k\}$ the subpath of $K$ starting at $v_0$ and ending at $a$, and by $L\bs\{l\}$ the subpath of $L$ starting at $w$ and ending at $v_t$. By Lemma \ref{how edges commute}(1), the weights of all edges in $K\bs\{k\}$ commute with the weights of all edges in $L$ and the weights of all edges in $L\bs\{l\}$ commute with the weights of all edges in $K$. By Lemma \ref{how edges commute}(2)(ii), we have $w(k)w(l)=q^{-1}w(l)w(k)$. Now 
\begin{align*}
w(K)w(L)&=w(K\bs\{k\})w(k)w(l)w(L\bs\{l\}) \\
        &=q^{-1}w(K\bs\{k\})w(l)w(k)w(L\bs\{l\}) \\
        &=q^{-1}w(l)w(L\bs\{l\})w(K\bs\{k\})w(k) \\
        &=q^{-1}w(L)w(K). \\
\end{align*}

\end{enumerate}
\end{full} 
\end{proof}
\end{lemma}

\begin{lemma}[cf. Lemma 3.6 of \cite{cas1}]\label{how paths that share only initial vertex commute lemma}
Let $K: v\implies c_i$ and $L: v\implies c_j$ be two paths in $\postc$ which share their initial vertex and no other vertex. Let $K$ be the path that starts with a horizontal edge and let $L$ be the path that starts with a vertical edge.
\begin{enumerate}[(1)]
\item If $L$ consists only of vertical edges, then $w(K)w(L)=w(L)w(K)$. 
\item If $L$ has a horizontal edge then $w(K)w(L)=qw(L)w(K)$.
\end{enumerate}

\begin{center}
 \begin{minipage}{45ex}
   \begin{tikzpicture}[xscale=1, yscale=1]
\draw[color=gray] (-0.5,3.5) rectangle (0.5,4.5);              
\draw[color=gray] (0.5,3.5) rectangle (1.5,4.5);               
\draw[color=gray] (1.5,3.5) rectangle (2.5,4.5);               
\draw[color=gray] (2.5,3.5) rectangle (3.5,4.5);               
\draw[color=gray] (3.5,3.5) rectangle (4.5,4.5);               
\draw[color=gray] (4.5,3.5) rectangle (5.5,4.5);               

\draw[color=gray] (-0.5,2.5) rectangle (0.5,3.5);              
\draw[color=gray] (0.5,2.5) rectangle (1.5,3.5);               
\draw[color=gray] (1.5,2.5) rectangle (2.5,3.5);               
\draw[color=gray] (2.5,2.5) rectangle (3.5,3.5);               
\draw[color=gray] (3.5,2.5) rectangle (4.5,3.5);               
\draw[color=gray] (4.5,2.5) rectangle (5.5,3.5);               
 
\draw[color=gray] (-0.5,1.5) rectangle (0.5,2.5);              
\draw[color=gray] (0.5,1.5) rectangle (1.5,2.5);               
\draw[color=gray] (1.5,1.5) rectangle (2.5,2.5);               
\draw[color=gray] (2.5,1.5) rectangle (3.5,2.5);               
\draw[color=gray] (3.5,1.5) rectangle (4.5,2.5);               
\draw[color=gray] (4.5,1.5) rectangle (5.5,2.5);               

\draw[color=gray] (-0.5,0.5) rectangle (0.5,1.5);              
\draw[color=gray] (0.5,0.5) rectangle (1.5,1.5);               
\draw[color=gray] (1.5,0.5) rectangle (2.5,1.5);               
\draw[color=gray] (2.5,0.5) rectangle (3.5,1.5);               
\draw[color=gray] (3.5,0.5) rectangle (4.5,1.5);               

\draw[fill=lightgray] (1.5,3.5) rectangle (2.5,4.5);               
\draw[fill=lightgray] (1.5,2.5) rectangle (2.5,3.5);               

\draw[color=gray] (-0.5,-0.5) rectangle (0.5,0.5);              

\node at (5,4) {$\bullet$};
\node at (4,4) {$\bullet$};
\node at (3,4) {$\bullet$};
\node at (3,3) {$\bullet$};
\node at (1,3) {$\bullet$};
\node at (1,2) {$\bullet$};
\node at (1,1) {$\bullet$};
\node at (0,1) {$\bullet$};
\node at (2,1) {$\bullet$};
\node at (3,1) {$\bullet$};
\node at (4,1) {$\bullet$};
\node at (0,0) {$\bullet$};
\node at (5,3) {$\bullet$};
\node at (5,2) {$\bullet$};
\node at (4,2) {$\bullet$};

\node at (5,4.2) {$v$};
\node [color=blue] at (0.23,-1.1) {$\bullet \ c_i$};
\node [color=blue] at (2.23,-0.1) {$\bullet \ c_j$};
\node at (3.4,4.3) {$K$};
\node at (3.3,1.3) {$L$};
\draw[->, thick,black] (4.9,4)--(4.1,4);
\draw[->, thick,black] (3.9,4)--(3.1,4);
\draw[->, thick,black] (3,3.9)--(3,3.1);
\draw[->, thick,black] (2.9,3)--(1.1,3);
\draw[->, thick,black] (1,2.9)--(1,2.1);
\draw[->, thick,black] (1,1.9)--(1,1.1);
\draw[->, thick,black] (0.9,1)--(0.1,1);
\draw[->, thick,black] (0,0.9)--(0,0.1);
\draw[->, thick,black] (0,-0.1)--(0,-0.9);
\draw[->, thick,black] (5,3.9)--(5,3.1);
\draw[->, thick,black] (5,2.9)--(5,2.1);
\draw[->, thick,black] (4.9,2)--(4.1,2);
\draw[->, thick,black] (4,1.9)--(4,1.1);
\draw[->, thick,black] (3.9,1)--(3.1,1);
\draw[->, thick,black] (2.9,1)--(2.1,1);
\draw[->, thick,black] (2,0.9)--(2,0.1);
  \end{tikzpicture}
  
   \footnotesize
    \emph{Postnikov diagram illustrating Lemma 5.12 (2)}

 \end{minipage}

\end{center} 
\begin{proof}
\begin{full} 
\begin{enumerate}[(1)]
\item This follows immediately from the fact that vertical edges have weight $1$.
\item Suppose that $L$ has a horizontal edge. Because of the Cauchon-Le condition, no vertex of $K$ lies (with respect to column coordinates) between the vertices of a horizontal edge of $L$ (see the beginning of the proof of Lemma \ref{how edges commute} for more details).

\textbf{Claim:} If $e$ is any horizontal edge of $L$ except the first horizontal edge of $L$, then $w(e)w(K)=w(K)w(e)$.

There are four possibilities for $e$:
\\ $\bullet$ {\bf Case (i)}: No vertex in $K$ shares a column coordinate with either vertex of $e$. In this case, Lemma \ref{how edges commute}(1) then shows that the weights of the horizontal edges of $K$ commute with $w(e)$, so that $w(e)$ commutes with $w(K)$.
\\ $\bullet$ {\bf Case (ii)}: There are two distinct horizontal edges $f'$ and $f''$ of $K$ such that $|\col(e)\cap \col(f')|=|\col(e)\cap \col(f'')|=1$, $\col_2(f')=\col_1(f'')=\col_2(e)$, and $\col(e)\cap \col(f)=\emptyset$ for all other edges $f$ of $K$.

\begin{center}
\begin{minipage}{45ex}
 \begin{tikzpicture}
\draw[color=gray] (0,2) rectangle (1,3); 
\draw[color=gray] (1,2) rectangle (2,3); 
\draw[color=gray] (2,2) rectangle (3,3); 
\draw[color=gray] (3,2) rectangle (4,3); 
\draw[color=gray] (0,1) rectangle (1,2); 
\draw[color=gray] (1,1) rectangle (2,2); 
\draw[color=gray] (2,1) rectangle (3,2); 
\draw[color=gray] (3,1) rectangle (4,2); 
\draw[color=gray] (0,0) rectangle (1,1); 
\draw[color=gray] (1,0) rectangle (2,1); 
\draw[color=gray] (2,0) rectangle (3,1); 
\draw[color=gray] (3,0) rectangle (4,1); 

\draw[fill=lightgray] (2,2) rectangle (3,3); 

\node at (3.5,2.5) {$\bullet$};
\node at (1.5,2.5) {$\bullet$};
\node at (1.5,1.5) {$\bullet$};
\node at (0.5,1.5) {$\bullet$};
\node at (1.5,0.5) {$\bullet$};
\node at (2.5,0.5) {$\bullet$};

\draw[->, thick,black] (3.4,2.5)--(1.6,2.5);
\draw[->, thick,black] (1.4,1.5)--(0.6,1.5);
\draw[->, thick,black] (2.4,0.5)--(1.6,0.5);
\draw[->, thick,black] (1.5,2.4)--(1.5,1.6);


\node at (2.5,2.75) {$f'$};
\node at (1.2,1.75) {$f''$};
\node at (2.2,0.75) {$e$};
 \end{tikzpicture}
 
 \footnotesize
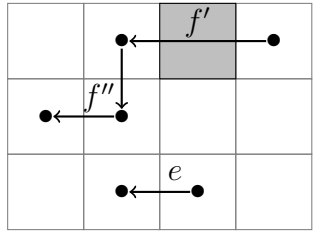
  \emph{Postnikov diagram illustrating Case (ii)}
 \end{minipage}
\end{center}

 In this case, Lemma \ref{how edges commute}(2) shows that $w(f')w(f'')w(e)=qq^{-1}w(e)w(f')w(f'')=w(e)w(f')w(f'')$. Now with Lemma \ref{how edges commute}(1), we can conclude that $w(e)w(K)=w(K)w(e)$.
\\  $\bullet$ {\bf Case (iii)}: There are two distinct horizontal edges $f'$ and $f''$ of $K$ such that $|\col(e)\cap \col(f')|=|\col(e)\cap \col(f'')|=1$, $\col_2(f')=\col_1(f'')=\col_1(e)$, and $\col(e)\cap \col(f)=\emptyset$ for all other edges of $K$. 

\begin{center}
\begin{minipage}{45ex}
 \begin{tikzpicture}
 \draw[fill=lightgray] (1,1) rectangle (2,3);
\draw[color=gray] (0,2) rectangle (1,3); 
\draw[color=gray] (1,2) rectangle (2,3); 
\draw[color=gray] (2,2) rectangle (3,3); 
\draw[color=gray] (3,2) rectangle (4,3); 
\draw[color=gray] (0,1) rectangle (1,2); 
\draw[color=gray] (1,1) rectangle (2,2); 
\draw[color=gray] (2,1) rectangle (3,2); 
\draw[color=gray] (3,1) rectangle (4,2); 
\draw[color=gray] (0,0) rectangle (1,1); 
\draw[color=gray] (1,0) rectangle (2,1); 
\draw[color=gray] (2,0) rectangle (3,1); 
\draw[color=gray] (3,0) rectangle (4,1); 

\node at (3.5,2.5) {$\bullet$};
\node at (2.5,2.5) {$\bullet$};
\node at (2.5,1.5) {$\bullet$};
\node at (0.5,1.5) {$\bullet$};
\node at (2.5,0.5) {$\bullet$};
\node at (1.5,0.5) {$\bullet$};

\draw[->, thick,black] (3.4,2.5)--(2.6,2.5);
\draw[->, thick,black] (2.4,1.5)--(0.6,1.5);
\draw[->, thick,black] (2.4,0.5)--(1.6,0.5);
\draw[->, thick,black] (2.5,2.4)--(2.5,1.6);


\node at (2.8,2.75) {$f'$};
\node at (2.2,1.75) {$f''$};
\node at (2.2,0.75) {$e$};
 \end{tikzpicture}
 
 \footnotesize
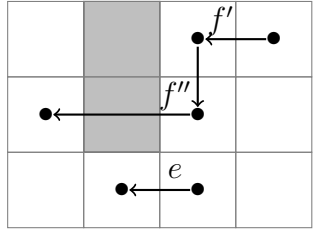
  \emph{Postnikov diagram illustrating Case (iii)}
 \end{minipage}
\end{center}

This case is similar to Case (ii).
\\ $\bullet$ {\bf Case (iv)}: There are edges $f',f'',f''$ of $K$ such that $|\col(e)\cap \col(f'')|=2$, $\col_2(f')=\col_1(e)$, and $\col_1(f''')=\col_2(e)$. 
\begin{center}
\begin{minipage}{45ex}
 \begin{tikzpicture}
\draw[color=gray] (0,3) rectangle (1,4); 
\draw[color=gray] (1,3) rectangle (2,4); 
\draw[color=gray] (2,3) rectangle (3,4); 
\draw[color=gray] (3,3) rectangle (4,4); 
 
\draw[color=gray] (0,2) rectangle (1,3); 
\draw[color=gray] (1,2) rectangle (2,3); 
\draw[color=gray] (2,2) rectangle (3,3); 
\draw[color=gray] (3,2) rectangle (4,3); 
\draw[color=gray] (0,1) rectangle (1,2); 
\draw[color=gray] (1,1) rectangle (2,2); 
\draw[color=gray] (2,1) rectangle (3,2); 
\draw[color=gray] (3,1) rectangle (4,2); 
\draw[color=gray] (0,0) rectangle (1,1); 
\draw[color=gray] (1,0) rectangle (2,1); 
\draw[color=gray] (2,0) rectangle (3,1); 
\draw[color=gray] (3,0) rectangle (4,1); 

\node at (3.5,3.5) {$\bullet$};
\node at (2.5,3.5) {$\bullet$};
\node at (2.5,2.5) {$\bullet$};
\node at (1.5,2.5) {$\bullet$};
\node at (1.5,1.5) {$\bullet$};
\node at (0.5,1.5) {$\bullet$};
\node at (2.5,0.5) {$\bullet$};
\node at (1.5,0.5) {$\bullet$};

\draw[->, thick,black] (3.4,3.5)--(2.6,3.5);
\draw[->, thick,black] (2.4,2.5)--(1.6,2.5);
\draw[->, thick,black] (2.5,3.4)--(2.5,2.6);
\draw[->, thick,black] (2.4,0.5)--(1.6,0.5);
\draw[->, thick,black] (1.5,2.4)--(1.5,1.6);
\draw[->, thick,black] (1.4,1.5)--(0.6,1.5);


\node at (2.8,3.75) {$f'$};
\node at (2.2,2.75) {$f''$};
\node at (1.25,1.75) {$f'''$};
\node at (2.2,0.75) {$e$};
 \end{tikzpicture}
 
 \footnotesize
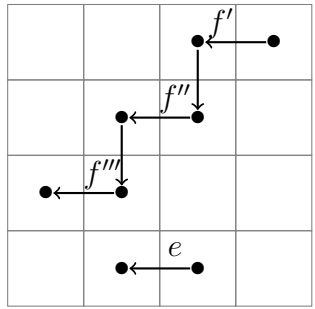
  \emph{Postnikov diagram illustrating Case (iv)}
 \end{minipage}
\end{center}

By Lemma \ref{how edges commute} parts (2) and (3), we have 
\[
w(f')w(f'')w(f''')w(e)=q^{-1}q^2q^{-1}w(e)w(f')w(f'')w(f''')=w(e)w(f')w(f'')w(f''')
\]
and now with Lemma \ref{how edges commute}(1), we can conclude that $w(e)w(K)=w(K)w(e)$. This establishes the claim that if $e$ is any horizontal edge of $L$ except the first horizontal edge of $L$, then $w(e)w(K)=w(K)w(e)$.

Let us turn now to the first horizontal edge $e_1$ of $L$. 

\textbf{Claim:} $w(K)w(e_1)=qw(e_1)w(K)$. 

Let us denote by $f_1$ the first horizontal edge of $K$. There are two cases to consider:

(i)  Let us suppose that $\col_2(f_1)<\col_2(e_1)$. Then Lemma \ref{how edges commute}(2) gives $w(f_1)w(e_1)=qw(e_1)w(f_1)$ and Lemma \ref{how edges commute}(1) allows us to conclude that $w(K)w(e_1)=qw(e_1)w(K)$. 

(ii) Let us suppose that $\col_2(f_1)=\col_2(e_1)$. Then the second horizontal edge $f_2$ of $K$ satisfies $\col_1(f_2)=\col_2(f_1)=\col_2(e_1)$. Then by Lemma \ref{how edges commute} parts (2) and (3), we have $w(f_1)w(f_2)w(e_1)=qw(e_1)w(f_1)w(f_2)$ and now with Lemma \ref{how edges commute}(1), we can conclude that $w(K)w(e_1)=qw(e_1)w(K)$. This completes the proof of the claim. (Note that the case $\col_2(f_1)>\col_2(e_1)$ cannot happen because Postnikov graphs are built from Cauchon-Le diagrams.)

The result now follows from the two claims which we have proven, namely that if $e$ is any horizontal edge of $L$ except the first horizontal edge of $L$, then $w(K)w(e)=w(e)w(K)$ and that if $e_1$ is the first horizontal edge of $L$, then $w(K)w(e_1)=qw(e_1)w(K)$.
\end{enumerate}
\end{full} 
\end{proof}
\end{lemma}


\subsection{Path systems}

\begin{definition}
Suppose that $I=\{i_1<\dots<i_t\}\subseteq \intint{1,m}$ and $J=\{j_1<\dots<j_t\}\subseteq \intint{1,n}$. An $R_{(I,J)}$-\emph{path system} in $\postc$ is a collection $\cp=(P_1,\ldots,P_t)$ of paths in $\postc$ starting respectively at the row vertices $r_{i_1},\ldots,r_{i_t}$ and ending respectively at the column vertices $c_{j_{\sigma_\cp(1)}},\ldots,c_{j_{\sigma_\cp(t)}}$ for some permutation $\sigma_\cp\in S_t$, called the {\em permutation of the path system $\cp$}.   
The path system $\cp$ is called \emph{vertex-disjoint} if no two of its paths share a vertex. 
The {\em weight of the path system $\cp=(P_1,\ldots,P_t)$} is defined simply as the ordered product $w(P_1)\cdots w(P_t)$ of the weights of the paths $P_1,\ldots,P_t$.
\end{definition}

\begin{example}\label{part ex'''}
{\rm 
Let $c=d=4$ and let $\lambda=(4,3,3,1)$. Below are the Cauchon-Le diagram on $Y_\lambda$ from Example \ref{partition-ex''}, with a vertex-disjoint $R_{(\{1,4\},\{1,4\})}$-path system marked in solid lines and a non vertex-disjoint $R_{(\{2,3\},\{2,3\})}$-path system marked in dotted and dashed lines; each path system has permutation $(1\ 2)\in S_2$. 
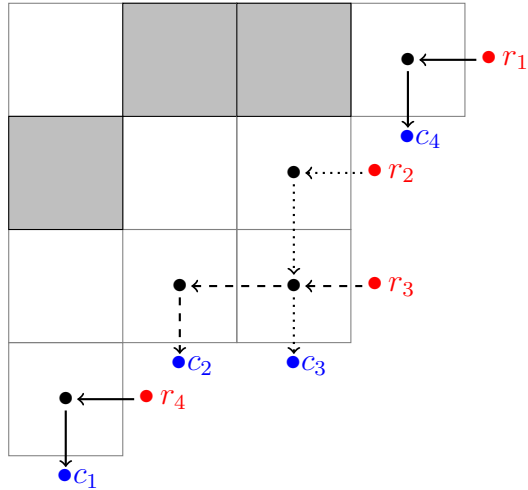
\begin{center}

   \begin{tikzpicture}[xscale=1.5, yscale=1.5]

\draw[color=gray] (0,3) rectangle (1,4);            
\draw[color=gray] (1,3) rectangle (2,4);            
\draw[color=gray] (2,3) rectangle (3,4);            
\draw[color=gray] (3,3) rectangle (4,4);            
\draw[color=gray] (0,2) rectangle (2,3);               
\draw[color=gray] (1,2) rectangle (1,3);               
\draw[color=gray] (2,2) rectangle (3,3);               
\draw[color=gray] (0,1) rectangle (2,2);               
\draw[color=gray] (1,1) rectangle (1,2);               
\draw[color=gray] (2,1) rectangle (3,2);               
\draw[color=gray] (0,0) rectangle (1,1);            

\draw[fill=lightgray] (1,3) rectangle (2,4);               
\draw[fill=lightgray] (0,2) rectangle (1,3);               
\draw[fill=lightgray] (2,3) rectangle (3,4);               

\node at (3.5, 3.5) {$\bullet$}; 
\node at (2.5, 2.5) {$\bullet$}; 
\node at (1.5, 1.5) {$\bullet$}; 
\node at (2.5, 1.5) {$\bullet$}; 
\node at (0.5, 0.5) {$\bullet$}; 

\node[color=red]  at (4.35,3.5) {$\bullet\, r_1$};
\node[color=red]  at (3.35,2.5) {$\bullet\, r_2$};
\node[color=red]  at (3.35,1.5) {$\bullet\, r_3$};
\node[color=red]  at (1.35,0.5) {$\bullet\, r_4$};

\node[color=blue]  at (0.5,-0.2) {$\hspace{0.32cm}\bullet c_1$};
\node[color=blue]  at (1.5,0.8) {$\hspace{0.32cm}\bullet c_2$};
\node[color=blue]  at (2.5,0.8) {$\hspace{0.32cm}\bullet c_3$};
\node[color=blue]  at (3.5,2.8) {$\hspace{0.32cm}\bullet c_4$};

\draw [<-,  thick,black] (3.6,3.5)--(4.1,3.5);
\draw [->, thick, dotted, black] (3.1,2.5)--(2.6,2.5);
\draw [<-, thick, dashed, black] (1.6,1.5)--(2.4,1.5);
\draw [<-, thick, dashed, black] (2.6,1.5)--(3.1,1.5);
\draw [<-, thick, black] (0.6,0.5)--(1.1,0.5);

\draw [->, thick, black] (0.5,0.4)--(0.5,-0.1);
\draw [->, thick, dashed, black] (1.5,1.4)--(1.5,0.9);
\draw [->, thick, dotted, black] (2.5,2.4)--(2.5,1.6);
\draw [->, thick, dotted, black](2.5,1.4)--(2.5,0.9);
\draw [->, thick, black] (3.5,3.4)--(3.5,2.9);

   \end{tikzpicture}
  \end{center}
}\end{example}

In the case that a Young diagram is rectangular, the permutation of a vertex disjoint 
path system must be the identity permutation. This needs not be the case for nonrectangular 
Young diagrams, and the following lemma deals with this problem. 


\begin{lemma}\label{same perm lemma}
Suppose that $I=\{i_1<\dots<i_t\}\subseteq \intint{1,c}$ and $J=\{j_1<\dots<j_t\}\subseteq \intint{1,d}$. Then all vertex-disjoint $R_{(I,J)}$-path systems in $\postc$ have the same permutation.
\end{lemma} 


\begin{proof} 
  \begin{full} 
The proof is by induction on $t$. The case $t=1$ is obvious; so, suppose that $t>1$ and that the result holds for vertex disjoint path systems of smaller size 
than $t$. 

 Choose $s$ as large as possible such that the Young diagram $Y_\lambda$ has a square in the $(i_s,j_t)$ position. Let $\cs$ denote any vertex-disjoint $R_{(I,J)}$-path system. We claim that the path $S_s$ in ${\cal S}$ starting at $r_{i_s}$ must finish at $c_{j_t}$. 

Suppose, for a contradiction, that the path $S_s$ in $\cs$ starting at $r_{i_s}$ does not end at $c_{j_t}$. 
 Let $l$ be such that the path $S_l$ in $\cs$ starting at $r_{i_l}$ ends at $c_{j_t}$, forcing $l<s$ because of the maximality of $s$. Suppose that $S_s$ ends at $c_{j_u}$, and note that $u<t$. As $\postc$ is planar, the paths 
$S_s: r_{i_s}\implies c_{j_u}$ and $S_l: r_{i_l}\implies c_{j_t}$ must cross, as
$r_{i_s}$ is to the right of the path $S_l$ and $c_{j_u}$ is to the left of $S_l$; this crossing must occur at a vertex by Proposition \ref{basic facts}(2). This gives the desired contradiction and proves the claim that the path $S_s$ in ${\cal S}$ starting at $r_{i_s}$ must finish at $c_{j_t}$.  

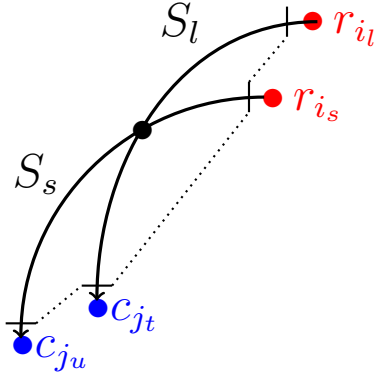
\begin{center}

\begin{tikzpicture}[xscale=1, yscale=1]
\node[scale=1.5, color=red] at (4.7,4.4) {$\bullet\, r_{i_l}$};
\node[scale=1.5, color=red] at (4.2,3.4) {$\bullet\, r_{i_s}$};

\node[color=blue, scale=1.5] at (0.65,0.1) {$\hspace{0.32cm}\bullet c_{j_u}$};
\node[color=blue, scale=1.5] at (1.6,0.6) {$\hspace{0.32cm}\bullet c_{j_t}$};

\draw[black,thick] (4.0,4.3) -- (4.0,4.7);
\draw[black,thick] (3.5,3.3) -- (3.5,3.7);
\draw[black,thick] (1.3,1.0) -- (1.7,1.0);
\draw[black,thick] (0.3,0.5) -- (0.7,0.5);

\draw[->,very thick, black] (4.4,4.5) to [out=180,in=90] (1.5,0.8);
\draw[->,very thick, black] (3.7,3.5) to [out=180,in=90] (0.5,0.3);

\draw[black,thick,dotted] (3.5,3.7) -- (4.0,4.3);
\draw[black,thick,dotted] (1.7,1.0) -- (3.5,3.3);
\draw[black,thick,dotted] (0.7,0.5) -- (1.3,1.0);

\node [scale=1.5, black, above] at (0.7,2.0) {$S_s$};
\node [scale=1.5, black, above] at (2.6,4) {$S_l$};

\node [scale=1.5, black] at (2.1,3.05) {$\bullet$};

\end{tikzpicture}
\end{center}

Consider any two vertex-disjoint $R_{(I,J)}$-path systems $\cp=(P_1,\ldots,P_t)$ and $\mathcal{Q}=(Q_1,\ldots,Q_t)$ in $\postc$. The paths $P_s$ and $Q_s$ which start at $r_{i_s}$ must finish at $c_{j_t}$. Now $\cp\bs\{P_s\}$ and $\mathcal{Q}\bs\{Q_s\}$ are two vertex-disjoint $R_{(I\bs\{i_s\},J\bs\{j_t\})}$-path systems and hence must have the same permutation by the induction hypothesis. The result follows immediately.
  \end{full} 
\end{proof}


\begin{lemma}[cf. Lemma 4.2 of \cite{cas1}]\label{adjacent paths share vertex lemma}
Suppose that $I=\{i_1<\dots<i_t\}\subseteq \intint{1,m}$ and $J=\{j_1<\dots<j_t\}\subseteq \intint{1,n}$. If $\cp=(P_1,\ldots,P_t)$ is a non-vertex-disjoint $R_{(I,J)}$-path system in $\postc$, then there exists $s\in \intint{1,t-1}$ such that $P_s$ and $P_{s+1}$ share a vertex. 
\end{lemma} 


\begin{proof}
  \begin{full} 
Let $d=\min \{|a-b|\ |\ a\neq b \tx{ and }P_a\tx{ and }P_b \tx{ share a vertex}\}$ and suppose that $d>1$. Let $a<b$ be such that $|a-b|=d$ and $P_a$ shares a vertex with $P_b$. Let $x$ be the first vertex which is common to $P_a$ and $P_b$ and consider the subpaths $P_a': r_{i_a}\implies x$ of $P_a$ and $P_b': r_{i_b}\implies x$ of $P_b$.   

Since $d>1$, there exists $\ell\in \intint{1,t}$ such that $a<\ell<b$. The path $P_\ell\in \cp$ which starts at $r_{i_\ell}$ must intersect either $P_a'$ or $P_b'$ and this intersection must be at a vertex of $\postc$ by Proposition \ref{basic facts}(2), contradicting the minimality of $d$.

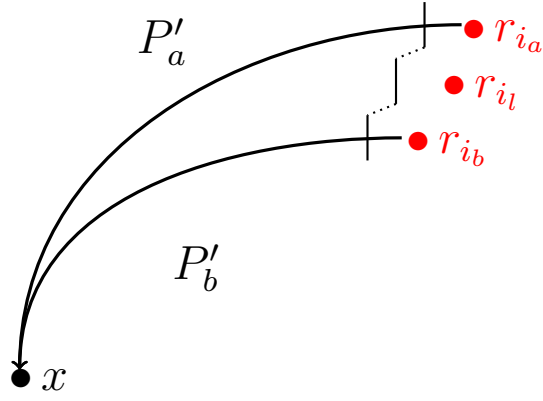
\begin{center}
\begin{tikzpicture}[xscale=1.5, yscale=1.5]
\node[scale=1.5, color=red] at (4.7,4.4) {$\bullet\, r_{i_a}$};
\node[scale=1.5, color=red] at (4.2,3.4) {$\bullet\, r_{i_b}$};
\node[scale=1.5, color=red] at (4.5,3.9) {$\bullet\, r_{i_l}$};

\draw[black,thick] (4.0,4.3) -- (4.0,4.7);
\draw[black,thick] (3.75,3.8) -- (3.75,4.2);
\draw[black,thick] (3.5,3.3) -- (3.5,3.7);

\draw[->, very thick, black] (4.325,4.5) to [out=180,in=90] (0.45,1.45);
\draw[->, very thick, black] (3.8,3.5) to [out=180,in=90] (0.45,1.45);

\draw[black,thick,dotted] (4.0,4.3) -- (3.75,4.2);
\draw[black,thick,dotted] (3.75,3.8) -- (3.5,3.7);

\node [scale=1.5, black, above] at (0.6125,1.125) {$\bullet\, x$};

\node [scale=1.5, black, above] at (1.7,4) {$P_a'$};
\node [scale=1.5, black, above] at (2,2) {$P_b'$};
\end{tikzpicture}
\end{center}
 
  \end{full} 
\end{proof}


\subsection{Path matrices and their pseudo quantum minors}
\label{subsection:pathmatrix}

\begin{definition}
Let $\lambda=(\lambda_1,\ldots,\lambda_c)$ be  a partition with $d= \lambda_1\geq\cdots\geq \lambda_c\geq 1$, where $c\leq m$ and $d\leq n$. 
Define the \emph{path matrix} $M_C=(M_C[i,j])$ of $C$ to be the $c\times d$ matrix with entries from the quantum torus associated with $C$ such that for each $(i,j)\in \llb 1,c\rrb\times\llb 1,d\rrb$, the entry $M_C[i,j]$ is the sum of the weights of all paths from $r_i$ to $c_j$ in $\postc$. For $I=\{i_1<\ldots<i_t\}\subseteq \llb 1,c \rrb$ and $J=\{j_1<\ldots<j_t\}\subseteq \llb 1,d\rrb$, we define the 
pseudo quantum minor $[I\ |\ J]$ of $M_C$ as follows 
\[
 [I\ |\ J]:=\sum_{\sigma\in S_t}(-q)^{\ell(\sigma)}M_C[i_1,j_{\sigma(1)}]\cdots M_C[i_t,j_{\sigma(t)}].
\]
\end{definition}

In the case of a rectangular Young diagram, \cite[Theorem 4.4]{cas1} shows that the quantum minor $[I\ |\ J]$ of $M_C$ is equal to the sum of the weights of all vertex-disjoint $R_{(I,J)}$-path systems in $\postc$. In the nonrectangular case, we have to take account of the permutation associated with the vertex disjoint path systems (cf. Lemma~\ref{same perm lemma}), to obtain the next theorem. 
 

\begin{theorem}[cf. Theorem 4.4 of \cite{cas1}]\label{computing q det with paths thm}
Let $I\subseteq \llb 1,c\rrb$ and $J\subseteq \llb 1,d\rrb$ have the same cardinality and let $\sigma_{(I,J)}$ be the permutation of all vertex-disjoint $R_{(I,J)}$-path systems (see Lemma \ref{same perm lemma}). Then the pseudo quantum minor $[I\ |\ J]$ of $M_C$ is given by 
\begin{equation}\label{sum of vd paths}
[I\ |\ J]=(-q)^{\ell(\sigma_{(I,J)})}\sum_{\cp}w(\cp),
\end{equation}
where $\cp$ runs over all vertex-disjoint $R_{(I,J)}$-path systems in $\postc$. In particular, if there are no vertex-disjoint $R_{(I,J)}$-path systems in $\postc$, then $[I\ |\ J]=0$. 
\end{theorem} 


\begin{proof}
  \begin{full} 
For ease of notation, let us take $I=J=\{1,\ldots,t\}$ (the proof for general $I$ and $J$ is the same but the notation is more unwieldy). 

By the definition of the path matrix, we have 
\begin{align*}
[I\ |\ J]&=\sum_{\sigma\in S_t}(-q)^{\ell(\sigma)}M_C[1,\sigma(1)]\cdots M_C[t,\sigma(t)] \\
         &=\sum_{\sigma\in S_t}(-q)^{\ell(\sigma)}\left( \sum_{P_1: r_1\implies c_{\sigma(1)}}w(P_1) \right)\left( \sum_{P_2: r_2\implies c_{\sigma(2)}}w(P_2) \right)\cdots \left( \sum_{P_t: r_t\implies c_{\sigma(t)}}w(P_t) \right) \\
         &=\sum_{\cp} (-q)^{\ell(\sigma_\cp)}w(\cp), \\
\end{align*} 
where, in the final sum, $\cp$ runs over all $R_{(I,J)}$-path systems. 

We claim that
\begin{equation}\label{cancellation}
\sum_{\cp\in \cn}(-q)^{\ell(\sigma_\cp)}w(\cp)=0.
\end{equation}
where $\cn$ is the set of non-vertex-disjoint $(R_I,C_J)$-path systems. 

To show that \eqref{cancellation} holds, we construct a fixed-point-free involution $\pi: \cn\to \cn$ which satisfies
\begin{equation}\label{pi inv minus}
(-q)^{\ell(\sigma_\cp)}w(\cp)=-(-q)^{\ell(\sigma_{\pi(\cp)})}w(\pi(\cp))
\end{equation}
for every $\cp\in \cn$. 

Let $\cp=(P_1,\ldots,P_t)\in \cn$ and let $i$ be minimal such that $P_i$ and $P_{i+1}$ share a vertex (this $i$ exists by Lemma \ref{adjacent paths share vertex lemma}). Let $x$ be the last vertex shared by $P_i$ and $P_{i+1}$ and let $K_1: r_i\implies x$ and $L_1: x\implies c_{\sigma_\cp(i)}$ be subpaths of $P_i$ so that $P_i=K_1L_1$; define $K_2$ and $L_2$ from $P_{i+1}$ similarly. For any $j\in \llb 1,t\rrb$, set
\[
\piecewiseseven{\pi(P_j):}{K_1L_2}{j=i}{K_2L_1}{j=i+1}{P_j}{\tx{otherwise}}
\]
(see Example \ref{action of pi} below for an example of the action of $\pi$).
Define $\pi(\cp)$ to be the $R_{(I,J)}$-path system $(\pi(P_1),\ldots,\pi(P_t))$. 
This gives us a map $\pi: \cn\to \cn$ which is clearly an involution and which clearly has no fixed points. In order to prove \eqref{pi inv minus}, we may assume without loss of generality that $\sigma_\cp(i)<\sigma_\cp(i+1)$, so that $\sigma_{\pi(\cp)}=\sigma_\cp(i\ i+1)$ satisfies $\ell(\sigma_{\pi(\cp)})=\ell(\sigma_\cp)+1$. Notice that because $x$ is the last vertex shared by $P_i$ and $P_{i+1}$, the assumption $\sigma_\cp(i)<\sigma_\cp(i+1)$ forces $L_1$ to start with a horizontal edge.

We claim that $w(P_i)w(P_{i+1})=qw(\pi(P_i))w(\pi(P_{i+1}))$. There are two cases to consider:
\begin{enumerate}[(i)]
\item Suppose that $L_2$ has a horizontal edge. Then 
\begin{align*}
w(P_i)w(P_{i+1})&=w(K_1)w(L_1)w(K_2)w(L_2) \\
                &=qw(K_1)w(K_2)w(L_1)w(L_2)\ \ \ \ \ (\tx{Lemma \ref{how the head an tail commute lemma}}) \\
                &=q^2w(K_1)w(K_2)w(L_2)w(L_1)\ \ \ \ \ (\tx{Lemma \ref{how paths that share only initial vertex commute lemma}}) \\
                &=q^2q^{-1}w(K_1)w(L_2)w(K_2)w(L_1)\ \ \ \ \ (\tx{Lemma \ref{how the head an tail commute lemma}}) \\
                &=qw(\pi(P_i))w(\pi(P_{i+1})). \\
\end{align*}
\item Suppose that $L_2$ consists of vertical edges. Then $w(L_2)=1$ and we have 
\begin{align*}
w(P_i)w(P_{i+1})&=w(K_1)w(L_1)w(K_2)w(L_2) \\
                &=w(K_1)w(L_2)w(L_1)w(K_2) \\
                &=qw(K_1)w(L_2)w(K_2)w(L_1)\ \ \tx{(Lemma \ref{how the head an tail commute lemma})} \\
                &=qw(\pi(P_i))w(\pi(P_{i+1})).
\end{align*}
\end{enumerate}
Hence we have 
\begin{align*}
w(\cp)&= \left(\prod_{j=1}^{i-1}w(P_j)\right)w(P_i)w(P_{i+1})\left(\prod_{j=i+2}^{t}w(P_j)\right) \\
         &=\left(\prod_{j=1}^{i-1}w(\pi(P_j))\right)qw(\pi(P_i))w(\pi(P_{i+1}))\left(\prod_{j=i+2}^{t}w(\pi(P_j))\right) \\
         &=qw(\pi(\cp)). \\
\end{align*}
Now 
\begin{align*}
(-q)^{\ell(\sigma_\cp)}w(\cp)&=(-q)^{\ell(\sigma_\cp)}qw(\pi(\cp)) \\
                                &=-(-q)^{\ell(\sigma_\cp)+1}w(\pi(\cp)) \\
                                &=-(-q)^{\ell(\sigma_{\pi(\cp)})}w(\pi(\cp)), \\
\end{align*}
proving that $\pi: \cn\to\cn$ satisfies \eqref{pi inv minus}; the claim \eqref{cancellation} follows immediately. 
Moreover, the claim \eqref{cancellation} immediately gives 
\[
[I\ |\ J]=\sum_\cp(-q)^{\ell(\sigma_\cp)}w(\cp),
\]
where $\cp$ runs over all vertex-disjoint $R_{(I,J)}$-path systems in $\postc$. Lemma \ref{same perm lemma} shows that $\sigma_\cp=\sigma_{(I,J)}$ for all such $\cp$, giving the result.
  \end{full} 
\end{proof}



\begin{example}\label{action of pi}
Below left is an example of a non-vertex-disjoint $R_{(\{1,2\},\{1,3\})}$-path system $\cp=(P_1,P_2)$ on the Postnikov graph of a Cauchon-Le diagram. Below right is the non-vertex-disjoint $R_{(\{1,2\},\{1,3\})}$-path system $\pi(\cp)=(\pi(P_1),\pi(P_2))$.

\begin{center}
 \begin{minipage}{0.4675\textwidth}
  \begin{center}
   \begin{tikzpicture}[xscale=0.9,yscale=0.9]

\draw[color=gray] (0,4) rectangle (1,5); 
\draw[color=gray] (1,4) rectangle (2,5); 
\draw[color=gray] (2,4) rectangle (3,5); 
\draw[color=gray] (3,4) rectangle (4,5); 
\draw[color=gray] (4,4) rectangle (5,5); 
\draw[color=gray] (5,4) rectangle (6,5); 

\draw[color=gray] (0,3) rectangle (1,4);            
\draw[color=gray] (1,3) rectangle (2,4);            
\draw[color=gray] (2,3) rectangle (3,4);            
\draw[color=gray] (3,3) rectangle (4,4);            
\draw[color=gray] (4,3) rectangle (5,4);            
\draw[color=gray] (5,3) rectangle (6,4);            

\draw[color=gray] (0,2) rectangle (1,3);               
\draw[color=gray] (1,2) rectangle (2,3);               
\draw[color=gray] (2,2) rectangle (3,3);               
\draw[color=gray] (3,2) rectangle (4,3);               
\draw[color=gray] (4,2) rectangle (5,3);               

\draw[color=gray] (0,1) rectangle (1,2);               
\draw[color=gray] (1,1) rectangle (2,2);               
\draw[color=gray] (2,1) rectangle (3,2);               
\draw[color=gray] (3,1) rectangle (4,2);               
\draw[color=gray] (4,1) rectangle (5,2);               

\draw[color=gray] (0,0) rectangle (1,1);            
\draw[color=gray] (1,0) rectangle (2,1);            
\draw[color=gray] (2,0) rectangle (3,1);            
\draw[color=gray] (3,0) rectangle (4,1);            
\draw[color=gray] (4,0) rectangle (5,1);            

\draw[fill=lightgray] (4,4) rectangle (5,5); 
\draw[fill=lightgray] (5,4) rectangle (6,5); 
\draw[fill=lightgray] (2,4) rectangle (3,5); 
\draw[fill=lightgray] (2,3) rectangle (3,4);            
\draw[fill=lightgray] (2,2) rectangle (3,3);               
\draw[fill=lightgray] (5,3) rectangle (6,4); %

\node at (1.5, 2.5) {$\bullet$}; 
\node at (4.5, 3.5) {$\bullet$}; 
\node at (4.5, 2.5) {$\bullet$}; 
\node at (3.5, 2.5) {$\bullet$}; 
\node at (3.5, 3.5) {$\bullet$}; 
\node at (3.5, 1.5) {$\bullet$}; 
\node at (3.5, 0.5) {$\bullet$}; 
\node at (3.5, 4.5) {$\bullet$}; 
\node at (2.5, 0.5) {$\bullet$}; 
\node at (0.5, 1.5) {$\bullet$}; 
\node at (1.5, 1.5) {$\bullet$}; 
\node at (0.5, 0.5) {$\bullet$}; 

\node[color=red]  at (6.5,4.5) {$\bullet\, r_1$};
\node[color=red]  at (6.5,3.5) {$\bullet\, r_2$};

\node[color=blue]  at (0.5,-0.4) {$\hspace{0.4cm}\bullet c_1$};
\node[color=blue]  at (2.5,-0.4) {$\hspace{0.4cm}\bullet c_3$};

\draw [<-, thick, black] (0.6,1.5)--(1.4,1.5);
\draw [->, thick, black] (3.4,2.5)--(1.6,2.5);
\draw [->, thick, black] (0.5,1.4)--(0.5,0.6);
\draw [->, thick, black] (0.5,0.4)--(0.5,-0.3);
\draw [->, thick, black] (1.5,2.4)--(1.5,1.6);
\draw [->, thick, black] (3.5,4.4)--(3.5,3.6);
\draw [->, thick, black] (3.5,3.4)--(3.5,2.6);
\draw [->, thick, black] (6.2,4.5)--(3.6,4.5);

\draw [->, thick, dashed, black] (6.2,3.5)--(4.6,3.5);
\draw [->, thick, dashed, black] (4.5,3.4)--(4.5,2.6);
\draw [->, thick, dashed, black] (4.4,2.5)--(3.6,2.5);
\draw [->, thick, dashed, black] (3.5,2.4)--(3.5,1.6);
\draw [->, thick, dashed, black] (3.5,1.4)--(3.5,0.6);
\draw [->, thick, dashed, black] (3.4,0.5)--(2.6,0.5);
\draw [->, thick, dashed, black] (2.5,0.4)--(2.5,-0.3);

\node at (3.25,2.75) {$x$};
\node at (5.5,4.75) {$P_1$};
\node at (5.5,3.75) {$P_2$};

\node[color=blue]  at (0,1) {$\hspace{1cm}$};

   \end{tikzpicture}
   
    \footnotesize
    \emph{$P_1$ marked with straight lines. $P_2$ marked with dashed lines.}
  \end{center}
 \end{minipage}\qquad
 \begin{minipage}{0.4675\textwidth}
  \begin{center}
   \begin{tikzpicture}[xscale=0.9,yscale=0.9]


\draw[color=gray] (0,4) rectangle (1,5); 
\draw[color=gray] (1,4) rectangle (2,5); 
\draw[color=gray] (2,4) rectangle (3,5); 
\draw[color=gray] (3,4) rectangle (4,5); 
\draw[color=gray] (4,4) rectangle (5,5); 
\draw[color=gray] (5,4) rectangle (6,5); 

\draw[color=gray] (0,3) rectangle (1,4);            
\draw[color=gray] (1,3) rectangle (2,4);            
\draw[color=gray] (2,3) rectangle (3,4);            
\draw[color=gray] (3,3) rectangle (4,4);            
\draw[color=gray] (4,3) rectangle (5,4);            
\draw[color=gray] (5,3) rectangle (6,4);            

\draw[color=gray] (0,2) rectangle (1,3);               
\draw[color=gray] (1,2) rectangle (2,3);               
\draw[color=gray] (2,2) rectangle (3,3);               
\draw[color=gray] (3,2) rectangle (4,3);               
\draw[color=gray] (4,2) rectangle (5,3);               

\draw[color=gray] (0,1) rectangle (1,2);               
\draw[color=gray] (1,1) rectangle (2,2);               
\draw[color=gray] (2,1) rectangle (3,2);               
\draw[color=gray] (3,1) rectangle (4,2);               
\draw[color=gray] (4,1) rectangle (5,2);               

\draw[color=gray] (0,0) rectangle (1,1);            
\draw[color=gray] (1,0) rectangle (2,1);            
\draw[color=gray] (2,0) rectangle (3,1);            
\draw[color=gray] (3,0) rectangle (4,1);            
\draw[color=gray] (4,0) rectangle (5,1);            

\draw[fill=lightgray] (4,4) rectangle (5,5); 
\draw[fill=lightgray] (5,4) rectangle (6,5); 
\draw[fill=lightgray] (2,4) rectangle (3,5); 
\draw[fill=lightgray] (2,3) rectangle (3,4);            
\draw[fill=lightgray] (2,2) rectangle (3,3);               
\draw[fill=lightgray] (5,3) rectangle (6,4); %

\node at (1.5, 2.5) {$\bullet$}; 
\node at (4.5, 3.5) {$\bullet$}; 
\node at (4.5, 2.5) {$\bullet$}; 
\node at (3.5, 2.5) {$\bullet$}; 
\node at (3.5, 3.5) {$\bullet$}; 
\node at (3.5, 1.5) {$\bullet$}; 
\node at (3.5, 0.5) {$\bullet$}; 
\node at (3.5, 4.5) {$\bullet$}; 
\node at (2.5, 0.5) {$\bullet$}; 
\node at (0.5, 1.5) {$\bullet$}; 
\node at (1.5, 1.5) {$\bullet$}; 
\node at (0.5, 0.5) {$\bullet$}; 

\node[color=red]  at (6.5,4.5) {$\bullet\, r_1$};
\node[color=red]  at (6.5,3.5) {$\bullet\, r_2$};

\node[color=blue]  at (0.5,-0.4) {$\hspace{0.4cm}\bullet c_1$};
\node[color=blue]  at (2.5,-0.4) {$\hspace{0.4cm}\bullet c_3$};

\draw [->, thick, black] (3.5,4.4)--(3.5,3.6);
\draw [->, thick, black] (3.5,3.4)--(3.5,2.6);
\draw [->, thick, black] (6.2,4.5)--(3.6,4.5);

\draw [->, thick, black] (3.5,2.4)--(3.5,1.6);
\draw [->, thick,  black] (3.5,1.4)--(3.5,0.6);
\draw [->, thick,  black] (3.4,0.5)--(2.6,0.5);
\draw [->, thick,  black] (2.5,0.4)--(2.5,-0.3);

\draw [->, thick, dashed, black] (6.2,3.5)--(4.6,3.5);
\draw [->, thick, dashed, black] (4.5,3.4)--(4.5,2.6);
\draw [->, thick, dashed, black] (4.4,2.5)--(3.6,2.5);
\draw [->, thick, dashed, black] (3.4,2.5)--(1.6,2.5);
\draw [->, thick, dashed, black] (1.5,2.4)--(1.5,1.6);
\draw [<-, thick, dashed, black] (0.6,1.5)--(1.4,1.5);%
\draw [->, thick, dashed, black] (0.5,1.4)--(0.5,0.6);
\draw [->, thick, dashed, black] (0.5,0.4)--(0.5,-0.3);

\node at (3.25,2.75) {$x$};
\node at (5.5,4.75) {$\pi(P_1)$};
\node at (5.5,3.75) {$\pi(P_2)$};

   \end{tikzpicture}
   
    \footnotesize
    \emph{$\pi(P_1)$ marked with straight lines. $\pi(P_2)$ marked with dashed lines.}
  \end{center}
 \end{minipage} \\~\\
\end{center}
\end{example}



\begin{example}
Let $c=d=4$, let $\lambda=(4,3,3,1)$, and let $C$ be the Cauchon-Le diagram on $Y_\lambda$ from 
Example~\ref{partition-ex''}. Below are $C$ and $\postc$:
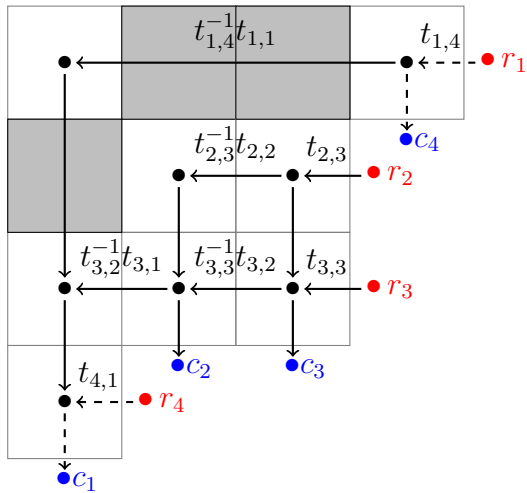
  \begin{center}
   \begin{tikzpicture}[xscale=1.5, yscale=1.5]

\draw[color=gray] (0,3) rectangle (1,4);            
\draw[color=gray] (1,3) rectangle (2,4);            
\draw[color=gray] (2,3) rectangle (3,4);            
\draw[color=gray] (3,3) rectangle (4,4);            
\draw[color=gray] (0,2) rectangle (2,3);               
\draw[color=gray] (1,2) rectangle (1,3);               
\draw[color=gray] (2,2) rectangle (3,3);               
\draw[color=gray] (0,1) rectangle (2,2);               
\draw[color=gray] (1,1) rectangle (1,2);               
\draw[color=gray] (2,1) rectangle (3,2);               
\draw[color=gray] (0,0) rectangle (1,1);            

\draw[fill=lightgray] (1,3) rectangle (2,4);               
\draw[fill=lightgray] (0,2) rectangle (1,3);               
\draw[fill=lightgray] (2,3) rectangle (3,4);               

\node at (0.5, 3.5) {$\bullet$}; 
\node at (3.5, 3.5) {$\bullet$}; 
\node at (1.5, 2.5) {$\bullet$}; 
\node at (2.5, 2.5) {$\bullet$}; 
\node at (0.5, 1.5) {$\bullet$}; 
\node at (1.5, 1.5) {$\bullet$}; 
\node at (2.5, 1.5) {$\bullet$}; 
\node at (0.5, 0.5) {$\bullet$}; 

\node[color=red]  at (4.35,3.5) {$\bullet\, r_1$};
\node[color=red]  at (3.35,2.5) {$\bullet\, r_2$};
\node[color=red]  at (3.35,1.5) {$\bullet\, r_3$};
\node[color=red]  at (1.35,0.5) {$\bullet\, r_4$};

\node[color=blue]  at (0.5,-0.2) {$\hspace{0.32cm}\bullet c_1$};
\node[color=blue]  at (1.5,0.8) {$\hspace{0.32cm}\bullet c_2$};
\node[color=blue]  at (2.5,0.8) {$\hspace{0.32cm}\bullet c_3$};
\node[color=blue]  at (3.5,2.8) {$\hspace{0.32cm}\bullet c_4$};

\draw [<-, thick, black] (0.6,3.5)--(3.4,3.5);
\draw [<-, thick, dashed, black] (3.6,3.5)--(4.1,3.5);
\draw [<-, thick, black] (2.6,2.5)--(3.1,2.5);
\draw [<-, thick, black] (1.6,2.5)--(2.4,2.5);
\draw [<-, thick, black] (0.6,1.5)--(1.4,1.5);
\draw [<-, thick, black] (1.6,1.5)--(2.4,1.5);
\draw [<-, thick, black] (2.6,1.5)--(3.1,1.5);
\draw [<-, thick, dashed, black] (0.6,0.5)--(1.1,0.5);

\draw [->, thick, black] (0.5,3.4)--(0.5,1.6);
\draw [->, thick, black] (0.5,1.4)--(0.5,0.6);
\draw [->, thick, dashed, black] (0.5,0.4)--(0.5,-0.1);
\draw [->, thick, black] (1.5,2.4)--(1.5,1.6);
\draw [->, thick, black] (1.5,1.4)--(1.5,0.9);
\draw [->, thick, black] (2.5,2.4)--(2.5,1.6);
\draw [->, thick, black] (2.5,1.4)--(2.5,0.9);
\draw [->, thick, dashed, black] (3.5,3.4)--(3.5,2.9);

\node [above] at (3.8,3.5) {$t_{1,4}$};
\node [above] at (2,3.5) {$t_{1,4}^{-1}t_{1,1}$};
\node [above] at (2,2.5) {$t_{2,3}^{-1}t_{2,2}$};
\node [above] at (2.8,2.5) {$t_{2,3}$};
\node [above] at (1,1.5) {$t_{3,2}^{-1}t_{3,1}$};
\node [above] at (2,1.5) {$t_{3,3}^{-1}t_{3,2}$};
\node [above] at (2.8,1.5) {$t_{3,3}$};
\node [above] at (0.8,0.5) {$t_{4,1}$};

   \end{tikzpicture}
  \end{center}

The only vertex-disjoint $R_{(\{1,4\},\{1,4\})}$-path system in $\postc$ is that which is marked with dashed  lines above; this path system has weight $t_{1,4}t_{4,1}$ and has permutation $(1,2)\in S_2$, whose length is $1$. Theorem~\ref{computing q det with paths thm} predicts that the 
pseudo quantum minor $[14\ |\ 14]$ of $M_C$ is $-qt_{1,4}t_{4,1}$.
Computing the pseudo quantum minor $[14\ |14]$ of $M_C$ directly, we indeed get
\[
[14\ |\ 14]=M_C[1,1]M_C[4,4]-qM_C[4,1]M_C[1,4]=0-qt_{4,1}t_{1,4}=-qt_{1,4}t_{4,1}.
\] 

There are no vertex-disjoint $R_{(\{2,3\},\{1,2\})}$-path systems in $\postc$, so that Theorem \ref{computing q det with paths thm} predicts that the quantum minor $[23\ |\ 12]$ of $M_C$ is zero. Computing the quantum minor $[23\ |\ 12]$ of $M_C$ directly, we indeed get
\begin{align*}
[23\ |\ 12]&=M_C[2,1]M_C[3,2]-qM_C[3,1]M_C[2,2] \\
           &=(t_{2,2}t_{3,2}^{-1}t_{3,1}+t_{2,3}t_{3,3}^{-1}t_{3,1})(t_{3,2})-q(t_{3,1})(t_{2,3}t_{3,3}^{-1}t_{3,2}+t_{2,2}) \\
           &=t_{2,2}t_{3,2}^{-1}t_{3,1}t_{3,2}+t_{2,3}t_{3,3}^{-1}t_{3,1}t_{3,2}-qt_{3,1}t_{2,3}t_{3,3}^{-1}t_{3,2}-qt_{3,1}t_{2,2} \\
           &= qt_{2,2}t_{3,1} +qt_{3,1}t_{2,3}t_{3,3}^{-1}t_{3,2}-qt_{3,1}t_{2,3}t_{3,3}^{-1}t_{3,2}-qt_{3,1}t_{2,2} \\
           &=0.
\end{align*}
\end{example}

\subsection{Vanishing of pseudo quantum minors of a path matrix}

\begin{definition}[cf. Definition 3.1.7 of \cite{cas2}]
{\rm Let $v\in W_C$ be a vertex of a path $P: r_i\implies c_j$ in $\postc$. Let $e$ be the edge of $P$ which ends at $v$ and let $f$ be the edge of $P$ which begins at $v$. Then we say that $v$ is a \emph{$\Gamma$-turn} of $P$ (or that $P$ has a \emph{$\Gamma$-turn} at $v$) if $e$ is horizontal and $f$ is vertical and that $v$ is a \emph{$\Le$-turn} of $P$ (or that $P$ has a $\Le$-turn at $v$) if $e$ is vertical and $f$ is horizontal.}
\end{definition}


\begin{proposition}[cf. Proposition 3.1.8 of \cite{cas2}]\label{only need the turns proposition}
Let $\cp: r_i\implies c_j$ be a path in $\postc$. If $v_1,v_2,\ldots,v_s$ is the sequence of all $\Gamma$-turns and $\Le$-turns in $\cp$, then $v_a$ is a $\Gamma$-turn for odd values of $a$ and a $\Le$-turn for even values of $a$, $s$ is odd, and 
\[
w(\cp)=t_{v_1}t_{v_2}^{-1}t_{v_3}\cdots t_{v_{s-1}}^{-1}t_{v_s}.
\]
\end{proposition} 

\begin{proof}
  \begin{full} 
It is clear that $v_a$ is a $\Gamma$-turn for $a$ odd and a $\Le$-turn for $a$ even. Since $\cp$ ends with a vertical edge, $s$ must be odd. Consider the subpaths:
\[ P_1: r_i\implies v_1,\ P_2: v_1\implies v_2,\ \ldots,\ P_s: v_{s-1}\implies v_s,\ P_{s+1}: v_s\implies c_j 
\]
of $\cp$. For $a\in \llb 1,s+1\rrb$ even (that is for $a=2,4,\ldots,s-1,s+1$), the path $P_a$ consists solely of vertical edges and hence $w(P_a)=1$. It follows that

\begin{align*}
w(\cp)&=w(P_1)w(P_2)\cdots w(P_s)w(P_{s+1}) \\
    &=w(P_1)w(P_3)\cdots w(P_{s-2})w(P_s) \\
    &=t_{v_1}w(P_3)\cdots w(P_{s-2})w(P_s) \ \ \ \tx{(by Proposition \ref{basic facts}(4))}.
\end{align*}
However, $P_3,\ldots P_{s-2},P_s$ are internal horizontal paths in $\postc$ and their respective weights are 
$t_{v_2}^{-1}t_{v_3},\ldots,t_{v_{s-3}}^{-1}t_{v_{s-2}},t_{v_{s-1}}^{-1}t_{v_s}$ by Proposition \ref{basic facts}(3). The result follows.
  \end{full} 
\end{proof}



\begin{notation}
Let $M$ be any $c\times d$ matrix with entries from $\Z$. Then $\ul{t}^M$ denotes the element $\prod_{(i,j)\in W_C}t_{i,j}^{M[i,j]}$, where the factors appear in lexicographical order.
\end{notation}



\begin{theorem}[cf. Theorem 4.1.9 of \cite{cas2}]\label{main vanishing pseudo minors theorem}
Let $I\subseteq \llb 1,c\rrb$ and $J\subseteq \llb 1,d\rrb$ have the same cardinality. Then the pseudo quantum minor $[I\ |\ J]$ of $M_C$ is zero if and only if there does not exist a vertex-disjoint $R_{(I,J)}$-path system in $\postc$.
\end{theorem} 


\begin{proof}
  \begin{full} 
For ease of notation, let us take $I=J=\{1,\ldots,t\}$ (the proof for general $I$ and $J$ is the same but notationally more unwieldy).

By Proposition \ref{only need the turns proposition}, the weight of any vertex-disjoint $R_{(I,J)}$-path system $\cp$ is equal to $q^\alpha\ul{t}^{M_\cp}$ for some integer $\alpha$, where the $c\times d$ matrix $M_\cp=(M_\cp[i,j])$ is defined as follows:
\[
\piecewiseseven{M_\cp[i,j]}{1}{\tx{if there is a path in }\cp\tx{ with a }\Gamma\tx{-turn at }(i,j);}{-1}{\tx{if there is a path in }\cp\tx{ with a }\Le\tx{-turn at }(i,j);}{0}{\tx{otherwise}.}
\]
Let $\cp=(P_1,\ldots,P_t)$ and $\cq=(Q_1,\ldots,Q_t)$ be vertex-disjoint $R_{(I,J)}$-path systems satisfying $M_\cp=M_\cq$.
Fix any $i\in \llb 1,t\rrb$ and let $(i,\ell)$ be the first vertex where $P_i$ turns and $(i,\ell')$
be the first vertex where $Q_i$ turns. Suppose that $\ell'>\ell$, so that $P_i$ goes horizontally straight through $(i,\ell')$ and in particular, $(i,\ell')$ is a vertex 
of $P_i$ but neither a $\Gamma$-turn nor a $\Le$-turn of $P_i$. However, since $(i,\ell')$ is a $\Gamma$-turn of $Q_i$ and $M_\cp=M_\cq$, there must be a path $P\neq P_i$ in $\cp$ which 
has a $\Gamma$-turn at $(i,\ell')$, which is a contradiction since $\cp$ is a vertex-disjoint 
path system. Hence $\ell'\ngtr \ell$. A similar argument shows that $\ell\ngtr \ell'$, so that $\ell=\ell'$; that is, the first turning vertices of $P_i$ and $Q_i$ coincide. A similar argument can be 
applied to the remaining turning vertices (if any) of $P_i$ and $Q_i$ to show that $P_i$ and $Q_i$ have the same turning vertices and hence $P_i=Q_i$. 
Since $i\in \llb 1,t\rrb$ was chosen arbitrarily, we conclude that $\cp=\cq$.  

We have shown that if $\cp=(P_1,\ldots,P_t)$ and $\cq=(Q_1,\ldots,Q_t)$ are distinct vertex-disjoint $R_{(I,J)}$-path systems, then $M_\cp\neq M_\cq$ and hence $M_\cp[i,j]\neq M_{\cq}[i,j]$ for some $(i,j)\in W_C$.

It follows easily that if there exists at least one vertex-disjoint $R_{(I,J)}$-path system in $\postc$ 
then $[I\ |\ J]$ is a nontrivial linear combination of pairwise distinct lex-ordered monomials in the $t_{i,j}^{\pm 1}$ where $(i,j)\in W_C$ and hence $[I\ |J]\neq 0$, because 
the lex-ordered monomials in the $t_{i,j}^{\pm 1}$  with $(i,j)\in W_C$ form a basis for the quantum torus associated with $C$. Theorem \ref{computing q det with paths thm} gives the converse. 
\end{full} 
\end{proof}


\section{$\ch$-primes in partition subalgebras: membership}

The aim of this section is to give a characterisation of those pseudo quantum minors that belong to an $\ch$-prime in a partition subalgebra. To achieve this aim, we will use the parametisation of $\ch$-primes by Cauchon-Le diagrams through the deleting derivations algorithm for partition subalgebras. 

Throughout this section we assume that $q\in \k^*$ is not a root of unity.

\subsection{Partition subalgebras as QNAs}

Let $\lambda = \{\lambda_1\geq \lambda_2\geq\dots\geq\lambda_m\}$ be a partition with associated Young diagram $Y_\lambda$ and $n\geq \lambda_1$. Recall that $\ylk$ is the $\k$-subalgebra of $\oqmmnk$ generated by those 
$x_{ij}$ that fit  into the Young diagram for $\lambda$. The partition subalgebra  $\ylk$ can be presented as a QNA with the variables $x_{ij}$ added in lexicographic order, and with the torus $\ch$ acting via restriction from the action on $\oqmmnk$. More precisely, adding the generators $x_{i,j}$ ($(i,j)\in Y_\lambda$) in lexicographical order, we may write the partition subalgebra $\ylk$ as an iterated Ore extension 
\begin{equation}\label{partlambda as Ore ext}
\ylk=\K[x_{1,1}]\cdots [x_{i,j};\sigma_{i,j},\delta_{i,j}]\cdots [x_{m,\lambda_m};\sigma_{m,\lambda_m},\delta_{m,\lambda_m}],
\end{equation}
where for each $(a,b)\in Y_\lambda$, the automorphism $\sigma_{a,b}$ and the left $\sigma_{a,b}$-derivation $\delta_{a,b}$ are defined such that for each $(i,j)\in Y_\lambda$ satisfying $(i,j)<_{\tx{lex}}(a,b)$, we have 
\begin{equation}\label{auto}
\piecewise{\sigma_{a,b}(x_{i,j})}{q^{-1}x_{i,j}}{\tx{if }i=a\tx{ or }j=b}{x_{i,j}}{\tx{otherwise}}
\end{equation}
and 
\begin{equation}\label{der}
\piecewise{\delta_{a,b}(x_{i,j})}{(q^{-1}-q)x_{i,b}x_{a,j}}{i<a\tx{ and }j<b}{0}{\tx{otherwise.}}
\end{equation}
It is easy to check that $\ylk$ is a QNA for this presentation as an iterated Ore extension (and torus $\ch = (\k ^*)^{m+n}$ acting via restriction from the action on $\oqmmnk$). As a consequence, Cauchon's deleting derivations theory applies to the QNA $\ylk$. We make explicit the results of Section \ref{canonicalembedding} in the context of partition subalgebras in the next sections. 

\subsection{Deleting derivations algorithm for partition subalgebras} 

Set $E_\lambda=(Y_\lambda\bs\{(1,1)\})\sqcup \{(m,\lambda_m+1)\}$. For $(a,b)\in Y_\lambda$, let $(a,b)^+$ be the smallest (with respect to the lexicographical order) element of $E_\lambda$ satisfying $(a,b)^+>_\lex(a,b)$. Clearly $(1,1)^+$ and $(m,\lambda_m+1)$ are respectively the smallest and largest elements of $E_\lambda$ with respect to the lexicographical order. Moreover $E_\lambda=\{(a,b)^+~|~ (a,b)\in Y_\lambda\}$. For $(a,b)\in E_\lambda$, let $(a,b)^-$ be the largest (with respect to the lexicographical order) element of $Y_\lambda$ satisfying $(a,b)^-<_\lex(a,b)$.

With this notation, the deleting derivations algorithm constructs, for each $(a,b) \in E_\lambda$, a family $\{x_{1,1}^{(a,b)}, \dots , x_{m,\lambda_m}^{(a,b)} \}$ of elements of ${\rm Frac} (\ylk )$ defined as follows: 
\begin{enumerate}
\item $x_{i,j}^{(m,\lambda_m+1)}:=x_{i,j}$ for all $(i,j)\in Y_\lambda$;
\item Assume that $(a,b) \neq (m,\lambda_m+1)$. Then $x_{a,b}^{(a,b)^+} \neq 0$ and 
\begin{equation*}
\piecewise{x_{i,j}^{(a,b)}:}{\displaystyle{\sum_{n=0}^{+\infty}\frac{(1-q_{a,b})^{-n}}{[n]!_{q_{a,b}}}\delta_{a,b}^n\circ\sigma_{a,b}^{-n}(x_{i,j}^{(a,b)^+})(x_{a,b}^{(a,b)^+})^{-n}  }}{\tx{if }(i,j)<_{\tx{lex}} (a,b);}{x_{i,j}^{(a,b)^+}}{\tx{if }(i,j)\geq_{\tx{lex}} (a,b),}
\end{equation*}
that is 
\begin{equation}\label{mhm2}
\piecewise{x_{i,j}^{(a,b)}:}{\displaystyle{ x_{i,j}^{(a,b)^+}-x_{i,b}^{(a,b)^+}(x_{a,b}^{(a,b)^+})^{-1}x_{a,j}^{(a,b)^+}}}{\tx{if }i<a \tx{ and } j<b;}{x_{i,j}^{(a,b)^+}}{\tx{otherwise}.}
\end{equation}
\end{enumerate}
As usual, for all $(a,b)\in E_\lambda$, $\ylk ^{(a,b)}$ denotes the subalgebra of ${\rm Frac} ( \ylk ) $ generated by the $x_{i,j}^{(a,b)}$ with $(i,j) \in Y_{\lambda}$. 
Moreover, we set $t_{i,j}:=x_{i,j}^{(1,2)}$. Recall that the subalgebra $\overline{\ylk}=\ylk^{(1,2)}$ of ${\rm Frac}(\ylk) $ generated by the $t_{i,j}$ with $(i,j) \in Y_{\lambda}$ is a quantum affine space in the $t_{i,j}$, and $t_{i,j}$ and $t_{k,l}$ commute unless $(i,j)$ and $(k,l)$ are in the same row or column (that is, unless $(k,l)$ is West or North of $(i,j)$), in which case $t_{i,j}t_{k,l}=q^{-1}t_{k,l}t_{i,j}$.

\subsection{Canonical embedding and $\ch$-primes}

By \cite[Section 4.3]{c1} (see also Section \ref{canonicalembedding}), for each $(a,b)\in E_\lambda\bs\{(c,\lambda_c+1)\}$, there is an injection 
\[
\varphi_{a,b} : \Spec(\ylk^{(a,b)^+})\to \Spec(\ylk^{(a,b)}).
\] 
We shall not describe again the construction of this injection but we shall recall some of its useful properties.

Let $(a,b)\in E_\lambda\bs\{(m,\lambda_m+1)\}$ and let $Q$ be a prime ideal of $\ylk^{(a,b)^+}$. 
The Lemmas \cite[Lemme 5.3.1 and Lemme 5.3.2]{c1} give isomorphisms
\begin{equation}\label{isos of hom images}
{\rm Frac} ( \ylk ^{(a,b)^+}/Q)\xrightarrow{\cong}{\rm Frac} ( \ylk ^{(a,b)}/\varphi_{a,b}(Q)).
\end{equation}

Fix a prime ideal $P$ of $\ylk$. For each $(a,b)\in E_\lambda$, set $P^{(a,b)}=\varphi_{a,b}\circ\cdots\circ \varphi_{m,\lambda_m}(P)\in \Spec (\ylk^{(a,b)})$ (which gives $P^{(m,\lambda_m+1)}=P$) and for each $(i,j)\in Y_\lambda$, let $\chi_{i,j}^{(a,b)}$ be the canonical image of $x_{i,j}^{(a,b)}$ in $\ylk^{(a,b)}/P^{(a,b)}$. 
Let us denote by $G$ the total ring of fractions of $\ylk/P$ (which is a division ring since all prime ideals of $\ylk$ are completely prime) and by varying $(a,b)$ over $E_\lambda\bs\{(m,\lambda_m+1)\}$ and $Q$ over $P^{(1,1)^{+}},\ldots,P^{(m,\lambda_m+1)}$ in the isomorphism \eqref{isos of hom images}, let us identify the total ring of fractions of each noetherian domain $\ylk^{(a,b)}/P^{(a,b)}$ ($(a,b)\in E_\lambda$) with $G$.

Some immediate consequences of this setup (noted in \cite[Proposition 5.4.1]{c1}) are that for each $(a,b)\in E_\lambda$,
\begin{itemize}
\item $\ylk^{(a,b)}/P^{(a,b)}$ is the subalgebra of $G$ generated by $(\chi_{i,j}^{(a,b)})_{(i,j)\in Y_\lambda}$;
\item there is a morphism of algebras $f_{a,b}: \ylk^{(a,b)}\to G$ which sends each $x_{i,j}^{(a,b)}$ ($(i,j)\in Y_\lambda$) to $\chi_{i,j}^{(a,b)}$;
\item the kernel of $f_{a,b}$ is $P^{(a,b)}$ and its image is $\ylk^{(a,b)}/P^{(a,b)}$. 
\end{itemize}

Suppose that $(a,b)\in E_\lambda\bs\{(m,\lambda_m+1)\}$.  By \cite[Proposition 5.4.2]{c1}, we may construct the generators $\chi_{i,j}^{(a,b)}$ ($(i,j)\in Y_\lambda$) of the algebra $\ylk^{(a,b)}/P^{(a,b)}$ from the generators $\chi_{i,j}^{(a,b)^+}$ ($(i,j)\in Y_\lambda$) of the algebra $\ylk^{(a,b)^+}/P^{(a,b)^+}$ as follows: 
\begin{lemma}\label{pre rest alg}
Suppose that $(a,b)\in E_\lambda\bs\{(m,\lambda_m+1)\}$. Then for all $(i,j)\in Y_\lambda$, we have
\[
\piecewise{\chi_{i,j}^{(a,b)^+}}{\chi_{i,j}^{(a,b)} +\chi_{i,b}^{(a,b)}(\chi_{a,b}^{(1,1)^+})^{-1}\chi_{a,j}^{(1,1)^+}}{\tx{if }i<a,\ j<b,\tx{ and }\chi_{a,b}^{(1,1)^+}\neq 0;}{\chi_{i,j}^{(a,b)}}{\tx{otherwise.}}
\]
\end{lemma}

\subsection{Cauchon-Le diagrams in $\ylk$}
\label{section: Cauchon-Le diagrams in partition subalgebras}
Let us now assume that $P$ is not just a prime ideal but an $\ch$-prime ideal of $\ylk$. 

The \emph{canonical embedding} $\varphi : \Spec(\ylk)\to \Spec(\ylk^{(1,1)^+})$ is defined by $\varphi=\varphi_{(1,1)^+}\circ\cdots\circ\varphi_{m,\lambda_m}$ (see \cite{c1} or Section \ref{canonicalembedding}). By the results of Cauchon described in Section \ref{canonicalembedding} the action of $\ch$ on $\ylk$ induces an action of $\ch$ on the quantum affine space $\overline{\ylk}=\ylk^{(1,1)^+}$ such that $\varphi$ sends $P$ to an $\ch$-prime ideal $\varphi(P)$ ($=P^{(1,1)^+}$) of $\overline{\ylk}$ and $\varphi(P)$ is generated by $\{t_{i,j}\ |\ (i,j)\in B\}$ for some subset $B$ of $Y_\lambda$. In the terminology of Section \ref{canonicalembedding}, $B$ is a Cauchon diagram for the QNA $\ylk$. 

As mentioned at the end of Section \ref{canonicalembedding}, we modify the visual presentation of Cauchon diagrams as defined in Section \ref{canonicalembedding} to take advantage of the Young diagram that is intrinsic to the partition subalgebra $\ylk$. More precisely, we colour the squares of the Young diagram $Y_\lambda$ in the following way: for $(i,j)\in Y_\lambda$, if $(i,j)\in B$, then assign colour black to the square of $Y_\lambda$ in the $(i,j)$-position and if $(i,j)\notin B$, then assign colour white to the square of $Y_\lambda$ in the $(i,j)$-position; call the resulting diagram $C$. By \cite[Theorem 3.5]{llr-grass}, the diagram $C$ is a Cauchon-Le diagram (see Subsection \ref{subsection:CauchonLeDiag-PostGraph}) and all Cauchon-Le diagrams on $Y_\lambda$ arise from $\ch$-prime ideals of $\ylk$ in this way, giving us a one-to-one correspondence 
\begin{equation}\label{H-primes of partlambda}
\ch-\Spec \ylk \longleftrightarrow \tx{ Cauchon-Le diagrams on the Young diagram }Y_\lambda.
\end{equation}
Recall that $W_C $ denotes the set of white boxes of the Cauchon-Le diagram $C $.

As we have identified the division ring $G={\rm Frac} ( \ylk /P)$ with the total ring of fractions of each noetherian domain $\ylk^{(a,b)}/P^{(a,b)}$ ($(a,b)\in E_\lambda$), we have in particular identified $G$ with the total ring of fractions of the quantum affine space  
\begin{equation}\label{heres B}
\ylk^{(1,1)^+}/\varphi(P).
\end{equation}
For $(i,j)\in Y_\lambda$, let $\overline{t}_{i,j}$ denote the canonical image of $t_{i,j}$ in the algebra \eqref{heres B}, so that $\overline{t}_{i,j}=\chi_{i,j}^{(1,1)^+}$ and we may realise
$G$ as the total ring of fractions of the uniparameter quantum torus $\mathcal{T}_{C}$ which is generated by $\{\overline{t}_{i,j}^{\pm 1}\ |\ (i,j)\in W_C\}$ with relations 
\begin{equation}\label{torus rels}
\begin{array}{ll}
\overline{t}_{i,j}\overline{t}_{i,l}=q\overline{t}_{i,l}\overline{t}_{i,j} & \tx{ if } (i,j),(i,l)\in W_C \tx{ and }  j<l; \\
\overline{t}_{i,j}\overline{t}_{k,j}=q\overline{t}_{k,j}\overline{t}_{i,j} & \tx{ if } (i,j),(k,j)\in W_C \tx{ and } i<k; \\
\overline{t}_{i,j}\overline{t}_{k,l}=\overline{t}_{k,l}\overline{t}_{i,j}  &  \tx{ if } (i,j),(k,l)\in W_C,\ k\neq i,\tx{ and } j\neq l; \\
\overline{t}_{i,j}\overline{t}_{i,j}^{-1}=1 & \tx{ if } (i,j)\in W_C. \\
\end{array}
\end{equation}

The following result allows us to express the $\chi_{i,j}^{(a,b)^+}$ in terms of the $\chi_{k,l}^{(a,b)}$.

\begin{corollary}\label{realstartup}
For all $(a,b)\in E_\lambda\bs\{(m,\lambda_m+1)\}$ and all $(i,j)\in Y_\lambda$, we have 
\[
\piecewise{\chi_{i,j}^{(a,b)^+}}{\chi_{i,j}^{(a,b)} +\chi_{i,b}^{(a,b)}\overline{t}_{a,b}^{-1}\overline{t}_{a,j}}{\tx{if }i<a,\ j<b,\tx{ and }(a,b)\in W_\mathcal{C};}{\chi_{i,j}^{(a,b)}}{\tx{otherwise.}}
\]
\begin{proof}
Since $\chi_{a,b}^{(1,1)^+}=\overline{t}_{a,b}$ is nonzero if and only if $(a,b)\in W_C$ and since $\chi_{a,j}^{(1,1)^+}=\overline{t}_{a,j}$, this result is an immediate consequence of Lemma \ref{pre rest alg}.
\end{proof}
\end{corollary}

\subsection{Pseudo quantum minors belonging to an $\ch$-primes in $\ylk$}
\label{section-cauchon-le}

Let $P$ be an $\ch$-prime ideal of the partition subalgebra $\ylk$ corresponding to a Cauchon-Le diagram $C$ on the Young diagram $Y_\lambda$ and let $I\subseteq\llb 1,m\rrb$ and $J\subseteq \llb 1,n\rrb$ have the same cardinality. As in Section \ref{subsection:pathmatrix}, we denote by $M_C$ the path matrix associated to $C$.


\begin{theorem}[cf. Lemma 5.4 of \cite{cas1}]\label{path matrix is meaningful}
For each $(i,j)\in Y_\lambda$, $M_C[i,j]$ is the canonical image in $\ylk / P$ of $x_{i,j}$, namely $M_C[i,j]=\chi_{i,j}^{(m,\lambda_m+1)}$. For each $(i,j)\in \llb 1,m\rrb \times \llb 1,n\rrb\bs Y_\lambda$, $M_C[i,j]$ is zero.
\end{theorem} 


\begin{proof}
  \begin{full} 
It is obvious that $M_C[i,j]$ is zero for each $(i,j)\in \llb 1,m\rrb \times \llb 1,n \rrb\bs Y_\lambda$; 
so for the rest of this proof all boxes shall be in $Y_\lambda$.

For any $(a,b),(i,j)\in Y_\lambda$, let us define $M_C^{(a,b)}[i,j]$ to be the sum of the weights of all paths $P: r_i\implies c_j$ in $\postc$ which have no $\Le$-turn after $(a,b)$ with respect to the lexicographical order (that is, whose $\Le$-turns $v$ all satisfy $v\leq_\lex (a,b)$). 
It will suffice to show that for any $(a,b),(i,j)\in Y_\lambda$, we have 
\begin{equation}\label{we aim for}
M_C^{(a,b)}[i,j]=\chi_{i,j}^{(a,b)^+};
\end{equation}
setting $(a,b)=(m,\lambda_m)$ in \eqref{we aim for} gives the result.
We prove the claim \eqref{we aim for} by induction on $(a,b)$. If $(i,j)\in B_C$, then there is no path $P: r_i\implies c_j$ in $\postc$ which has no $\Le$-turn after $(1,1)$ and we have $M^{(1,1)}[i,j]=0=\overline{t}_{i,j}=\chi_{i,j}^{(1,1)^+}$. If $(i,j)\in W_C$, then the only path in $\postc$ 
from $r_i$ to $c_j$ which has no $\Le$-turn after $(1,1)$ is the path which runs horizontally from $r_i$ to $(i,j)$ and then vertically from $(i,j)$ to $c_j$; this path has weight $\overline{t}_{i,j}=\chi_{i,j}^{(1,1)^+}$ by Proposition \ref{basic facts}(4), so that $M^{(1,1)}[i,j]=\chi_{i,j}^{(1,1)^+}$. 

Let $(a,b)\in E_\lambda\bs\{(m,\lambda_m+1)\}$ be such that 
\begin{equation}\label{the hyp of the ind}
M_C^{(a,b)^-}[i,j]=\chi_{i,j}^{(a,b)}
\end{equation}
for all $(i,j)\in Y_\lambda$. For any $(i,j)\in Y_\lambda$, let us define $F_{i,j}$ to be the set of all paths in $\postc$ from $r_i$ to $c_j$ which have a $\Le$-turn at $(a,b)$ and no later $\Le$-turn; it will suffice to show that for each $(i,j)\in Y_\lambda$, $\chi_{i,j}^{(a,b)^+}$ is obtained from $\chi_{i,j}^{(a,b)}$ by adding $\sum_{P\in F_{i,j}}w(P)$. 

We may assume that $i<a$, $j<b$, and $(a,b)\in W_C$ (since otherwise $F_{i,j}$ is empty and Corollary \ref{realstartup} gives $\chi_{i,j}^{(a,b)^+}=\chi_{i,j}^{(a,b)}$). By Corollary \ref{realstartup}, we have $\chi_{i,j}^{(a,b)^+}=\chi_{i,j}^{(a,b)} +\chi_{i,b}^{(a,b)}\overline{t}_{a,b}^{-1}\overline{t}_{a,j}$, so that it will suffice to show that 
\begin{equation}\label{sts}
\sum_{P\in F_{i,j}}w(P)=\chi_{i,b}^{(a,b)}\overline{t}_{a,b}^{-1}\overline{t}_{a,j}.
\end{equation}
There are two cases to consider:

\begin{enumerate}[(a)]
\item Suppose that $(a,j)\in B_C$. Then $F_{i,j}$ is empty and $\overline{t}_{a,j}=0$; \eqref{sts} follows immediately.

\item Suppose that $(a,j)\in W_C$. 

Let $F_i$ be the set of all paths in $\postc$ from $r_i$ to $c_b$ which have no $\Le$-turn after $(a,b)^-=(a, b-1)$, so that $\sum_{Q\in F_i}w(Q)=M_C^{(a,b)^-}[i,b]$ and hence $\sum_{Q\in F_i}w(Q)=\chi_{i,b}^{(a,b)}$ by the induction hypothesis \eqref{the hyp of the ind}. Let the path $K_j: (a,b)\implies c_j$ be given by concatenating the horizontal path $(a,b)\implies (a,j)$ with the vertical path $(a,j)\implies c_j$. Proposition \ref{basic facts}(3) gives $w(K_j)=\overline{t}_{a,b}^{-1}\overline{t}_{a,j}$. Let $L_b$ be the vertical path from $(a,b)$ to $c_b$. For any path $P\in F_{i,j}$, the subpath $P': r_i\implies (a,b)$ of $P$ is such that $P=P'K_j$, $P'L_b\in F_i$, and $w(P'L_b)=w(P')$. Notice that each path in $F_i$ has the form $P'L_b$ for a unique path $P\in F_{i,j}$.

We have
\begin{align*}
\sum_{P\in F_{i,j}} w(P)&=\sum_{P\in F_{i,j}} w(P')w(K_j) \\
                    &=\left(\sum_{P\in F_{i,j}} w(P'L_b)\right)w(K_j) \\
                    &=\left(\sum_{Q\in F_i}w(Q)\right)w(K_j) \\                    
                    &=\chi_{i,b}^{(a,b)}w(K_j) \\
                    &=\chi_{i,b}^{(a,b)}\overline{t}_{a,b}^{-1}\overline{t}_{a,j},
\end{align*}
establishing \eqref{sts}.
\end{enumerate}
The proof is complete.
  \end{full} 
\end{proof}


As an immediate corollary of Theorems \ref{main vanishing pseudo minors theorem} and \ref{path matrix is meaningful}, we get the main result of this section, which is a generalisation of the main result of \cite{cas1}:


\begin{theorem}[cf. Theorem 5.6 of \cite{cas1}]\label{which pseudo minors}
Assume that $q\in \k^*$ is not a root of unity. Let $P$ be an $\ch$-prime ideal of the partition subalgebra $\ylk$ corresponding to a Cauchon-Le diagram $C$ on the Young diagram $Y_\lambda$ and let $I\subseteq\llb 1,m\rrb$ and $J\subseteq \llb 1,n\rrb$ have the same cardinality. Then the pseudo quantum minor $[I\ |\ J]$ of $\ylk$ belongs to $P$ if and only if there exists no vertex-disjoint $R_{(I,J)}$-path system in the Postnikov graph 
$\postc$ of $C$. 
\end{theorem}


\section{Primes in $\oqmmnr$ from Cauchon-Le diagrams}\label{Primes in oqmmnr from Cauchon-Le diagrams}

In Section \ref{qgrass}, we will use Theorem \ref{which pseudo minors} together with noncommutative dehomogenisation to identify the quantum Pl\"ucker coordinates that are contained in an $\ch$-prime $P$ of $\oqgmnk$. However, we would like to show that these quantum Pl\"ucker coordinates generate $P$ as an ideal. This is a known result in the classical (commutative) case and our strategy is to pull back the result from the classical case by using the technique developed in \cite{good-len-q-transc}. In order to do this, we need to work with $q$ transcendental over a base field $\k$ and with $R:=\k[q^{\pm1}]$ as the base ring so that we can ``set $q=1$''. 

With this in mind, we need to re-develop some of the theory of the earlier sections for partition subalgebras of $\oqmmnr$. In this section, we focus on the case of $\oqmmnr$.

Let $\f=\k(q)$ be the field of fractions of $R$. In this section, we start by considering certain prime ideals in $\oqmmnr$. 
We need to be careful here, as $\ch =(\f^*)^{m+n}$ does not act on 
$\oqmmnr$; so it does not make sense to refer to $\ch$-prime ideals of $\oqmmnr$. 
Instead, we proceed as follows. Let $\Pi$ denote the set of all quantum minors of 
$\oqmmnr$. Let $C$ be a Cauchon-Le diagram on an $m\times n$ array. Let $\Pi^C$ 
be the set of quantum minors $[I\mid J]$ such that there are no 
vertex disjoint $R_{(I,J)}$-path systems in $\postc$, and let $\ideal{\Pi^C}_R$ denote 
the ideal of $\oqmmnr$ generated by $\Pi^C$. The aim is to show that 
$\ideal{\Pi^C}_R$ is a completely prime ideal of $\oqmmnr$. When the base ring is a field, Casteels, 
\cite{cas1}, has shown that $\Pi^C$ is a generating set for  $P$, the 
$\ch$-prime ideal corresponding to the Cauchon-Le diagram $C$. In fact, in \cite{cas2}, he goes further and shows that $\Pi^C$ is a Gr\"obner basis for $P$ either as a left ideal or as a right ideal. As this is a crucial result for us, we state his theorem below. 

\begin{notation}
If If $M=(M_{i,j}) \in \matnonneg$, then we set: 
$$\vect{x}^M=x_{1,1}^{M_{1,1}} x_{1,2}^{M_{1,2}} \cdots x_{m,n}^{M_{m,n}},$$
where the indices are in the lexicographic order, from the smallest to the largest when one goes from left to right. Such a monomial is called a {\em lexicographic monomial}. 
\end{notation}
Note that lexicographic monomials form a $R$-basis (resp. $\f$-basis) of $\oqmmnr$ (resp. $\oqmmnf$). This basis is referred to as the {\em PBW basis} of $\oqmmnr$ (resp. $\oqmmnf$), and the expression of an element $a \in \oqmmnr$ (resp. $a \in \oqmmnf$) in this basis is called the {\em lexicographic expression} of $a$.

For details about Gr\"obner bases, see either  \cite{cas2}, or \cite{bgv}. 

In this section, by a symbol $q^{\bullet}$, we mean some power of $q$ with exponent in $\mz$.

\subsection{Matrix lexicographic order and lexicographic expression}

In order to discuss Gr\"obner bases, one first has to choose an ordering on the generators of an algebra. We give the generators $x_{ij}$ of $\oqmmnr$, or $\oqmmnf$, the matrix lexicographic order as defined below. 
~\\

\begin{definition} 
\label{revlex} 
{\rm The \emph{matrix lexicographic order} $\prec$ on $m\times n$ 
matrices with integer coefficients is defined as follows. If $M\neq N\in \mat$, let $(k,\ell)$ be the least coordinate in which $M$ and $N$ differ. Then we set \begin{align*}M\prec N \Leftrightarrow (M)_{k,\ell}<(N)_{k,\ell}.\end{align*}
If $M\prec N$ are both in $\matnonneg$, then the matrix 
lexicographic order induces a total order (that we also call matrix lexicographic) on the lexicographic monomials of $\oqmmnr$ and $\oqmmnf$ by setting 
\begin{align*}\vect{x}^M\prec \vect{x}^{N} \Leftrightarrow M\prec N.\end{align*}
}\end{definition}


On the generators, under the matrix lexicographic order, we have $$x_{i,j}\prec x_{k,\ell} \Leftrightarrow (i,j)>(k,\ell);$$
so that, for example, in quantum $2\times 2$ matrices 
$x_{22}\prec x_{21}\prec x_{12}\prec x_{11}$.\\

For a given element $a$, the {\em leading monomial}, $\lm(a)$ is the monomial 
$\vect{x}^{N} $ that is the greatest monomial in the $\prec$ order
that occurs with a nonzero coefficient. 
The {\em leading exponent}, $\lexp(a)$ is $N$ where 
$\lm(a)=\vect{x}^{N} $. If $\vect{x}^W$ is any other monomial that occurs with a nonzero coefficient in the expression of $a$ in the PBW basis then we know that $W\prec N$. 
For example 
the quantum minor $[12|12]=x_{11}x_{22}- qx_{12}x_{21}$ has leading monomial $x_{11}x_{22}$.  More generally, let $I=\{i_1 < \dots < i_t\} \in \{1, \dots , m\}$ and $J=\{j_1 < \dots < j_t\} \in \{1, \dots , n\}$. Then Casteels prove \cite[Lemma 4.2.8]{cas2}  that the leading monomial of the quantum minor $[I \mid J]$ of $\oqmmnf$ is given by:
$$\lm ([I  \mid J]) = x_{i_1,j_1} \cdot x_{i_2,j_2} \cdots x_{i_t,j_t}.$$
In particular, observe that 
$$[I \mid J]=\lm ([I  \mid J]) + \mbox{smaller terms},$$
that is the coefficients of the leading monomial of a quantum minor is always equal to $1$.

If $i<k$ and $j<\ell$ then 
$$x_{k,\ell}x_{i,j}=x_{i,j}x_{k,\ell}-(q-q^{-1})x_{i,\ell}x_{k,j}.$$ On the other hand, we also have $$x_{i,\ell}x_{k,j}\prec x_{i,j}x_{k,\ell}.$$ 

By repeated application of this fact and the other relations for $\oqmnr$, we obtain the following, which is a special case of the more general \cite[Chapter 2, Proposition 2.4]{bgv}, although there the authors are working over a field.

\begin{proposition} \label{straightening}
For $M,N\in \matnonneg$, the lexicographic expression of 
$\vect{x}^M\vect{x}^N$
 in $\oqmnr$ is 
 \begin{equation} \label{straightening law}
 \vect{x}^M\vect{x}^N=q^\bullet \vect{x}^{M+N} + \sum_{L\in\matnonneg} \alpha_L\vect{x}^{L}, 
 \end{equation}
where $\alpha_L\in\mz[q^{\pm 1}]$ and for every $\alpha_L\neq 0$, one has $L\prec M+N$. 
\end{proposition}
\begin{proof} 
For this proof, note the following easy properties of $\prec$ for any $A,B,C\in M_{m,p}(\mathbb{Z}_{\geq 0})$.
\begin{itemize}
\item $A\prec A+B$;
\item $A\prec B \implies A+C\prec B+C$. 
\end{itemize}

One can show easily by induction on the matrix lexicographic order of $M$ that it is enough to prove \eqref{straightening law} when $M$ is elementary, that is, when $M=E_{a,b}$, where $E_{a,b} \in \matnonneg$ has all its entries equal to zero, except its $(a,b)$-entry which is equal to $1$. Indeed \eqref{straightening law} clearly holds when $M=\mathbf{0}$. Suppose that the first nonzero entry of $M$ is $M_{i,j}$ and set $M'=M-E_{i,j}\prec M$. Then 
\begin{align*}
\vect{x}^M\vect{x}^N&=x_{i,j}\vect{x}^{M'}\vect{x}^N \\
                    &=x_{i,j}\left( q^\bullet \vect{x}^{M'+N} +\sum_{L\prec M'+N}\alpha_L\vect{x}^L\right),
\end{align*}
where the latter equality holds by the induction hypothesis;
the assumption that \eqref{straightening law} holds when $M$ is elementary now allows us to conclude.

So let us assume from now on that $M$ is elementary; let $M=E_{a,b}$. We proceed by induction on the matrix lex order of $N$ (the result clearly holds when $N=\mathbf{0}$). Let $N_{i,j}$ be the first nonzero entry of $N$ and set $N'=N-E_{i,j}\prec N$. We may assume that $(i,j)<(a,b)$ as otherwise the result is trivial. There are two cases to consider:

\textbf{Case 1:} $i=a$ or ($i<a$ and $b\geq j$).  
\begin{align*}
\vect{x}^M\vect{x}^N&=x_{a,b}x_{i,j}\vect{x}^{N'} \\
                    &=q^\bullet x_{i,j}(x_{a,b}\vect{x}^{N'}) \\
                    &=q^\bullet x_{i,j}\left( q^\bullet \vect{x}^{N'+E_{a,b}}+\sum_{L\prec N'+E_{a,b}}\alpha_L\vect{x}^L \right)\ \ \ \text{(by the induction hypothesis)} \\
\end{align*}
Using the induction hypothesis again, we may settle \textbf{Case 1}.

\textbf{Case 2:} $i<a$ and $j<b$. 
\begin{align*}
\vect{x}^M\vect{x}^N&=x_{a,b}x_{i,j}\vect{x}^{N'} \\
                    &=x_{i,j}x_{a,b}\vect{x}^{N'}+(q^{-1}-q)x_{i,b}x_{a,j}\vect{x}^{N'} \\
                    &=x_{i,j}\left( q^\bullet \vect{x}^{N'+E_{a,b}}+\sum_{L\prec N'+E_{a,b}}\alpha_L\vect{x}^L \right)+(q^{-1}-q)x_{i,b}\left( \vect{x}^{N'+E_{a,j}}+
\sum_{L\prec N'+E_{a,j}}\alpha_L\vect{x}^L \right)\  \\
& (\text{by the induction hypothesis}) \\
\end{align*}
Using the induction hypothesis again, we may settle \textbf{Case 2}.
\end{proof}


\begin{definition} \label{dividedef}
{\rm Let $M, N\in \matnonneg$. We say that $\vect{x}^M$ \emph{occurs in} 
$\vect{x}^N$ if $(M)_{i,j}\leq (N)_{i,j}$ for all $(i,j)\in \llb 1,m\rrb \times \llb 1,n\rrb$. 
}\end{definition}

As a consequence of Proposition \ref{straightening}, we obtain the below result. It is similar to \cite[Corollary 4.2.4]{cas2}, but the extra information on the coefficients will be crucial in the next section. 

\begin{corollary}  \label{strcor}
Let $M,N\in\matnonneg$. If $\vect{x}^M$ occurs in $\vect{x}^N$, then there exist matrices $L\prec N$, and  
$\alpha_L\in\mz[q^{\pm 1}]$  such that
 \begin{align*} \vect{x}^N & = q^\bullet\vect{x}^{N-M}\vect{x}^{M} + \sum_L \alpha_L \vect{x}^{L}.
\end{align*}
with $L\prec N$ for each $L$ such that $\alpha_L\neq 0$.
\end{corollary}

\subsection{Gr\"obner bases and completely prime ideals of $\oqmmnr$}

Let $J$ be a right (resp. left) ideal of $\oqmmnf$. Recall (see for instance \cite[Definition 4.2.7]{cas2})  that $G=\{g_1,\dots,g_t\} \in J$ is a {\em Gr\"obner basis} for $J$ if for all $a \in J$, there exists $g_i \in G$ such that $\lm(g_j)$ occurs in $\lm(a)$. While it is not a requirement of the definition, as in the commutative setting, it is easy to show that $G$ is indeed a basis of the right (resp. left) ideal $J$.

We can now state Casteels' result.

\begin{theorem}{\rm \cite[Theorem 4.4.1]{cas2}}
Let $C$ be a Cauchon-Le diagram on an $m\times n$ array and let $\Pi^C$ 
be the set of quantum minors $[I\mid J]$ such that there are no 
vertex disjoint $R_{(I,J)}$-path systems in $\postc$. Then $\ideal{\Pi^C}_{\f}$ 
is a completely prime ideal of $\oqmmnf$ and  $\Pi^C$ is a Gr\"obner basis for 
$\ideal{\Pi^C}_{\f}$ considered either as a left ideal or as a right ideal. 
\end{theorem}


\begin{proof} 
This is the case $t=mn$ in \cite[Theorem 4.4.1]{cas2}. 
\end{proof} 


We will use this result to show that $\ideal{\Pi^C}_R$ is a completely prime ideal of $\oqmmnr$.


\begin{theorem}
Let $C$ be an $m\times n$ Cauchon-Le diagram  
and let $\Pi^C$ 
be the set of quantum minors $[I\mid J]$ such that there are no 
vertex disjoint $R_{(I,J)}$-path systems in $\postc$. Let $I^C_R$ denote the left ideal of 
$\oqmmnr$ generated by 
$\Pi^C$. Then 

\[
I^C_R= \ideal{\Pi^C}_{\f}\cap \oqmmnr.
\]
\end{theorem}

\begin{proof} 

It is obvious that 
$ I^C_R\subseteq  \ideal{\Pi^C}_{\f}\cap \oqmmnr $; 
so we need to show that 
$  \ideal{\Pi^C}_{\f}\cap \oqmmnr 
 \subseteq  I^C_R.$\\
 
 Suppose that $ \ideal{\Pi^C}_{\f}\cap\oqmmnr
 \not\subseteq  I^C_R.$ We will derive a contradiction from this 
 assumption. \\

 Set ${\mathcal S}:=  \ideal{\Pi^C}_{\f}\cap\oqmmnr
 \backslash  I^C_R$, a nonempty set by assumption. Choose 
 $a\in{\mathcal S}$ with smallest possible leading monomial $\lm(a)$. \\

We know that $\ideal{\Pi^C}_{\f}$ is generated by $\Pi^C$ as a left ideal of 
$\oqmmnf$; so we may write 
$$
a =  \sum_{\ell = 1}^{t}\, b_{I_\ell}\vect{x}^{I_\ell}
=
\sum_{\ell=1}^{s} c_l\Delta_\ell
$$
with $I_{\ell} \in \matnonneg$, $0\neq b_{I_\ell}\in \k[q^{\pm1}]$ 
and  $0\neq c_\ell\in\oqmmnf$ for all $\ell \in \{1, \dots , t\}$, and with $\Delta_\ell\in\Pi^C$ for all $\ell\in \{ 1, \dots , s\}$. Note that 
each $\Delta_\ell\in I^C_R$ and let us emphasise that 
$b_{I_\ell}\in \k[q^{\pm1}]$ as this will be crucial 
later in the proof. 

As $\Pi^C$ is a Gr\"obner basis for $\ideal{\Pi^C}_{\f}$ as a left ideal of $\oqmmnf$, 
the previous equality implies that there exists $\ell\in\{1,\dots,t\}$ such that 
$\lm(\Delta_\ell)$ occurs in $\lm(a)$. Without loss of generality, assume that $\ell=1$ and that $\lm(a)=\vect{x}^{I_1}$. Suppose that 
$ \lm(\Delta_1)=\vect{x}^M$. Write  $\Delta_1=\vect{x}^M+u$ and note that 
$\lm(u)\prec \vect{x}^M$.\\

By Corollary~\ref{strcor}, we can write 
\begin{eqnarray*}
\vect{x}^{I_1} &=&q^\bullet \vect{x}^{I_1-M}\vect{x}^M + \sum_{L\prec I_1}\,
\alpha_L\vect{x}^L\\
&=&q^\bullet\vect{x}^{I_1-M}\Delta_1 - q^\bullet\vect{x}^{I_1-M}u + 
\sum_{L\prec I_1}\,
\alpha_L\vect{x}^L
\end{eqnarray*} 
We use this expression for $\vect{x}^{I_1}$ in the PBW expression for $a$ to obtain
\begin{eqnarray*}
a&=& b_{I_1}\vect{x}^{I_1} +  \sum_{\ell = 2}^{t}\, b_{I_\ell}\vect{x}^{I_\ell}\\
&=&b_{I_1}q^\bullet\vect{x}^{I_1-M}\Delta_1 -b_{I_1}q^\bullet\vect{x}^{I_1-M}u
+\sum_{L\prec I_1}\,
\alpha_L\vect{x}^L\\
\end{eqnarray*}

If $-b_{I_1}q^\bullet\vect{x}^{I_1-M}u
+\sum_{L\prec I_1}\,
\alpha_L\vect{x}^L=0$ then 
 $a=b_{I_1}q^\bullet\vect{x}^{I_1-M}\Delta_1$ is in  $I^C_R$ 
(here, we are using our previous observation that $b_{I_1}\in \k[q^{\pm1}]$), contradicting our assumption that $a\in{\mathcal S}$. \\

Hence, we may assume that 
\begin{equation}
\label{eqa'}
0\neq -b_{I_1}q^\bullet\vect{x}^{I_1-M}u
+\sum_{L\prec I_1}\,
\alpha_L\vect{x}^L = a-b_{I_1}q^\bullet\vect{x}^{I_1-M}\Delta_1.
\end{equation}

So $ a':= a-b_{I_1}q^\bullet\vect{x}^{I_1-M}\Delta_1 $ is a nonzero element of $\mathcal{S}$ (since $a\in \mathcal{S}$ and $b_{I_1}q^\bullet\vect{x}^{I_1-M}\Delta_1 \in I^C_R$).

On the other hand, note that 
$$\lm(-b_{I_1}q^\bullet\vect{x}^{I_1-M}u)= \lm(\vect{x}^{I_1-M})\lm(u)
\prec  \lm(\vect{x}^{I_1-M})\lm(\vect{x}^M) =\lm(\vect{x}^{I_1}).$$
Having noticed this, inspection of the left hand side of Equation (\ref{eqa'}) above reveals that 
\[
\lm(a')=\lm(-b_{I_1}q^\bullet\vect{x}^{I_1-M}u
+\sum_{L\prec I_1}\,
\alpha_L\vect{x}^L)\prec \lm(a),
\]
 a contradiction to the choice of $a$ having least leading monomial 
 among members of 
 ${\mathcal S}$. Hence, $I^C_R= \ideal{\Pi^C}_{\f}\cap \oqmmnr$, as required.
\end{proof}

An immediate corollary of this result is that the ideal $\ideal{\Pi^C}_R$ is a 
completely prime ideal of $\oqmmnr$.

\begin{corollary}  \label{corollary-cp-over-R}
The ideal $\ideal{\Pi^C}_R$ is a completely prime ideal of $\oqmmnr$ 
that  is generated by 
$\Pi^C$ as a left ideal and as a right ideal. 
\end{corollary}

\begin{proof} 
As $I^C_R= \ideal{\Pi^C}_{\f}\cap \oqmmnr$, by the previous result, we see that 
$I^C_R$ is a two-sided ideal and is also completely prime, because 
$\ideal{\Pi^C}_{\f}$ is a completely prime ideal of $\oqmmnf$. Now, $\Pi^C\subseteq I^C_R$ 
and this forces 
$\ideal{\Pi^C}_R\subseteq  I^C_R\subseteq \ideal{\Pi^C}_R$. Hence, 
$I^C_R= \ideal{\Pi^C}_R$; and so 
$ \ideal{\Pi^C}_R$ is a completely prime ideal that is generated by $\Pi^C$ as a left ideal of $\oqmmnr$. 
In the previous theorem, we could have chosen to use right ideals rather than 
left ideals; so  $ \ideal{\Pi^C}_R$ is also generated by $\Pi^C$ 
as a right ideal of $\oqmmnr$. 

\end{proof}



\section{$\ch$-primes in partition subalgebras: generation}\label{Gens_for_H_primes_in_part_subalgs}

Let $\f$ be a field containing a non root of unity $q$ and let $R$ be any subring of $\f$ containing $q$ and $q^{-1}$. 

Let $\lambda = \{\lambda_1\geq \lambda_2\geq\dots\geq\lambda_m\}$ ($\lambda_1\leq n$) be a partition with associated Young diagram $Y_\lambda$ and consider the partition subalgebras $\ylr$ and $\ylf$ of $\qmatr$ and $\oqmmnf$ respectively. The partition subalgebra $\ylf$ can be presented as a QNA extension in many ways. In fact, it is straightforward to check that a QNA presentation is associated with any ordering of the variables $x_{ij}$ which satisfies the following property: suppose that $i<j$ and $k<l$, then $x_{jl}$ should be later in the ordering than 
 $x_{ik}, x_{il}$ and $x_{jk}$. A specific such ordering would be the lexicographic ordering; however, for some partition subalgebras, we also need another ordering which we now describe. 

Suppose that $\ylf$ is a proper partition subalgebra of $\oqmmnf$ and notice that this forces the partition $\lambda=(\lambda_1,\ldots,\lambda_m)$ to satisfy $\lambda_m<n$. 
Let us choose the QNA presentation of $\ylf$ with the variables $x_{ij}$ ordered lexicographically.

Choose $u$ to be minimal such that $\lambda_u<n$, set $v=\lambda_u+1$, and consider the partition $\lambda':=(\lambda_1,\ldots,\lambda_{u-1},v,\lambda_{u+1},\ldots,\lambda_m)$. Set $z=x_{uv}\in \yldashf$.
The following picture illustrates the idea (for $m=n=5$).
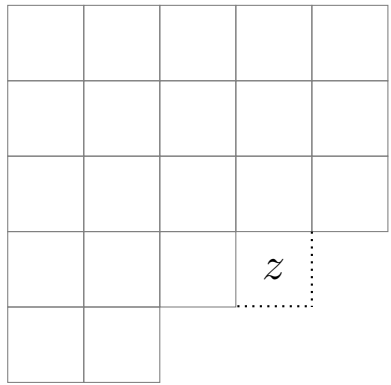
\begin{center}
\begin{tikzpicture}[xscale=1, yscale=1]




\draw[color=gray] (0,4) rectangle (1,5);            
\draw[color=gray] (1,4) rectangle (2,5);            
\draw[color=gray] (2,4) rectangle (3,5);            
\draw[color=gray] (3,4) rectangle (4,5);            
\draw[color=gray] (4,4) rectangle (5,5);            

\draw[color=gray] (0,3) rectangle (1,4);               
\draw[color=gray] (1,3) rectangle (2,4);               
\draw[color=gray] (2,3) rectangle (3,4);               
\draw[color=gray] (3,3) rectangle (4,4);               
\draw[color=gray] (4,3) rectangle (5,4);               

\draw[color=gray] (0,2) rectangle (1,3);            
\draw[color=gray] (1,2) rectangle (2,3);            
\draw[color=gray] (2,2) rectangle (3,3);            
\draw[color=gray] (3,2) rectangle (4,3);               
\draw[color=gray] (4,2) rectangle (5,3);               

\draw[color=gray] (0,1) rectangle (1,2);            
\draw[color=gray] (1,1) rectangle (2,2);            
\draw[color=gray] (2,1) rectangle (3,2);            

\draw[color=gray] (0,0) rectangle (1,1);            
\draw[color=gray] (1,0) rectangle (2,1);            

\draw[thick,dotted]  (3.02,1.01) -- (4,1.01);
\draw[thick,dotted]  (4,2) -- (4,1);

\node at (3.5, 1.5) {\Large{$z$}}; 

\end{tikzpicture}
\end{center}

We obtain a QNA presentation for $\yldashf$ by using the variables for 
$\ylf$ in the (lexicographical) order already chosen, together with  $z=x_{uv}$ as the final variable in the ordering. More precisely:
\begin{lemma}\label{Cauchon Ext}
There is an automorphism $\sigma$ of $\ylf$ and a left $\sigma$-derivation $\delta$ of $\ylf$ such that 
\[
\piecewise{\sigma(x_{a,b})}{q^{-1}x_{a,b}}{\tx{if }a=u\tx{ or }b=v;}{x_{a,b}}{\tx{otherwise,}}
\]
\[
\piecewise{\delta(x_{a,b})}{(q^{-1}-q)x_{a,v}x_{u,b}}{\tx{if }a<u\tx{ and }b<v;}{0}{\tx{otherwise,}}
\]
and there is a QNA presentation for $\yldashf$ given by
\begin{equation}\label{obv}
\yldashf=\ylf[z; \sigma, \delta].
\end{equation}
\end{lemma}

We shall later rely on the following proposition, which essentially says that this nonstandard QNA expression for $\yldashf$ does not affect its Cauchon diagrams (Cauchon diagrams in the sense of Section \ref{section-Cauchon diagrams for a QNA}). First we introduce some new notation. Let $G$ be any division ring and let $\{y_{i,j}\}_{(i,j)\in Y}$ be a collection of elements of $G$ indexed by the boxes of a Young diagram $Y$. For any $(a,b)\in Y$, define the collection $\{y_{i,j}^{\ang{a,b}}\}_{(i,j)\in Y}$ of elements of $G$ by:
\begin{equation}\label{formal}
\piecewise{y_{i,j}^{\ang{a,b}}}{y_{i,j}-y_{i,b}y_{a,b}^{-1}y_{a,j}}{i<a,\ j<b, y_{a,b}\neq 0}{y_{i,j}}{\tx{otherwise.}}
\end{equation}
Set $Y:=Y_{\lambda'}$ (and assume for ease of notation that this Young diagram is not just a column, so that $(1,1)^+=(1,2)$), $G:=\yldashf / J$ for some $\ch$-prime ideal $J$, and $y_{i,j}:=x_{i,j}+J\in G$. Then by \cite[Proposition 5.4.2]{c1}, 
\begin{itemize}
\item with respect to the usual QNA structure of $\yldashf$, we have $x^{(1,2)}_{i,j}+J^{(1,2)}=\left({{{{({y_{i,j}^{\ang{m,\lambda_m}}})}^\cdot}^\cdot}^\cdot}\right)^{\ang{1,2}}$ for all $(i,j)\in Y$, where $\ang{m,\lambda_m}\to \ang{1,2}$ proceeds backwards along the boxes of the Young diagram $Y=Y_{\lambda'}$;
\item with respect to the alternative QNA structure of $\yldashf$ (see \eqref{obv}), we have $x_{i,j}^{(1,2)}+J^{(1,2)}=
{\left({{{{(y_{i,j}^{\ang{u,v}})}^{\ang{m,\lambda_m}}}^\cdot}^\cdot}^\cdot\right)}^{\ang{1,2}}$ for all $(i,j)\in Y$, where $\ang{m,\lambda_m}\to \ang{1,2}$ 
proceeds backwards along the boxes of the Young diagram $Y_\lambda$. 
\end{itemize}
So if we can show that 
$\left({{{{({y_{i,j}^{\ang{m,\lambda_m}}})}^\cdot}^\cdot}^\cdot}\right)^{\ang{1,2}}
=
{\left({{{{(y_{i,j}^{\ang{u,v}})}^{\ang{m,\lambda_m}}}^\cdot}^\cdot}^\cdot\right)}^{\ang{1,2}}$ for all $(i,j)\in Y$, 
then for both QNA expressions, the Cauchon diagrams for $J$ (in the sense of Section \ref{section-Cauchon diagrams for a QNA}; that is, those $(i,j)$ such that $x_{i,j}^{(1,2)}\in J^{(1,2)}$) must coincide; this is the purpose of the following lemma.  
\begin{lemma}\label{formal dda in general division rings}
Let $G$ be any division ring and let $\{y_{i,j}\}_{(i,j)\in Y}$ be a collection of elements of $G$ indexed by the boxes of a Young diagram $Y$. 

If $(a',b')$ is strictly south-west of $(a,b)$ in $Y$ (that is, if $a'>a$ and $b'<b$), then $\{y_{i,j}^{\ang{a,b}}\}_{(i,j)\in Y}$ and $\{y_{i,j}^{\ang{a',b'}}\}_{(i,j)\in Y}$ -- which are themselves collection of elements of $G$ indexed by the boxes of the Young diagram $Y$ -- satisfy
\[
\left({y_{i,j}^{\ang{a,b}}}\right)^{\ang{a',b'}}=
\left({y_{i,j}^{\ang{a',b'}}}\right)^{\ang{a,b}}\ \tx{ for all }(i,j)\in Y. 
\]
\begin{proof}
Notice that $y_{a,b}^{\ang{a',b'}}=y_{a,b}$ and $y_{a',b'}^{\ang{a,b}}=y_{a',b'}$. The result is clear if $i\geq a$, $j\geq b'$, $y_{a,b}=y_{a,b}^{\ang{a',b'}}= 0$, or $y_{a',b'}=y_{a',b'}^{\ang{a,b}} = 0$, so let us assume that $i<a$, $j<b'$, $y_{a,b}=y_{a,b}^{\ang{a',b'}}\neq 0$, and $y_{a',b'}=y_{a',b'}^{\ang{a,b}}\neq 0$.
We have  
\begin{align*}
\left( y_{i,j}^{\ang{a',b'}} \right)^{\ang{a,b}} 
&= y^{\ang{a',b'}}_{i,j}-y_{i,b}^{\ang{a',b'}}\left(y_{a,b}^{\ang{a',b'}}\right)^{-1}y_{a,j}^{\ang{a',b'}} \\
&=  y_{i,j}-y_{i,b'}y_{a',b'}^{-1}y_{a',j}-y_{i,b}y_{a,b}^{-1}\left(y_{a,j}-y_{a,b'}y_{a',b'}^{-1}y_{a',j}\right)
\end{align*}
and 
\begin{align*}
\left( y_{i,j}^{\ang{a,b}} \right)^{\ang{a',b'}} 
&= y^{\ang{a,b}}_{i,j}-y_{i,b'}^{\ang{a,b}}\left(y_{a',b'}^{\ang{a,b}}\right)^{-1}y_{a',j}^{\ang{a,b}} \\
&= y_{i,j}-y_{i,b}y_{a,b}^{-1}y_{a,j}-\left(  
y_{i,b'}-y_{i,b}y_{a,b}^{-1}y_{a,b'}
\right)y_{a',b'}^{-1}y_{a',j}.
\end{align*}
It is trivial to check that the above expressions for $\left( y_{i,j}^{\ang{a',b'}} \right)^{\ang{a,b}}$ and $\left( y_{i,j}^{\ang{a,b}} \right)^{\ang{a',b'}}$ coincide. 
\end{proof}
\end{lemma}
\begin{remark}
In adapting the above lemma to $\yldashf$, one should think of $(a,b)$ as $(u,v)$ and $(a',b')$ as any box occurring after $(u,v)$ in $Y_{\lambda'}$. 
\end{remark}
The following proposition is an immediate consequence of Lemma \ref{formal dda in general division rings} and the discussion preceding it.
\begin{proposition}\label{same Cauchon sets}
Let $J$ be any $\ch$-prime ideal of $\yldashf$. Then the Cauchon diagram for $J$ is the same whether one uses the standard lexicographical QNA presentation for $\yldashf$ or the QNA presentation for $\yldashf$ given in \eqref{obv}, where $z=x_{u,v}$ appears at the end.  As a consequence, the set of Cauchon diagrams in both cases coincide with the set of Cauchon-Le diagrams on $Y_{\lambda'}$. 
\end{proposition}

Both the automorphism $\sigma$ and the left $\sigma$-derivation $\delta$ of $\ylf$ clearly restrict to $\ylr$; we shall abuse notation slightly and refer to these restrictions also as $\sigma$ and $\delta$. The equality \eqref{obv} restricts to give
\[
\yldashr=\ylr[z; \sigma, \delta].
\]
In the rest of this section, in order to save space in displays, we will set 
\[
A_\lambda:= \ylr, A_{\lambda'}:= \yldashr,\ B_\lambda:= \ylf, \tx{ and } B_{\lambda'}:= 
\yldashf. 
\]
In particular, 
\[
B_{\lambda'}=B_\lambda[z; \sigma, \delta] \ \tx{ and } \ A_{\lambda'}=A_\lambda[z; \sigma, \delta].
\]
\begin{remark}
Notice that if we take $R=\f$, then $A_\lambda=B_\lambda$ and $A_{\lambda'}=B_{\lambda'}$. 
\end{remark}
By \cite[Lemme 2.1]{c1}, the powers of $z$ form a left and a right Ore set in $B_{\lambda'}$; inspection of the proof shows that although not all of the assumptions of \cite[Lemme 2.1]{c1} hold for $A_{\lambda'}$, the proof is valid for showing that the 
powers of $z$ form a left and a right Ore set in $A_{\lambda'}$. As an immediate consequence of \cite[Proposition 2.2, Proposition 2.3(2)]{c1}, there is an embedding 
\[
\theta:B_\lambda \hookrightarrow B_{\lambda'}[z^{-1}]
\]
such that for all $(i,j)\in Y_\lambda$, 
\begin{equation}\label{theta formula}
\piecewise{\theta(x_{i,j})}{x_{i,j}-x_{i,v}z^{-1}x_{u,j}}{\tx{if }i<u\tx{ and }j<v;}{x_{i,j}}{\tx{otherwise.}}
\end{equation}

\begin{notation}
For all $(i,j)\in Y_\lambda$, set $y_{i,j}:= \theta(x_{i,j})$. 
\end{notation}
From now on, we will identify $B_\lambda$ with its image under $\theta$ in $B_{\lambda'}[z^{-1}]$. By \cite[Propositions 2.3, 2.4]{c1}, we have
\begin{equation}\label{identifying over F}
B_\lambda[z^{\pm1}; \sigma] = B_{\lambda'}[z^{-1}].
\end{equation}
The embedding $\theta:B_\lambda \hookrightarrow B_{\lambda'}[z^{-1}]$ clearly restricts to an embedding $\theta: A_\lambda\hookrightarrow A_{\lambda'}[z^{-1}]$; from now on, we will identify $A_\lambda$ with its image under $\theta$ in $A_{\lambda'}[z^{-1}]$. The equality \eqref{identifying over F} restricts to 
\begin{equation}\label{identifying over R}
A_\lambda[z^{\pm1}; \sigma] = A_{\lambda'}[z^{-1}].
\end{equation}

From now on, we shall work in $A_{\lambda'}[z^{-1}]$ and $B_{\lambda'}[z^{-1}]$. This means in particular that \begin{itemize}
\item $x_{i,j}$ (with $(i,j)\in Y_{\lambda'}$) shall always denote a standard generator of $A_{\lambda'}$ and $B_{\lambda'}$;
\item $A_\lambda$ and $B_\lambda$ denote respectively the subalgebras of $A_{\lambda'}[z^{-1}]$ and $B_{\lambda'}[z^{-1}]$ generated by $\{y_{i,j}\ |\ (i,j)\in Y_\lambda\}$.
\end{itemize}

The following technical lemma will be useful. 

\begin{lemma}\label{restricting plussing}
Let $J,J'$ be ideals of $B_\lambda,B_{\lambda'}$ respectively and consider the ideals 
$I=J\cap A_\lambda,I'=J'\cap A_{\lambda'}$ of $A_\lambda,A_{\lambda'}$ respectively. 
\begin{enumerate}[(i)]
\item If $z\notin J'$, then we have $J'[z^{-1}]\cap A_{\lambda'}[z^{-1}]=I'[z^{-1}]$.
\item If $z\notin J'$ and $\bigoplus_{i\in \Z}Jz^i = J'[z^{-1}]$, then  
$\bigoplus_{i\in \Z}Iz^i = I'[z^{-1}]$.
\end{enumerate}
\begin{proof}
\begin{enumerate}[(i)]
\item We claim that if $p\in J'$, $m\geq 0$ are such that there exist $s\in A_{\lambda'}$, $n\geq 0$ with $pz^{-m}=sz^{-n}$, then either $p$ or $s$ belongs to $I'$. Indeed
\begin{itemize}
\item If $m=n$, then obviously $p=s\in J'\cap A_{\lambda'}=I'$. 
\item If $n>m$, then $s=pz^{n-m}$, which belongs to $J'$ because $p$ does; hence $s\in J'\cap A_{\lambda'}=I'$. 
\item If $m>n$, then $p=sz^{m-n}$, which belongs to $A_{\lambda'}$ because $s$ and $z$ do; hence $p\in J'\cap A_{\lambda'}=I'$. 
\end{itemize}
From this claim the required result follows easily.
\item  By part (i), we have 
\begin{equation}\label{intersecting}
\bigoplus_{i\in \Z}Jz^i\cap A_\lambda[z^{\pm 1};\sigma] 
= 
J'[z^{-1}]\cap A_{\lambda'}[z^{-1}] = I'[z^{-1}].
\end{equation}
By the independence of the powers of $z$ over $B_\lambda\supseteq A_\lambda$, we get 
\begin{equation}\label{first equality required}
\bigoplus_{i\in \Z}Jz^i\cap A_\lambda[z^{\pm 1};\sigma]
=
\bigoplus_{i\in \Z} (J\cap A_\lambda)z^i=\bigoplus_{i\in \Z}Iz^i.
\end{equation}
Now \eqref{intersecting} and \eqref{first equality required} give $\bigoplus_{i\in \Z}Iz^i = I'[z^{-1}]$, as required.
\end{enumerate}

\end{proof}
\end{lemma}

\begin{proposition}\label{comb}
Fix any ideal $I$ of $A_\lambda$ and let $I'$ be the ideal $\bigoplus_{i\in \Z}Iz^i\cap A_{\lambda'}$ of $A_{\lambda'}$. Every pseudo quantum minor of $A_{\lambda'}$ belonging to $I'$ can be expressed as an $R[z^{\pm 1}]$-linear combination (with coefficients on the right) of pseudo quantum minors of $A_\lambda$ which belong to $I$. \\

\end{proposition} 

\begin{proof}
\begin{full} 
Let us set $x_{ij}:=0$ for all $(i,j)\notin Y_{\lambda'}$ and $y_{ij}:=0$ for all $(i,j)\notin Y_\lambda$. Recall that $z=x_{uv}$ and note the following facts which will be used often in this proof:
\begin{itemize}
\item $x_{ij}=y_{ij}$ when $i\geq u$ or $j\geq v$ (except when $(i,j)=(u,v)$)
\item $y_{ij}=0$ when $i\geq u$ and $j\geq v$.
\end{itemize}

Let us use subscripts to indicate the partition subalgebra in which we are taking pseudo quantum minors. Consider a pseudo quantum minor $[i_1\cdots i_l\ |\ j_1\cdots j_l]_{A_{\lambda'}}$ of $A_{\lambda'}$ which belongs to $I'$. We wish to express $[i_1\cdots i_l\ |\ j_1\cdots j_l]_{A_{\lambda'}}$ as an $R[z^{\pm 1}]$-linear combination (with coefficients on the right) of pseudo quantum minors of $A_\lambda$ which belong to $I$. Because the powers of $z$ are independent over $A_\lambda$ and $I'\subseteq \bigoplus_{i\in \Z}Iz^i$, it will be enough to show that there exist pseudo quantum minors $\delta_{-1},\delta_0,\delta_1$ of $A_\lambda$ and $\mu_{-1},\mu_0,\mu_1\in R^\ast\cup \{0\}$ such that
\[
[i_1\cdots i_l\ |\ j_1\cdots j_l]_{A_{\lambda'}}=\mu_{-1}\delta_{-1}z^{-1}+\mu_0\delta_0 +\mu_1\delta_1z.
\]
If $u<i_1$ or $v<j_1$, then $x_{\bullet\bullet}=y_{\bullet\bullet}$ for each entry in the two pseudo quantum minors; so 
\[
[i_1\cdots i_l\ |\ j_1\cdots j_l]_{A_{\lambda'}}=[i_1\cdots i_l\ |\ j_1\cdots j_l]_{A_{\lambda}}.
\] 
Thus we may assume that $u\geq i_1$ and $v\geq j_1$; from here we break the proof down into four distinct and exhaustive cases.\\


\noindent \emph{\textbf{Case 1: $u\in \{i_1,\ldots,i_l\}$ and $v\notin \{j_1,\ldots,j_l\}$.}}
\\ In this case, we claim that 
\begin{equation}\label{one in one out}
[i_1\cdots i_l\ |\ j_1\cdots j_l]_{A_{\lambda'}}=[i_1\cdots i_l\ |\ j_1\cdots j_l]_{A_{\lambda}}.\\
\end{equation} 
We prove this claim by induction on $(\# i's>u)+(\# j's>v)$. Assume that $(\# i's>u)+(\# j's>v)=0$; that is, assume that $u=i_l$ and $v>j_l$. Then the pseudo quantum minors are actually genuine quantum minors and \eqref{one in one out} follows from 
\cite[Lemme 4.2.1]{cauchon2}.
Assume that $(\# i's>u)+(\# j's>v)>0$. The inductive step differs slightly depending on whether or not $i_l=u$:

Assume that $i_l\neq u$. Then $i_l>u$ and we have 
\begin{align*}
[i_1\cdots i_l\ |\ j_1\cdots j_l]_{A_{\lambda'}} 
&=^1\sum_{p=1}^l(-q)^{l-p}[i_1\cdots i_{l-1}\ |\ \widehat{j_p}]_{A_{\lambda'}}x_{i_lj_p} \\
&=^2\sum_{p=1}^l(-q)^{l-p}[i_1\cdots i_{l-1}\ |\ \widehat{j_p}]_{A_{\lambda}}y_{i_lj_p} \\
&=^3[i_1\cdots i_l\ |\ j_1\cdots j_l]_{A_{\lambda}},
\end{align*}
where in ($=^1$) and ($=^3$), we are using $q$-Laplace expansion on the last row with entries on the right (see Lemma \ref{row Laplace} (2)), and ($=^2$) follows from the inductive hypothesis.

On the other hand, assume that $i_l=u$. Then $j_l>v$ and, noting that $x_{uj_l}=y_{uj_l}=0$, we have
\begin{align*}
[i_1\cdots i_l\ |\ j_1\cdots j_l]_{A_{\lambda'}} 
&=^1\sum_{p=1}^{l-1}(-q)^{l-p}[\widehat{i_p}\ |\ j_1\dots j_{l-1}]_{A_{\lambda'}}x_{i_pj_l} \\
&=^2\sum_{p=1}^{l-1}(-q)^{l-p}[\widehat{i_p}\ |\ j_1\dots j_{l-1}]_{A_{\lambda}}y_{i_pj_l}  \\
&=^3[i_1\cdots i_l\ |\ j_1\cdots j_l]_{A_{\lambda}},
\end{align*}
where in ($=^1$) and ($=^3$), we are using $q$-Laplace expansion on the last column with entries on the right (see Corollary \ref{column Laplace} (2)), and ($=^2$) follows from the inductive hypothesis.

This finishes Case 1. \\


\noindent \emph{\textbf{Case 2: $u\notin \{i_1,\ldots,i_l\}$ and $v\in \{j_1,\ldots,j_l\}$.}} 
\\ In this case, we claim that 
\begin{equation}\label{one out one in}
[i_1\cdots i_l\ |\ j_1\cdots j_l]_{A_{\lambda'}}=[i_1\cdots i_l\ |\ j_1\cdots j_l]_{A_{\lambda}}.
\end{equation} 

We prove this claim by induction on $(\# i's>u)+(\# j's>v)$. Assume that $(\# i's>u)+(\# j's>v)=0$; that is, assume that $i_l<u$ and $j_l=v$. Then the pseudo quantum minors are actually genuine quantum minors and \eqref{one out one in} follows from 
\cite[Lemme 4.2.2]{cauchon2}. Assume that $(\# i's>u)+(\# j's>v)>0$. The inductive step differs slightly depending on whether or not 
$j_l=v$.

Assume that $j_l\neq v$. Then $j_l>v$ and we have 
\begin{align*}
[i_1\cdots i_l\ |\ j_1\cdots j_l]_{A_{\lambda'}} 
&=^1\sum_{p=1}^{l}(-q)^{l-p}[\widehat{i_p}\ |\ j_1\dots j_{l-1}]_{A_{\lambda'}}x_{i_pj_l} \\
&=^2\sum_{p=1}^{l}(-q)^{l-p}[\widehat{i_p}\ |\ j_1\dots j_{l-1}]_{A_{\lambda}}y_{i_pj_l}  \\
&=^3[i_1\cdots i_l\ |\ j_1\cdots j_l]_{A_{\lambda}},
\end{align*}
where in ($=^1$) and ($=^3$), we are using $q$-Laplace expansion on the last column with entries on the right (see Corollary \ref{column Laplace} (2)), and ($=^2$) follows from the inductive hypothesis.

On the other hand, assume that $j_l=v$. Then $i_l>u$ and, noting that $x_{i_lv}=y_{i_lv}=0$, we have
\begin{align*}
[i_1\cdots i_l\ |\ j_1\cdots j_l]_{A_{\lambda'}} 
&=^1\sum_{p=1}^{l-1}(-q)^{l-p}[i_1\cdots i_{l-1}\ |\ \widehat{j_p}]_{A_{\lambda'}}x_{i_lj_p} \\
&=^2\sum_{p=1}^{l-1}(-q)^{l-p}[i_1\cdots i_{l-1}\ |\ \widehat{j_p}]_{A_{\lambda}}y_{i_lj_p} \\
&=^3[i_1\cdots i_l\ |\ j_1\cdots j_l]_{A_{\lambda}},
\end{align*}
where in ($=^1$) and ($=^3$), we are using $q$-Laplace expansion on the last row with entries on the right (see Lemma \ref{row Laplace} (2)), and ($=^2$) follows from the inductive hypothesis.

This finishes Case 2.\\



\noindent \emph{\textbf{Case 3: $u\in \{i_1,\ldots,i_l\}$ and $v\in \{j_1,\ldots,j_l\}$.}}
\\ Where $v=j_k$ and $u=i_t$, we claim that
\begin{equation}\label{both in}
[i_1\cdots i_l\ |\ j_1\cdots j_l]_{A_{\lambda'}}=(-q)^{2l-k-t}[\{i_1,\ldots,i_l\}\bs\{u\}\ |\ \{j_1,\ldots,j_l\}\bs\{v\}]_{A_{\lambda}}z.
\end{equation}
We prove this claim by induction on $(\# i's>u)+(\# j's>v)$. Assume that $(\# i's>u)+(\# j's>v)=0$; that is, assume that $u=i_l$ and $v=j_l$. Then the pseudo quantum minors are actually genuine quantum minors and \eqref{one in one out} follows from 
\cite[Lemme 4.1.2]{cauchon2}. 
Assume that $(\# i's>u)+(\# j's>v)>0$. 
The inductive step differs slightly depending on whether or not $i_l=u$:

Assume that $i_l\neq u$. Then $i_l>u$ and, noting that $x_{i_lj_p}=y_{i_lj_p}=0$ for $p\geq k$, we have
\begin{align*}
[i_1\cdots i_l\ |\ j_1\cdots j_l]_{A_{\lambda'}}
&=^0
\sum_{p=1}^{k-1}(-q)^{l-p}[\widehat{i_l}\ |\ \widehat{j_p}]_{A_{\lambda'}}x_{i_lj_p}\\
&=^1
\sum_{p=1}^{k-1}(-q)^{l-p}[\widehat{i_l}\ |\ \widehat{j_p}]_{A_{\lambda'}}y_{i_lj_p} \\
&=^2
\sum_{p=1}^{k-1}(-q)^{l-p}(-q)^{(l-1)-t+l-k}[\widehat{u}\,\widehat{ i_l}\ |\ \widehat{j_p}\widehat{v}]_{A_{\lambda}}zy_{i_lj_p} \\
&=^3
(-q)^{2l-k-t}\left(\sum_{p=1}^{k-1}(-q)^{l-1-p}[\{i_1,\ldots,i_{l-1}\}\bs\{u\}\ |\ \widehat{j_p}\widehat{v}]_{A_{\lambda}}y_{i_lj_p}\right)z \\
&=^4
(-q)^{2l-k-t}[\{i_1,\ldots,i_l\}\bs\{u\}\ |\ \{j_1,\ldots,j_l\}\bs\{v\}]_{A_{\lambda}}z, \\
\end{align*}
where in ($=^0$) and ($=^4$), we are using $q$-Laplace expansion on the last row with entries on the right (see Lemma \ref{row Laplace} (2)),  ($=^1$) follows from relevant $x_{\bullet\bullet}$ and $y_{\bullet\bullet}$ being equal, ($=^2$) follows from the inductive hypothesis, and ($=^3$) follows from the commutation between $z$ with relevant $y_{\bullet\bullet}$.

On the other hand, if $i_l=u$, then $j_l>v$ and so, noting that $x_{uj_l}=y_{uj_l}=0$, we have
\begin{align*}
[i_1\cdots i_l\ |\ j_1\cdots j_l]_{A_{\lambda'}}
&=^0\sum_{p=1}^{l-1}(-q)^{l-p}[\widehat{i_p}\ |\ j_1\cdots j_{l-1}]_{A_{\lambda'}}x_{i_pj_l} \\
&=^1\sum_{p=1}^{l-1}(-q)^{l-p}[\widehat{i_p}\ |\ j_1\cdots j_{l-1}]_{A_{\lambda'}}y_{i_pj_l} \\
&=^2\sum_{p=1}^{l-1}(-q)^{l-p}(-q)^{(l-1)-k}[\widehat{i_p}\ \widehat{u}\ |\ \{j_1,\ldots,j_{l-1}\}\bs\{v\}]_{A_{\lambda}}zy_{i_pj_l} \\
&=^3(-q)^{l-k}\sum_{p=1}^{l-1}(-q)^{l-1-p}[\widehat{i_p}\ \widehat{u}\ |\ \{j_1,\ldots,j_{l-1}\}\bs\{v\}]_{A_{\lambda}}y_{i_pj_l}z \\
&=^4 (-q)^{l-k}[\{i_1,\ldots,i_l\}\bs\{u\}\ |\ \{j_1,\ldots,j_l\}\bs\{v\}]_{A_{\lambda}}z,
\end{align*}
where in ($=^0$) and ($=^4$), we are using $q$-Laplace expansion on the last column with entries on the right (see Corollary \ref{column Laplace} (2)),  ($=^1$) follows from relevant $x_{\bullet\bullet}$ and $y_{\bullet\bullet}$ being equal, ($=^2$) follows from the inductive hypothesis, and ($=^3$) follows from the commutation between $z$ with relevant $y_{\bullet\bullet}$.

This finishes Case 3.\\



\noindent \emph{\textbf{Case 4: $u\notin \{i_1,\ldots,i_l\}$ and $v\notin \{j_1,\ldots,j_l\}$.}}
\\ Where $k$ is maximal such that $j_k<v$ and $t$ is maximal such that $i_t<u$, we claim that 
\begin{equation}\label{neither in}
[i_1\cdots i_l\ |\ j_1\cdots j_l]_{A_{\lambda'}}=[i_1\cdots i_l\ |\ j_1\cdots j_l]_{A_{\lambda}}-(-q)^{k+t-2l}[\{i_1,\ldots,i_l\}\sqcup\{u\}\ | \{j_1,\ldots,j_l\}\sqcup \{v\}]_{A_{\lambda}}z^{-1}.
\end{equation}
We prove this claim by induction on $(\# i's>u)+(\# j's>v)$. Assume that $(\# i's>u)+(\# j's>v)=0$; that is, assume that $u>i_l$ and $v>j_l$. Then 
\cite[Proposition 3.1.4.3]{launois-thesis} gives 
\begin{equation}\label{stephane thesis}
[i_1\cdots i_l\ |\ j_1\cdots j_l]_{A_{\lambda'}}=[i_1\cdots i_l\ |\ j_1\cdots j_l]_{A_\lambda}-\sum_{p=1}^l(-q)^{(l+1)-p}[i_1\cdots \widehat{i_p}\cdots i_l\ u \ |\ j_1\cdots j_l]_{A_\lambda}y_{i_pv}z^{-1}
\end{equation}
(all pseudo quantum minors appearing in \eqref{stephane thesis} are genuine quantum minors).
However it follows from q-Laplace expansion with the last column on the right (Lemma \ref{column Laplace}) and from the fact that $y_{uv}=0$ that 
$$\sum_{p=1}^l(-q)^{(l+1)-p}[i_1\cdots \widehat{i_p}\cdots i_l\ u \ |\ j_1\cdots j_l]_{A_\lambda}y_{i_pv}=[i_1\cdots i_l\ u\ |\ j_1\cdots j_l\ v]_{A_\lambda},$$ so that \eqref{neither in} holds. 

Assume that $(\# i's>u)+(\# j's>v)>0$. The inductive step differs slightly depending on whether or not $i_l>u$.

Assume that $i_l>u$. Then, noting that $x_{i_lj_p}=y_{i_lj_p}=0$ for $p>k$, we have 
\begin{align*}
&{ \ \ \ \ } [i_1\cdots i_l\ |\ j_1\cdots j_l]_{A_{\lambda'}} \\
&=^0\sum_{p=1}^k(-q)^{l-p}[i_1\cdots i_{l-1}\ |\ \widehat{j_p}]_{A_{\lambda'}}x_{i_lj_p} \\
&=^1\sum_{p=1}^k(-q)^{l-p}[i_1\cdots i_{l-1}\ |\ \widehat{j_p}]_{A_{\lambda'}}y_{i_lj_p} \\
&=^2\sum_{p=1}^k(-q)^{l-p}\left( [i_1\cdots i_{l-1}\ |\ \widehat{j_p}]_{A_{\lambda}}-(-q)^{t-(l-1)+k-l}[\{i_1,\ldots, i_{l-1}\}\sqcup\{u\}\ |\ \widehat{j_p}\sqcup\{v\}]_{A_{\lambda}}z^{-1} \right)y_{i_lj_p} \\
&=^3 \sum_{p=1}^k(-q)^{l-p}[i_1\cdots i_{l-1}\ |\ \widehat{j_p}]_{A_{\lambda}}Y_{i_lj_p}-(-q)^{k+t-2l}\sum_{p=1}^k(-q)^{l+1-p}[\{i_1,\ldots, i_{l-1}\}\sqcup\{u\}\ |\ \widehat{j_p}\sqcup\{v\}]_{A_{\lambda}}y_{i_lj_p}z^{-1} \\
&=^4[i_1\cdots i_l\ |\ j_1\cdots j_l]_{A_{\lambda}}-(-q)^{k+t-2l}[\{i_1,\ldots,i_l\}\sqcup\{u\}\ | \{j_1,\ldots,j_l\}\sqcup \{v\}]_{A_{\lambda}}z^{-1}.
\end{align*}
where in ($=^0$) and ($=^4$), we are using $q$-Laplace expansion on the last row with entries on the right (see Lemma \ref{row Laplace} (2)),  ($=^1$) follows from relevant $x_{\bullet\bullet}$ and $y_{\bullet\bullet}$ being equal, ($=^2$) follows from the inductive hypothesis, and ($=^3$) follows from the commutation between $z^{-1}$ with relevant $y_{\bullet\bullet}$.

On the other hand, assume that $i_l\ngtr u$. Then $i_l<u$, $j_l>v$, and 
we have
\begin{align*}
&{ \ \ \ \ } [i_1\cdots i_l\ |\ j_1\cdots j_l]_{A_{\lambda'}} \\
&=^0\sum_{p=1}^l(-q)^{l-p}[\widehat{i_p}\ |\ j_1\cdots j_{l-1}]_{A_{\lambda'}}x_{i_pj_l} \\
&=^1\sum_{p=1}^l(-q)^{l-p}[\widehat{i_p}\ |\ j_1\cdots j_{l-1}]_{A_{\lambda'}}y_{i_pj_l} \\
&=^2\sum_{p=1}^l(-q)^{l-p}\left( [\widehat{i_p}\ |\ j_1\cdots j_{l-1}]_{A_{\lambda}}-(-q)^{k-(l-1)}[\widehat{i_p}\sqcup\{u\}\ |\ \{j_1,\ldots,j_{l-1}\}\sqcup\{v\}]_{A_{\lambda}}z^{-1} \right)y_{i_pj_l} \\
&=^3\sum_{p=1}^l(-q)^{l-p} [\widehat{i_p}\ |\ j_1\cdots j_{l-1}]_{A_{\lambda}}Y_{i_pj_l}-(-q)^{k-l}\sum_{p=1}^l(-q)^{l+1-p}[\widehat{i_p}\sqcup\{u\}\ |\ \{j_1,\ldots,j_{l-1}\}\sqcup\{v\}]_{A_{\lambda}}y_{i_pj_l}z^{-1} \\
&=^4[i_1\cdots i_l\ |\ j_1\cdots j_l]_{A_{\lambda}}-(-q)^{k-l}[\{i_1,\ldots,i_l\}\sqcup\{u\}\ | \{j_1,\ldots,j_l\}\sqcup \{v\}]_{A_{\lambda}}z^{-1},
\end{align*}
where in ($=^0$) and ($=^4$), we are using $q$-Laplace expansion on the last column with entries on the right (see Corollary \ref{column Laplace} (2)),  ($=^1$) follows from relevant $x_{\bullet\bullet}$ and $y_{\bullet\bullet}$ being equal, ($=^2$) follows from the inductive hypothesis, and ($=^3$) follows from the commutation between $z^{-1}$ with relevant $y_{\bullet\bullet}$.

This finishes Case 4 and completes the proof.

\end{full} 
\end{proof}

\begin{remark}
Though we did not say it explicitly, in the proof of Proposition \ref{comb} we showed that that every pseudo quantum minor in $A_{\lambda'}$ can be expressed as an $R[z^{\pm 1}]$-linear combination (with coefficients on the right) of pseudo quantum minors in $A_\lambda$. We can also get this result simply by setting $I=A_\lambda$ in Proposition \ref{comb}.
\end{remark}

\begin{corollary}\label{gen by q minors R case}
Let $P'$ be a completely prime ideal of $A_{\lambda'}$ that does not contain $z$ and that is generated as a right ideal by pseudo quantum minors of $A_{\lambda'}$. Assume there is an ideal $P$ of $A_\lambda$ satisfying $\bigoplus_{i\in \Z} Pz^i = P'[z^{-1}]$. Then $P$ is completely prime and is generated as a right ideal by pseudo quantum minors of $A_\lambda$.
\begin{proof}
Clearly $P=P'[z^{-1}]\cap A_\lambda$, from which it follows easily that $P$ is completely prime.

Where $\delta_1',\ldots,\delta_n'$ are all the pseudo quantum minors of $A_{\lambda'}$ belonging to $P'$, by assumption we have $P'=\delta_1'A_{\lambda'}+\cdots +\delta_n'A_{\lambda'}$. Since $P\subset P'[z^{-1}]$, we have 
\[
P\subset \delta_1'A_{\lambda'}[z^{-1}]+\cdots +\delta_n'A_{\lambda'}[z^{-1}]=
\delta_1'A_{\lambda}[z^{\pm 1}; \sigma]+\cdots +\delta_n'A_{\lambda}[z^{\pm 1}; \sigma].
\] 
Let $\delta_1,\ldots,\delta_m$ be the pseudo quantum minors of $A_\lambda$ belonging to $P$. Notice that $P'=\bigoplus_{i\in \Z}Pz^i \cap A_{\lambda'}$; this allows us to invoke Proposition \ref{comb} to give
\begin{equation}\label{the desired comb}
P\subset \delta_1A_\lambda[z^{\pm 1};\sigma]+\cdots + \delta_mA_\lambda[z^{\pm 1};\sigma].
\end{equation}
From the independence of the powers of $z$ over $A_\lambda$, we conclude that
\[
P=\delta_1A_\lambda+\cdots + \delta_mA_\lambda.
\]
\end{proof}
\end{corollary}

\begin{proposition}\label{H-prime inductive step}
If the $\ch$-prime ideals of $B_{\lambda'}$ are generated as right ideals by pseudo quantum minors of $B_{\lambda'}$, then the same statement is true for $B_\lambda$. 
\end{proposition}
\begin{proof}
\begin{full} 
Let $J$ be an $\ch$-prime ideal of $B_\lambda$ and 
let $J'=\bigoplus_{i\in \Z}Jz^i\cap B_{\lambda'}$; clearly $z\notin J'$. By \cite[Theorem 2.3]{llr-ufd}, $\bigoplus_{i\in \Z}Jz^i$ is an $\ch$-prime ideal of $B_\lambda[z^{\pm 1};\sigma]=B_{\lambda'}[z^{-1}]$; hence $J'$ is an $\ch$-prime ideal of $B_{\lambda'}$ and 
\begin{equation}\label{plussing up}
\bigoplus_{i\in \Z}Jz^i = J'[z^{-1}]. 
\end{equation}
Since $J'$ is completely prime (in fact all prime ideals of $B_{\lambda'}$ are completely prime) and does not contain $z$, we may conclude via Corollary \ref{gen by q minors R case} (with $R=\f$).
\end{full} 
\end{proof}

We are now ready to state the main result of this section. 
\begin{theorem}\label{gen by pseudos}
Assume that $q\in \f^*$ is not a root of unity. The $\ch$-prime ideals in a partition subalgebra of quantum matrices are generated as right ideals by pseudo quantum minors. 
\end{theorem}

\begin{proof} 
\begin{full} 
The $\ch$-primes in $\oqmmnf$ are generated as right ideals by quantum minors, see 
\cite{cas1, cas2}. For a partition subalgebra of $\oqmmnf$, induction on the number of missing boxes, plus 
Proposition~\ref{H-prime inductive step} gives the result. 
\end{full} 
\end{proof}

\section{Primes in $\ylr$ from Cauchon-Le diagrams}\label{Primes in ylr from Cauchon-Le diagrams}

Let us continue with the notation and conventions of Section \ref{Gens_for_H_primes_in_part_subalgs} and let us add the following conventions: 
\begin{itemize}
\item There is a subfield $\k$ of $\f$ over which $q$ is transcendental. 
\item $R$ is the Laurent polynomial algebra $\k[q^{\pm 1}]$. 
\item $\f=Frac(R)=\k(q)$.
\end{itemize}
Let $C$ be a Cauchon-Le diagram on $Y_\lambda$ and let $C'$ be the Cauchon-Le diagram on $Y_{\lambda'}$ that has the 
same black boxes as $C$ has on $Y_\lambda$

\begin{notation}{\rm  
Let $\Pi_{\lambda}$ denote the set of pseudo quantum minors in $\ylr$ (or, equivalently, in $\ylf$) and let $\pcl$ denote the set of pseudo quantum minors $[I\mid J]$ for which there are no vertex disjoint families of $R_{(I,J)}$-paths in the Postnikov graph of $C$. Define $\Pi_{\lambda'}$ and $\pcldash$ similarly.
}\end{notation}

Let $\idealpclr$ be the ideal in $\ylr$ generated by 
$\pcl$. We aim to show that $\idealpclr$ is a completely prime ideal 
in $\ylr$ that is generated as a right ideal by $\pcl$. The proof will be by induction. The partition $Y_\lambda$ sits 
in an $m\times n$ array, and the induction will be on the difference between 
$mn$ and the number of boxes in $Y_\lambda$. The base case, 
where this difference is equal to zero, is that of $\oqmmnr$, and this result has 
been established in Section \ref{Primes in oqmmnr from Cauchon-Le diagrams} (see Corollary \ref{corollary-cp-over-R}). Thus, we need only deal with the inductive step. Before embarking on the proof, we need some notation and 
preparatory 
results. 
\\

Let $\idealpclrdash$ be the ideal in $\yldashr$ generated by $\pcldash$. Over the next few results,  under the assumption that $J'_R:=\idealpclrdash$ is completely prime and generated as a right ideal by $\pcldash$, we will  show that $\idealpclr$ is also completely prime and is generated as a right ideal by $\pcl$. 
\\

Recall that, as in the previous section, we set $A_\lambda:= \ylr$, $A_{\lambda'}:= \yldashr$, $B_\lambda:= \ylf$ and $B_{\lambda'}:= \yldashf$. Moreover, $z$ still denotes the new variable corresponding to the extra box in $Y_{\lambda'}$.
\\

Now, let $J'_{\f} := \idealpclfdash$ be the ideal of $\yldashf$ generated by $\pcldash$. Notice that, by Theorems  \ref{which pseudo minors} and \ref{gen by pseudos}, $J'_{\f}$ is the $\ch$-invariant (completely) prime ideal of $\yldashf$ with Cauchon-Le diagram $C'$. Since the box in $C'$ corresponding to $z$ is white, \cite[Proposition 5.4.2]{c1} (which is rewritten for our current context as Lemma 4.3.3 in \cite{nolan}) guarantees that $z\notin J'_{\f}$. Set $J^0_{\f}:= J_{\f}'[z^{-1}]\cap \ylf$ and notice that, by \cite[Theorem 2.3]{llr-ufd}, $J^0_{\f}$ is an $\ch$-prime ideal of $\ylf$ satisfying 
\begin{equation} \label{observation-a}
\bigoplus_{i\in \Z}J^0_{\f} z^i=J'_{\f}[z^{-1}].
\end{equation}  
By Lemma~\ref{lemma-equal-cauchon-diagrams}, the black boxes in the Cauchon-Le diagram of $J^0_{\f}$ in $Y_\lambda$ are the same as the 
black boxes in the Cauchon-Le diagram $C'$ of $J'_{\f}$ in $Y_{\lambda'}$. (Note that, here, we are really talking about the Cauchon diagram of  $J'_{\f}$ with respect to the alternative QNA presentation of $\yldashf$ where $z$ appears at the end. However, Lemma \ref{same Cauchon sets} shows that this Cauchon diagram coincides with the Cauchon-Le diagram of $J'_{\f}$ for the standard QNA presentation of $\yldashf$.) Hence, $J^0_{\f}$ is the $\ch$-prime ideal of $\ylf$ with Cauchon-Le diagram $C$; it follows by Theorem~\ref{which pseudo minors} that $J^0_{\f}\cap\Pi_\lambda = \pcl$.
\\

Assume that $J'_R= \idealpclrdash$ is completely prime and, noting that $z\notin J'_R$ since $z\notin J'_{\f}\supset J'_R$, set $J^0_R:= J'_R[z^{-1}] \cap A_\lambda$. Then $J^0_R$ is a completely prime  ideal of $A_\lambda$.

\begin{lemma} \label{lemma-pqm-in-j-zero-r}
With the notation and assumptions above, 
\[
J^0_R\cap\Pi_\lambda = \pcl.
\]
\end{lemma} 
\begin{proof} 
It follows from the definitions that $J^0_R\subseteq J^0_{\f}$. Recalling that Theorem~\ref{which pseudo minors} gives $J^0_{\f}\cap\Pi_\lambda = \pcl$, we have $J^0_R\cap\Pi_\lambda \subseteq J^0_{\f}\cap\Pi_\lambda=\Pi_\lambda^C$. 
Thus, all we need to show is that $\pcl\subseteq J^0_R$. 
Suppose that $\delta\in\pcl$. Then 
\[
\delta\in J^0_{\f}=J'_{\f}[z^{-1}] \cap B_\lambda.
\]
Hence, there exists a nonnegative integer $i$ such that $\delta z^i\in J'_{\f} = J_R'[(R^*)^{-1}]$, where $R^*$ denotes the nonzero elements of $R$. 
Thus, there exists $0\neq w\in R$ such that $\delta z^iw\in J_R'$. As $J_R'$ is a completely prime ideal, either $w\in J_R'$,  or $\delta z^i\in J_R'$. If the first possibility occurs then $w\in J_{\f}'$, which is impossible as $w$ is a unit in $B_{\lambda'}$. Hence, $\delta\in J_R'[z^{-1}]\cap A_\lambda =J^0_R$, as required.
 \end{proof} 

\begin{corollary}\label{corollary-main}
Suppose that that $J'_R:=\idealpclrdash$ is completely prime and generated as a right ideal by $\pcldash$. Then $\idealpclr$
is also completely prime and is generated as a right ideal by $\pcl$. 
\end{corollary}

\begin{proof}
Recall that $J'_{\f} = J'_R[(R^*)^{-1}]$ and that we have assumed that $J_R'$ is completely prime, so that $J'_R=J'_{\f}\cap A_{\lambda'}$.  It follows that $J_R^0=J_{\f}^0\cap A_\lambda$. Indeed
\begin{align*}
J_R^0 &= J_R'[z^{-1}]\cap A_\lambda \\
      &= (J_{\f}'\cap A_{\lambda'})[z^{-1}]\cap A_\lambda \\
      &= J_{\f}'[z^{-1}]\cap A_{\lambda'}[z^{-1}]\cap A_\lambda \ \ \ \ \tx{(by Lemma \ref{restricting plussing}(i))} \\
      &= J_{\f}'[z^{-1}]\cap A_\lambda \\
      &= (J_{\f}'[z^{-1}]\cap B_\lambda)\cap A_\lambda \\
      &=J_{\f}^0\cap A_\lambda.
\end{align*}

Hence, by Lemma \ref{restricting plussing}(ii) and \eqref{observation-a}, we get 
\begin{equation}\label{required identity}
\bigoplus_{i\in \Z}J^0_R z^i=J'_R[z^{-1}].
\end{equation}
Now Corollary~\ref{gen by q minors R case} shows that $J^0_R$ is completely prime and is generated as a right ideal by the pseudo quantum minors which it contains. Thus we deduce from Lemma~\ref{lemma-pqm-in-j-zero-r} that $J^0_R$ is completely prime and is generated as a right ideal by $\pcl$; that is, $J^0_R=\idealpclr$ is completely prime and is generated as a right ideal by $\pcl$, as required. 
\end{proof}  


We now have all the necessary ingredients to prove the main theorem of this 
section. 

\begin{theorem} \label{theorem-ideal-is-prime}
Let $\k$ be a field and let $q$ be an element that is transcendental over $\k$. 
Set $R:=\k[q^{\pm1}]$. Fix a partition $\lambda$ and let $C$ be a Cauchon-Le 
diagram on the Young diagram $Y_\lambda$. Let $\pcl$ be the set of pseudo 
quantum minors $[I\mid J]$ for which there are no vertex disjoint families 
of $R_{(I,J)}$-paths in the Postnikov graph $\postc$ of $C$, and let 
$P:=\idealpclr$ be the ideal of $\ylr$ generated by $\pcl$. 

Then $P$ is a completely prime ideal of $\ylr$ and $P$ is generated as a right ideal by $\pcl$. 
\end{theorem} 

\begin{proof} 
Let $m,n$ be minimal such that $Y_\lambda$ sits in an $m\times n$ array. 
The proof will be by induction on the number $t$ that is the difference between $mn$ and the number of boxes in $Y_\lambda$. \\

If $t=0$ then $Y_\lambda$ is itself an $m\times n$ array, and $\ylr = \oqmmnr$. 
In this case, the result is given by Corollary~\ref{corollary-cp-over-R}.\\

Now, assume that $t>0$ and that the result is true for values less than $t$. \\

Augment $\lambda$ to $\lambda'$ such that the Young diagram $Y_{\lambda'}$ of $\lambda'$ is obtained from $Y_\lambda$ by adding a box in the manner described at the beginning of Section \ref{Gens_for_H_primes_in_part_subalgs}.
Let $C'$ be the Cauchon-Le diagram on  $Y_{\lambda'}$ that has the same black boxes as $C$ has on $Y_\lambda$, 
and let $\idealpclrdash$ be the ideal in $\yldashr$ generated by $\pcldash$. Let $t'$ 
be the difference between $mn$ and the number of boxes in $Y_{\lambda'}$. 
Then $t'=t-1$ and so the result holds in $\yldashr$, by the inductive hypothesis; 
that is, $P':=\idealpclrdash$ is a completely prime ideal of $\yldashr$ that is generated as a 
right ideal of $\yldashr$ by $\pcldash$.\\

Corollary~\ref{corollary-main} shows that $\idealpclr$ is a completely prime ideal 
that is generated as a right ideal by $\pcl$, as required to prove the inductive step.

\end{proof} 

Of course, one can also show that the ideal $P$ from Theorem \ref{theorem-ideal-is-prime} is generated as a left ideal by $\pcl$.




\section{$\ch$-primes in the quantum grassmannian: membership}\label{qgrass}
Fix positive integers $m<n$ and let $q\in \k^*$. Assume that $q$ is not a root of unity. The  {\em quantised homogeneous coordinate ring of the $m\times n$ grassmannian over a field $\k$} (informally known as the \emph{$(m\times n)$ quantum grassmannian}), and denoted by $\oqgmnk$, is defined to be the subalgebra of $\oqmmnk$ generated by the maximal quantum minors of the matrix 
\begin{equation}\label{gen matrix}
\left(\begin{array}{ccc}
x_{1,1}&\cdots & x_{1,n} \\
\vdots &\ddots & \vdots \\
x_{m,1}&\cdots & x_{m,n} \\
\end{array}\right)
\end{equation}
of canonical generators of $\oqmmnk$. By \cite[Theorem 1.1]{klr}, the quantum grassmannian $\oqgmnk$ is a noetherian domain. 

An $m\times m$ quantum minor of the matrix \eqref{gen matrix} must involve each of the $m$ rows of \eqref{gen matrix}; so in order to specify a maximal quantum minor  one needs only specify $m$ of the $n$ columns. As such, the generators of $\oqgmnk$ are written as $[\gamma_1 \cdots \gamma_m]$ where $1\leq \gamma_1<\gamma_2<\cdots<\gamma_m\leq n$; that is, $[\gamma_1 \cdots \gamma_m]$  denotes the maximal quantum minor $[1\cdots m\ |\ \gamma_1\cdots\gamma_m]$ of $\oqmmnk$. 
Such sets $\gamma:=\{\gamma_1<\gamma_2<\cdots<\gamma_m\}$ are called {\em index sets}, and the maximal quantum minors  $[ \gamma_1\cdots\gamma_m]$ are called the  \emph{quantum Pl\"ucker coordinates} of $\oqgmnk$. The set of quantum Pl\"ucker coordinates of $\oqgmnf$ is denoted $\Pi_{m,n}$ or $\Pi$ when  $m$ and $n$ are understood. We shall often identify  $\Pi$ with the set of all $m$-element subsets of $\llb 1,n \rrb$ in the obvious way.

There is a natural partial order on $\Pi$ given by 
\begin{equation}\label{partial order on Pi}
[\gamma_1\cdots \gamma_m]\leq [\gamma_1'\cdots \gamma_m']\iff (\gamma_i\leq \gamma_i'\tx{ for all }i\in \llb 1,m \rrb).
\end{equation}
Next, there is an action of the torus $\ch:=(\k^*)^n$ on $\oqgmnk$ coming from the column action 
on quantum matrices. Thus, 
\begin{equation}\label{H action on qgrass}
(\alpha_1,\ldots,\alpha_n)\cdot [\gamma_1\cdots \gamma_m]:=\alpha_{\gamma_1}\cdots \alpha_{\gamma_m}[\gamma_1\cdots \gamma_m].
\end{equation}
By \cite[Corollary 2.1]{klr}, the algebra $\oqgmnk$ has a $\k$-basis consisting of products of quantum Pl\"ucker coordinates. Since quantum Pl\"ucker coordinates are clearly $\ch$-eigenvectors with rational eigenvalues, it follows easily that the action of $\ch$ on $\oqgmnk$ is rational.

The set of $\ch$-invariant prime ideals of $\oqgmnk$ has been studied in \cite{llr-grass}, building on earlier work in \cite{klr, lr2, lr3} and \cite{lruss}. Our aim in this chapter is to develop a graph-theoretical method for deciding whether or not a given quantum Pl\"ucker coordinate belongs to a given $\ch$-prime ideal of $\oqgmnk$, by using and refining the methods developed in these papers to reduce this problem to the problem of deciding when a given pseudo quantum minor in a partition subalgebra is in an $\ch$-prime ideal of the partition subalgebra, a problem that we have solved in Section~\ref{section-cauchon-le}. 

We start by summarising the known results that we need from these aforementioned papers. 


\subsection{Quantum graded algebras with a straightening law}\label{section-known}

The quantum grassmannian $\oqgmnk$, equipped with the partially ordered set 
$(\Pi, \leq)$ is a quantum graded algebra with a straightening law in the sense of the following 
definition, see \cite[Theorem 3.4.4]{lr2}. In particular, $\oqgmnk$ is $\mathbb{N}$-graded, with all quantum Pl\"ucker coordinates in degree one.

Let $A$ be an algebra and $\Pi$ a finite subset of 
elements of $A$ with a partial order $<_\st$. A  
{\em standard monomial} on $\Pi$ is an element
of $A$ which is either $1$ or of the form $\alpha_1\cdots\alpha_s$, 
for some $s\geq 1$, where $\alpha_1,\dots,\alpha_s \in \Pi$ and
$\alpha_1\le_\st\dots\le_\st\alpha_s$. 

\begin{definition} \label{recall-q-gr-asl}
{\rm Let $A$ be an ${\mathbb N}$-graded $\k$-algebra and $\Pi$ a finite subset 
of $A$ equipped with a partial order $<_\st$. 
We say that $A$ is a {\em quantum graded algebra with a straightening law} 
({\em QGASL} for short) on the poset $(\Pi,<_\st)$ 
if the following conditions are satisfied.\\
(1) The elements of $\Pi$ are homogeneous with positive degree.\\
(2) The elements of $\Pi$ generate $A$ as a $\k$-algebra.\\
(3) The set of standard monomials on $\Pi$ is a linearly independent set.\\
(4) If $\alpha,\beta\in\Pi$ are not comparable for $<_\st$, 
then $\alpha\beta$ 
is a linear combination of terms $\lambda$ or $\lambda\mu$, where 
$\lambda,\mu\in\Pi$, $\lambda\le_\st\mu$ and $\lambda<_\st\alpha,\beta$.\\
(5) For all $\alpha,\beta\in\Pi$, there exists $c_{\alpha\beta} \in \k^\ast$ 
such that $\alpha\beta-c_{\alpha\beta}\beta\alpha$ is a linear combination of 
terms $\lambda$ or $\lambda\mu$, where $\lambda,\mu\in\Pi$,
$\lambda\le_\st\mu$ and $\lambda<_\st\alpha,\beta$.
}
\end{definition}

By \cite[Proposition 1.1.4]{lr2}, if $A$ is a QGASL on the
partially ordered set $(\Pi,<_\st)$, then the set of standard monomials on $\Pi$ forms an
$\k$-basis of $A$. Hence, in the presence of a standard monomial basis, the
structure of a QGASL may be seen as providing more detailed 
information on the way standard monomials multiply and commute.\\

For any $\gamma\in \Pi$, set $\Pi_\gamma=\{\alpha\in \Pi\ |\ \alpha\ngeq \gamma\}$. 
By \cite[Theorem 5.1]{llr-grass}, for every $P\in \Spec(\oqgmnk)$ other than 
the irrelevant ideal $\ang{\Pi}$, there is a unique $\gamma\in \Pi$ such that $\gamma\notin P$ and $\Pi_\gamma\subseteq P$.
For any $\gamma\in \Pi$, let $\hspec_\gamma(\oqgmnk)$ denote the subspace of $\Spec(\oqgmnk)$ consisting of all those $\ch$-prime ideals $J$ such that $\gamma\notin J$ and $\Pi_\gamma\subseteq J$; we have 
\begin{equation}\label{splitting into cells}
\hspec(\oqgmnk)=\bigsqcup_{\gamma\in \Pi}\ts{\hspec_\gamma}(\oqgmnk)\sqcup \ang{\Pi}.
\end{equation}

In order to understand the whole of $\hspec(\oqgmnk)$ it is useful to study the individual 
subsets $\hspec_\gamma(\oqgmnk)$ of this partition. This is done via the notions of quantum 
Schubert varieties and quantum Schubert cells, as outlined below. 


\begin{convention}
For the rest of this chapter, let us fix some $\gamma=[\gamma_1 \cdots\gamma_m]\in \Pi$. \end{convention}

If $J\in\hspec_\gamma(\oqgmnk)$, then by the definition of $\hspec_\gamma(\oqgmnk)$, we know that $\gamma\notin J$ and that $\alpha\in J$ for all $\alpha\in \Pi$ such that $\alpha\ngeq \gamma$. What remains is to decide which other quantum Pl\"ucker coordinates belong to $J$; that is, given $\alpha\in \Pi$ such that $\alpha>\gamma$, we seek to decide whether or not $\alpha$ belongs to $J$. 
The key to achieving this goal is to exploit the correspondence (established in \cite{llr-grass}) between $\hspec_\gamma \oqgmnk$ and the $\ch$-spectrum of a certain partition subalgebra of $\co_{q^{-1}}(M_{m,n-m}(\k))$. We shall describe this correspondence below.


\subsection{Noncommutative dehomogenisation}\label{subsection-dehom}

The process of \emph{noncommutative dehomogenisation}, introduced in \cite[Section 3]{klr}, is the foundation for the construction in \cite{llr-grass} of a biincreasing one-to-one correspondence between $\hspec_\gamma (\oqgmnk)$ and $\hspec(\partitionk)$, where $\lambda$ is a partition associated to $\gamma$.

Let $R=\bigoplus_{i\in \N}R_i$ be an $\N$-graded $\f$-algebra, and let $x$ be a homogeneous normal regular element of degree one. Set $S:=R[x^{-1}]$. The algebra $S$ is $\mz$-graded with 
$S=\bigoplus_{l\in \Z}S_l$ where $S_l:=\sum_{t=0}^\infty R_{l+t}x^{-t}$. (In this sum, we take 
$R_i=0$ for $i<0$.)\\

\begin{definition}\label{dehom def}
{\rm 
Let $R=\bigoplus_{i\in \N}R_i$ be an $\N$-graded $\k$-algebra and let $x$ be a homogeneous regular normal element of $R$ of degree one. The \emph{noncommutative dehomogenisation} of $R$ at $x$, written $\Dhom(R,x)$, is the subalgebra $S_0=\sum_{t=0}^\infty R_tx^{-t}=\bigcup_{t=0}^\infty R_tx^{-t}$ of the $\Z$-graded algebra $R[x^{-1}]=S=\bigoplus_{l\in \Z}S_l$. 
}\end{definition}

Denote by $\sigma$ the conjugation automorphism of $S$ given by $\sigma(s)=xsx^{-1}$ for all $s\in S$. It is easy to check that $\sigma$ restricts to an automorphism of $\Dhom(R,x)=S_0$ (which we shall also denote by $\sigma$). By \cite[Lemma 3.1]{klr}, the inclusion $\Dhom(R,x)\hookrightarrow R[x^{-1}]$ extends to an isomorphism 
\[
\Dhom(R,x)[y^{\pm 1};\sigma]\xrightarrow{\cong} R[x^{-1}]
\]
which sends $y$ to $x$. 


\subsection{Quantum Schubert varieties and quantum Schubert cells}

The ideal $\ang{\Pi_\gamma}$ of $\oqgmnk$ is completely prime, by \cite[Corollary 3.1.7]{lr2}; and so the noetherian algebra $S(\gamma):=\oqgmnf/\ang{\Pi_\gamma}$ is a domain. It is well known that $\oqgmnk$ is an $\N$-graded $\k$-algebra with each quantum Pl\"ucker coordinate being homogeneous of degree $1$. As  the elements of $\Pi_\gamma$ are homogeneous, there is an induced $\N$-grading on $S(\gamma)$. By \cite[Remark 1.4]{llr-grass}, $\ol{\gamma}\in S(\gamma)$ is a homogeneous regular normal element of degree one, so that we may dehomogenise $S(\gamma)$ at $\ol{\gamma}$ (in fact this follows from a more general result, see \cite[Lemma 1.2.1]{lr2}).


\begin{definition}{\rm 
The algebra $S(\gamma):=\oqgmnk/\ang{\Pi_\gamma}$ is called the \emph{quantum Schubert variety} associated to $\gamma$.
The algebra $S^o(\gamma):=\Dhom(S(\gamma),\ol{\gamma})$ is called the \emph{quantum Schubert cell} associated to $\gamma$. 
}\end{definition}


\begin{remark}{\rm 
We shall later describe an isomorphism (established in \cite[Theorem 4.7]{llr-grass}), between the quantum Schubert cell $S^o(\gamma)$ and a partition subalgebra of $\co_{q^{-1}}(M_{m,n-m}(\k))$.
}\end{remark}


\begin{definition}{\rm 
The \emph{ladder} associated to $\gamma$ is denoted by $\cl_\gamma$ and defined by 
\[
\cl_\gamma=\{(i,j)\in \llb 1,m \rrb\times \llb 1,n \rrb\ |\ j>\gamma_{m+1-i}\tx{ and }j\neq \gamma_l\ \tx{ for all } l\in \llb 1,m \rrb \}.
\]
}\end{definition}
 
 A generating set for the quantum Schubert cell $S^o(\gamma)$ was described in \cite[Proposition 4.4]{llr-grass}: if, for $(i,j)\in \cl_\gamma$, one defines $m_{i,j}:=[\{\gamma_1,\ldots,\gamma_m\}\bs \{\gamma_{m+1-i}\}\sqcup \{j\}]$ (which clearly belongs to $\Pi\bs\Pi_\gamma$, so that $\ol{m_{i,j}}\in S(\gamma)$ is nonzero and homogeneous of degree $1$), then the quantum Schubert cell $S^o(\gamma)$ is generated by $\{\ol{m_{i,j}}\bar{\gamma}^{-1}\ |\ (i,j)\in \cl_\gamma\}$. Let us set $\widetilde{m_{i,j}}:=\ol{m_{i,j}}\bar{\gamma}^{-1}$ for all $(i,j)\in \cl_\gamma$.

Since $\ang{\Pi_\gamma}$ is clearly an $\ch$-invariant ideal of $\oqgmnk$, the action of $\ch$ on $\oqgmnk$ descends to $S(\gamma)$. Since $\ol{\gamma}$ is an $\ch$-eigenvector of $S(\gamma)$, the action of $\ch$ on $S(\gamma)$ extends to $S(\gamma)[\ol{\gamma}^{-1}]$. This action of $\ch$ on $S(\gamma)[\ol{\gamma}^{-1}]$ restricts to $S^o(\gamma)$; indeed for any $\widetilde{m_{i,j}}$ with $(i,j)\in \cl_\gamma$, and any $(\alpha_1,\ldots,\alpha_n)\in \ch$, an elementary calculation shows that
\begin{equation}\label{action of H on So}
(\alpha_1,\ldots,\alpha_n)\cdot \widetilde{m_{i,j}}=\alpha_{\gamma_{m+1-i}}^{-1}\alpha_j \widetilde{m_{i,j}}.
\end{equation}

Recall from the general theory of noncommutative dehomogenisation that when $\sigma$ is the restriction to $S^o(\gamma)$ of the automorphism of $S(\gamma)[\ol{\gamma}^{-1}]$ given by $s\mapsto\ol{\gamma}s\ol{\gamma}^{-1}$ for all $s\in S(\gamma)[\ol{\gamma}^{-1}]$, the inclusion $S^o(\gamma)\hookrightarrow S(\gamma)[\ol{\gamma}^{-1}]$ extends to an isomorphism
\begin{equation}\label{dehom}
S^o(\gamma)[y^{\pm 1};\sigma]\to (\oqgmnk/\ang{\Pi_\gamma})[\ol{\gamma}^{-1}]
\end{equation}
which sends $y$ to $\ol{\gamma}$. Notice here that by \cite[Lemma 3.1.4(v)]{lr2}, the automorphism $\sigma$ multiplies each $\widetilde{m_{i,j}}$ ($(i,j)\in \cl_\gamma$) by $q$. The action of $\ch$ on $(\oqgmnf/\ang{\Pi_\gamma})[\ol{\gamma}^{-1}]$ passes to $S^o(\gamma)[y^{\pm 1};\sigma]$ via the isomorphism \eqref{dehom} and this action of $\ch$ on $S^o(\gamma)[y^{\pm 1};\sigma]$ restricts to the action of $\ch$ on $S^o(\gamma)$ described in \eqref{action of H on So}. 
In particular, the isomorphism \eqref{dehom} is $\ch$-equivariant where $\ch$ acts on $S^o(\gamma)$ as in \eqref{action of H on So} and each $(\alpha_1,\ldots,\alpha_n)\in \ch$ acts on $y$ as follows 
\begin{equation}\label{H action on y}
(\alpha_1,\ldots,\alpha_n)\cdot y=\alpha_{\gamma_1}\cdots\alpha_{\gamma_m}y
\end{equation}
(cf. \eqref{H action on qgrass}).


\subsection{Quantum ladder matrix algebras}
It was shown in \cite{llr-grass} that the quantum Schubert cell $S^o(\gamma)$ can be identified with a well-behaved subalgebra of $\oqmmnk$, which can in turn be identified with a partition subalgebra of $\co_{q^{-1}}(M_{m,n-m}(\k))$. We describe these isomorphisms in detail in this section.


\begin{definition}{\rm 
The \emph{quantum ladder matrix algebra} associated to $\gamma$ is the subalgebra of $\oqmmnk$ generated by all those $x_{i,j}$ with $(i,j)\in \cl_\gamma$; it is denoted by $\ladderk$.
}\end{definition}


By \cite[Lemma 4.6]{llr-grass}, there is an isomorphism
\begin{align}\label{So to ladder}
\begin{split}
S^o(\gamma) &\xrightarrow{\cong} \ladderk \\
\widetilde{m_{i,j}}&\mapsto x_{i,j}.
\end{split}
\end{align}

One may obtain the generators of $\ladderk$ as follows. Consider the matrix 
\begin{equation}\label{the gens pre ladder}
\left(
\begin{array}{ccc}
x_{1,1}&\cdots&x_{1,n} \\
\vdots&\ddots&\vdots \\
x_{m,1}&\cdots&x_{m,n} \\
\end{array}
\right)
\end{equation}
of canonical generators of $\co_q(M_{m,n}(\k))$ and recall that $\gamma=[\gamma_1\cdots\gamma_m]$. For each $i\in \llb 1,m \rrb$, remove the $i^{\tx{th}}$-last entry of the $\gamma_i^{\tx{th}}$ column of \eqref{the gens pre ladder} (namely the entry $x_{m+1-i,\gamma_i}$) and replace it with a bullet. For each bullet, replace all matrix entries which are to its left and all matrix entries which are below it with stars. Then the quantum ladder matrix algebra $\co_q(M_{m,n,\gamma}(\k))$ is the subalgebra of $\co_q(M_{m,n}(\k))$ which is generated by the entries of the matrix \eqref{the gens pre ladder} which survive this process (that is, which are not replaced by a bullet or a star).


\begin{example}\label{ladder ex}
{\rm Let $\gamma$ be the maximal quantum minor $[1347]$ of $\co_q(G_{4,8}(\k))$ and consider the matrix 
\[
\left(
\begin{array}{cccccccc}
x_{1,1}&x_{1,2}&x_{1,3}&x_{1,4}&x_{1,5}&x_{1,6}&x_{1,7}&x_{1,8} \\
x_{2,1}&x_{2,2}&x_{2,3}&x_{2,4}&x_{2,5}&x_{2,6}&x_{2,7}&x_{2,8} \\
x_{3,1}&x_{3,2}&x_{3,3}&x_{3,4}&x_{3,5}&x_{3,6}&x_{3,7}&x_{3,8} \\
x_{4,1}&x_{4,2}&x_{4,3}&x_{4,4}&x_{4,5}&x_{4,6}&x_{4,7}&x_{4,8} \\
\end{array}
\right)
\]
of canonical generators of $\co_q(M_{4,8}(\k))$.
Applying the prescribed procedure, we are left with 
\begin{equation}\label{48 ladder generators}
\left(
\begin{array}{cccccccc}
\ast&\ast&\ast&\ast&\ast&\ast&\bullet&x_{1,8} \\
\ast&\ast&\ast&\bullet&x_{2,5}&x_{2,6}&\ast&x_{2,8} \\
\ast&\ast&\bullet&\ast&x_{3,5}&x_{3,6}&\ast&x_{3,8} \\
\bullet&x_{4,2}&\ast&\ast&x_{4,5}&x_{4,6}&\ast&x_{4,8} \\
\end{array}
\right)
\end{equation}
The quantum ladder matrix algebra $\co_q(M_{4,8,\gamma}(\k))$ is the subalgebra of $\co_q(M_{4,8}(\k))$ generated by those $x_{i,j}$ appearing in \eqref{48 ladder generators}. After rotating \eqref{48 ladder generators} through $180^\circ$ and deleting the columns containing bullets, notice that the generators of $\co_q(M_{4,8,\gamma}(\k))$ lie in the Young diagram below
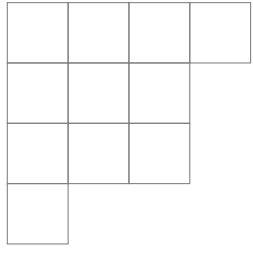
\begin{equation}\label{1347}
\begin{tikzpicture}[xscale=0.8, yscale=0.8]
\draw[color=gray] (0,0) rectangle (1,1);
\draw[color=gray] (0,1) rectangle (1,2);
\draw[color=gray] (1,1) rectangle (2,2);
\draw[color=gray] (2,1) rectangle (3,2);
\draw[color=gray] (0,2) rectangle (1,3);
\draw[color=gray] (1,2) rectangle (2,3);
\draw[color=gray] (2,2) rectangle (3,3);
\draw[color=gray] (0,3) rectangle (1,4);
\draw[color=gray] (1,3) rectangle (2,4);
\draw[color=gray] (2,3) rectangle (3,4);
\draw[color=gray] (3,3) rectangle (4,4);
\end{tikzpicture}
\end{equation}
In fact it turns out that the quantum ladder matrix algebra $\co_q(M_{4,8,\gamma}(\k))$ is isomorphic to the partition subalgebra of $\co_{q^{-1}}(M_{4,4}(\k))$ corresponding to the partition whose Young diagram is \eqref{1347}. 
}\end{example}



\begin{notation}\label{associated partition}
{\rm Notice that  $\gamma_i-i=|\{a\in \llb 1,n \rrb\bs\gamma\ |\ a<\gamma_i\}|$ 
for each $i\in\llb 1,m \rrb$. It follows easily that if we define $\lambda_i=n-m-(\gamma_i-i)$ for each $i\in \llb 1,m \rrb$, then $(\lambda_1,\ldots,\lambda_m)$ is a partition with $n-m\geq \lambda_1\geq \lambda_2\geq\cdots\geq \lambda_m\geq 0$.
Let $c$ be as large as possible such that $\lambda_c\neq 0$ and denote by $\lambda$ the partition $(\lambda_1,\ldots,\lambda_c)$. Recall that $\partitionk$ denotes the partition subalgebra of $\co_{q^{-1}}(M_{m,n-m}(\k))$ associated to the partition $\lambda$. 
}\end{notation}

Note that the south and east borders of $Y_{\lambda}$ give rise to path of length $n$, from the north-east corner to the south-west corner of the $m\times (n-m)$ rectangle. Label each of edges of this path with the numbers $1$ through $n$ (starting from the north-east corner). Then the elements of $\gamma$ coincide with the vertical steps in this numbering. The following example illustrates this construction when $\gamma =[1347]$.
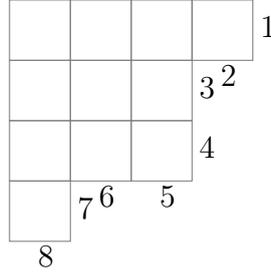
\begin{equation}\label{1347}
\begin{tikzpicture}[xscale=0.8, yscale=0.8]
\draw[color=gray] (0,0) rectangle (1,1);
\draw[color=gray] (0,1) rectangle (1,2);
\draw[color=gray] (1,1) rectangle (2,2);
\draw[color=gray] (2,1) rectangle (3,2);
\draw[color=gray] (0,2) rectangle (1,3);
\draw[color=gray] (1,2) rectangle (2,3);
\draw[color=gray] (2,2) rectangle (3,3);
\draw[color=gray] (0,3) rectangle (1,4);
\draw[color=gray] (1,3) rectangle (2,4);
\draw[color=gray] (2,3) rectangle (3,4);
\draw[color=gray] (3,3) rectangle (4,4);
\node at (4.25,3.55) {$1$};
\node at (3.6,2.75) {$2$};
\node at (3.25,2.55) {$3$};
\node at (3.25,1.55) {$4$};
\node at (2.6,0.75) {$5$};
\node at (1.6,0.75) {$6$};
\node at (1.25,0.55) {$7$};
\node at (0.6,-0.25) {$8$};
\end{tikzpicture}
\end{equation}


Let $\{a_1<\cdots<a_{n-m}\}=\llb 1,n \rrb\bs\gamma$ and notice that all elements of $\cl_\gamma$ take the form $(i,a_j)$ for some $i\in \llb 1,m \rrb$ and some $j\in \llb 1,n-m\rrb$.
The following result appears in the proof of \cite[Theorem 4.7]{llr-grass}. We write down the maps explicitly here as we shall need them later.


\begin{lemma}\label{ladder to partition}
There is an isomorphism 
\[
f: \co_q(M_{m,n,\gamma}(\k))\xrightarrow{\cong} \partitionk
\] 
such that
\begin{itemize}
\item $f(x_{i,a_j})=x_{m+1-i,n-m+1-j}$ for each $(i,a_j)\in \cl_\gamma$;
\item $f^{-1}(x_{i,j})=x_{m+1-i,a_{n-m+1-j}}$ for each $(i,j)\in Y_\lambda$.
\end{itemize}
\end{lemma} 


\begin{proof} By the proof of \cite[Corollary 5.9]{gl-winding}, there is an isomorphism $\co_q(M_n(\k))\xrightarrow{\cong} \co_{q^{-1}}(M_n(\k))$ which sends each $x_{i,j}$ to $x_{n+1-i,n+1-j}$; this isomorphism can be thought of as rotating the matrix of canonical generators for $\co_q(M_n(\k))$ through $180^\circ$. 

There is an isomorphism 
$\delta: \co_q(M_{m,n}(\k))\xrightarrow{\cong}\co_{q^{-1}}(M_{m,n}(\k))$ 
with $\delta(x_{i,j})=x_{m+1-i,n+1-j}$
for each $(i,j)\in \llb 1,m \rrb\times\llb 1,n \rrb$, 
(this isomorphism can be thought of as rotating the matrix of canonical generators for $\co_q(M_{m,n}(\k))$ through $180^\circ$). This isomorphism is constructed by identifying
$\co_q(M_{m,n}(\k))$ with the subalgebra of $\co_q(M_n(\k))$ generated by the last $m$ rows of the matrix of canonical generators for $\co_q(M_n(\k))$, identifying $\co_{q^{-1}}(M_{m,n}(\k))$ with the subalgebra of $\co_{q^{-1}}(M_n(\k))$ generated by the first $m$ rows of the matrix of canonical generators for $\co_{q^{-1}}(M_n(\k))$, and applying the isomorphism described in the previous paragraph. 

There is an isomorphism $\delta(\ladderk)\xrightarrow{\cong} \partitionk$ which sends each $\delta(x_{i,a_j})=x_{m+1-i,n+1-a_j}$ ($(i,a_j)\in \cl_\gamma$) to $x_{m+1-i,n-m+1-j}$. Composing this isomorphism with $\delta$ (or rather the restriction of $\delta$ to $\ladder$) gives the desired isomorphism $f$.
\end{proof}


The isomorphism $f$ is simpler than the notation of Lemma \ref{ladder to partition} might make it seem. The following example should illuminate the idea.


\begin{example}
In the situation of Example \ref{ladder ex}, where $\gamma$ is the quantum Pl\"ucker coordinate $[1347]$ of $\co_q(G_{4,8}(\k))$, the generators of the quantum ladder matrix algebra $\co_q(M_{4,8,\gamma}(\k))$ are those appearing below
\[
\left(
\begin{array}{cccccccc}
\ast&\ast&\ast&\ast&\ast&\ast&\bullet&x_{1,8} \\
\ast&\ast&\ast&\bullet&x_{2,5}&x_{2,6}&\ast&x_{2,8} \\
\ast&\ast&\bullet&\ast&x_{3,5}&x_{3,6}&\ast&x_{3,8} \\
\bullet&x_{4,2}&\ast&\ast&x_{4,5}&x_{4,6}&\ast&x_{4,8} \\
\end{array}
\right).
\]
The action of the isomorphism $\delta (\co_q(M_{4,8}(\k)))\xrightarrow{\cong}\co_{q^{-1}}(M_{4,8}(\k))$ may be understood as rotating this picture through $180^\circ$:
\begin{equation}\label{sigh}
\left(
\begin{array}{cccccccc}
x_{1,1}&\ast& x_{1,3}&x_{1,4}&\ast&\ast&x_{1,7}&\bullet \\
x_{2,1}&\ast&x_{2,3}&x_{2,4}&\ast&\bullet&\ast&\ast \\
x_{3,1}&\ast&x_{3,3}&x_{3,4}&\bullet&\ast&\ast&\ast \\
x_{4,1}&\bullet&\ast&\ast&\ast&\ast&\ast&\ast
\end{array}
\right)
\end{equation}
Let $\lambda$ be the partition associated to $\gamma$ as in Notation \ref{associated partition}, whose Young diagram is 
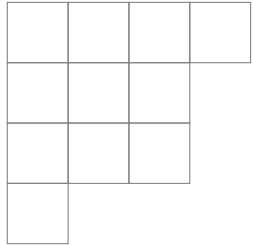
\begin{center}
\begin{tikzpicture}[xscale=0.8, yscale=0.8]
\draw[color=gray] (0,0) rectangle (1,1);
\draw[color=gray] (0,1) rectangle (1,2);
\draw[color=gray] (1,1) rectangle (2,2);
\draw[color=gray] (2,1) rectangle (3,2);
\draw[color=gray] (0,2) rectangle (1,3);
\draw[color=gray] (1,2) rectangle (2,3);
\draw[color=gray] (2,2) rectangle (3,3);
\draw[color=gray] (0,3) rectangle (1,4);
\draw[color=gray] (1,3) rectangle (2,4);
\draw[color=gray] (2,3) rectangle (3,4);
\draw[color=gray] (3,3) rectangle (4,4);
\end{tikzpicture}
\end{center}
The subalgebra of $\co_{q^{-1}}(M_{4,8}(\k))$ generated by the $x_{i,j}$ appearing in \eqref{sigh} is clearly isomorphic to the partition subalgebra $\partitionk$ of $\co_{q^{-1}}(M_{4,4}(\k))$.
\end{example}


The following is a more explicit statement of \cite[Theorem 4.7]{llr-grass}.


\begin{theorem}\label{So to partition}
There is an isomorphism $\theta: S^o(\gamma)\xrightarrow{\cong} 
\partitionk$ such that
\begin{itemize}
\item $\theta(\wt{m_{i,a_j}})=x_{m+1-i,n-m+1-j}$ for each $(i,a_j)\in \cl_\gamma$;
\item $\theta^{-1}(x_{i,j})=\wt{m_{m+1-i,a_{n-m+1-j}}}$ for each $(i,j)\in Y_\lambda$.
\end{itemize}
\end{theorem} 


\begin{proof}
  \begin{full} 
When $g$ is the isomorphism $S^o(\gamma)\xrightarrow{\cong} \ladderk$ given in \eqref{So to ladder} and $f$ is the isomorphism $\ladderk\xrightarrow{\cong} \partitionk$ given in Lemma \ref{ladder to partition}, the desired isomorphism $\theta$ is given by $f\circ g$.
\end{full} 
\end{proof}


We may pass the action of $\ch$ on $S^o(\gamma)$ through $\theta$ to get an action of $\ch$ on $\partition$ described by 
\begin{equation}\label{induced action of H on partition}
(\alpha_1,\ldots,\alpha_n)\cdot x_{i,j}=\alpha_{\gamma_i}^{-1}\alpha_{a_{n-m+1-j}}x_{i,j}
\end{equation}
for all $(\alpha_1,\ldots,\alpha_n)\in \ch$ and all $(i,j)\in Y_\lambda$. With this action of $\ch$ on $\partition$, the isomorphism $\theta$ is $\ch$-equivariant. 

\vspace{2mm}


\noindent \danger\textbf{WARNING}\danger  \vspace{2mm}  Because it allows the isomorphism $\theta$ to be $\ch$-equivariant, the $\ch$-action on $\partitionk$ which we shall use is that given in \eqref{induced action of H on partition}; this is NOT the usual action of $\ch$ on $\partitionk$ (which is the restriction of the action of $\ch$ on $\co_{q^{-1}}(M_{m,n-m}(\k))$). 


\vspace{2mm}

In spite of the warning above, the following lemma shows that in fact we may use the term \emph{$\ch$-prime ideal of $\partition$} without ambiguity (cf. commentary in \cite{llr-grass} before Theorem 4.8).


\begin{lemma}\label{H pres}
The same subsets (and in particular the same prime ideals) of $\partition$ are invariant under $\ch$ whether one uses the action of $\ch$ described in \eqref{induced action of H on partition} or the restriction of the action of $\ch$ on on $\co_{q^{-1}}(M_{m,n-m}(\k))$. 
\end{lemma} 


\begin{proof}
  \begin{full}
 Let us use ``$\cdot$'' to denote the action of $\ch$ on $\partition$ transferred through $\theta$ from the action on $S^o(\gamma)$ (described in \eqref{induced action of H on partition}), let us use ``$\#$'' to denote the standard action of $\ch$ on $\partition$, and let us fix any $\alpha=(\alpha_1,\ldots,\alpha_n)\in \ch$.

Define $\alpha'=(\alpha'_1,\ldots,\alpha'_n),\ \alpha''=(\alpha''_1,\ldots,\alpha''_n)\in \ch$ by $\alpha'_{i}=\alpha_{\gamma_i}^{-1}$ for all $i\in \llb 1,m \rrb$, $\alpha'_{m+j}=\alpha_{a_{n-m+1-j}}$ for all $j\in \llb 1, n-m\rrb$, $\alpha''_{\gamma_i}=\alpha_i^{-1}$ for all $i\in \llb 1,m \rrb$ and $\alpha''_{a_{n-m+1-j}}=\alpha_{m+j}$ for all $j\in \llb 1, n-m\rrb$. 

One checks easily that if $(i,j)\in Y_\lambda$, then $\alpha\cdot x_{i,j}=\alpha' \# x_{i,j}$ and $\alpha \# x_{i,j} = \alpha''\cdot x_{i,j}$. Since these $x_{i,j}$ generate $\partitionk$, we have $\alpha\cdot x=\alpha' \# x$ and $\alpha \# x = \alpha''\cdot x$ for all $x\in \partitionk$.
The result follows.
  \end{full}
\end{proof}



\subsection{The correspondence between $\ch$-primes and Cauchon-Le diagrams}\label{sec-corresp-h-primes-cauchon-diags}
Recall that we have set $\{a_1<\cdots<a_{n-m}\}=\llb 1,n \rrb\bs\gamma$ and that all elements of $\cl_\gamma$ take the form $(i,a_j)$ for some $i\in \llb 1,m \rrb$ and some $j\in \llb 1, n-m\rrb$.

When $\sigma$ is the automorphism of $\partitionk$ which multiplies each $x_{i,j}$  by $q$, the $\ch$-equivariant isomorphism $\theta: S^o(\gamma)\xrightarrow{\cong} \partitionk$ (from Theorem \ref{So to partition}) and the $\ch$-equivariant dehomogenisation isomorphism 
\[
S^o(\gamma)[y^{\pm 1};\sigma]\xrightarrow{\cong} (\oqgmnk/\ang{\Pi_\gamma})[\ol{\gamma}^{-1}]
\]
given in \eqref{dehom} induce an $\ch$-equivariant isomorphism
\begin{align}\label{the godfather iso}
\begin{split}
\Phi: \partitionk [y^{\pm 1};\sigma]&\xrightarrow{\cong} (\oqgmnk/\ang{\Pi_\gamma})[\ol{\gamma}^{-1}] \\
x_{i,j}&\mapsto \wt{m_{m+1-i,a_{n-m+1-j}}} \hspace{7mm} ((i,j)\in Y_\lambda) \\
y&\mapsto \ol{\gamma}. \\
\end{split}
\end{align}
whose inverse we shall denote by $\Psi$.


\begin{remark}\label{ends up in part}
Recall that the dehomogenisation isomorphism 
\[
S^o(\gamma)[y^{\pm 1};\sigma]\xrightarrow{\cong} 
(\oqgmnk/\ang{\Pi_\gamma})[\ol{\gamma}^{-1}]
\]
 extends the inclusion $S^o(\gamma)\hookrightarrow (\oqgmnk/\ang{\Pi_\gamma})[\ol{\gamma}^{-1}]$, so that for any $x\in S^o(\gamma)$, we have $\Psi(x)=\theta(x)\in \partitionk$. 
\end{remark}


By \cite[Theorem 5.4]{llr-grass},
there is a bi-increasing bijection 
\begin{equation}\label{godmother corresp}
\xi: \hspec_\gamma( \oqgmnk)\xrightarrow{\cong}\hspec(\partitionk)
\end{equation}
such that 
\[
\xi(P)= \Psi(\ol{P}[\ol{\gamma}^{-1}])\cap \partitionk
\]
for any $P\in \hspec_\gamma \oqgmnk$ (with the convention that $\ol{P}:=P/\ang{\Pi_\gamma}$)
and 
$\xi^{-1}(Q)$ is the preimage in $\oqgmnk$ of 
\[
\Phi\left(\bigoplus_{i\in \Z}Qy^i\right)\cap \left(\oqgmnk/\ang{\Pi_\gamma}\right).
\]
for any $Q\in \hspec(\partitionk)$.
Recall the one-to-one correspondence \eqref{H-primes of partlambda} (first established in \cite[Theorem 3.5]{llr-grass}) between the $\ch$-prime ideals of $\partitionk$ and the Cauchon-Le diagrams on the Young diagram $Y_\lambda$. Composing this correspondence with $\xi$ gives the one-to-one correspondence
\begin{equation}\label{major corresp}
\ts{\ch-\Spec_\gamma \oqgmnk} \longleftrightarrow \tx{Cauchon-Le diagrams on }Y_\lambda
\end{equation}
which was established in \cite[Corollary 5.5]{llr-grass}: 
any $P\in \ch-\Spec_\gamma \oqgmnk$ corresponds to the Cauchon-Le diagram of 
the $\ch$-prime ideal $\xi(P)$ of $\partitionk$ and any Cauchon-Le diagram $C$ 
on the Young diagram $Y_\lambda$ corresponds to the image under $\xi^{-1}$ of 
the $\ch$-prime ideal of $\partitionk$ which has Cauchon-Le diagram $C$. 


\subsection{Quantum Pl\"ucker coordinates in $\ch$-primes}
\label{section-9-6}

Fix  $P\in \ch-\Spec_\gamma\oqgmnk$ and  denote by $C$ the Cauchon-Le diagram on $Y_\lambda$ which corresponds to $P$ under \eqref{major corresp}. 
We seek to identify those quantum Pl\"ucker coordinates that belong to $P$ by 
considering the Postnikov graph of $C$. Recall that by the definition of 
$\hspec_\gamma(\oqgmnk)$ (see the beginning of Section \ref{section-known}), if $\alpha\ngeq\gamma$ 
then $\alpha\in P$; so we need only consider $\alpha\geq\gamma$. 
Fix such an 
 $\alpha\in \Pi$. Notice that there exist 
 $1\leq i_1<\cdots < i_t\leq m$ and $1\leq j_1<\cdots <j_t\leq n-m$ such that $a_{j_l}>\gamma_{i_l}$ for all $l=1,\ldots,t$ and $\alpha=[(\gamma\bs \{\gamma_{i_1},\ldots,\gamma_{i_t}\})\sqcup\{a_{j_1},\ldots,a_{j_t}\}]$.

Notice that when $h_0=(\alpha_1,\ldots,\alpha_n)\in \ch$ is such that $\alpha_i=q^2$ if $i\notin \{\gamma_1,\ldots,\gamma_m\}$ and $\alpha_i=q$ otherwise, the isomorphism $\sigma$ of $\partitionk$ (which multiplies each $x_{i,j}$ ($(i,j)\in Y_\lambda$) by $q$) coincides with the action of $h_0$. Moreover $h_0\cdot y=q^my$ by \eqref{H action on y}. Hence the algebra $\partitionk[y^{\pm 1};\sigma]$, along with its $\ch$-action, satisfies \cite[Hypothesis 2.1]{llr-ufd}. We shall use this fact in the proof of the following 
proposition.


\begin{proposition}\label{precondition}
The condition that $\alpha$ belongs to $P$ is equivalent to the condition that $\Psi(\bar{\alpha}\bar{\gamma}^{-1})$ belongs to $\xi(P)$.
\end{proposition} 


\begin{proof}
  \begin{full} 
By \cite[Lemma 2.2]{llr-ufd}, we have $\bigoplus_{i\in \Z}\xi(P)y^i=\Psi((P/\ang{\Pi_\gamma})[\ol{\gamma}^{-1}])$, so that the isomorphism $\Psi$ induces an isomorphism
\[
\frac{(\oqgmnk/\ang{\Pi_\gamma})[\ol{\gamma}^{-1}]}{(P/\ang{\Pi_\gamma})[\ol{\gamma}^{-1}]}\xrightarrow{\cong}\frac{\partitionk[y^{\pm 1};\sigma]}{\bigoplus_{i\in \Z}\xi(P)y^i},
\]
which in turn induces an isomorphism 
\[
\ol{\Psi}: \frac{\oqgmnk}{P}[\ol{\gamma}^{-1}]\xrightarrow{\cong} \frac{\partitionk[y^{\pm 1};\sigma]}{\bigoplus_{i\in \Z}\xi(P)y^i}.
\]
Now, $\alpha\in P$ if and only if $\bar{\alpha}\bar{\gamma}^{-1}=0$ in $(\oqgmnk/P)[\bar{\gamma}^{-1}]$, and this  is true if and only if $\ol{\Psi}(\bar{\alpha}\bar{\gamma}^{-1})=0$. Since 
\[
\ol{\Psi}(\bar{\alpha}\bar{\gamma}^{-1})=\ol{\Psi(\bar{\alpha}\bar{\gamma}^{-1})}\in \partitionk[y^{\pm 1};\sigma]/\bigoplus_{i\in \Z}\xi(P)y^i,
\]
we find that 
$\alpha\in P$ if and only if $\Psi(\bar{\alpha}\bar{\gamma}^{-1})\in 
\bigoplus_{i\in \Z}\xi(P)y^i$. 

However the element 
$\bar{\alpha}\bar{\gamma}^{-1}$ of $(\oqgmnk/\ang{\Pi_\gamma})[\ol{\gamma}^{-1}]$ 
in fact belongs to $S^o(\gamma)$, so that 
$\Psi(\bar{\alpha}\bar{\gamma}^{-1})\in \partitionk$ by 
Remark \ref{ends up in part}; hence $\alpha\in P$ if and only if 
\[
\Psi(\bar{\alpha}\bar{\gamma}^{-1})\in \left(\bigoplus_{i\in \Z}\xi(P)y^i\right)\cap \partitionk=\xi(P).
\] 
  \end{full} 
\end{proof}


For the proof of the next theorem, we shall need two sets of relations, known to hold in quantum grassmannians, called the \emph{generalised quantum Pl\"ucker relations} and the \emph{quantum Muir Law of extensible minors}.\\

\begin{theorem} 
\label{theorem-quantum-plucker-relations}
~
\begin{enumerate}
\item \textbf{Generalised quantum Pl\"ucker relations} \cite[Theorem 2.1]{klr}: \\
Let $J_1,J_2,K\subseteq \llb 1,n \rrb$ be such that $|J_1|,|J_2|\leq m$ and $|K|=2m-|J_1|-|J_2|>m$. Then  
\begin{equation}\label{gen q plucker rels}
\sum_{K'\sqcup K''=K}(-q)^{\ell(J_1;K')+\ell(K';K'')+\ell(K'';J_2)}[J_1\sqcup K'][K''\sqcup J_2]=0,
\end{equation}
where for any two sets $I,J$ of integers, $\ell(I;J)$ denotes the cardinality of the set $\{(i,j)\in I\times J\ |\ i>j\}$.
\item \textbf{Quantum Muir Law} (adapted from \cite[Proposition 1.3]{lr2}): \\
Let $r$ be a positive integer. For $s\in \llb 1,r\rrb$, let $I_s,J_s$ be $m$-element subsets of $\llb 1,n \rrb$  and let $c_s\in \f$ be such that $\sum_{s=1}^r c_s[I_s][J_s]=0$ in $\oqgmnf$. Suppose that $D$ is a subset of $\llb 1,n \rrb$ such that $\left(\bigcup_{s=1}^r I_s\right)\cup \left(\bigcup_{s=1}^r J_s\right)$ does not intersect $D$. Then 
in $\co_q(G_{m+|D|,n}(\f))$, we have
\begin{equation}\label{Muir}
\sum_{s=1}^r c_s[I_s\sqcup D][J_s\sqcup D]=0.
\end{equation}
\end{enumerate}
\end{theorem}


Before reading the following proof, the reader might want to revisit the construction, given in Notation \ref{associated partition}, of the partition $\lambda$ from the quantum Pl\"ucker coordinate $\gamma$.

\begin{theorem}\label{where it goes} Let $\alpha \in \Pi$ with $\alpha \geq \gamma$, and write $\alpha=[(\gamma\bs \{\gamma_{i_1},\ldots,\gamma_{i_t}\})\sqcup\{a_{j_1},\ldots,a_{j_t}\}]$ with $1\leq i_1<\cdots < i_t\leq m$, $1\leq j_1<\cdots <j_t\leq n-m$ and $a_{j_l}>\gamma_{i_l}$ for all $l=1,\ldots,t$.

The isomorphism 
\[
\Psi: (\oqgmnk/\ang{\Pi_\gamma})[\ol{\gamma}^{-1}]\xrightarrow{\cong} \partitionk[y^{\pm 1};\sigma]
\]
sends $\bar{\alpha}\bar{\gamma}^{-1}$ to 
\[
(-q)^{\ell(\{\gamma_{i_1},\ldots,\gamma_{i_t}\};\ \{a_{j_1},\ldots,a_{j_t}\})}[i_1\cdots i_t\ |\ n-m+1-j_t\cdots n-m+1-j_1].
\]
\end{theorem} 


\begin{proof}
  \begin{full} 
Suppose that $t=1$. Then $\bar{\alpha}\bar{\gamma}^{-1}=[(\gamma\bs\{\gamma_{i_1}\})\sqcup\{a_{j_1}\}]=\ol{m}_{m+1-i_1,a_{j_1}}\bar{\gamma}^{-1}=\widetilde{m_{m+1-i_1,a_{j_1}}}$, which is sent by $\Psi$ to $x_{i_1,n-m+1-j_1}=[i_1\ |\ n-m+1-j_1]$. Since $\ell(\{\gamma_{i_1}\}; \{a_{j_1}\})=0$, the claim holds. 

We proceed by induction on $t$. (In order to keep the notation managable here, we denote a singleton set by its element; that is, we write a singleton set $\{z\}$ simply as $z$.)

Let us set $\widetilde{a}=\{a_{j_2},\ldots,a_{j_t}\}$ and $\widetilde{\gamma}=\{\gamma_{i_1},\ldots,\gamma_{i_t}\}$. Applying the generalised quantum Pl\"ucker relations \eqref{gen q plucker rels} with $J_1=\widetilde{a}$, $J_2=\emptyset$, $K=a_{j_1}\sqcup \widetilde{\gamma}$, and noticing that 
$\ell(\gamma_{i_l};(a_{j_1}\sqcup\widetilde{\gamma})\bs\gamma_{i_l})=\ell(\gamma_{i_l};a_{j_1})+l-1$ (for all $l=1,\ldots,t$) and $\ell(\wt{a};a_{j_1})=t-1$, we see that the following holds in $\co_q(G_{t,n}(\f))$:
\[
\sum_{l=1}^t(-q)^{\ell(\widetilde{a};\gamma_{i_l})+\ell(\gamma_{i_l};a_{j_1})+l-1}[\widetilde{a}\sqcup\gamma_{i_l}][(a_{j_1}\sqcup\widetilde{\gamma})\bs \gamma_{i_l}]
+(-q)^{t-1+\ell(a_{j_1};\wt{\gamma})}[a_{j_1}\cdots a_{j_t}][\widetilde{\gamma}]=0.
\]

Notice that no element of $\gamma\bs\widetilde{\gamma}$ appears in any of the quantum Pl\"ucker coordinates in the above display, so that the quantum version of Muir's Law \eqref{Muir} with $D=\gamma\bs\wt{\gamma}$ shows that in $\co_q(G_{t+|D|,n}(\f))=\oqgmnf$, we have
\[
\sum_{l=1}^t(-q)^{\ell(\widetilde{a};\gamma_{i_l})+\ell(\gamma_{i_l};a_{j_1})+l-1}[(\widetilde{a}\sqcup\gamma_{i_l})\sqcup (\gamma\bs\wt{\gamma})]\,[\overbrace{((a_{j_1}\sqcup\widetilde{\gamma})\bs \gamma_{i_l})\sqcup (\gamma\bs\wt{\gamma})}^{(\gamma\bs\gamma_{i_l})\sqcup a_{j_1}}]
\]
\[
+(-q)^{t-1+\ell(a_{j_1};\wt{\gamma})}\overbrace{[(\gamma\bs\wt{\gamma})\sqcup\{a_{j_1},\ldots, a_{j_t}\}]}^{\alpha}\,\overbrace{[\widetilde{\gamma}\sqcup (\gamma\bs\wt{\gamma})]}^{\gamma}=0;
\]
that is, 
\[
\sum_{l=1}^t(-q)^{\ell(\widetilde{a};\gamma_{i_l})+\ell(\gamma_{i_l};a_{j_1})+l-1}[(\widetilde{a}\sqcup\gamma_{i_l})\sqcup (\gamma\bs\wt{\gamma})]\,[(\gamma\bs\gamma_{i_l})\sqcup a_{j_1}]+(-q)^{t-1+\ell(a_{j_1};\wt{\gamma})}\alpha\gamma=0.
\]

Let $s$ be maximal such that $a_{j_1}>\gamma_{i_s}$, so that in $S(\gamma)=\oqgmnf/\ang{\Pi_\gamma}$, we have $\ol{[(\gamma\bs\gamma_{i_l})\sqcup a_{j_1}]}=0$ for $l>s$. Notice that if $l\leq s$, then $\ell(\wt{a};\gamma_{i_l})=t-1$ and
 $\ell(\gamma_{i_l};a_{j_1})=0$, while $\ell(a_{j_1};\wt{\gamma})=s$; so that we may conclude from the above display that the following holds in $S(\gamma)$:
\[
\bar{\alpha}\bar{\gamma}=-\sum_{l=1}^s(-q)^{l-1-s}\ol{[(\widetilde{a}\sqcup\gamma_{i_l})\sqcup (\gamma\bs\wt{\gamma})]}\,\ol{m}_{m+1-i_l,a_{j_1}}. 
\]
Now \cite[Lemma 3.1.4 (v)]{lr2} gives $\gamma m_{m+1-i_l,a_{j_1}}=q m_{m+1-i_l,a_{j_1}}\gamma$ for all $l=1,\ldots,s$, so that in $S^o(\gamma)$, we have
\[
\bar{\alpha}\bar{\gamma}^{-1}=
-q\sum_{l=1}^s(-q)^{l-1-s}\,\ol{[(\widetilde{a}\sqcup\gamma_{i_l})\sqcup (\gamma\bs\wt{\gamma})]}\,\ol{\gamma}^{-1}\,\ol{m}_{m+1-i_l,a_{j_1}}\ol{\gamma}^{-1}. 
\] 

Now if we write $[\widehat{i_l}\ |\ \widehat{n-m+1-j_1}]$ for $[i_1\cdots\widehat{i_l}\cdots i_t\ |\ n-m+1-j_t\cdots n-m+1-j_2]$ the induction hypothesis gives 
\[
\Psi(\bar{\alpha}\bar{\gamma}^{-1})=
\sum_{l=1}^s(-q)^{l-s}(-q)^{\ell(\{\gamma_{i_1},\ldots,\widehat{\gamma_{i_l}},\ldots,\gamma_{i_t}\},\wt{a})}\,[\widehat{i_l}\ |\ \widehat{n-m+1-j_1}]x_{i_l, n-m+1-j_1}.
\]
For $l\leq s$, we have 
\begin{align*}
{}&\ell(\{\gamma_{i_1},\ldots,\gamma_{i_t}\};\ \{a_{j_1},\ldots,a_{j_t}\})\\
&=\ell(\{\gamma_{i_1},\ldots,\widehat{\gamma_{i_l}},\ldots,\gamma_{i_t}\},\wt{a})+\ell(\gamma_{i_l},\{a_{j_1},\ldots,a_{j_t}\})+\ell(\{\gamma_{i_1},\ldots,\gamma_{i_t}\};a_{j_1}) \\
&=\ell(\{\gamma_{i_1},\ldots,\widehat{\gamma_{i_l}},\ldots,\gamma_{i_t}\},\wt{a})+0+t-s.
\end{align*}
Now
\[
\Psi(\bar{\alpha}\bar{\gamma}^{-1})=\sum_{l=1}^s(-q)^{\ell(\{\gamma_{i_1},\ldots,\gamma_{i_t}\};\ \{a_{j_1},\ldots,a_{j_t}\})+l-t}
\,[\widehat{i_l}\ |\ \widehat{n-m+1-j_1}]x_{i_l, n-m+1-j_1}.
\]
If $l>s$, then $a_{j_1}<\gamma_{i_l}$ and hence $|\{j\ |\ a_j<\gamma_{i_l}\}|\geq j_1$. Now since  $|\{j\ |\ a_j<\gamma_{i_l}\}|=\gamma_{i_l}-i_l$ (see Notation \ref{associated partition}), we have 
\begin{equation}\label{needed for outside Young diag}
\gamma_{i_l}-i_l\geq j_1\ \ \ \ \ \tx{for all }l>s. 
\end{equation}
If $l>s$, then \eqref{needed for outside Young diag} shows that $n-m+1-j_1>n-m+i_l-\gamma_{i_l}$ and hence $(i_l, n-m+1-j_1)\notin Y_\lambda$ since row $i_l$  
of the Young diagram $Y_\lambda$ has only $\lambda_{i_l}=n-m+i_l-\gamma_{i_l}$ squares (again see Notation \ref{associated partition}), so that our convention (see Section \ref{partition subalgebra}) says that $x_{i_l, n-m+1-j_1}=0$ in $\partitionk$. 
Hence we get the following expression for $\Psi(\bar{\alpha}\bar{\gamma}^{-1})$:
\[
\Psi(\bar{\alpha}\bar{\gamma}^{-1})=(-q)^{\ell(\{\gamma_{i_1},\ldots,\gamma_{i_t}\};\ \{a_{j_1},\ldots,a_{j_t}\})}\sum_{l=1}^t(-q^{-1})^{t-l}\,[\widehat{i_l}\ |\ \widehat{n-m+1-j_1}]x_{i_l, n-m+1-j_1}.
\]
Quantum Laplace expansion in $\partitionk$ with the last column on the right (see Corollary \ref{column Laplace}(2))\footnote{Care is needed with the parameters $q$ and $q^{-1}$ here because in the proof of Theorem \ref{where it goes}, we are working with a partition subalgebra of $\co_{q^{-1}}(M_{m,n-m}(\k))$, whereas Corollary \ref{column Laplace} is stated for partition subalgebras of $\oqmmnk$.} shows that, as required, we have
\[
\Psi(\bar{\alpha}\bar{\gamma}^{-1})=(-q)^{\ell(\{\gamma_{i_1},\ldots,\gamma_{i_t}\};\ \{a_{j_1},\ldots,a_{j_t}\})}[i_1\cdots i_t\ |\ n-m+1-j_t\cdots n-m+1-j_1].
\]
  \end{full} 
\end{proof}

For later use, we also compute the image of a pseudo quantum minors of $ \partitionk$ by $\Psi^{-1}$.

\begin{corollary}\label{cor-pluckers-match-minors}
Let $[i_1\cdots i_t\mid i_1'\cdots i_t']$ be a nonzero pseudo quantum minor in $ \partitionk$. Set  $j_1:=n-m+1-i'_t<\cdots < j_t:= n-m+1-i'_1$ and $\alpha := [\gamma \bs\{\gamma_{i_1},\ldots,\gamma_{i_t}\}\sqcup \{a_{j_1},\ldots,a_{j_t}\}]$.
 
Then $\alpha>\gamma$ and 
\[
\Psi^{-1}([i_1\cdots i_t\mid i_1'\cdots i_t'])= (-q)^{-\ell(\{\gamma_{i_1},\ldots,\gamma_{i_t}\};\ \{a_{j_1},\ldots,a_{j_t}\})} \bar{\alpha}\bar{\gamma}^{-1}.
\]
\begin{proof} We claim that 
\begin{equation}\label{equation-wanted}
\gamma_{i_l}<a_{j_l} \ \tx{ for all }l=1,\ldots,t 
\end{equation}
or equivalently that there is a permutation $\sigma\in S_t$ such that 
\begin{equation}\label{equation-equiv-wanted}
\gamma_{i_{\sigma(l)}}<a_{j_l} \ \tx{ for all }l=1,\ldots,t.
\end{equation}

Define $C_0$ to be the Cauchon-Le diagram on $Y_\lambda$ where all boxes are white (this corresponds to the $\ch$-prime ideal $\ang{0}$ of $\partitionk$). Since $[i_1\cdots i_t\mid i_1'\cdots i_t']$ is nonzero, Theorem \ref{which pseudo minors} (with $P=\ang{0}$) implies that there is a vertex-disjoint path system in $\mathrm{Post}(C_0)$ from $\{r_{i_1},\ldots,r_{i_t}\}$ to $\{c_{i_1'},\ldots,c_{i_t'}\}$. In the labeling of source and target vertices introduced just after Notation \ref{associated partition} (see \eqref{1347}), this is a vertex-disjoint path system from $\{\gamma_{i_1},\ldots,\gamma_{i_t}\}$ to $\{a_{j_1},\ldots,a_{j_t}\}$; it follows immediately that \eqref{equation-equiv-wanted} holds, that is \eqref{equation-wanted} holds, as required.
\end{proof}
\end{corollary}


Recall the biincreasing bijection 
\[
\xi: \ts{\ch-\Spec_\gamma} \oqgmnk\xrightarrow{\cong}\ch-\Spec \partitionk
\] 
given in \eqref{godmother corresp}. As in immediate consequence of Proposition \ref{precondition} and Theorem \ref{where it goes}, we get

\begin{corollary}\label{Cor minor}
Let $\alpha \in \Pi$ with $\alpha > \gamma$, and write $\alpha=[(\gamma\bs \{\gamma_{i_1},\ldots,\gamma_{i_t}\})\sqcup\{a_{j_1},\ldots,a_{j_t}\}]$ with $1\leq i_1<\cdots < i_t\leq m$, $1\leq j_1<\cdots <j_t\leq n-m$ and $a_{j_l}>\gamma_{i_l}$ for all $l=1,\ldots,t$.

The condition that $\alpha$ belongs to $P$ is equivalent to the condition that the pseudo quantum minor 
\[
[i_1\cdots i_t\ |\ n-m+1-j_t\cdots n-m+1-j_1]=(-q)^{-\ell(\{\gamma_{i_1},\ldots,\gamma_{i_t}\};\ \{a_{j_1},\ldots,a_{j_t}\})}\Psi(\bar{\alpha}\bar{\gamma}^{-1})
\] 
of $\partitionk$ belongs to the $\ch$-prime ideal $\xi(P)$ of $\partitionk$. 
\end{corollary}

Recall from Notation \ref{associated partition} that we have set $\lambda_i=n-m-(\gamma_i-i)$ for each $i\in \llb 1,m \rrb$, chosen $c$ as large as possible such that $\lambda_c\neq 0$, and defined the partition $\lambda$ by $\lambda=(\lambda_1,\ldots,\lambda_c)$. When $d=\lambda_1$, if we can show that $\{i_1,\ldots,i_t\}\subseteq \llb 1,c \rrb$ and $\{n-m+1-j_t,\ldots, n-m+1-j_1\}\subseteq \llb 1,d\rrb$, then the question of whether or not the pseudo quantum minor $[i_1\cdots i_t\ |\ n-m+1-j_t\cdots n-m+1-j_1]$ of $\partition$ is in $\xi(P)$ can be settled by the graph-theoretic method of Theorem \ref{which pseudo minors}.


\begin{lemma}\label{it fits}
We have
\begin{enumerate}[(i)]
\item $\{i_1,\ldots,i_t\}\subseteq \llb 1,c \rrb$ and 
\item $\{n-m+1-j_t,\ldots, n-m+1-j_1\}\subseteq \llb 1,d\rrb$.
\end{enumerate}
\end{lemma} 


\begin{proof}
  \begin{full} 
\begin{enumerate}[(i)]
\item Clearly it will suffice to show that $\lambda_{i_t}>0$. Recall from Notation \ref{associated partition} that 
\begin{align*}
\lambda_{i_t}&=n-m-(\gamma_{i_t}-i_t) \\
             &=n-m-|\{a\in \llb 1,n \rrb\bs\gamma\ |\ a<\gamma_{i_t}\}|. \\
\end{align*} 
Now if $\lambda_{i_t}=0$, then $|\{a\in \llb 1,n \rrb\bs\gamma\ |\ a<\gamma_{i_t}\}|=n-m$ so that every $a\in \llb 1,n \rrb\bs\gamma=\{a_1<\cdots<a_{n-m}\}$ satisfies $a<\gamma_{i_t}$; this is impossible since $a_{j_t}>\gamma_{i_t}$.
\item Clearly it will suffice to show that $\lambda_1\geq n-m+1-j_1$. Recall from Notation \ref{associated partition} that $\lambda_1=n-m-(\gamma_1-1)$, so that it will suffice to show that $j_1\geq \gamma_1$. Since $\alpha>\gamma$, $\gamma$ cannot be the largest element $[n-m+1\cdots n]$ of $\Pi$ with respect to the partial order on $\Pi$, so that $\gamma_1\in \llb 1, n-m\rrb$. Notice that $a_j=j$ for all $j<\gamma_1$ and $a_{\gamma_1}>\gamma_1$, so that $\inf\{j\in \llb 1, n-m\rrb\ |\ a_j>\gamma_1\}=\gamma_1$. Since $a_{j_1}>\gamma_{i_1}\geq \gamma_1$, we must have $j_1\geq \gamma_1$, as required.  
\end{enumerate}
  \end{full} 
\end{proof}


This brings us to the main result of this section, which tells us that $\alpha$ belongs to $P$ if and only if there exists no vertex-disjoint $R_{(I,J)}$-path system in the Postnikov graph of $C$, where $I:=\{i_1,\ldots,i_t\}$ and $J:=\{n-m+1-j_t,\ldots, n-m+1-j_1\}$. For the sake of completeness, we include our conventions in the statement of the theorem.



\begin{theorem}\label{theorem-main-pluckers}
Assume that $q\in \k^*$ is not a root of unity. Let $P\neq \ang{\Pi}$ be an $\ch$-prime ideal of $\oqgmnk$ and let $\gamma=[\gamma_1<\cdots<\gamma_m]$ be the unique quantum Pl\"ucker coordinate such that 
$P\in \hspec_\gamma(\oqgmnk)$ (see \cite[Theorem 5.1]{llr-grass}). Let $\lambda$ be the partition corresponding to $\gamma$ as in Notation \ref{associated partition} and let $Y_\lambda$ be the Young diagram of $\lambda$. Let $C$ be the Cauchon-Le diagram on $Y_\lambda$ corresponding to $P$ as in \eqref{major corresp}. Set $\{a_1<\cdots<a_{n-m}\}=\llb 1,n \rrb\bs\gamma$. Let $\alpha\in \Pi$ be such that $\alpha>\gamma$ and let $1\leq i_1<\cdots < i_t\leq m$, $1\leq j_1<\cdots <j_t\leq n-m$ be such that $\alpha=[(\gamma\bs \{\gamma_{i_1},\ldots,\gamma_{i_t}\})\sqcup\{a_{j_1},\ldots,a_{j_t}\}]$ and $a_{j_l}>\gamma_{i_l}$ for all $l\in \llb 1,t\rrb$. Then the quantum Pl\"ucker coordinate $\alpha$ belongs to $P$ if and only if there does not exist a vertex-disjoint $R_{(I,J)}$-path system in the Postnikov graph $\postc$ of $C$, where   $I:=\{i_1,\ldots,i_t\}$ and $J:=\{n-m+1-j_t,\ldots, n-m+1-j_1\}$.
\end{theorem} 


\begin{proof}
  \begin{full} 
This follows immediately from Corollary \ref{Cor minor}, Lemma \ref{it fits}, and Theorem \ref{which pseudo minors}. 
  \end{full} 
\end{proof}

For example, let us consider the Cauchon-Le diagram below, and denote by $P$ the corresponding $\ch$-prime in $\oqgmnk$. Then, $P$ belongs to $\hspec_{[135]}(\oq(G_{36}(\k)))$

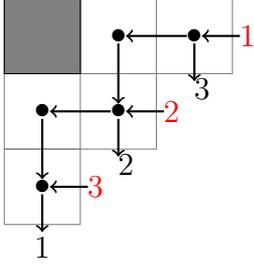
\begin{tikzpicture}[xscale=1, yscale=1]



\draw[color=gray] (0,2) rectangle (1,3);            
\draw[color=gray] (1,2) rectangle (2,3);            
\draw[color=gray] (2,2) rectangle (3,3);            

\draw[color=gray] (0,1) rectangle (2,2);               
\draw[color=gray] (1,1) rectangle (1,2);               

\draw[color=gray] (0,0) rectangle (1,1);            


\draw[fill=gray] (0,2) rectangle (1,3);               

\node[scale=1] at (3.2,2.5) {$\color{red}1$};
\node[scale=1] at (2.6,1.8) {$3$};
\node[scale=1] at (2.2,1.5) {$\color{red}2$};
\node[scale=1] at (1.6,0.8) {$2$};
\node[scale=1] at (1.2,0.5) {$\color{red}3$};
\node[scale=1] at (0.5,-0.3) {$1$};

\node at (1.5, 2.5) {$\bullet$}; 
\node at (2.5, 2.5) {$\bullet$}; 
\node at (0.5, 1.5) {$\bullet$}; 
\node at (1.5, 1.5) {$\bullet$}; 
\node at (0.5, 0.5) {$\bullet$}; 

\draw [<-, thick, black] (2.6,2.5)--(3.1,2.5); 
\draw [<-, thick, black] (1.6,2.5)--(2.4,2.5);

\draw [<-, thick, black] (1.6,1.5)--(2.1,1.5);
\draw [<-, thick, black] (0.6,1.5)--(1.4,1.5);

\draw [<-, thick, black] (0.6,0.5)--(1.1,0.5);

\draw [->, thick, black] (1.5,2.4)--(1.5,1.6);  
\draw [->, thick, black] (2.5,2.4)--(2.5,1.9);  
\draw [->, thick, black] (0.5,1.4)--(0.5,0.6);  
\draw [->, thick, black] (1.5,1.4)--(1.5,0.9);  
\draw [->, thick, black] (0.5,0.4)--(0.5,-0.1);  

\end{tikzpicture}

Moreover, there is a vertex disjoint set of paths from $\{{\color{red}1},{\color{red}2}\}$ to $\{2,3\}$ so $[245]$ is not in the prime $P$. On the other hand, there is no vertex disjoint set of paths from  $\{{\color{red}1},{\color{red}2}\}$  to $\{1,2\}$ so $[456]$ is in $P$.

As the reader has seen on the previous examples, the labeling of the sources and sinks is not helpful when one wants to decide whether a given quantum Pl\"ucker coordinate belongs to a specific $\ch$-prime (given by its Cauchon-Le diagram). It is more natural to label the sources of  $\postc$ by the $\gamma_i$ (that is, one replaces source $r_i$ by $\gamma_i$), and the sinks by the $a_j$ from the left (that is, sink $c_{n-m+1-j}$ is replaced  by $a_j$). Note that this labeling coincide with the path-numbering introduced after Notation \ref{associated partition} (see also \eqref{1347}). 

We denote by  $\rm{Post}'(C)$ the Postnikov graph $\postc$ with this alternative diagram. With this new convention, the example above becomes:

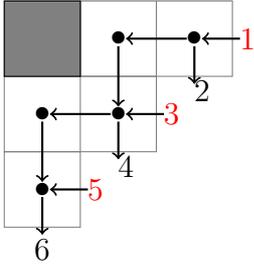
\begin{tikzpicture}[xscale=1, yscale=1]



\draw[color=gray] (0,2) rectangle (1,3);            
\draw[color=gray] (1,2) rectangle (2,3);            
\draw[color=gray] (2,2) rectangle (3,3);            

\draw[color=gray] (0,1) rectangle (2,2);               
\draw[color=gray] (1,1) rectangle (1,2);               

\draw[color=gray] (0,0) rectangle (1,1);            


\draw[fill=gray] (0,2) rectangle (1,3);               

\node[scale=1] at (3.2,2.5) {$\color{red}1$};
\node[scale=1] at (2.6,1.8) {$2$};
\node[scale=1] at (2.2,1.5) {$\color{red}3$};
\node[scale=1] at (1.6,0.8) {$4$};
\node[scale=1] at (1.2,0.5) {$\color{red}5$};
\node[scale=1] at (0.5,-0.3) {$6$};

\node at (1.5, 2.5) {$\bullet$}; 
\node at (2.5, 2.5) {$\bullet$}; 
\node at (0.5, 1.5) {$\bullet$}; 
\node at (1.5, 1.5) {$\bullet$}; 
\node at (0.5, 0.5) {$\bullet$}; 

\draw [<-, thick, black] (2.6,2.5)--(3.1,2.5); 
\draw [<-, thick, black] (1.6,2.5)--(2.4,2.5);

\draw [<-, thick, black] (1.6,1.5)--(2.1,1.5);
\draw [<-, thick, black] (0.6,1.5)--(1.4,1.5);

\draw [<-, thick, black] (0.6,0.5)--(1.1,0.5);

\draw [->, thick, black] (1.5,2.4)--(1.5,1.6);  
\draw [->, thick, black] (2.5,2.4)--(2.5,1.9);  
\draw [->, thick, black] (0.5,1.4)--(0.5,0.6);  
\draw [->, thick, black] (1.5,1.4)--(1.5,0.9);  
\draw [->, thick, black] (0.5,0.4)--(0.5,-0.1);  

\end{tikzpicture}

Moreover, there is a vertex disjoint set of paths from $\{{\color{red}1},{\color{red}3}\}$ to $\{2,4\}$ so $[245]=[\gamma \setminus\{{\color{red}1},{\color{red}3}\} \cup \{2,4\}]$ is not in the prime $P$. On the other hand, there is no vertex disjoint set of paths from  $\{{\color{red}1},{\color{red}3}\}$  to $\{4,6\}$ so $[456]=[\gamma \setminus\{{\color{red}1},{\color{red}3}\} \cup \{4,6\}]$ is in $P$.

We finish this section by restating Theorem \ref{theorem-main-pluckers} with the new labeling of the sources and sinks of $\postc$. 

\begin{theorem}\label{theorem-main-pluckers-2}
Assume that $q\in \k^*$ is not a root of unity.  Let $P\neq \ang{\Pi}$ be an $\ch$-prime ideal of $\oqgmnk$ and let $\gamma=[\gamma_1<\cdots<\gamma_m]$ be the unique quantum Pl\"ucker coordinate such that 
$P\in \hspec_\gamma(\oqgmnk)$ (see \cite[Theorem 5.1]{llr-grass}). Let $\lambda$ be the partition corresponding to $\gamma$ as in Notation \ref{associated partition} and let $Y_\lambda$ be the Young diagram of $\lambda$. Let $C$ be the Cauchon-Le diagram on $Y_\lambda$ corresponding to $P$ as in \eqref{major corresp}. Set $\{a_1<\cdots<a_{n-m}\}=\llb 1,n \rrb\bs\gamma$. Let $\alpha\in \Pi$ be such that $\alpha>\gamma$ and let $1\leq i_1<\cdots < i_t\leq m$, $1\leq j_1<\cdots <j_t\leq n-m$ be such that $\alpha=[(\gamma\bs \{\gamma_{i_1},\ldots,\gamma_{i_t}\})\sqcup\{a_{j_1},\ldots,a_{j_t}\}]$ and $a_{j_l}>\gamma_{i_l}$ for all $l\in \llb 1,t\rrb$. Then the quantum Pl\"ucker coordinate $\alpha$ belongs to $P$ if and only if there does not exist a vertex-disjoint $R_{(\{\gamma_{i_1},\ldots,\gamma_{i_t}\}, \{a_{j_1},\ldots,a_{j_t}\})}$-path system in the Postnikov graph $\rm{Post}'(C)$ of $C$.
\end{theorem}


\section{$\ch$-primes in $\oqgmnk$: generation}

In this section, we prove the main result of this article. More precisely, we prove that $\ch$-primes of $\oqgmnf$ are generated by quantum Pl\"ucker coordinates when $q$ is transcendental (over the prime field of $\f$). 

In order to prove this result, we proceed as follows. First, we use noncommutative dehomogenisation to prove that the result is true up to localisation. Then we use results from \cite{good-len-q-transc} to prove the main result under the assumption that certain ideals of the homogeneous coordinate ring of the corresponding grassmannian are prime. Using results of \cite{kls-proj,oh} on positroid varieties, we conclude that the main result is true over an algebraically close field. Finally, we extend the result to arbitrary fields using ideas from \cite{gll2}.

\subsection{Noncommutative dehomogenisation over $R$} \label{section-dehomogenisation}

Throughout this section, let $R$ be a commutative noetherian domain containing a nonroot of unity $q \in R^{\ast}$. Let $\f$ denote the field of fractions of $R$, and let $\gamma \in \Pi$. We set $\{a_1<\cdots<a_{n-m}\}:=\llb 1,n \rrb\bs\gamma$.  \\

We need to check that the dehomogenisation isomorphism obtained in Section~\ref{qgrass} restricts to the quantum grassmannian over a commutative noetherian domain $R$ from the quantum grassmannian over the field of fractions $\f$ of $R$. \\

Recall that $\oqgmnr$ and $\oqgmnf$ are both $\mn$-graded, with all quantum Pl\"ucker coordinates in degree one. \\

Set $\sgr:=\oqgmnr/\langle\Pi_\gamma\rangle_R$ and 
$\sgf:=\oqgmnf/\langle\Pi_\gamma\rangle_\f$. (Here, we use a subscript $_{\f}$ in order to avoid confusion.) There is a natural map 
$\phi:\oqgmnr\goesto\sgf$ induced by the inclusion of $\oqgmnr$ into 
$\oqgmnf$ followed by the natural homomorphism onto the factor $\sgf$. 
Obviously, $\langle\Pi_\gamma\rangle_R\subseteq\ker(\phi)$ and we claim that 
$\ker(\phi)=\langle\Pi_\gamma\rangle_R$. Suppose that $a\in\ker(\phi)$. Note that
any standard monomial that involves an $\alpha\not\geq\gamma$ is certainly in 
$\langle\Pi_\gamma\rangle_R$; so, without loss of generality, assume that 
$a=\sum\,r_iS_i$ for some $r_i\in R$ and $S_i$ standard monomials that do not
involve such $\alpha$. Let $\widebar{b}$ denote the coset $b+\langle\Pi_\gamma\rangle_\f$. Then, $\sum\,\widebar{r_i}\widebar{S_i} = \phi(a)=0$. However, 
$\sgf$ is a graded quantum algebra with a straightening law, with the 
$\widebar{S_i}$ forming a basis over $\f$, by \cite[Corollary 1.2.6]{lr2}; and so 
each $\widebar{r_i}=0$. Hence, each $r_i\in\langle\Pi_\gamma\rangle_\f\cap R =0$. 
Thus, $a\in\langle\Pi_\gamma\rangle_R$, as required to show that 
$\ker(\phi)=\langle\Pi_\gamma\rangle_R$. 

It follows that $\sgr\subseteq\sgf$, via the inclusion 
$ a+\langle\Pi_\gamma\rangle_R\goesto a+\langle\Pi_\gamma\rangle_\f$. 
As a consequence, $\sgr$ is a domain and $\langle\Pi_\gamma\rangle_R$ 
is a completely prime ideal of $\oqgmnr$.
Also, $\sgr$ and $\sgf$ inherit an $\mn$-grading from the $\mn$-grading on 
$\oqgmnr$ and $\oqgmnf$, respectively; and 
the  inclusion we have just obtained 
preserves the $\mn$-grading on $\sgr$ and $\sgf$. Each of these two rings is 
$\mn$-graded and the inclusion respects the graded components. 
\\

It follows easily from \cite[Theorem 3.4.2]{lr2} that $\widebar{\gamma}$ 
commutes up to a power of $q$ with any $\widebar{\alpha}$ for any 
$\alpha>\gamma$. Thus, $\widebar{\gamma}$ is a regular normal element in 
both $\sgr$ and $\sgf$, and we may invert the powers of $\widebar{\gamma}$ to 
obtain  $\sgr[\widebar{\gamma}^{\,-1}]\subseteq\sgf[\widebar{\gamma}^{\,-1}]$. Each of these rings is 
$\mz$-graded and the inclusion respects the graded components. In particular, 
considering the zero components, we see that $\sgro\subseteq \sgfo$. (Note the use of subscripts compared to previous sections.  Again this is just to avoid confusion.)
\\

By Theorem \ref{So to partition}, there is an isomorphism $\theta: \sgfo\xrightarrow{\cong} 
\partition$ such that, with the notation of Theorem \ref{So to partition}, we have:
$$\theta(\wt{m_{i,a_j}})=x_{m+1-i,n-m+1-j}$$ for each $(i,a_j)\in \cl_\gamma$;
and 
$$\theta^{-1}(x_{i,j})=\wt{m_{m+1-i,a_{n-m+1-j}}}$$
 for each $(i,j)\in Y_\lambda$.

This isomorphism $\theta$ clearly restricts to an isomorphism $\sgro\cong \co_{q^{-1}}(Y_\lambda(R))$. Then the noncommutative dehomogenisation isomorphism 
\begin{equation}\label{eq-psi}
\Psi: \sgf[\widebar{\gamma}^{\,-1}]\xrightarrow{\cong}\co_{q^{-1}}(Y_\lambda(\f))[y^{\pm 1};\sigma]
\end{equation}
from Section \ref{sec-corresp-h-primes-cauchon-diags} (see \eqref{the godfather iso}) restricts to an isomorphism
\begin{equation}\label{eq-psi-2}
\sgr[\widebar{\gamma}^{\,-1}]\xrightarrow{\cong}\co_{q^{-1}}(Y_\lambda(R))[y^{\pm 1};\sigma], 
\end{equation}
which we shall also call $\Psi$, by tracing through the maps above. We summarise our discussion in the following proposition. 
\begin{proposition}[Dehomogenisation over $R$]
Let $\gamma \in \Pi$, and set $\{a_1<\cdots<a_{n-m}\}:=\llb 1,n \rrb\bs\gamma$. Then there is an $R$-isomorphism
\begin{equation}\label{eq-psi-2}
\Psi: \sgr[\widebar{\gamma}^{\,-1}]\xrightarrow{\cong}\co_{q^{-1}}(Y_\lambda(R))[y^{\pm 1};\sigma]
\end{equation}
defined by $\Psi(\bar{\gamma})=y$ and $\Psi ( \wt{m_{m+1-i,a_{n-m+1-j}}} )=x_{i,j}$ for all $(i,j)\in Y_\lambda $.
\end{proposition}



\subsection{Quantum Pl\"ucker coordinates generate $\ch$-primes: algebraically closed case}

In this section, we set $R:=\k[q^{\pm1}]$ with $q$ transcendental over a field $\k$, and $\f:=\k(q^{\pm1})$ denotes the field of fractions of $R$.\\

The aim in this section is to show that any $\ch$-prime ideal in $\oqgmnf$ is generated by the quantum Pl\"ucker coordinates that it contains, under the assumption that the corresponding list of (classical) Pl\"ucker coordinate generates a prime ideal of the homogeneous coordinate ring of the grassmannian $G_{mn}(\k)$.  

Let $P$ be a $\ch$-prime ideal in $\oqgmnf$. As the augmentation ideal is clearly generated by quantum Pl\"ucker coordinates, we assume that $P$ is distinct from the augmentation ideal. 

As usual, we denote by $\gamma=[\gamma_1<\cdots<\gamma_m]$ the unique quantum Pl\"ucker coordinate such that $\gamma\notin P$ but $\alpha\in P$ for all $\alpha\ngeq\gamma$. Let $\lambda$ be the partition associated with $\gamma$ and let $C$ be the Cauchon-Le diagram on $Y_\lambda$ corresponding to $P$. 

\begin{notation}
\begin{enumerate} 
\item We set $\Pcl:=\{\alpha\in \Pi\ |\ \alpha\in P\}$.
\item We define $\{a_1<\cdots<a_{n-m}\}:=\llb 1,n \rrb\bs\gamma$.
\item Let $\alpha\in \Pi$ be such that $\alpha > \gamma$. Then there exists $1\leq i_1<\cdots < i_t\leq m$, $1\leq j_1<\cdots <j_t\leq n-m$ such that $\alpha=[(\gamma\bs \{\gamma_{i_1},\ldots,\gamma_{i_t}\})\sqcup\{a_{j_1},\ldots,a_{j_t}\}]$ and $a_{j_l}>\gamma_{i_l}$ for all $l\in \llb 1,t\rrb$. 

We set $I_{\alpha} := \{\gamma_{i_1},\ldots,\gamma_{i_t}\}$ and $J_{\alpha}:=\{a_{j_1},\ldots,a_{j_t}\}$. 
\end{enumerate}
\end{notation}

With the above notation, it follows from Theorem \ref{theorem-main-pluckers-2} that 
$$\Pcl= \Pi_\gamma\sqcup  \{\alpha \in \Pi\ |\ \alpha>\gamma\tx{ and there is no vertex-disjoint }R_{(I_{\alpha},J_{\alpha})}\tx{-path system in }\rm{Post}'(C)\}.
$$

\begin{lemma}
\label{lemma-generation-up-to-localisation}
The ideal $(P /\langle\Pi_\gamma\rangle)[\widebar{\gamma}^{-1}]$
is completely prime in $(\oqgmnf/\langle\Pi_\gamma\rangle)[\gamma^{-1}]$ and is generated as a right ideal by $\Pcl$.
\end{lemma}
\begin{proof}
Note that complete primeness is clear.  

We use the notation from Sections \ref{sec-corresp-h-primes-cauchon-diags} and \ref{section-9-6}. In particular,we still denote by $\xi$ the biincreasing bijection 
\[
\xi: \ts{\ch-\Spec_\gamma} \oqgmnf\xrightarrow{\cong}\ch-\Spec \partition
\] 
given in \eqref{godmother corresp}. 

Recall that $Q:=\xi (P)=  \Psi(\ol{P}[\ol{\gamma}^{-1}])\cap \partition $ and $P=\Phi\left(\bigoplus_{i\in \Z}Qy^i\right)\cap \left(\oqgmnf/\ang{\Pi_\gamma}\right)$.

As the Cauchon-Le diagram associated to $Q$ is $C$, it follows from Theorems \ref{which pseudo minors} and \ref{gen by pseudos} that $Q$ is generated as a right ideal by the pseudo quantum minors $[I\mid J]$ of $\partition$ such that there are no vertex-disjoint $R_{(I,J)}$-path system in $\rm{Post}(C)$. 

Hence we deduce from Theorem \ref{where it goes} that $\ol{P}[\ol{\gamma}^{-1}]=\Phi\left(\bigoplus_{i\in \Z}Qy^i\right)$ is generated by $\Pcl$ as a right ideal of $(\oqgmnf/\langle\Pi_\gamma\rangle)[\gamma^{-1}]$, as desired. 
\end{proof}

As extension and contraction are inverse bijections between the prime spectra of $\oqgmnf/\langle\Pi_\gamma\rangle$ and $(\oqgmnf/\langle\Pi_\gamma\rangle)[\gamma^{-1}]$, we deduce from the previous lemma the following remark. 

\begin{remark}
If $\Pcl$ generates a prime ideal in $\oqgmnf$, then $\ang{\Pcl}_{\f}=P$. 
\end{remark}
 So, we are aiming to show that the set $\Pcl$ generates a prime ideal in 
$\oqgmnf$. 
Let $\idealPclr$ denote the ideal of $\oqgmnr$ generated by $\Pcl$.

\begin{notation}
Set $\{\alpha_1,\dots,\alpha_s\}:=\Pcl$. To avoid confusion later, in this section, we use overbars to denote images modulo $q-1$ and we denote the 
image of $\gamma$ in $\sgr:=\oqgmnr/\langle\Pi_\gamma\rangle_R$ by 
$\mu$. Similarly, we write $\mu_i$ for the image 
of $\alpha_i$ in $\sgr$, for $i=1,\dots, s$.
\end{notation}

Before proving the main result of this section, we establish an analogue of Lemma \ref{lemma-generation-up-to-localisation} over $R$. 

\begin{lemma}
The ideal $(\idealPclr/\langle\Pi_\gamma\rangle_R)[\mu^{-1}]$
is completely prime 
in $(\oqgmnr/\langle\Pi_\gamma\rangle_R)[\mu^{-1}]$ and is generated as a right ideal by $\mu_1,\ldots,\mu_s$.
\begin{proof}
As in Theorem  \ref{theorem-ideal-is-prime} (with $q$ replaced by $q^{-1}$), $\pcl$ denotes the set of pseudo quantum minors $[I\mid J]$ in $\co_{q^{-1}}(Y_\lambda(R))$ for which there are no vertex disjoint families of $R_{(I,J)}$-paths in the Postnikov graph $\postc$ of $C$.  
 
Recall that, where $\sgr =\oqgmnr/\langle\Pi_\gamma\rangle_R$ and $\sgro\cong \co_{q^{-1}}(Y_\lambda(R))$, Corollaries \ref{cor-pluckers-match-minors} and \ref{Cor minor} show that the noncommutative dehomogenisation isomorphism 
\begin{equation}\label{noncomm-dehom-isom}
\Psi : \sgr[\mu^{-1}]\xrightarrow{\cong}\sgro[y^{\pm 1};\sigma], 
\end{equation}
established in \eqref{eq-psi-2}, sends the right ideal of $S_\gamma(R)[\mu^{-1}]$ generated by $\mu_1,\ldots,\mu_s$ (or equivalently by $\mu_1\mu^{-1},\ldots,\mu_s\mu^{-1}$) to the right ideal of $\co_{q^{-1}}(Y_\lambda(R))[y^{\pm 1};\sigma]$ generated by $\Pi_\lambda^C$, which is a two-sided prime ideal by Theorem \ref{theorem-ideal-is-prime}. The result follows immediately. 
\end{proof}
\end{lemma}

We are now ready to prove the main result of this section. To state it, we need to introduce a few additional notation.

Set $B_R:= (\idealPclr/\langle\Pi_\gamma\rangle_R)[\mu^{-1}]
\cap\oqgmnr/\langle\Pi_\gamma\rangle_R$, and note that $B_R$ is completely 
prime. 

Set $A_R:= \idealPclr/\langle\Pi_\gamma\rangle_R$. Clearly, $A_R\subseteq B_R$. 

Similarly, define $A_{\f}$ and $B_{\f}$. \\

We wish to prove that $A_{\f}=B_{\f}$. If we can do this then $A_{\f}$ is completely prime (since $B_{\f}$ is the localisation at $R^\ast$ of the completely prime ideal $B_R$ and hence $B_{\f}$ is itself completely prime) and it follows that $\ang{P_\lambda^C}_{\f}$ is completely prime.\\

The idea is to use \cite[Proposition 2.1]{good-len-q-transc} whose right-module version is reproduced next for the reader's convenience. 

\begin{proposition}\label{proposition-qtransc}
Let  $k\subset K$ be  a field extension and $q\in K\setminus k$ transcendental over $k$ (so  that the $k$-subalgebra $R=k[q,q^{-1}] \subset K$ is a Laurent polynomial ring).  Let us denote
reduction modulo $q-1$ by overbars, that is, given any
right $R$-module homomorphism $\phi:A\rightarrow B$, we write $\widebar{\phi}:
\widebar{A}\rightarrow \widebar{B}$ for the induced map $A/(q-1)A \rightarrow
B/(q-1)B$. 

 Let $A \stackrel{\phi}\goesto B \stackrel{\psi}\goesto C$ be a complex of right $R$-modules, such that $C$ is torsionfree. Suppose that there are right $R$-module decompositions
$$
A= \bigoplus_{j\in J} A_j, \qquad\qquad B= \bigoplus_{j\in J} B_j,
\qquad\qquad C=
\bigoplus_{j\in J} C_j
$$
such that $B_j$ is finitely generated, $\phi(A_j)\subseteq B_j$, and
$\psi(B_j)\subseteq C_j$ for all $j\in J$.

If the reduced complex $\widebar{A} \stackrel{\widebar{\phi}}\goesto 
\widebar{B} \stackrel{\widebar{\psi}}\goesto 
\widebar{C}$ is exact, then so is 
$$\xymatrixcolsep{3pc}
\xymatrix{
 A \otimes_R K \ar[r]^{\id\otimes\phi} & B \otimes_R K \ar[r]^{\,\id\otimes\psi\,} & C \otimes_R K. }$$
\end{proposition} ~\\

We would like to use this result with $k=\k$ and $A:=A_R, B:=B_R$ while $C:=0$. \\

If we assign degree one to each (image of a) quantum Pl\"ucker coordinate, 
then $A$ and $B$ both have $R$-module decompositions of the type required by the 
proposition, with $J=\mz$.\\

Let $\phi:A\goesto B$ be inclusion and $\psi:B\goesto C$ be the zero map. We aim to show that 
$$\widebar{A} \stackrel{\widebar{\phi}}\goesto \widebar{B} \stackrel{\widebar{\psi}}\goesto \widebar{C}$$
 is exact; that is, we aim to show that $\widebar{\phi}(\widebar{A})=\widebar{B}$, 
 and to do this it is enough to show that $B_R\subseteq A_R + (q-1)B_R$. \\
 
Let $b\in B_R$. Then there is an $i\geq 0$ with $b\mu^i\in A_R$. Reducing modulo $q-1$, we obtain 
$\widebar{b}\widebar{\mu}^i \in \ideal{\widebar{\mu}_1,\dots,\widebar{\mu}_t}$. \\

{\bf Assume that $\ideal{\widebar{\mu}_1,\dots,\widebar{\mu}_t}$ is a prime ideal and $\widebar{\mu}$ is not in this ideal}; then we deduce that 
$ \widebar{b}\in\ideal{\widebar{\mu}_1,\dots,\widebar{\mu}_t}$.\\

Hence, $b=\mu_1c_1+\dots+\mu_tc_t +(q-1)d$ for some $c_i, d\in\sgr$.\\

Now, $\mu_1c_1+\dots+\mu_tc_t\in A_R\subseteq B_R$; and so $(q-1)d\in B_R$. 
As $B_R$ is completely prime and $q-1\not\in B_R$, we obtain that $d\in B_R$ and so $B_R\subseteq A_R+(q-1)B_R$, as required. \\

Thus, the conditions of Proposition~\ref{proposition-qtransc} are satisfied, and we deduce 
that 
$$\xymatrixcolsep{3pc}
\xymatrix{
 A \otimes_R \f \ar[r]^{\id\otimes\phi} & B \otimes_R \f\ar[r]^{\,\id\otimes\psi\,} & C \otimes_R \f }$$
 is exact when $\f=\k(q^{\pm1})$. 
\\

As $C \otimes_R \f=0$, it follows that $ A \otimes_R \f = B \otimes_R \f$; that is, $A_{\f}=B_{\f}$ as right ideals.

The discussion above is summarised in the following proposition.
\begin{proposition} 
Let $\k$ be a field, and let $\f:=\k(q)$ be the field of fractions of the Laurent polynomial ring $\k[q^{\pm 1}]$. 
Assume that $\ideal{\widebar{\mu}_1,\dots,\widebar{\mu}_t}$ is a prime ideal of the homogeneous coordinate ring of the classical grassmannian $G_{mn}(\f)$ and $\widebar{\mu}$ is not in this ideal. 
Then $P=\ideal{\mu_1,\dots,\mu_t}$ is a (completely) prime ideal of $\oqgmnf$, and $P$ is generated by $\mu_1,\dots,\mu_t$ as a right (resp. left) ideal.
\end{proposition} 

Unfortunately, it is not known in general whether $\ideal{\widebar{\mu}_1,\dots,\widebar{\mu}_t}$ is a prime ideal. It is only known under the assumption that the base field $\k$ is algebraically closed. We thank Thomas Lam for the explanation below. \\

Recall that a {\em positroid} is a matroid that can be represented by a $m\times n$ matrix with nonnegative maximal minors. In other words, positroids are the set of maximal minors that do not vanish on an $m\times n$ matrix with nonnegative maximal minors. These matroids were introduced and studied by Postnikov in \cite[Definition 3.2]{post} (see also \cite[Definition 3]{oh}). Crucially for us, Oh proved that each positroid is exactly the intersection of cyclically shifted Schubert matroids (we refer the reader to \cite[Theorem 6]{oh} for more details). \\

Let $\k$ be an algebraic closed field in any finite characteristic. It is shown in \cite[Corollary 4.3]{kls-proj} that projections of Richardson varieties are compatibly Frobenius split. Positroid
varieties defined in \cite[Section 5.2]{kls} are special cases of these.  By \cite[Lemma 4.1]{kls-proj}, components of intersections of Frobenius split subvarieties are automatically Frobenius split and therefore reduced.  By standard arguments, the ``reduced'' statement also holds in characteristic 0 (see for example \cite[Section 1.6]{bk}).\\

Since each positroid variety is the (reduced) intersection of cyclically rotated Schubert varieties by \cite{oh}, the above comment implies that it must be a scheme-theoretic intersection.  It is known for algebraically closed fields in any characteristic that the ideal of a Schubert variety is linearly generated by the Pl\"ucker variables vanishing on it, see for example \cite{ram}.\\

It follows that the ideal of a positroid variety is linearly generated by Pl\"ucker variables not in its positroid.  This ideal is prime because a positroid variety is irreducible. \\

Finally, recall that Postnikov parametrised positroids by Cauchon-Le diagrams in a Young diagram that fits in a rectangle of size $m \times n-m$. Moreover, \cite[Theorem 6.5]{post} (see also \cite[Proposition 13]{oh}) shows that the Pl\"ucker coordinate $\alpha \in \Pi$  belongs to the positroid associated to the Cauchon-Le diagram $C$  on the Young diagram $Y_{\lambda}$  if and only if there exists a vertex-disjoint $R_{(\{\gamma_{i_1},\ldots,\gamma_{i_t}\}, \{a_{j_1},\ldots,a_{j_t}\})}$-path system in the Postnikov graph $\rm{Post}'(C)$ of $C$, where $\{a_1<\cdots<a_{n-m}\}=\llb 1,n \rrb\bs\gamma$ and $1\leq i_1<\cdots < i_t\leq m$, $1\leq j_1<\cdots <j_t\leq n-m$ with $\alpha=[(\gamma\bs \{\gamma_{i_1},\ldots,\gamma_{i_t}\})\sqcup\{a_{j_1},\ldots,a_{j_t}\}]$ and $a_{j_l}>\gamma_{i_l}$ for all $l\in \llb 1,t\rrb$.\\

Comparing the above with Theorem \ref{theorem-main-pluckers-2}, we obtain that the set $\widebar{\mu}_1,\dots,\widebar{\mu}_t$ is exactly the set of Pl\"ucker coordinates that don't belong to the positroid associated to the Cauchon-Le diagram $C$. Hence we conclude that $\ideal{\widebar{\mu}_1,\dots,\widebar{\mu}_t}$ is prime and does not contain $\widebar{\mu}$. \\

We are now ready to state the main result of this section.

\begin{theorem} 
\label{theorem-generation} 

Let $\k$ be an algebraically closed field and let $\f:=\k(q)$ be the field of fractions of the Laurent polynomial ring $\k[q^{\pm 1}]$. Let $P$ be an $\ch$-prime ideal of $\oqgmnf$. 

Then $P$ is generated as a right ideal by the quantum Pl\"ucker coordinates contained in $P$.
\end{theorem}

Note that if $P$ is the $\ch$-prime corresponding to the Cauchon-Le diagram $C$ 
on the Young diagram $Y_{\lambda}$ then the quantum Pl\"ucker coordinates 
that it contains are specified in Theorem~\ref{theorem-main-pluckers}. 
\\

In the next section, we extend this result to the quantum grassmannian over any field that contains a transcendental element $q$. 



\subsection{Quantum Pl\"ucker coordinates generate $\ch$-primes: general case}

In this section, $\f$ denotes an field and $q \in \f ^*$ is not a root of unity. 

In order to extend Theorem \ref{theorem-generation}  to arbitrary field, we will use an idea from \cite{gll2} which is based on the notion of strongly $\ch$-rational ideal. 
\begin{definition} {\rm 
Let $R$ be an integral domain that is an  $\f$-algebra and  supports a rational action of a torus $\ch$. Let $P$ be a completely prime $\ch$-prime ideal of $R$. Then $P$ is
{\em strongly $\ch$-rational} provided that the set of 
$\ch$-invariant elements in the centre of the 
division ring of fractions of $R/P$ is precisely the field $\f$. 
If all $\ch$-prime ideals of $R$ are completely prime and also strongly $\ch$-rational then we say that 
$R$ is a {\em strongly $\ch$-rational $\f$-algebra}.
}
\end{definition}

By 
\cite[Theorem II.6.4]{bg-book}, 
quantum nilpotent algebras are strongly $\ch$-rational algebras. \\

In what follows, $D(R)$ will denote the division ring of fractions of an integral domain $R$ that is a $\f$-algebra, and $ZD(R)$ will denote the centre of this division ring.  In all occurences, $D(R)$ will have an 
induced $\ch$-action from $R$, and the task will be to show that the set of 
$\ch$-invariant 
central elements $ZD(R)^{\ch}:=\{d\in ZD(R)\mid h\cdot d = d~{\rm for~all~}h\in H\}$ is precisely 
$\f$.

\begin{lemma} \label{lemma-shr} 
Let $R$ be a  strongly $\ch$-rational $\f$-algebra
that contains a nonzero normal $\ch$-eigenvector $u$. 
Set $T:=R[u^{-1}]$ and note that there is a natural induced action of $\ch$ on $T$. 
Let $Q$ be an $\ch$-prime ideal of $T$. Then 
$T/Q$ is an integral domain with an induced action of $\ch$, and 
$ZD(T/Q)^{\ch}=\f$. 
\end{lemma}

\begin{proof} 

Note that $Q=(Q\cap R)T$ and that $Q\cap R$ is  a completely prime ideal of $R$ 
that does not contain $u$. It follows that the image $\widebar{u}$ of $u$ in 
$R/(Q\cap R)$ is a 
nonzero normal element of 
$R/(Q\cap R)$. Also, $R/(Q\cap R)\subseteq T/Q$, via the natural map  embedding that  
sends $r+Q\cap R$ to $r+Q$. This map induces an $\ch$-isomorphism 
$T/Q\cong\left(R/(Q\cap R)\right)[\widebar{u}^{-1}]$ which we will consider to be an 
identification. Then 
\[
D(T/Q)=D(\left(R/(Q\cap R)\right)[\widebar{u}^{-1}])=D(R/(Q\cap R)),
\]
where the final equality holds because $\widebar{u}^{-1}\in D(R/(Q\cap R))$. 
Hence $ZD(T/Q)^{\ch}=ZD(R/(Q\cap R)^{\ch} = \f$, as $R$ is a  strongly 
$\ch$-rational $\f$-algebra.

\end{proof} 

\begin{theorem}\label{theorem-shr}
The quantum grassmannian $\oqgmnf$ is a strongly $\ch$-rational algebra. 
\end{theorem} 

\begin{proof} 
Note that the augmentation ideal is clearly strongly $\ch$-rational. 

Let $P$ be an $\ch$-prime ideal of $\oqgmnf$ distinct from the augmentation ideal. Suppose that 
$\gamma$ is the quantum Pl\"ucker coordinate such that $\gamma\not\in P$ 
but $\alpha\in P$ for all $\alpha\ngeq\gamma$. Then 
$\Pi_\gamma\subseteq P$. Set 
$Q:=P/\ideal{\Pi_\gamma}\,\triangleleft\, S(\gamma)$, 
and note that $Q$ is an $\ch$-prime ideal of $S(\gamma)$.  
As $\oqgmnf/P\cong S(\gamma)/Q$, it is 
enough to show that $Q$ is strongly $\ch$-rational in $S(\gamma)$. \\

Set $u:=\widebar{\gamma}\in S(\gamma)$ and recall that $u$ is a nonzero 
normal element in $S(\gamma)$. It follows that $\widebar{u}:=u+Q$ is a nonzero 
normal element of $S(\gamma)/Q$. It also follows that $Q[u^{-1}]$ is an 
$\ch$-prime ideal in $S(\gamma)[u^{-1}]$ and that 
$S(\gamma)[u^{-1}]/Q[u^{-1}]\cong (S/Q)[\widebar{u}^{\,-1}]$. \\

As $\widebar{u}^{\,-1}\in D(S/Q)$, we see that 
\[
ZD(S/Q)^{\ch}=ZD((S/Q)[\widebar{u}^{\,-1}])^{\ch}=ZD(S(\gamma)[u^{-1}]/Q[u^{-1}])^{\ch}.
\]
In view of this, it is enough to show that $ZD(S(\gamma)[u^{-1}]/Q[u^{-1}])^{\ch}=\f$.\\

In order to prove this latter equality, we use the ($\ch$-equivariant) dehomogenisation isomorphism (more precisely, we use the $\ch$-equivariant isomorphism \eqref{the godfather iso})
\[
S^o(\gamma)[y^{\pm 1};\sigma]\cong  
S(\gamma)[u^{-1}]\,.
\]
to transfer the problem into the problem of showing that
$ZD(S^o(\gamma)[y^{\pm 1};\sigma]/J)^{\ch} =\f$,
for each $\ch$-prime ideal 
$J$ of $S^o(\gamma)[y^{\pm 1};\sigma]$.

Note that $S^o(\gamma)$ is a torsionfree quantum nilpotent algebra, by Theorem~\ref{So to partition} and Proposition~\ref{prop-partition-cgl}. It is then easy to check that $S^o(\gamma)[y;\sigma]$ is also a torsionfree quantum nilpotent algebra. 
Hence, $S^o(\gamma)[y;\sigma]$ is strongly $\ch$-rational. 
As $y$ is a nonzero normal $\ch$-eigenvector in  $S^o(\gamma)[y;\sigma]$, 
the equality we desire is obtained by using Lemma~\ref{lemma-shr}
\end{proof} 

In order to prove that $\ch$-primes of $\oqgmnf$ are generated by quantum Pl\"ucker coordinates when $q$ is transcendental, we need results from \cite{gll2} and \cite{gl-ijm}. For the reader's convenience, we have extracted the necessary results.

\begin{proposition}\label{proposition-H-bijection}
 Let $\k\subseteq K$ be  fields, let $A$ be a noetherian $\k$-algebra and let $\ch$ be a group acting on $A$ 
by $\k$-algebra automorphisms. 

Assume that all $\ch$-prime ideals of $A$ are strongly $\ch$-rational. Then the rule 
$P\mapsto P\otimes_{\k}K$ gives a bijection from 
$\ch{\rm -Spec}(A)
\longrightarrow \ch{\rm -Spec}(A\otimes_{\k} K)$.
\end{proposition}

\begin{proof}
This is a special case of \cite[Proposition 3.3]{gl-ijm} where, in the statement of that 
proposition, we set $A_1:=A, A_2:=K, P_1=P, P_2=0, H_1:=\ch$ and $H_2$ to be the trivial group. 
\end{proof} 

\begin{lemma}\label{lemma-same-hspec}
{\rm \cite[Lemma 1.3]{gll2}}
Let $K_1\subseteq K_2$ be infinite fields and let $A$ be a noetherian $K_2$-algebra 
supporting a rational action of a torus $\ch_2:=(K_2^*)^r$ by $K_2$-algebra automorphisms.
Set $\ch_1:=(K_1^*)^r$, which acts on $A$ by restriction of the $\ch_2$-action. 
Then the $\ch_1$-prime ideals of $A$ coincide with the $\ch_2$-prime ideals of $A$. 
\end{lemma} 

\begin{proposition}\label{proposition-bijection} 
{\rm (c.f. \cite[Proposition 1.4]{gll2})}
Let $K_1\subseteq K_2$ be infinite fields, let $q\in K_1^*$ be a nonroot of unity and identify the algebra 
$\oqgmnktwo$ with $\oqgmnkone\otimes K_2$. Set $\ch_i:=(K_i^*)^n$ for $i=1,2$ and let $\ch_i$ act on $\oqgmnki$ by $K_i$-automorphisms in the standard way. Then the rule 
$P\mapsto P\otimes_{K_1}K_2$ gives a bijection from $\ch_1{\rm -Spec}(\oqgmnkone)
\longrightarrow \ch_2{\rm -Spec}(\oqgmnktwo)$. 
\end{proposition} 

\begin{proof}
Set $A_i=\oqgmnki$ and recall that $\ch_i$ acts rationally on $A_i$ and that every 
$\ch_i$-prime ideal $P$ of $A_i$ is stongly $\ch_i$-rational. 

The action of $\ch_1$ on $A_1$ extends naturally to an action of $\ch_1$ on $A_2$ by $K_2$-automorphisms, and it follows from  
Proposition~\ref{proposition-H-bijection} 
that the rule 
$P\mapsto P\otimes_{K_1}K_2$ gives a bijection from $\ch_1{\rm -Spec}(A_1)
\longrightarrow \ch_1{\rm -Spec}(A_2)$.
In view of Lemma~\ref{lemma-same-hspec}, $\ch_1{\rm-Spec}(A_2)=
\ch_2{\rm-Spec}(A_2)$, and the proposition is proved. 
\end{proof} 

A version of the following theorem is obtained for quantum matrices  
in  \cite[Theorem 1.5]{gll2}, and the proof we give is similar to the proof of that theorem. 

\begin{theorem}
\label{theorem-full-generation}
Let $\f$ be a field that contains an element $q$ which is 
transcendental (over the prime field of $\f$). Then every $\ch$-prime ideal of $\oqgmnf$ 
is generated as a right (or left) ideal by the quantum Pl\"ucker coordinates 
that it contains. 
\end{theorem}

\begin{proof} 
First, let $\k$ be the prime subfield of $\f$, and consider the subfield $K_1:=\k(q)\subseteq\f$, and identify 
$\oqgmnf$ with $\oqgmnkone\otimes_{K_1}\f$. Set $\ch_1:=(K_1^*)^n$ 
with the standard action on $\oqgmnkone$. By 
Proposition ~\ref{proposition-bijection}, any $\ch$-prime ideal of $\oqgmnf$ 
is of the form $P\otimes_{K_1}\f$ for some $\ch_1$-prime ideal 
$P$ of $\oqgmnkone$. Hence, it suffices to prove that $P$ is generated, as 
a right (or left) ideal by the quantum Pl\"ucker coordinates that it contains. 

Now, suppose without loss of generality, that $q$ is transcendental over the algebraic closure $\overline{\k}$ of $\k$ and set $\ch_2:= (\overline{\k}^*)^n$ with the standard action on $\oqgmnktwo$, with 
$K_2=\overline{\k}(q)$. By Proposition~\ref{proposition-bijection}, 
$P\otimes_{K_1}K_2$ is an $\ch_2$-prime ideal of $\oqgmnktwo$, and thus, 
by Theorem~\ref{theorem-generation} $P\otimes_{K_1}K_2$ is generated as a right ideal by the set $X$ of quantum Pl\"ucker coordinates that it contains. Note that $X$ is also the set of quantum 
Pl\"ucker coordinates contained in $P$, and let $P'$ be the right ideal 
of $\oqgmnkone$ generated by $X$. Then 
$P'\otimes_{K_1}K_2=P\otimes_{K_1}K_2$, and, 
consequently, $P'=P$. Therefore, $P$ is generated as a right ideal by $X$, and similarly as a left ideal. 
\end{proof} 



\section{The link with the total nonnegative grassmannian and applications} 

\def\cellm{{\rm Cell}(\cm)}
\def\gmnk{G_{mn}(\k)}
\def\R{\mathbb{R}}
\def\gmntnn{G^{\geq 0}_{mn}(\R)}



The positroid cell stratification of the totally nonnegative
grassmannian has been intensively studied over the last dozen or so
years following Postnikov's groundbreaking paper \cite{post}, see, for
example, \cite{ama,arw,lam2,oh}. Besides its intrinsic beauty, there are also
applications in the study of partial differential equations \cite{kw}, scattering amplitudes \cite{arkani}, and juggling \cite{kls}. On seeing Postnikov's paper,
the first two authors observed that the Le diagrams that Postnikov
introduced to parameterise the  positroid cells were the same as the
diagrams introduced by Cauchon to study the $\ch$-prime spectrum of
quantum matrices. This lead to several papers investigating this
connection, culminating in the present work. In the first paper in the
series, \cite{llr-grass} it was shown that the positroid cells of the
totally nonnegative grassmannian were in natural bijection with the
$\ch$-prime spectrum of the quantum grassmannian via what we are now
calling Cauchon-Le diagrams.  We also conjectured at the time that
this bijection would be a homeomorphism between the partially ordered
sets provided by the positroid cells under closure and the $\ch$-prime
spectrum under inclusion. Further, we conjectured that the
$\ch$-primes would be generated by the quantum Pl\"ucker
coordinates that they contain, and that the containment of a quantum
Pl\"ucker coordinate in an $\ch$-prime ideal could be read off from
the Postnikov graph. All of these conjectures have now been answered
in the present work. However, at the time, we did not have the tools
in the quantum setting to verify these conjectures. One setting where
we were able to make progress was for quantum matrices, which occur as
the ``big cell'' in the quantum grassmannian.  The first author had
already shown that, in the case of a transcendental deformation
parameter $q$, each $\ch$-prime ideal was generated by the quantum
minors that it contained, \cite{launois-generation}, verifying a
conjecture of Goodearl and the second author,
\cite{gl-winding}. Yakimov also produced a proof of this result,
\cite{yak-invariant-primes}, and Casteels replaced the transcendental
restriction by the natural condition that $q$ be a not a root of
unity, \cite{cas2}. Casteels' use of (noncommutative) Gr\"obner basis
techniques was crucial to our present work.

In two papers with Goodearl, the first two authors were able to show
that the membership problem for quantum minors in $\ch$-prime ideals
exactly matches to the corresponding problem for minors belonging
to a positroid cell, \cite{gl1, gl2}. Our first task in this section
is to show that the same correspondence holds with respect to
(quantum) Pl\"ucker coordinates in the quantum and totally nonnegative
grassmannians.

We start by outlining the key properties concerning the positroid cell 
stratification of the totally nonnegative grassmannian.

\subsection{The positroid stratification of $\gmntnn$}

The {\em totally nonnegative $m\times n$ grassmannian}, denoted by
$\gmntnn$, consists of those points in the $m\times n$ real
grassmannian which can be represented by a $m\times n$ real matrix
whose Pl\"ucker coordinates are all $\geq 0$.

The study of the cell decomposition of the totally nonnegative grassmannian 
was initiated by Postnikov in  \cite{post}. 

Given a set $\cm$ of $m$-element subsets of $\{1,\dots,n\}$,
the cell in $\gmntnn$ determined by $\cm$, and denoted by $\cellm$,
consists of those points $p$ in $\gmntnn$ for which the Pl\"ucker
coordinates of $p$ that are in $\cm$ are zero, while those not in $\cm$
are greater than zero. For a given choice of $\cm$, the corresponding
cell may well be empty: for example, in the $2\times 4$ totally
nonnegative grassmannian the Pl\"ucker relation
$[12][34]-[13][24]+[14][23]=0$ shows that it is impossible to have a
point where $[12]=[23] = 0$ while the remaining four Pl\"ucker
coordinates are greater than zero. (As a precursor to what we are to
discuss below, notice that essentially the same reasoning, using the
quantum Pl\"ucker relation $[12][34]-q[13][24]+q^2[14][23]=0$, shows
that if a (completely) prime ideal $P$ in the $2\times 4$ quantum grassmannian
contains $[12]$ and $[23]$ then it must also contain at least one of
$[13]$ and $[24]$.)

Postnikov made the following definition.

\begin{definition} 
If $\cellm$ is nonempty then $\cellm$ is a {\em positroid cell} in $\gmntnn$. 
\end{definition} 

It is obvious that the positroid cells give a partition of the totally
nonnegative grassmannian, but much more is true: the positroid cell
decomposition is a stratification, in the sense that the closure of a
positroid cell is a union of positroid cells.

Postnikov showed that the positroid cells are parameterised by a
variety of combinatorial objects.  Two that are relevant to us in this
section are Le-diagrams and Grassmann necklaces. Le-diagrams are what we have
been calling Cauchon-Le diagrams and we will continue to use that terminology. 
Grassmann necklaces are defined later in the section.

A fundamental question is to describe the points in a positroid cell
when given the Cauchon-Le diagram corresponding to the cell.  Oh
obtains the following result and attributes it as a corollary of
\cite[Theorem 6.5]{post}

\begin{theorem}{\rm \cite[Proposition 13]{oh}}\label{oh-vd-paths} 
Let $C$ be a Cauchon-Le diagram defined on a partition $Y_\lambda$
which fits into a $m\times (n-m)$ rectangle, and let $P_C$ denote the
positroid cell in $\gmntnn$ corresponding to $C$.  Let $[I]$ be the Pl\"ucker
coordinate corresponding to the partition $\lambda$; so that the rows
of $Y_\lambda$ are indexed by $I$. Then a Pl\"ucker coordinate
$[J]$ vanishes on all points of $P_C$ if and only if there are no
vertex disjoint $R_{(I,J)}$-path systems in the Postnikov graph
$\mathrm{Post}'(C)$ of $C$.
\end{theorem}

\subsection{Links between the quantum and totally nonnegative grassmannians} 

Comparison of Theorem~\ref{oh-vd-paths} with our Theorem~\ref{theorem-main-pluckers} 
immediately gives the following theorem. 

\begin{theorem}\label{theorem-triple-correspondence}

Let $C$ be a Cauchon-Le diagram that fits into a $m\times (n-m)$ 
rectangle. Let $P_C$ be the positroid cell corresponding to $C$. 
Let $\k$ be a field, $q\in \k^*$ not a root of unity, and let $Q_C$ be the $\ch$-prime ideal of 
$\oqgmnk$ corresponding to $C$. Let $\cf$ be the set of 
 Pl\"ucker coordinates in $\ogmnr$ that vanish on the whole of 
$P_C$ and let  $\cf_q$ be the set of quantum Pl\"ucker coordinates 
that are contained in $Q_C$. 

Then $\cf$ and $\cf_q$ are essentially the same, in the sense that 
the Pl\"ucker coordinates in $\cf$ are obtained from the quantum Pl\"ucker 
coordinates in $\cf_q$ by setting $q=1$. 
\end{theorem} 

\begin{proof}
Membership of $\cf$ and $\cf_q$ are both recognised by the same test: nonexistence of vertex disjoint 
families of paths in $\mathrm{Post}'(C)$, by Theorem~\ref{theorem-main-pluckers-2} and Theorem~\ref{oh-vd-paths}.
\end{proof} 

In view of this theorem, for the rest of this section, we will assume that $\k$ is a field and that $q\in \k^*$ is not a root of unity. We can then use the theorem to transfer many results concerning the positroid cells of $\ogmntnn$ to the $\ch$-prime spectrum of $\oq(G_{mn}(\k))$.



\subsection{Grassmann necklaces}

In \cite{post}, Postnikov associated various combinatorial objects with 
the nonempty positroid cells of the totally nonempty grassmannian.  
In view of the connection that we have established between positroid cells in the totally nonnegative grassmannian and the $\ch$-spectrum of the quantum grassmannian, we can potentially use these combinatorial objects to study the quantum grassmannian. 
Here, we use the notion of Grassmann necklace to study the Zariski topology 
of $\hspec(\oqgmnk)$; that is, to decide when $Q\subseteq P$ for $\ch$-prime ideals 
$P,Q$. \\

Recall that the {\em irrelevant ideal} of $\oqgmnk$ is the ideal generated by 
all of the quantum Pl\"ucker coordinates. This ideal is an $\ch$-prime ideal that contains all other $\ch$-prime ideals and so in what follows we will usually assume that we are dealing with $\ch$-prime ideals that are distinct from the irrelevant ideal.\\ 

A  {\em matroid of rank $m$} on the set $\setn$ is a
non-empty collection of $m$-element subsets $\cm$ of $\setn$ called {\em
bases} of $\cm$, that satisfy the {\em exchange axiom}: for any $I,J\in\cm$
and $i\in I$, there exists $j\in J$ such that
$I\backslash\{i\}\sqcup\{j\}\in\cm$, see, for example,  \cite[2.3]{post}. 

Postnikov showed \cite[2.3]{post} that the set of Pl\"ucker coordinates that do not vanish on a given positroid cell in $\gmntnn$ forms a matroid of rank $m$. In view of Theorem~\ref{theorem-triple-correspondence}, the same conclusion applies to the set of quantum Pl\"ucker coordinates that are not contained in a given $\ch$-prime ideal of $\oqgmnk$ distinct from the irrelevant ideal. For the convenience of the reader, we give a direct proof which is essentially the same as that for $\gmntnn$, but uses quantum Pl\"ucker relations rather than classical Pl\"ucker relations. 


\begin{proposition}
Let $q\in \k^*$ and assume that $q$ is not a root of unity.  Let $P$ be an $\ch$-prime of $\oqgmnk$ that is not the irrelevant ideal  and let $\cm$ be the set of quantum Pl\"ucker coordinates that are not in $P$. Then $\cm$ is a matroid of rank $m$. 
\end{proposition} 


\begin{proof}\begin{full}   Let $[I]$ and $[J]$ be two quantum Pl\"ucker coordinates that are not in $P$. We need to verify
the exchange axiom for $I$ and $J$. Fix $i\in I$. If $i\in J$ then the
exchange axiom is trivially satisfied by choosing $j=i$. Thus, assume that
$i\not\in J$.

The generalised quantum Pl\"ucker
relations of \cite[Theorem 2.1]{klr} (see also Theorem \ref{theorem-quantum-plucker-relations}), applied with $J_1=\inoti$,
$J_2=\emptyset$ and
$K=\jwithi$ give a relation 
\[
\sum_{K'\sqcup K''=K} (-q)^{\bullet}[J_1\sqcup K'][K'']=0 
\] 
in
$\oqgmnk$. (Here, by a symbol $(-q)^{\bullet}$, 
we mean some power
of $-q$ with exponent in $\mz$.) 

Consider the various terms in the above expression. When $K'= \{i\}$ then
$J_1\sqcup K' = I$ and $K'' = J$. Otherwise, $K' = \{j\}$ for some $j\in J$
and each of the other terms is of the form
$[\inotiwithj][J\backslash\{j\}\sqcup \{i\}]$. If every term $[\inotiwithj]$
is in $P$ then we obtain $[I][J]\in P$. However, $P$ is completely prime; so
either $[I]\in P$ or $[J]\in P$, a contradiction. Thus there exists $j\in J$ with
$[\inotiwithj]\in\cm$ and the exchange axiom is verified. 
\end{full}\end{proof}


We know that the quantum grassmannian is a graded quantum algebra with a straightening law (abbreviated to QGASL)
 on the usual poset $\Pi$, see \cite{lr2} and Section~\ref{section-known} earlier in this paper.
However, we can put other partial orderings on the quantum Pl\"ucker coordinates and again  get a QGASL structure. In particular, for each $i$ with $1\leq
i\leq n$ we can define the {\em $i$-ordering}, denoted by $\lessi$. In this
ordering, we have 
\[
i\lessi i+1 \lessi i+2\lessi \dots \lessi n\lessi 1\lessi \dots \lessi i-1.
\]

There is then an induced partial ordering on the quantum Pl\"ucker coordinates given by 
$$[a_1 \lessi \dots \lessi a_m] \leqqi [b_1\lessi \dots \lessi b_m]$$ if and only if $a_j\leqqi b_j$
for each $1\leq j\leq m$.  

We will refer to the set of quantum Pl\"ucker coordinates with this partial ordering as
$\Pi_i$. Note that $\lessone$ is just the usual ordering on the natural
numbers and that $\Pi_1$ is the usual $\Pi$. 

As an example, we show the four partial orders on $\gqtwofour$ in Figure~\ref{figure-2x4-orderings}.


\begin{figure}[ht]
\ignore{
{\hskip 10ex}
 \parbox{25ex}{
$$\xymatrixrowsep{2.4pc}\xymatrixcolsep{3.2pc}\def\objectstyle{\scriptstyle}
\xymatrix@!0{
 & [34] \edge[d]&\\
 &[24] \edge[dl] \edge[dr]\\
[14]\edge[dr]&&[23]\edge[dl]\\
&[13]\edge[d]&\\
&[12]&
}$$
}{\hskip 10ex}
 \parbox{25ex}{
$$\xymatrixrowsep{2.4pc}\xymatrixcolsep{3.2pc}\def\objectstyle{\scriptstyle}
\xymatrix@!0{
 & [14] \edge[d]&\\
 &[13] \edge[dl] \edge[dr]\\
[12]\edge[dr]&&[34]\edge[dl]\\
&[24]\edge[d]&\\
&[23]&
}$$
}

{\hskip 10ex} 
 \parbox{25ex}{
$$\xymatrixrowsep{2.4pc}\xymatrixcolsep{3.2pc}\def\objectstyle{\scriptstyle}
\xymatrix@!0{
 & [12] \edge[d]&\\
 &[24] \edge[dl] \edge[dr]\\
[23]\edge[dr]&&[14]\edge[dl]\\
&[13]\edge[d]&\\
&[34]&
}$$
}
{\hskip 10ex} \parbox{25ex}{
$$\xymatrixrowsep{2.4pc}\xymatrixcolsep{3.2pc}\def\objectstyle{\scriptstyle}
\xymatrix@!0{
 & [23] \edge[d]&\\
 &[13] \edge[dl] \edge[dr]\\
[34]\edge[dr]&&[12]\edge[dl]\\
&[24]\edge[d]&\\
&[14]&
}$$
}
\caption{The four  $i$-orderings on $\gqtwofour$.}
\label{figure-2x4-orderings}
}
\end{figure}



\begin{theorem} \label{theorem-qasl} 
The quantum grassmannian $\oqgmnk$ is a QGASL on the poset $\Pi_i$ for each $i$
with $1\leq i\leq n$. 
\end{theorem}

\begin{proof}\begin{full}  
See \cite[Theorem 3.5]{lruss} 
\end{full}\end{proof}

\begin{theorem} \label{theorem-qasl-for-i-order}
Consider the QGASL structure on $\oqgmnk$ determined by the poset $\Pi_i$. Let
$P$ be an $\ch$-prime ideal of $\oqgmnk$. Then there is a unique quantum Pl\"ucker coordinate
 $[I_i]$ in $\Pi_i$ with the property that $[I_i]\not\in P$ but $[J]\in P$ for all
$J\not\geqqi I_i$.   
\end{theorem}

\begin{proof}\begin{full}  
This is proved in \cite[Theorem 5.1]{llr-grass} for the standard order $\leq_1$. Inspection 
of the proof reveals that it works for any quantum graded algebra with a straightening law; so the result follows from the previous theorem.
\end{full}\end{proof} 



We adapt Postnikov's definition of a Grassmann necklace, see \cite[Definition 16.1]{post}, from the totally nonnegative case to the quantum case in the following definition.

\begin{definition}{\rm 
A {\em Grassmann necklace in $\oqgmnk$} is a sequence $\ci = ([I_1], \dots,
[I_n])$ of quantum Pl\"ucker coordinates such that for each $1\leq i\leq n$, if $i\in I_i$
then $I_{i+1}= (I_i\backslash\{i\})\sqcup \{j\}$ for some $1\leq j\leq n$,
and, if $i\not\in I_i$ then $I_{i+1}= I_i$. (Here, indices are taken modulo
$n$, so that $I_{n+1}=I_1$.)
}\end{definition} 

In \cite{post}, Postnikov shows that each positroid cell in $\gmntnn$ determines a Grassmann necklace. In view of Theorem~\ref{theorem-triple-correspondence}, 
the same result is available for $\ch$-prime ideals in the quantum grassmannian $\oqgmnk$. Nonetheless, we give a direct proof in the quantum setting for the convenience of the reader.


\begin{theorem} 
Let $q\in \k^*$ and assume that $q$ is not a root of unity. Let $P$ be an $\ch$-prime ideal of $\oqgmnk$ and let $\ci = ([I_1], \dots,
[I_n])$, where $[I_i]$ is the quantum Pl\"ucker coordinate of $\oqgmnk$ determined by Theorem~\ref{theorem-qasl-for-i-order}. 
Then $\ci$ is a Grassmann necklace. 
\end{theorem}


\begin{proof}\begin{full}  For each $i=1,\dots,n$ we need to show the correct relationship between $I_i$ and $I_{i+1}$. The proof is the same in all cases, but we will present the 
proof in the case that $i=1$ to simplify the notation.

Note that if $a\leqqone b$ and $a\neq 1$ then $a\leqqtwo b$. As a consequence, if
$[I]$ and $[J]$ are quantum Pl\"ucker coordinates with $I\leqqone J$ and $1\not\in I$ then $I\leqqtwo J$. 

Let $[I_1] =[i_1\lessone i_2\lessone\dots\lessone i_m]$ and 
$[I_2] =[j_1\lessone j_2\lessone\dots\lessone j_m]$. 

It follows from the previous theorem that $I_1\leqqone I_2$, as $[I_2]\not\in
P$. As a consequence, $i_s\leqqone j_s$ for each $1\leq s\leq m$. Also, the
observation above shows that $i_s\leqqtwo j_s$ for each $2\leq s\leq m$ and
that $i_1\leqqtwo j_1$ whenever $i_1\neq 1$. 

In a similar way,  $I_2\leqqtwo I_1$, as $[I_1]\not\in
P$. \\


We consider three mutually exclusive and exhaustive cases. Case 1: $i_1\neq
1$. Case 2: $i_1=j_1 =1$. Case 3: $i_1=1$ and $j_1 \neq 1$.\\

\noindent{\bf Case 1} Assume that $i_1\neq 1$, and note that this implies that $1\not\in I_1$, so we need to prove that $I_2=I_1$. Now, $[I_1] =[i_1\lesstwo i_2\lesstwo\dots\lesstwo i_m]$ 
and $[I_2] =[j_1\lesstwo j_2\lesstwo\dots\lesstwo j_m]$. 
Also, it follows from the previous theorem that $I_2\leqqtwo I_1$, as
$[I_1]\not\in P$. Thus, $j_s\leqqtwo i_s$ for each $1\leq s\leq m$. 
However, we know from above that 
$i_s\leqqtwo j_s$ for each $1\leq s\leq m$. Thus, 
$i_s= j_s$ for each $1\leq s\leq m$; and so $I_2=I_1$, as required. \\

\noindent{\bf Case 2} Assume that $i_1=j_1 =1$. We need to show that $I_2=(I_1\backslash\{i_1\})\sqcup \{j_k\}$ for some $j_k\in I_2$.  Now, $[I_1] =[i_2\lesstwo
i_3\lesstwo\dots\lesstwo i_m\lesstwo i_1]$ and $[I_2] =[j_2\lesstwo
j_3\lesstwo\dots\lesstwo j_m\lesstwo j_1]$. Thus, $j_s\leqqtwo i_s$ for each
$1\leq s\leq m$, because $I_2\leqqtwo I_1$. As we already know from above that $i_s\leqqtwo j_s$ for each $2\leq s\leq m$ this shows that $i_s=j_s$
for each $2\leq s\leq m$. As we are in Case 2, we know that 
$i_1=j_1 =1$; so $I_2=I_1 =(I_1\backslash\{i_1\})\sqcup \{j_1\}$, as required. \\


\noindent{\bf Case 3} Assume that $i_1=1$ and $j_1 \neq 1$. In particular,
$1\not\in I_2$. We need to show that $I_2=(I_1\backslash\{i_1\})\sqcup \{j_k\}$ for some $j_k\in I_2$.\\

Recall that $i_1 =1$ is the largest element with respect 
to $\lesstwo$ so displaying $I_1$ in the order
$\lesstwo$, we have 
\[
[I_1] = [i_2\lesstwo\dots\lesstwo i_m\lesstwo i_1].
\]

Also, $j_1\neq 1$; so 
 displaying $I_2$ in the order
$\lesstwo$, we have 
\[
[I_2] = [j_1\lesstwo j_2\lesstwo\dots\lesstwo j_m].
\]
Thus, $j_t\leqqtwo i_{t+1}$ for each $t=1, \dots, m-1$, 
because $I_2\leqqtwo I_1$.\\

 We apply the generalised quantum Pl\"ucker
relations of \cite[Theorem 2.1]{klr} (see also Theorem \ref{theorem-quantum-plucker-relations}) with the following
assignations. Set $J_1:=I_1\backslash\{1\}$ and $J_2:=\emptyset$, 
while $K:=I_2\sqcup\{1\}$. There is a relation 
\[
\sum_{K'\sqcup K''=K} (-q)^{\bullet}[J_1\sqcup K'][K'']=0 
\] 
in
$\oqgmnk$. (As usual, by a symbol $(-q)^{\bullet}$, 
we mean some power of $-q$ with exponent in $\mz$.) 

Consider the various terms in the above expression. When $K'= \{1\}$ then
$J_1\sqcup K' = I_1$ and $K'' = I_2$. Thus, $[I_1][I_2]$ is one of the terms in
the above relation. Each of the other terms is of the form
$[(I_1\backslash\{1\})\sqcup\{j_s\}][(I_2\backslash\{j_s\})\sqcup\{1\}]$ for
some $1\leq s\leq m$. 

Now, $[I_1][I_2]\not\in P$; so there must be a term 
$[(I_1\backslash\{1\})\sqcup\{j_s\}][(I_2\backslash\{j_s\})\sqcup\{1\}]
\not\in P$. 
For such a term, set
$[X]:= [(I_1\backslash\{1\})\sqcup\{j_s\}]$ 
and 
$[Y]:= [(I_2\backslash\{j_s\})\sqcup\{1\}]$. 
Note that $[X],[Y]\not\in P$, as $P$ is a completely prime ideal. Hence,
$I_1\leqqone X,Y$ and $I_2\leqqtwo X,Y$.\\

We consider three subcases: (i) $s=1$; (ii) $1<s<m$; and (iii) $s=m$. \\


\noindent{\bf Subcase (i)} Assume that  $s=1$. Thus, 
$[X]:= [(I_1\backslash\{1\})\sqcup\{j_1\}]$ 
and 
$[Y]:= [(I_2\backslash\{j_1\})\sqcup\{1\}]$.

We know 
$j_1\leqqtwo i_2$ from above. In fact, $j_1\lesstwo
i_2$, since $j_1\not\in I_1$. As a result, displaying $X$ in $\lesstwo$, we
have
\[
[X]=[j_1\lesstwo i_2\lesstwo\dots\lesstwo i_m].
\]
Hence, $j_t\leqqtwo i_t$ for each $2\leq t\leq m$, because $I_2\leqqtwo X$. 
From the beginning of the proof we
know that $i_t\leqqtwo j_t$ for each $2\leq t\leq m$; so $j_t= i_t$ for each
$2\leq t\leq m$. Thus, 
\[
I_2 =\{j_1, j_2, \dots, j_m\} = \{j_1, i_2, \dots, i_m\}= 
(I_1\backslash\{i_1\})\sqcup\{j_1\},
\]
as required. \\




\noindent{\bf Subcase (ii)} Assume that  $1<s<m$. 
 We wish to display $X$ with respect
to $\lesstwo$ and to do this we need to know where to place $j_s$. We have
already observed that $j_s\leqqtwo i_{s+1}$, and, in fact $j_s\lesstwo
i_{s+1}$, since $j_s\not\in I_1$. From the beginning of the proof, we know
that $i_s\leqqtwo j_s$. In fact, $i_s<_2 j_s$, since $j_s\not\in I_1$. Thus,
$i_s\lesstwo j_s \lesstwo i_{s+1}$. Hence, displaying $X$ in the order
$\lesstwo$, we have 

\[
[X]= [i_2\lesstwo i_3\lesstwo\dots\lesstwo i_s\lesstwo 
j_s\lesstwo i_{s+1}\lesstwo\dots\lesstwo i_m].
\]
Thus, we see that 
\begin{eqnarray}\label{j<i}
j_1\leqqtwo i_2, \quad\dots, \quad j_{s-1}\leqqtwo i_s, 
\quad j_{s+1}\leqqtwo i_{s+1}, \quad\dots, \quad j_m\leqqtwo i_m,
\end{eqnarray}
because $I_2\leqqtwo X$. 

As $[Y]\not\in P$, we know that $I_1\leqqone Y$. Now displaying $[Y]:= [(I_2\backslash\{j_s\})\sqcup\{1\}]$ in the order
$\lessone$, we have 
\[
[Y] = [ i_1\lessone j_1\lessone\dots\lessone  
j_{s-1}\lessone j_{s+1} \lessone\dots\lessone j_m]
\]
(recalling that $i_1=1$). 
Hence, the fact that $I_1\leqqone Y$ shows that 
\begin{eqnarray*} 
i_2\leqqone j_1, \quad \dots,\quad i_s\leqqone j_{s-1}, \quad
i_{s+1}\leqqone j_{s+1}, \quad \dots, \quad  i_m\leqqone j_m.
\end{eqnarray*}

As $i_2> 1$ we get the same inequalities with respect to $\lesstwo$; that
is, 
\begin{eqnarray} \label{i<j}
i_2\leqqtwo j_1, \quad \dots,\quad i_s\leqqtwo j_{s-1}, \quad
i_{s+1}\leqqtwo j_{s+1}, \quad \dots, \quad  i_m\leqqtwo j_m.
\end{eqnarray} 

From~(\ref{j<i}) and~(\ref{i<j}), we see that 
\[
j_1=i_2,\quad\dots,\quad j_{s-1}= i_s, \quad 
j_{s+1} = i_{s+1},\quad\dots,\quad j_m=i_m.
\]

Therefore,
\begin{eqnarray*}
I_2 &=&\{j_1,\dots ,j_{s-1},j_s,j_{s+1},\dots, j_m\}\\
     &=&\{i_2,\dots ,i_s,j_s, i_{s+1},\dots,i_m\}\\
     &=& (I_1\backslash\{i_1\})\sqcup\{j_s\},
\end{eqnarray*}     
as required.\\




\noindent{\bf Subcase (iii)} Assume that  $s=m$. 
This case is similar to the previous
case; so we will be brief. Recall that  $i_m\leqqtwo j_m$, and, note that, in fact, $i_m<_2 j_m$ as $j_m\not\in I_1$. Consequently, 
\[
[X]= [(I_1\backslash\{i_1\})\sqcup\{j_m\}]= [i_2\lesstwo i_3\lesstwo\dots\lesstwo i_m\lesstwo j_m];
\]
and so 
\begin{eqnarray}\label{j<i-case-s=m}
j_1\leqqtwo i_2, \quad\dots, \quad j_{m-1}\leqqtwo i_m
\end{eqnarray}
because $I_2\leqqtwo X$. \\

Now $[Y]= [ i_1\lessone j_1\lessone\dots\lessone  
j_{m-1}]$ and so, as $I_1\leq_1 Y_1$, we see that 
\begin{eqnarray*} 
i_2\leqqone j_1, \quad \dots,\quad i_m\leqqone j_{m-1}
\end{eqnarray*} 
and deduce that 
\begin{eqnarray} \label{i<j-case-s=m}
i_2\leqqtwo j_1, \quad \dots,\quad i_m\leqqtwo j_{m-1},
\end{eqnarray} 
as $i_2\neq 1$.

From~(\ref{j<i-case-s=m}) and~(\ref{i<j-case-s=m}), we see that 
\[
j_1=i_2,\quad\dots,\quad j_{m-1}= i_m.
\]

Therefore,
\begin{eqnarray*}
I_2 &=&\{j_1,\dots ,j_{m-1},j_m\}\\
     &=&\{i_2,\dots ,i_m,j_m\}\\
     &=& (I_1\backslash\{i_1\})\sqcup\{j_m\},
\end{eqnarray*}     
as required.


\end{full}\end{proof} 




Let $P$ be an $\ch$-prime ideal in $\oqgmnk$ 
with associated Grassmann necklace $\ci = ([I_1],\dots, [I_n])$.
Then Theorem~\ref{theorem-qasl-for-i-order} shows that 
$P$ contains the quantum Pl\"ucker coordinates  in the set 
\[ 
\bigcup_{i=1}^n \,\{[J] \mid J\not\geqqi I_i\}.
\]
The first two authors and Rigal had conjectured earlier that these are precisely the quantum Pl\"ucker coordinates in $P$. We can now prove this conjecture by appealing to the connection with the totally nonnegative grassmannian. 

\begin{theorem}\label{theorem-union} 
Assume that $q \in \k^*$ is not a root of unity. Let $P$ be an $\ch$-prime ideal in $\oqgmnk$ 
with associated Grassmann necklace $\ci = ([I_1],\dots, [I_n])$.
Then the set of quantum Pl\"ucker coordinates contained in $P$ is equal to the 
set 
\[ 
\bigcup_{i=1}^n \,\{[J] \mid J\not\geqqi I_i\}.
\]
\end{theorem} 

\begin{proof} 
Let $\cf_q$ be the set of quantum Pl\"ucker coordinates contained in $P$. 
Then the corresponding set of Pl\"ucker coordinates $\cf$ determines a nonempty 
cell $S$ in the totally nonnegative grassmannian. The Grassmann necklace associated 
with $S$ is also given by $\ci = ([I_1],\dots, [I_n])$ (where we are now interpreting the 
$[I_j]$ as classical Pl\"ucker coordinates). 
By \cite[Proposition 16]{oh}, 
the members of $\cf$ can be described as in the claim of the theorem; and so 
the same result holds for $\cf_q$.
\end{proof} 

If  $P$ is the $\ch$-prime ideal $P$ of $\oqgmnf$ with 
Grassmann necklace $([I_1],\dots,[I_n])$, we set 
$\neck(P):=([I_1],\dots,[I_n])$.

Define a partial order on the set of Grassmann necklaces in the following way. Let  $Q$ be another 
$\ch$-prime ideal with $\neck(Q)=([J_1],\dots,[J_n])$. Then set 
$\neck(Q)\leq\neck(P)$ if and only if $J_i\leq_i I_i$ for each $i=1,\dots,n$. 

\begin{theorem}\label{proposition-necklace-order} 
Assume that $q \in \k^*$ is transcendental over the prime field of $\k$ and suppose that 
 $P$ and $Q$ are $\ch$-prime ideals of $\oqgmnk$.  Then 
$Q\subseteq P$ if and only if $\neck(Q)\leq\neck(P)$. 
\end{theorem} 

\begin{proof}\begin{full} Set $\neck(P)=  ([I_1],\dots,[I_n])$ and $\neck(Q)=([J_1],\dots,[J_n])$. 

($\Rightarrow$) 
Suppose that $Q\subseteq P$ but that there is an $i$ 
with $J_i\not\leq_i I_i$. Then $[I_i]\in Q\subseteq P$, a contradiction. Hence, 
$J_i\leq_i I_i$ for each $i=1,\dots,n$; that is, $\neck(Q)\leq\neck(P)$. 

($\Leftarrow$) Suppose that $Q\not\subseteq P$. We need both quantum 
Pl\"ucker coordinates and classical Pl\"ucker coordinates in this argument; so, if 
$[I]$ is a classical  Pl\"ucker coordinate then the corresponding quantum 
 Pl\"ucker coordinate is denoted by $[I]_q$. 
 
 Suppose that $P$ is associated with 
 the Cauchon-Le diagram $C$ and denote the corresponding positroid cell 
 by $S_C$. Similarly, suppose that $Q$ is associated with 
 the Cauchon-Le diagram $D$ with corresponding positroid cell 
 $S_D$. Then $[I]_q \in P$ if and only if $[I](s)=0$ for all $s\in S_C$, and 
 $[I]_q \in Q$ if and only if $[I](s)=0$ for all $s\in S_D$. 

As $Q\not\subseteq P$ there is a quantum Pl\"ucker coordinate 
$[I]_q \in Q\backslash P$ by Theorem \ref{theorem-full-generation}.  Hence, $[I](s)=0$ for all $s\in S_D$. Thus, by 
\cite[Proposition 16]{oh}, there exists a $j$ with $I\not\geq_j J_j$. 

Now, $[I]_q\not\in P$ and so there exists $s\in S_C$ with $[I](s)\neq 0$. 
It follows from \cite[Proposition 16]{oh} that $I\geq_i I_i$ for all $i$. In particular, $I \geq_j I_j$. 
If $I_j\geq_j J_j$ then $I\geq_j I_j\geq_jJ_j$, which is a contradiction. 
Thus, $I_j\not\geq_j J_j$ and so $\neck(P)\not\geq\neck(Q)$. 
\end{full}\end{proof} 

The previous proposition shows that we can decide containment between two 
$\ch$-primes $P$ and $Q$ of $\oqgmnk$ by comparing their Grassmann necklaces.
Let $C$  be the Cauchon-Le diagram associated with $P$. In \cite{oh} an algorithm is given to determine $\neck(P)$ 
from $C$ and so this gives an algorithm for the containment problem 
for $\hspec(\oqgmnk)$, as it is easy to compare quantum Pl\"ucker coordinates 
in each $i$-order. Another description of an algorithm to compute the Grassmann 
necklace from a Cauchon-Le diagram is given in \cite{cf}, at least in the case of 
$\oqmmnk$. This latter algorithm is perhaps easier to understand than Oh's algorithm, and can easily be adapted to work in the general $\oqgmnk$ case. An algorithm to 
perform the reverse procedure; that is, to reconstruct a Cauchon-Le diagram from a 
Grassmann necklace, is presented in \cite{af}.




\subsection{On Yakimov's conjecture} 

In  \cite{yak-qflag}, Yakimov parameterises the $\ch$-prime ideals in 
quantum partial flag varieties by pairs of elements in the associated Weyl group, and makes a conjecture about containments between such primes in terms of the parameterisation. 
He states that the conjecture was proved in the full flag case by Gorelik, 
\cite{gorelik}, but otherwise  remains open, even in the grassmannian 
case.  In this subsection, we give a characterisation of the poset of $\ch$-primes in the quantum grassmannian. \\

We begin by describing briefly Yakimov's parameterisation and his conjecture. Full details are available in \cite{yak-qflag}. \\

Let $G$  be a split, simply connected, semisimple algebraic group over a field $K$ 
 of characteristic zero and let $P_I$ be the standard parabolic subgroup associated 
 with a set of simple roots $I$. If $W$ is the Weyl group of $G$ 
then $W^I$ denotes the minimal length representatives of cosets in $W/W_I$ where $W_I$ is the parabolic 
subgroup of $W$ corresponding to $P_I$.

\begin{theorem}
\label{theorem-yakimov} 
{\rm (Yakimov, \cite[Theorem 1.1]{yak-qflag})} 
The $\ch$-invariant primes of the quantum partial flag variety $R_q[G/P_I]$ are parameterised by 
\[
S_{W,I}:= \{ (w,v)\in W^I\times W \mid v\leq w\}, 
\]
the partial ordering being the Bruhat order. \\
\end{theorem} 

\begin{conjecture}\label{conjecture-yakimov} 
{\rm 
(Yakimov, \cite[Conjecture 1.2]{yak-qflag}) For $(w,v), (w',v') \in S_{W,I}$ the containment 
$P_{w,v}\subseteq P_{w',v'}$ holds if and only if there exists $z\in W_I$ such that 
\[
w\geq w'z\qquad{\rm and}\qquad v\leq v'z.
\]
}
\end{conjecture} 

The inspiration for Yakimov's conjecture is the following. 
The algebra $R_q[G/P_I]$ is a quantisation of the projective Poisson variety
$(G/P_I, \pi_I )$, and there is a bijection between the $\ch$-prime ideals in 
$R_q[G/P_I]$ and the so-called torus orbits of symplectic leaves  $T^I_{w,v}$ of 
$(G/P_I, \pi_I )$, as these torus orbits are also parameterised by $S_{W,I}$.\\

The Zariski closures of  $T^I_{w,v}$ were explicitly determined in 
\cite{gy} and \cite{rietsch}:
 \[
\widebar{T^I_{w,v}} = \{T^I_{w',v'} \mid {\rm ~there~exists~} z\in W_I {\rm ~such~that~} 
w\geq w'z, v\leq v'z\}.
\]
The orbit method then suggests that the containments between $\ch$-prime 
ideals in $R_q[G/P_I]$ should mirror the description of the Zariski closures 
for the torus orbits  of symplectic leaves and this leads to Yakimov's conjecture. 
\\

In \cite[Theorem 1.8]{gy}, Goodearl and Yakimov give the result for membership 
of the closures in the Zariski topology of the torus orbits of symplectic leaves. They 
state that Rietsch has the same result in \cite[Proposition 7.2]{rietsch}, and also that she obtains the same result for closures of nonnegative cells in the corresponding nonnegative setting. Note that, at first sight, the parameterisations given by Rietsch and Goodearl-Yakimov 
are not the same. However, the  fact that they are indeed the same is covered in 
an appendix in a paper by He and Lam, \cite{hl}. \\

We now consider the special case of grassmannians; that is we assume that $G=SL_{n+1}(K)$ and $I=\{\alpha_1, \dots , \alpha_n\}\setminus \{\alpha_m\}$, so that $R_q[G/P_I] = \oqgmnk$. 
As we now know that containments between $\ch$-prime ideals in the quantum 
grassmannian correspond to membership of the closure of positroid 
cells in the totally nonnegative grassmannian by Theorem \ref{theorem-triple-correspondence}, we deduce from the above description the following result. \\

\begin{theorem} \label{theorem-containment} 
Assume that $q \in \k$ is transcendental over the prime field of $\k$. The following posets are isomorphic:
\begin{enumerate}
\item $\ch$-$\spec \oqgmnk$ (endowed with inclusion);
\item the set of torus orbits of symplectic leaves in the grassmannian  $SL_{n+1}(\mathbb{C}) / P_I$ (ordered by closure);
\item $S_{W,I}$, where $W$ is the symmetric group $W=S_{n+1}$ and $W_I$ is the subgroup generated by $s_1, \dots, s_{m-1},s_{m+1}, \dots , s_m$, and where the order $(w,v)< (w',v')$ in $S_{W,I}$ is defined by  
$$(w,v)< (w',v') \mbox{ if and only if there exists $z\in W_I$ such that }
w\geq w'z \mbox{ and } v\leq v'z.$$
\end{enumerate}
\end{theorem} 

Previously, only a bijection between $\ch$-$\spec \oqgmnk$ and $S_{W,I}$ had been established by Yakimov \cite{yak-qflag}.

The above result can be seen as a first step towards proving that the primitive spectrum of $\oqgmnk$ and the the space of symplective leaves in the corresponding grassmannian $SL_{n+1}(\mathbb{C}) / P_I$ are homeomorphic.

We note that, in order to prove the full Yakimov's conjecture in the quantum grassmannian case, additional work is needed: one still need to compare his parametrisation of $\ch$-primes by the set $S_{W,I}$ with our parametrisation of $\ch$-primes by Cauchon-Le diagrams. 





\subsection{Separating Ore sets} 

In \cite{bg}, Brown and Goodearl consider the problem of determining 
the Zariski topology of the prime spectrum of algebras $A$ equipped with 
a suitable rational action of a torus $\ch$. 
More precisely, they assume that they are dealing with a noetherian 
algebra over a field $\k$ such that $A$ satisfies the noncommutative nullstellensatz, that $\ch$ is a $\k$-torus acting rationally on $A$ by 
$\k$-algebra automorphisms and that $A$ has only finitely many 
$\ch$-prime ideals. The basic question is to decide, given two prime ideals 
$P,Q$, whether or not $Q\subseteq P$. If the stratification theorem 
applies and $P$ and $Q$ are in the same stratum then the problem is 
relatively easy to solve. The difficult part is when the two primes belong 
to different strata. In order to have a chance of progress one needs to be 
able to first decide this problem for $\ch$-prime ideals; that is, given 
$\ch$-prime ideals $P$ and $Q$, recognise when $Q\subseteq P$. In the case of quantum grassmannians, we have answered this question thanks to the notion of Grassmann necklace. 

In order to deal with the general situation, Brown and Goodearl introduce a notion that is subsequently named a {\em separating Ore set} in works by Casteels and Fryer, \cite{cf}, and Fryer and Yakimov, \cite{fy}. Explicit separating Ore sets have been constructed by Casteels and Fryer for quantum matrices \cite{cf}, and by Fryer and Yakimov for quantised coordinate rings of algebraic groups and quantum Schubert cells \cite{fy}. In this subsection, we show that Grassmann necklaces generate  separating Ore sets in the quantum grassmannian. 

As mentioned above, Casteels and Fryer, \cite{cf}  addressed this problem in the case of quantum matrices and our belief that Grassmann necklaces would generate separating Ore sets is a result of reading their paper. Our proof of the Ore condition for powers of a quantum Pl\"ucker coordinate  is via noncommutative dehomogenisation  from a result due to Skoda \cite{skoda} that any quantum minor in the algebra of quantum matrices generates an Ore set. \\

\begin{definition}\label{definition-separting-ore-set} 
{\rm Let $A$ be a $\k$-algebra that supports a rational action of  a $\k$-torus 
$\ch$ and let $P$ be an $\ch$-prime. A {\em separating Ore set} for $P$ in $\ch$-Spec$(A)$ is an Ore set $S_P$ with the following two 
properties: (i) $P\cap S_P=\emptyset$, (ii) if $Q\in\ch{\rm -Spec}(A)$ and $Q\not\subseteq P$ 
then $Q\cap S_P\neq\emptyset$. 
}\end{definition}

We begin by recalling some well-known facts about Ore sets. In order to 
avoid some technical details concerning denominator sets, we will assume throughout the discussion that we are dealing with rings $R$ that are 
(noncommutative) integral domains. 

Recall that a multiplicatively closed subset $S$ of nonzero elements
in a domain $R$ is said to satisfy the {\em right Ore condition} if
for each $a\in R$ and $s\in S$ there exist $t\in S$ and $r\in R$ such
that $at=sr$ (that is, $aS\cap sR\neq\emptyset$). If this equation
holds for a pair $(a,s)$ with $s \in S$, we say that the pair $(a,s)$ {\em satisfies
  the right Ore condition with respect to the set $S$} and that $(a,s)$ is a {\em right Ore pair (with respect to $S$)}. 
 If $(a,s)$  is a right Ore pair for each $a \in R$ then 
  we say that $s$ is a {\em right Ore element with respect to $S$}. 
  The left Ore condition is defined in a
  similar manner, and, if $S$ satisfies both the right and left Ore
  conditions then we say that $S$ is an {\em  Ore set} in $R$.\\

Our first aim in this section is to show that the set of powers of an
arbitrary quantum minor is an Ore set in the algebra of quantum
matrices. This result was first obtained by Skoda \cite{skoda}, but we
present a complete proof which may be more transparent than that
given in \cite{skoda}.

We begin with some preparatory results, some well-known and some taken from or inspired
by Skoda's work. As we will only
be dealing with domains, we assume throughout that $R$ is an integral
domain and that $S$ is a multiplicatively closed subset of nonzero
elements of $R$.

\begin{lemma} \label{lemma-products}
Suppose that $S$ is a multiplicatively closed set and that $S_0$ is a subset of 
$S$ such that each element of $S$ is a product of elements of $S_0$ (that is, $S$ is generated as a multiplicative set by $S_0$). 
Suppose that each  $s\in S_0$  is a right Ore element with respect to $S$. 
Then $S$ is a right Ore set. 
\end{lemma}

\begin{proof} 
Suppose  that $s_1,s_2$ are right Ore elements with respect to $S$. 
We will show  that $s_1s_2$ is a right Ore element with respect to $S$. 

Let $a\in R$. As $(a,s_1)$ is a right Ore pair, there exist $t_1\in S$ and 
$r_1\in R$ such that $at_1=s_1r_1$. As $(r_1, s_2)$ is a right Ore pair 
there exist $t_2\in S$ and $r_2\in R$ with $r_1t_2=s_2r_2$. Thus, 
\[
a(t_1t_2) = s_1r_1t_2=s_1s_2r_2, 
\]
which shows that 
$(a,s_1s_2)$ is a right Ore pair, because $t_1t_2\in S$; and so 
$s_1s_2$ is a right Ore element, as required. 

The claim of the lemma now follows by induction on the length of any element 
of $s\in S$ when represented as a product of elements of $S_0$. 
\end{proof} 

\begin{corollary} \label{corollary-oresets}
Let $S_1,\dots,S_n$ be right Ore sets in $R$ and let $S$ be the
multiplicatively closed set of all products of elements from
$S_1\cup\dots\cup S_n$. Then $S$ is a right Ore set.
\end{corollary} 

\begin{proof}
This follows immediately from the previous lemma.
\end{proof}

\begin{definition} For a multiplicatively closed set $S$ of the integral 
domain $R$, 
define $E(S)$ to be the subset of elements $e$ of $R$ 
such that $(e,s)$ is a right Ore pair with respect to $S$ for each $s\in S$; 
that is,  for each $s\in S$ there exists an element $s_1\in
S$ and $r_1\in R$ such that $es_1=sr_1$
\end{definition} 

Note that if we can prove that $E(S)=R$ then we have proved that $S$ is a right Ore set in 
$R$.

Lemma~\ref{lemma-products} above occurs as \cite[Lemma 3.1(iii)]{skoda}. 
As we need more from Skoda's lemma, we restate it in full here.

\begin{lemma} \label{lemma-check-ore} Suppose that $S$ is a multiplicatively 
closed subset of the integral domain $R$ and that 
$S\subseteq E(S)$. Then \\
(i) $E(S)$ is a subalgebra of $R$.\\
(ii) Suppose further that an element $f\in R$ satisfies the following property: 
for each $s\in S$ there exists $s_1\in S$ and $r_1\in R$ such that 
$fs_1-sr_1\in E(S)$. Then $f\in E(S)$. \\
(iii) Suppose that $S_0$ is  a subset of $S$ that generates $S$ as a multiplicative set
and that for each $s\in S_0$ and $r\in R$, there exists $s_1\in S$ and $r_1\in R$ such that 
$sr_1=rs_1$. Then $S$ is a right Ore set in $R$. 
\end{lemma} 

\begin{proof}See the proof of \cite[Lemma 3.1]{skoda}.
\end{proof}

Note that Part (i) of the previous lemma shows that we only need to
consider elements in a generating subset of $R$ when checking the Ore
condition for $S$.

\begin{lemma} \label{lemma-xij-ore}
Let $x_{kl}$ be a generator in $\oqmmnk$ and set
  $S:=\{x_{kl}^d \mid d\in \mn\}$. Then $S$ is an Ore set in $\oqmmnk$. 
\end{lemma}  

\begin{proof} 
It is obvious that $S\subseteq E(S)$; so Lemma~\ref{lemma-check-ore} is available for use. 

By Lemma~\ref{lemma-check-ore}(i), it is enough to check
  the right Ore condition for the pair $(a,s)=(x_{ij},x_{kl}^d)$ with $d \in \mn$.\\
  
If  $x_{ij}$ is weakly NE or weakly SW of $x_{kl}$ then the pair of variables
  commute up to a power of $q$ and so the Ore condition holds. Thus,
  we may assume that $x_{ij}$ is either strictly NW or strictly SE of
  $x_{kl}$. Assume that $x_{ij}$ is strictly NW of $x_{kl}$. Then, an easy induction left to the reader shows that 
\[
x_{ij}x_{kl}^{d+1} = x_{kl}^d \left( x_{ij}x_{kl} + (q^{2d+1}-q) x_{il}x_{kj} \right)
\]
for all $d \geq 1 $. This demonstrates the right Ore condition for the pair  $(x_{ij},x_{kl}^d)$, as required.

The case where  $x_{ij}$ is strictly SE of $x_{kl}$ follows by a similar calculation. 

Finally, the left Ore condition is checked in a similar way. 
\end{proof} 

\begin{remark} \label{rmk-xij-ore}
We note for later use that the proof of Lemma \ref{lemma-xij-ore} shows that for all $d\geq 1$, there exist $e \geq d$ and $r \in \oqmmnk$ such that 
\[
x_{ij}x_{kl}^{e} = x_{kl}^d r.
\]
\end{remark}

Recall that the quantum determinant $D_q$ of $\oqmnk$ is central in the noetherian domain $\oqmnk$. Hence we can form the localisation $\oqglnk :=\oqmnk [D_q^{-1}] $ of $\oqmnk$ at the multiplicative set generated by $D_q$. The resulting algebra $\oqglnk$ is the so-called {\em quantum general linear group}. This is a quantum group in the sense that it supports a Hopf algebra structure (see for instance \cite{Kassel}). In the following proof, we will make use of its antipode $S$ which we will apply to quantum minors of $\oqglnk$. Recall for later use that the 
formula for the antipode applied to a quantum minor is 
$S([I|J])=(-q)^{I-J}[\widetilde{J}|\widetilde{I}]D_q^{-1}$, where 
the exponent $I-J$ is the sum of the row indices in $I$ minus the sum of the column
indices in $J$ and $D_q$ is the quantum determinant, 
while $\widetilde{I}$ 
denotes the complement of $I$ in $\{1,\dots,n\}$,  see, for example, 
\cite[Lemma 4.1]{klr}.

\begin{theorem}\label{theorem-quantum-minor-ore} 
[Skoda, \cite{skoda}] 
Let $[K|L]$ be a quantum minor in $\oqmmnk$ and set $S:=\{[K|L]^t\mid t\in\mn\}$.
Then $S$ is an Ore set.
\end{theorem} 

\begin{proof} 
We may assume that $[K|L]$ is of size at least two, as the case of size one 
is dealt with in the previous lemma. 

As $S\subseteq E(S)$, Lemma~\ref{lemma-check-ore} is available for use. 

By Lemma~\ref{lemma-check-ore}(i), it is enough to check  the right Ore condition for the pair $(a,s)=(x_{ij},[K|L]^t)$ with $t \geq 1$.

There is a submatrix of generators of $\oqmmnk$ that contains these two elements and has at most one more row and column than $[K|L]$. We can perform all calculations within this subalgebra, and so, without loss of generality, we may assume that $[K|L]$ is an $(n-1)\times (n-1)$ quantum minor in $\oqmnk$. By inverting the quantum determinant we may work in $\oqglnk$ -- this makes available the antipode which is an anti-automorphism. Note that $n\geq 3$. 
  Our plan is to look at the images of the two elements under the antipode, establish the left Ore condition for two related elements and then apply the antipode once more to obtain the right Ore condition that we are seeking. 
 
 Let $c,d$ be such that $\{1,\dots,n\} = K\sqcup \{c\} =L\sqcup \{d\}$
and let $U,V$ be such that $\{1,\dots,n\} = \{i\}\sqcup U=\{j\}\sqcup V$. 
Note that $S(x_{dc})= (-q)^{d-c}[K|L]D_q^{-1}$ and that 
$S([V|U])=(-q)^{V-U}x_{ij}D_q^{-1}$.

By the previous result, the left Ore condition holds in $\oqmnk$ 
for the pair $([V|U], x_{dc}^t)$. 
Hence, there exist $s\in\mn$ and an element $r\in\oqmnk$ such that 
\[
rx_{dc}^t = x_{dc}^s[V|U].
\]
Observe that $s \geq t$ by Remark \ref{rmk-xij-ore}.

Note that the equation above is  a homogeneous equation, and so  $r$ will be a sum of products of 
$s-t+n-1$ variables multiplied by various scalars. Thus, $S(r)$ will be equal to 
$r_1D_q^{-(s-t+n-1)}$ for a suitable $r_1\in\oqmnk$. 
By applying the antipode to the above equation, we obtain
\[
\{(-q)^{d-c}[K|L]D_q^{-1}\}^t \cdot r_1D_q^{-(s-t+n-1)} = 
(-q)^{V-U}x_{ij}D_q^{-1}\{(-q)^{d-c}[K|L]D_q^{-1}\}^s
\]
so that 
\[
[K|L]^t.\{(-q)^\bullet r_1\} = x_{ij}[K|L]^sD_q^{n-2} .
\]
where $(-q)^\bullet $ is some power of $(-q)$ that we need not know explicitly. 
Note that $n-2> 0$; so all terms in this equation are in $\oqmnk$. 
Now $\langle D_q\rangle$ is  a completely prime ideal 
in $\oqmnk$ and $[K|L]$ is not in
this ideal. Hence, $r_1 = r_2D_q$ for some $r_2\in\oqmnk$. Thus, we
obtain
\[
[K|L]^t \cdot \{(-q)^\bullet r_2\} = x_{ij}[K|L]^sD_q^{n-3} .
\]
Continuing in this way, eventually we find an $r'\in\oqmnk$ such that 
\[
[K|L]^t \cdot \{(-q)^\bullet r'\} = x_{ij}[K|L]^s,
\]
and this confirms the Ore condition for the pair $(x_{ij},[K|L]^t)$. 
\end{proof} 

\begin{remark} \label{rmk-quantum-minor-ore} 
We note for later use that the proof of Theorem \ref{theorem-quantum-minor-ore}  shows that for all $t\geq 1$, there exist $s \geq t$ and $r \in \oqmmnk$ such that 
\[
[K|L]^t \cdot r = x_{ij}[K|L]^s.
\]
\end{remark}

We are now ready to prove that quantum Pl\"ucker coordinates generate Ore sets in $\oqgmnk$ by using the dehomogenisation isomorphism $\Psi$ from Theorem \ref{where it goes}. Indeed, $\Psi$ will allow us to transfer Skoda's result to quantum grassmannians. 

\begin{theorem} \label{theorem-plucker-ore}
Let $q$ be a nonzero element of a field $\k$. Let $[C]$ be a quantum Pl\"ucker coordinate in $\oqgmnk$ and let $S$ be the 
set of powers of $[C]$. Then $S$ is an Ore set in $\oqgmnk$. 
\end{theorem} 

\begin{proof}
If $[C]= [1,\dots,m]$ then $[C]$ is a normal element and so the result is 
well-known. Thus, we may assume that $[C]\neq [1,\dots,m]$. 
Let $[A]$ be any quantum Pl\"ucker coordinate. By Lemma~\ref{lemma-check-ore}
it is enough to show that $([A],[C]^t)$ is a right Ore pair with respect to $S$ for all $t \geq 1$; that is, 
we need to show that there exists $s\in\mn$ and $r\in\oqgmnk$ such that 
$[A][C]^s=[C]^tr$. 

We consider noncommutative dehomogenisation at $u:=[1,\dots,m]$, see Section \ref{subsection-dehom}. More precisely, we will make use of Theorem \ref{where it goes} in the case where $\gamma = u$. By Theorem \ref{where it goes} the elements $[A]u^{-1}$ and $[C]u^{-1}$ are sent to  quantum minors in $\oqmmnminusmk$ by the isomorphism $\Psi$, say $\Psi([A]u^{-1})=[I\mid J]$
is a $c\times c$ quantum minor and $\Psi([C]u^{-1})=[K\mid L]$ is a $d\times d$ quantum minor. Denote the generators of $\oqmmnminusmk$ by $x_{ij}$ and, in the next part of the argument, we consider the degree of elements in $\oqmmnminusmk$ with each  $x_{ij}$ having degree equal to one. By Skoda's work, Theorem~\ref{theorem-quantum-minor-ore} and Remark \ref{rmk-quantum-minor-ore}, there exists $s\in \mn$ with $s\geq t$ and an element $w\in \oqmmnminusmk$ such that 
\begin{eqnarray}\label{ore-for-minors} 
[I\mid J] \cdot [K\mid L]^s &=& [K\mid L]^t \cdot w
\end{eqnarray}  The degree of the term on the left hand side of this equation is 
$c +sd$ and $[K\mid L]$ has degree $d$; so $w$ is a linear combination of terms of $\oqmmnminusmk$ each 
having degree $e:=c+(s-t)d$. Consider one such term $w_1\dots w_e$, where 
each $w_i$ is some $x_{ij}$. 

By using the dehomogenisation isomorphism $\Psi$, 
we may write $w_i=\Psi([B_i]u^{-1})$ for some quantum Pl\"ucker coordinate $[B_i]$. 
Hence, $w_1\dots w_e = \psi(([B_1]u^{-1})\dots ([B_e]u^{-1}))$. Each $[B_i]$ 
$q$-commutes with $u$, so we may write 
$w_1\dots w_e =\Psi( \qdot[B_1]\dots[B_e]u^{-e})$. We may also rewrite the left hand side of Equation~\eqref{ore-for-minors} as $[I\mid J]. [K\mid L]^s = \Psi([A]u^{-1}) \cdot \Psi([C]u^{-1})^s = \qdot \Psi([A][C]^su^{-(s+1)})$. Hence, applying $\Psi^{-1}$ to Equation~\eqref{ore-for-minors}, we see that

\[
[A][C]^s u^{-(s+1)} = [C]^tz,
\]
where $z$ is a linear combination of terms of the form $[B_1]\dots[B_e]u^{-(e+t)}$.
Multiplying through by $u^{e+t}$, we obtain 
\[
[A][C]^s u^{e+t-s-1} = [C]^tz',
\]
where $z'$ is a linear combination of terms of the form $[B_1]\dots[B_e]$. It follows that $z'\in\oqgmnk$. Also, note that $e+t-s-1=c+(s-t)(d-1)-1\geq 0$ since $s-t \geq 0$. If $e+t-s-1=0$ then
this gives the required Ore condition. Otherwise, note that $u$ generates 
a prime ideal of $\oqgmnk$, and $[C]$ is not in this ideal; so we may rewrite 
$z'= u^{e+t-s-1}z''$ for some $z''\in\oqgmnk$. We then cancel  $u^{e+t-s-1}$
in this equation to obtain 
\[
[A][C]^s  = [C]^tz'',
\]
which is the required Ore condition. 
\end{proof} 

We are now in position to construct separating Ore sets for $\ch$-prime ideals of $\oqgmnk$.

\begin{theorem} 
Let $\k$ be a field and assume that $q \in \k^*$ is transcendental over the prime field of $\k$. 
Let $P$ be an $\ch$-prime ideal of $\oqgmnk$ with 
$\neck(P)=  ([I_1],\dots,[I_n])$. Let $N$ be the multiplicatively closed set 
generated by $[I_1],\dots,[I_n]$. Then $N$ is a separating Ore set for 
$P$. 
\end{theorem}

\begin{proof}\begin{full}
Suppose that $N$ is the multiplicatively closed set generated by 
$[I_1],\dots,[I_n]$. Then $N$ is an Ore set by 
Theorem~\ref{theorem-plucker-ore} and Corollary~\ref{corollary-oresets}.
Suppose that $Q$ is an $\ch$-prime ideal of $\oqgmnk$ with 
$Q\not\subseteq P$ and suppose that 
$\neck(Q)=([J_1],\dots,[J_n])$. 
Then $\neck(Q)\not\leq\neck(P)$ by Proposition~\ref{proposition-necklace-order} 
and so there exists $i\in\{1,\dots,n\}$ with $J_i\not\leq_i I_i$. This forces 
 $[I_i]\in Q$. However, $[I_i]\in N$, by definition. Thus, $Q\cap N  \neq\emptyset$. 
\end{full}\end{proof}

 \begin{corollary} 
 Let $\k$ be a field and assume that $q \in \k^*$ is transcendental over the prime field of $\k$. 
 Let $P$ be an $\ch$-prime ideal in $\oqgmnk$ with Grassmann necklace $N$. 
 Then $\left(\oqgmnk/P\right)[\widebar{N}^{-1}]$ is an $\ch$-simple ring, where $\widebar{N}$ is the image of $N$ in $\oqgmnk/P$.
 \end{corollary} 
\begin{proof}\begin{full}
 Set $R:=\oqgmnk/P$. Then $\widebar{N}$ generates an Ore set in $R$ and the localisation
 $R[\widebar{N}^{-1}]$ exists. As $P$ is an $\ch$-ideal, the action of $\ch$ on $\oqgmnk$ induces a natural action of 
 $\ch$ on $R$,  and as the elements in $N$ are $\ch$-eigenvectors, there is 
 an induced action of $\ch$ on  $R[\widebar{N}^{-1}]$. Suppose that 
  $R[\widebar{N}^{-1}]$ is not $\ch$-simple and choose a proper, nonzero $\ch$-prime ideal $I$ in  $R[\widebar{N}^{-1}]$. Then $I=(I\cap R)[\widebar{N}^{-1}]$, 
  and it follows that $0\neq I\cap R$. Let $Q$ be the $\ch$-prime ideal of 
  $\oqgmnk$ such that $P\subsetneq Q$ and $I\cap R= Q/P$. Then, 
  $Q\cap N\neq\emptyset$, by the previous Theorem. It follows that 
  $I=R[\widebar{N}^{-1}]$, which is a contradiction. Hence, 
  $R[\widebar{N}^{-1}]$ is $\ch$-simple, as required. 
  \end{full}\end{proof} 
  
  In order to pursue the Brown-Goodearl strategy (see \cite{bg}), the next step is to fully describe the centre of the $\ch$-simple ring  $\left(\oqgmnk/P\right)[\widebar{N}^{-1}]$. 
  




\subsection{Polynormality for $\ch$-prime ideals of $\oqgmnk$}

In this subsection, $\k$ is a field and $q \in \k^*$ is transcendental over the prime field of $\k$. 
 
Recall that a sequence of elements $u_1,\dots,u_n$ 
in a ring $R$ is a {\em polynormal sequence} if $u_1$ is a normal element of $R$ (that is, 
$u_1R=Ru_1$), and, for each $i=2,\dots,n$, the image of $u_i$ is a normal 
element of the factor ring $R/\ideal{u_1,\dots,u_{i-1}}$. 

In \cite{gl-winding}, Goodearl and the second author conjectured that the $\ch$-prime ideals in quantum matrix algebras  would have polynormal sequences of generators, and verified the conjecture for the $3\times 3$ case. In \cite{yak-polynormality}, Yakimov verified the conjecture in the general case. (In fact, Yakimov proved a polynormality result for the much wider class of quantum nilpotent algebras of type $U_{-}^w({\frak g})$, of which quantum matrix algebras are a special case.) Here, we establish the corresponding result for the quantum grassmannian.

\begin{theorem}
 Let $\k$ be a field and assume that $q \in \k^*$ is transcendental over the prime field of $\k$. Then each  
 $\ch$-prime ideal in $\oqgmnk$ is generated by a polynormal sequence of quantum Pl\"ucker coordinates. 
\end{theorem}
\begin{proof}
Let $P$ be an $\ch$-prime ideal in  $\oqgmnk$ and let $([I_1],\dots,[I_n])$ be the Grassmann necklace of $P$. By Theorem~\ref{theorem-union}, the set of quantum Pl\"ucker coordinates that are contained in $P$ is given by 
$
\Pi_P:=\bigcup_{i=1}^n \,\Pi_i, 
$
where $\Pi_i=\{[J] \mid J\not\geqqi I_i\}$.
As we are assuming that $q$ is transcendental, Theorem~\ref{theorem-full-generation} shows that $P$ is generated as a right (or left) ideal by $\Pi_P$. Therefore, it suffices to show that the elements in $\Pi_P$ can be listed in such a way that they form a polynormal sequence.

Now, $\oqgmnk$ is a quantum graded algebra with a straightening law with respect to $(\Pi_i,<_i)$, for each $i$, 
by Theorem~\ref{theorem-qasl-for-i-order}. In view of this fact, \cite[Proposition 1.2.2]{lr2} shows that the elements of $\Pi_i$ may be listed so that they form a polynormal sequence of generators of the ideal that they generate. Let $L_i$ be such a list for each $i$. Then 
the concatenated list $L_P:=(L_1,L_2,\dots,L_n)$ forms a polynormal sequence of generators for $P$. 
\end{proof}




\begin{minipage}{\textwidth}
\noindent S Launois\\
School of Mathematics, Statistics and Actuarial Science,\\
University of Kent\\
Canterbury, Kent, CT2 7FS,\\ UK\\[0.5ex]
email: S.Launois@kent.ac.uk \\

\noindent T H Lenagan\\
Maxwell Institute for Mathematical Sciences,\\
School of Mathematics, University of Edinburgh,\\
James Clerk Maxwell Building\\
The King's Buildings\\
Peter Guthrie Tait Road\\
Edinburgh EH9 3FD \\
           UK\\[0.5ex]
email: tom@maths.ed.ac.uk\\


\noindent B M Nolan\\
School of Mathematics, Statistics and Actuarial Science,\\
University of Kent\\
Canterbury, Kent, CT2 7FS,\\ UK\\[0.5ex]
email: bn62@kentforlife.net

\end{minipage}

\end{document}